\documentclass[12pt,a4paper]{amsart}
\usepackage{amssymb,stmaryrd}
\usepackage[a4paper,left=2.25cm,right=2.25cm,top=3cm,bottom=3cm,headsep=1cm]{geometry}

\usepackage[mathbf]{euler} 
\usepackage{mathtools,mathdots}
\usepackage[normalem]{ulem}
\usepackage{fancybox,kbordermatrix}

\newcommand\vf{\vec f}
\newcommand\vX{\vec u}
\newcommand\vY{\vec v}
\newcommand\vZ{\vec w}
\newcommand\vW{\vec x}
\newcommand\W{\mathcal W}
\renewcommand\emptyset{\varnothing}
\newcommand\loc{\mathrm{loc}}
\newcommand\killing[2]{\langle #1,#2\rangle}
\usepackage{tikz}\usetikzlibrary{fit,positioning,matrix,calc,decorations.markings,angles,decorations.pathmorphing,decorations.pathreplacing,arrows.meta}
\tikzset{mynode/.style={circle,draw=black,fill=black,inner sep=1.8pt,outer sep=0pt}}
\tikzset{edgelabel/.style={\mcol,inner sep=0pt}}
\tikzset{invlabel/.style={draw=black,text=black,circle,inner sep=0pt,minimum size=3mm}}
\tikzset{triv/.style={circle,fill=black,inner sep=0.2mm}}
\newcommand\tikzif[2][]{
\tikzifinpicture{#2}{\begin{tikzpicture}[#1]#2\end{tikzpicture}}
}
\tikzset{math mode/.style = {execute at begin node=$, execute at end node=$}}
\def\dr{red!80!black} 
\def\dg{green!80!black} 
\def\db{blue!80!black} 
\tikzset{d/.style={ultra thick}}
\tikzset{dr/.style={draw=\dr,d}}
\tikzset{dg/.style={draw=\dg,d}}
\tikzset{db/.style={draw=\db,d}} 
\def\mcol{}
\def\m{}
\tikzset{rt/.style={text=blue,execute at begin node=$\sf,execute at end node=$}}

\def\rr#1{r_{\mathsf{#1}}}
\tikzset{arrow/.style={postaction={decorate,decoration={markings,mark = at position #1 with {\arrow{Straight Barb[line width=0.2mm,length=1.2mm]}}}}},arrow/.default=0.5}
\tikzset{invarrow/.style={postaction={decorate,decoration={markings,mark = at position #1 with {\arrowreversed{Straight Barb[line width=0.2mm,length=1.2mm]}}}}},invarrow/.default=0.5}
\newlength\myshift
\newcommand\getshift{%
\pgfmathsetlength{\myshift}{0.8mm}
}
\def\noparenx#1{%
	\ifx\relax#1
	\else
	\if)#1%
	\else
	\if(#1%
	\else
	#1%
	\fi
	\fi
	\expandafter\noparenx
	\fi
}

\newcommand\rh[5][]{
\tikzif[baseline=0,scale=1.25]{
\getshift
\draw[thick,black] (0,0) coordinate (ad) -- node[edgelabel,left,xshift=\myshift] {$#2$} ++(-60:1) coordinate (ba) -- node[edgelabel,right,xshift=-\myshift] {$#3$} ++(60:1) coordinate (cb) -- node[edgelabel,right,xshift=-\myshift] {$#4$} ++(120:1) coordinate (dc) -- node[edgelabel,left,xshift=\myshift] {$#5$} cycle;
\ifx\&#1\&\else\node[invlabel] at (0.5,0) {$\ss#1$};\fi
}}

\newcommand\uptri[4][]{
  \tikzif[baseline=0.34cm,scale=1.25]{
    \getshift
    \draw[thick,black] (-0.5,0) -- node[edgelabel] (horiz) {$\vphantom{0}\smash{#3}$} ++(0:1) -- node[edgelabel,right,xshift=-\myshift] (NE) {$#4$} ++(120:1) -- node[edgelabel,left,xshift=\myshift] (NW) {$#2$} ++(240:1) -- cycle; 
    \ifx\&#1\&\else\node[invlabel] at (0,0.33) {$\ss#1$};\fi
  }}
\newcommand\downtri[4][]{
  \tikzif[baseline=-0.54cm,scale=1.25]{
    \getshift
    \begin{scope}[scale=-1]
      \draw[thick,black] (-0.5,0) -- node[edgelabel] (horiz) {$\vphantom{0}\smash{#3}$} ++(0:1) -- node[edgelabel,left,xshift=\myshift] (NE) {$#4$} ++(120:1) -- node[edgelabel,right,xshift=-\myshift] (NW) {$#2$} ++(240:1) -- cycle; 
      \ifx\&#1\&\else\node[invlabel] at (0,0.33) {$\ss#1$};\fi
    \end{scope}
  }}
\newcommand\Deltatri{{\tikz[baseline=0,scale=0.3]{\uptri{}{}{}}}}
\newcommand\nablatri{{\tikz[baseline=-0.25cm,scale=0.3]{\downtri{}{}{}}}}
\usepackage{hyperref}
\allowdisplaybreaks[1]
\renewcommand\ss{\scriptstyle}
\newcommand\sss{\scriptscriptstyle}
\newcommand\Uq{{\mathcal U}_q}
\newcommand\gqg{\Uq(\mathfrak{g}[z^\pm])}
\newcommand\aqg{\Uq(\mathfrak{a}_d[z^\pm])}
\newcommand\xqg{\Uq(\mathfrak{x}_{2d}[z^\pm])}

\newcommand\ZZ{{\mathbb Z}}
\newcommand\NN{{\mathbb N}}
\newcommand\QQ{{\mathbb Q}}
\newcommand\CC{{\mathbb C}}
\newcommand\PP{{\mathbb P}}
\newcommand\CP{{\mathbb{CP}}}
\newcommand\RR{{\mathbb R}}
\newcommand\Otimes\bigotimes
\theoremstyle{plain}
\newtheorem{thm}{Theorem}[section]

\newtheorem{prop}[thm]{Proposition}
\newtheorem{lem}[thm]{Lemma}
\newtheorem{cor}[thm]{Corollary}
\newtheorem{ex}[thm]{Example}
\theoremstyle{definition}

\theoremstyle{remark}
\newtheorem*{rmk*}{Remark}
\newcommand\defn[1]{\textbf{#1}}
\newcommand\into\hookrightarrow
\newcommand\onto\twoheadrightarrow
\newcommand\otno\twoheadleftarrow
\newcommand\dom\backslash
\newcommand\tensor\otimes
\newcommand\Tensor\bigotimes
\newcommand\iso\cong
\newcommand\lie[1]{{\mathfrak{#1}}}
\DeclareMathOperator\attr{attr}
\newcommand\calM{{\mathcal M}}
\newcommand\horizbar[1][]{\tikzif[baseline=0]{\draw[thick] (0,0) -- node{$#1$} (0.5,0); }}
\newcommand\NW[1][]{\tikzif[baseline=0]{\draw[thick] (0.5,0) -- node{$#1$} ++(120:0.5); }}
\newcommand\NE[1][]{\tikzif[baseline=0]{\draw[thick] (0,0) -- node{$#1$} ++(60:0.5); }}
\tikzset{gauged/.style={rectangle,rounded corners=2mm,draw,inner sep=0.5mm,minimum size=4mm}}
\tikzset{framed/.style={rectangle,draw,inner sep=0.5mm,minimum size=4mm}}
\tikzset{script math mode/.style = {execute at begin node=$\ss , execute at end node=$}}
\newcommand\St{\text{St}}
\newcommand\actson\circlearrowright
\newcommand\diag{{\mathop{diag}}}
\newcommand\rhocek{{\check\rho}}

\newcommand\rem[2][]{}
\long\def\junk#1{}
\title[Schubert puzzles and integrability II]{Schubert puzzles and integrability II: \\ multiplying motivic Segre classes}
\author{Allen Knutson}
\address{Allen Knutson, Cornell University, Ithaca, New York}
\email{allenk@math.cornell.edu}
\thanks{AK was supported by NSF grant 1953948.}
\author{Paul Zinn-Justin}
\address{Paul Zinn-Justin, School of Mathematics and Statistics, The University of Melbourne, 
Victoria 3010, Australia}
\email{pzinn@unimelb.edu.au}
\thanks{PZJ was supported by ARC grant FT150100232.}
\date{\today}
\begin{document}

\begin{abstract}
  In {\em Schubert Puzzles and Integrability I}
  we proved several ``puzzle rules'' for computing products
  of Schubert classes in $K$-theory (and sometimes equivariant $K$-theory)
  of $d$-step flag varieties. The principal tool was ``quantum integrability'',
  in several variants of the Yang--Baxter equation; this let us recognize the 
  Schubert structure constants as $q\to0$ limits of certain matrix entries
  in products of $R$- (and other) matrices 
  of $\gqg$-representations. In the present work we give
  direct cohomological interpretations of those same matrix entries 
  but at finite $q$:
  they compute products of ``motivic Segre classes'', closely related to
  $K$-theoretic Maulik--Okounkov stable classes living on the
  {\em cotangent bundles} of the flag varieties. Without $q\to0$, 
  we avoid some divergences that blocked fuller understanding of $d=3,4$.
  The puzzle computations are then explained (in cohomology only
  in this work, not $K$-theory) in
  terms of Lagrangian convolutions between Nakajima quiver varieties.
  More specifically, the conormal bundle to the diagonal inclusion 
  of a flag variety factors through a quiver variety that is not a 
  cotangent bundle, and it is on {\em that} intermediate quiver variety that
  the $R$-matrix calculation occurs.
\end{abstract}

\maketitle

{\small
  \tableofcontents
}

\section{Introduction}\label{sec:intro}
\subsection{What puzzles ``really'' compute}
Equivariant puzzles were introduced in \cite{KT} to study Grassmannian 
Schubert calculus, and were related to quantum integrable systems in 
\cite{artic46,artic68}. In our previous paper \cite{artic71} 
we extended this connection beyond $1$-step flag manifolds (i.e., Grassmannians)
to $2$- and $3$-step flag manifolds. This required recognizing the
puzzle labels (for a fixed number of steps $d$) as indexing the 
basis vectors in three representations of a quantized loop algebra $\gqg$
(one representation for each edge angle: \NE, \horizbar, \NW).

The state sums computed using the $R$-matrices of those representations then
came with an extra parameter $q$ not appearing in the Schubert calculus
structure constants; to be rid of it, we took $q\to 0$ or $q\to \infty$.
This point was subtle enough that at $d=3$, we were unable to take this
limit without first losing equivariance, and so we only discovered (and
proved) in this way a formula for {\em non}equivariant $K$-theoretic Schubert
calculus on $3$-step flag manifolds. 

The purpose of the present paper is to give a cohomological interpretation 
of these puzzle state sums at general $q$, not just as $q^\pm \to 0$
(and in particular, to recover equivariance in $3$-step, and extend
to $4$-step).
They are again structure constants for multiplication of a certain 
ring-with-basis; the ring is no longer cohomology (or equivariant, or
$K$-cohomology) of a $d$-step flag manifold, but rather of its
cotangent bundle\footnote{Of course the manifold and its cotangent 
  bundle are equivariantly homotopic, but the extra parameter is only
  geometrically natural on the cotangent bundle.},
and properly speaking the basis elements live only
in an equivariant localization (we must formally invert the class of 
the zero section in the cotangent bundle).

These cotangent bundle calculations can to some extent be interpreted
on the base; our best result in this direction is theorem \ref{thm:EP},
a puzzle formula for the Euler characteristic of the intersection
of three generically situated Bruhat cells. (In the traditional setting,
one only computes this number when the intersection is $0$-dimensional.)

\subsection{The ring with basis}\label{ssec:ringwbasis}

In \cite{MO-qg, Ok-K} are defined ``stable bases'' of the 
cohomology and $K$-theory of ``symplectic resolutions'', a class of spaces
that includes cotangent bundles to flag manifolds. These bases
depend on a choice of symplectic circle action with isolated fixed points,
and are closely analogous to Schubert bases of cohomology of flag
manifolds. In both cases, one considers the Bia\l ynicki--Birula
attracting sets of the fixed points.  In the Schubert situation the
classes are those of the closures of the attracting sets.
In the symplectic-resolution situation, {\em if} an attracting set is
closed -- which is automatic when the resolution is affine
enough; see lemma \ref{lem:attrclosed} --
then the associated ``stable class'' in cohomology is indeed that of the 
attracting set. 
\junk{
For a simple but very important example, consider the
  space $G/T$ (not $G/B$), a typical $G$-orbit on $\lie{g}^*$ and hence
  symplectic.  The $T$-fixed points are exactly of the form $N(T)/T$,
  and for a regular dominant coweight the attracting sets are
  of the form $BwT/T$, $w\in W := N(T)/T$. 
  Each one is closed -- obviously $BeT/T = B/T$ is closed since
  $B\leq G$ is closed, and the orbits $BwT/T$ are permuted by the
  {\em right} action of $W$. 
}

On a general symplectic resolution $M$ an attracting set
(while automatically Lagrangian) need not be closed.
(This unfortunate behavior occurs in the case of most interest 
in this paper, the cotangent bundle $T^*(P_-\dom G)$ to a flag manifold
$P_-\dom G$.)
However, symplectic resolutions always have $T$-equivariant
deformations $\widetilde M$ to affine varieties $M_a$ on which the
attracting sets {\em are} closed. (For example, $T^*(P_-\dom G)$ has a 
Grothendieck--Springer deformation to the affine variety $L\dom G$, 
where $L$ is a Levi subgroup of the parabolic $P_-$.)

Using such a deformation, we can define the stable class
{\em in cohomology} associated to a 
fixed point as follows. We deform the symplectic resolution $M$ to affine $M_a$,
obtain the (closed, Lagrangian) attracting cycle $C_a$ inside $M_a$
and follow that subvariety $C_a$ back through the
degeneration, to $C \subseteq M$. 

In the $H^*_{T\times \CC^\times}(T^* (P_-\dom G)) \iso H^*_T(P_-\dom G)[\hbar]$
case this $C$ is a union of conormal varieties to 
Schubert varieties (with complicated multiplicities).
This is already sufficient to show that the Schubert
classes arise as the $\hbar$-leading terms in the stable classes,
where $\hbar$ is the weight of the dilation action on the cotangent fibers.
In particular, if we invert $\hbar$ then the stable classes
are a basis over $H^*_T[\hbar^\pm]$.
(The thesis \cite{Voula} concerns the product structure in this basis,
but only when $P_-\backslash G$ is projective space.)

In $K$-theory, using the classes of the structure sheaves of the
attracting sets (or their degeneration from affine) doesn't lead to 
an especially good basis; it seems that one wants the classes of some
other sheaves. However, there is no unique limit when degenerating sheaves,
so we can't use the above limiting trick to define
``$K$-theoretic stable classes''.
Instead we give in proposition \ref{prop:charact}
a recurrence relation on the equivariant $K$-theoretic
stable classes $\St^\lambda$, closely related to ones in
\cite{SZZ-Kstable,aluffi2020motivic}.
\junk{PZJ: don't think the paragraph above [actually the redundant
  junked one in next subsection] is correct.
  AK: No? I thought this was a right-operator (theirs, necessitating $G/B$
  not $G/P$) vs. left-operator (ours) issue.}

Let $i:P_-\backslash G \to T^*(P_-\backslash G)$ be the zero section, 
and $[P_-\backslash G]$ the class of its image. 
Since for all fixed points $\mu$ the point restriction 
$[P_-\backslash G]|_\mu$ is nonzero, we get a well-defined element 
$[P_-\backslash G]^{-1}$ in the appropriately localized equivariant $K$-theory 
$K^{\loc}_{T\times \CC^\times}(T^*(P_-\backslash G))
$.
Finally, following the terminology of \cite[footnote 1]{aluffi2020motivic}
we define for now the \defn{motivic Segre class} $S^\lambda$ 
as the ratio $\St^\lambda \big/ [P_-\backslash G]$ in this localization.
(An alternate, more explicit, definition is given in \S \ref{sec:stable}.)

\subsection{Our main results}\label{ssec:mainres}
\junk{
  We {\em define}\/ the motivic Segre classes using a recursion
  based on the standard trigonometrix $R$-matrix for $A_d$'s 
  defining representation. To demonstrate their connection to the 
  $K$-theoretic stable basis defined in the literature,
  we show our recursion is equivalent to a different recursion used in 
  the characterization in \cite{SZZ-Kstable}. 
}
For $d\leq 4$, we define a second group $X_{2d}$, three representations 
$V_1(z_1),V_2(z_2),V_3(z_3)$ of its quantized loop algebra $\xqg$,
and intertwiners 
\newcommand\cR{{\check R}}
\[
\cR :\ V_1(z_1) \tensor V_2(z_2) \to V_2(z_2) \tensor V_1(z_1), \qquad
U :\ V_1(q^{h_d/3}z) \tensor V_2(q^{-h_d/3} z) \to V_3(z)
\]
($h_d$ being the dual Coxeter number of $X_{2d}$)
that satisfy the Yang--Baxter and bootstrap equations exactly as in 
\cite[\S 3]{artic71}. 
\junk{(To efficiently check those equations, in \S ?? we 
  make use of a braid group action on quantized loop algebras.)}
Puzzles provide a way to compute the matrix entries in a massive
product of $R$s and $U$s, and via essentially the same proof
as in \cite[\S 3]{artic71} we show that puzzles compute the
product of the motivic Segre classes. (There are some new puzzle pieces 
required at finite $q$, that were suppressed as $q\to\infty$.)

One key difference between the $d=4$ case, vs. the $d=1,2,3$ cases
we studied in \cite{artic71}, is that the representations $V_a(z_i)$
are not minuscule -- rather, each is the adjoint-plus-trivial 
representation of $\lie{e}_8$, bearing a $9$-dimensional central weight space.
To write down $\cR,U$ as matrices requires picking bases of this
weight space and its dual. (These choices were unique up to a
simple gauge equivalence, in the $d=1,2,3$ cases.)
We do this in \S \ref{ssec:d4}.

The close relation between SSM classes and Euler characteristics
gives us a corollary that can be stated completely within the Grassmannian:
a puzzle formula (in \S\ref{ssec:SSM}) for the topological Euler
characteristic of the intersection of three generically translated
Bruhat cells. In particular, this gives a positivity result for $d\leq 3$
for these Euler characteristics (times the sign familiar from $K$-theory)
for which there was no known geometric proof, and based on it
we conjectured that the positivity holds for triple intersections of
Bruhat cells in general $G/P$. Since this paper was written, that
conjecture has been proven, in \cite{SSW}.

\junk{
 1) we have an algebraic way to define a bunch of classes in K_T, which we then divide by [zero section].
 2) we can show that puzzles can compute the structure constants in that basis
 3) for d<4 we can take q-limits of those to get the K_T and dual K_T bases on the zero section
 4) for d=1 we have a geometric understanding of the partial puzzles.
 Is there anything else this paper is about, that the intro should address? Of course (1),(2) involve a whole lotta algebra, e.g. the braid action giving the three embeddings.
}

Define a \defn{positivity notion} in Schubert calculus
to be a submonoid $M$ (under $+$) of the coefficient
ring (either the cohomology of a point, or a localization thereof) s.t.
$M \cap -M = 0$.
The principal consequence is that a sum of elements of $M$ is zero
only if each term is zero.  Call a formula \defn{manifestly $M$-positive}
if it is a sum of terms, each in $M\setminus 0$.

We show that our puzzle formulae for the
coefficients in the product of motivic Segre classes
are manifestly $M$-positive for some $M$ to be specified,
for $d=1,2$ equivariantly, and $d=1,2,3$ nonequivariantly.

\subsection{Plan of the paper}
In \S\ref{sec:stable} we define the motivic Segre classes on type $A$
partial flag manifolds. In \S\ref{sec:puzzle} we state the
main formula (a bit schematically, with detailed fugacities in
\S\ref{sec:exd1} and appendix \ref{app:Rd1};
see also \S\ref{ssec:d123}-\ref{ssec:d4} for the $d=2,3,4$ cases
in cohomology. \rem{Where is best to   send people who want to compute
  in $K$-theory?}
In particular, in \S\ref{ssec:limits} we collect the
various limiting versions of our main theorem, from $K$ to $H$,
from motivic Segre classes to Schubert classes, and from equivariant
to nonequivariant. In \S\ref{sec:exd1} we provide full details on
the $d=1$ (Grassmannian) case, and give in \S\ref{ssec:loop}
a loop-model interpretation of the nonequivariant puzzles;
in this model we can sum over fewer puzzles, at the cost of
measuring global features of each (the number of loops).
In \S\ref{sec:coho} we pass from $K$-theory to ordinary cohomology,
where we can more reasonably explore the $d=3,4$ cases.
This is also where we give a puzzle formula for the Euler characteristic
of the intersection of three Bruhat cells.
We discuss positivity of our puzzle rules in \S\ref{sec:positivity}.

In the remaining two sections, we give a retrodiction of our results
(at least in cohomology) using quiver varieties. In \S\ref{sec:Nakajima}
we recall the definitions and results we need about them,
and give a new result (proposition \ref{prop:reflect}) about recognizing
quiver varieties as cotangent bundles of partial flag varieties.
This is also where we recognize the type $A_d$ $R$-matrix as sitting
inside the type $X_{2d}$ $R$-matrix (lemma \ref{lem:ressinglev2}).
In \S\ref{sec:geominterp} we sneak up on the $X_{2d}$-based puzzles
through several steps: replace the diagonal inclusion $M\to M\times M$
(whose pullback defines the multiplication on cohomology)
by the conormal bundle of its graph (lemma \ref{lem:zerosection}),
factor that correspondence through an $X_{2d}$ quiver variety
(proposition \ref{prop:composite}), obtain the SSM classes using
stable envelopes (recalled in \S\ref{ssec:envelopes}), and show that
the failure of the resulting Lagrangian-correspondence squares to commute
can be computed using puzzles (\S \ref{sssec:following}).

Appendix~\ref{app:qg} gives the presentation of the quantized affine
algebras and their representations that we need. Appendix~\ref{app:Rd1}
computes the full details of these representations, for $d=1$.
Appendix~\ref{app:exd4} provides examples of $d=4$ puzzles, whereas Appendix~\ref{app:scal}
is a table of scalar products that is useful for computing $H_T$-fugacities of $d\le 3$ puzzles.

\junk{somewhere, general blah blah; $T^*(P_-\backslash G) \overset p\to P_-\backslash G$, 
torus action $\hat T=T\times \CC^\times$, 
fixed points indexed by single-number strings.
with the weird convention, meant to reconcile everyone, that
$\sigma\lambda=\sigma(\lambda_1,\ldots,\lambda_n):=(\lambda_{\sigma^{-1}(1)},
\ldots,\lambda_{\sigma^{-1}(n)})$, and we identify $\sigma$ and $\sigma\omega$,
where $\omega$ is the unique weakly increasing string with the right content
(define similarly $\alpha$ as the unique weakly decreasing string).
Bruhat order ($\alpha$ smallest, $\omega$ largest).
$K$-theory.
Most of the time, we shall work in {\em localized}\/ $K$-theory, i.e.,
the base ring is the field $K_{\hat T}(pt)[(1-w)^{-1},w\in \text{weights}]$.
careful that even that's too big, because of fusion: explicitly should be something like
making $1-q^2z_i/z_j$ invertible where $i,j$ are in the same sequence $\{1,\ldots,n\}$ or $\{n+1,\ldots,2n\}$,
(and possibly $1-z_i/z_j$ for all $i,j$ as well?).
also for the rep theory part we probably want to tensor with $\CC$?
restriction to fixed points map an isomorphism after localization.
let's just define the scaling parameter to be $q^{-2}$, mentioning that
there's an arbitrariness in the sign of $q$.}

In our next paper \cite{artic83} in this series we will exploit some
other quiver varieties to solve some additional families of Schubert
calculus problems.

Finally, we mention that all results of this paper were carefully
checked by computer with the help of {\it Macaulay2} \cite{M2}.
In particular, the puzzle rules of \cite{artic71} and of the present
paper are implemented in the package {\it {C}otangent{S}chubert} \cite{artic82}.

\subsection*{Acknowledgments} We thank Iva Halacheva, David Hernandez, Leonardo
Mihalcea, and Hiraku Nakajima for discussions on many related topics.

\section{Stable and motivic Segre classes}\label{sec:stable}
\subsection{Geometric setup}\label{sec:geom}
\subsubsection{Cotangent bundles of flag varieties}
Let $G=GL_n(\CC)$, $B_\pm$ denote the upper/lower triangular matrices,
with intersection $T$ the diagonal matrices, and let $P_- \geq B_-$ be
a parabolic subgroup with Levi factor $\prod_{i=0}^d GL_{p_i}(\CC)$.
The main actors of this paper are the 
coset space $P_-\dom G$
and its cotangent bundle $T^*(P_-\dom G) \overset p\to P_-\backslash G$.
The torus $T$ naturally acts on $P_-\dom G$,
but the cotangent bundle $T^*(P_-\dom G)$ has an extra
circle action by scaling of the fiber, resulting in an action of $\hat
T:=T\times \CC^\times$.  The $T$-fixed points in $P_-\dom G$ (or $\hat
T$-fixed points in $T^*(P_-\dom G)$) are indexed by elements of
$W_P\dom W$, where $W=N(T)/T\cong \mathcal S_n$ and $W_P=W\cap P_-
\cong \prod_{i=0}^d \mathcal S_{p_i} \dom \mathcal S_n$. We denote by
$\alpha$ (resp.\ $\omega$) the smallest (resp.\ largest) element in
$W/W_P$ w.r.t.\ the Bruhat order. As in \cite{artic71}, we identify
elements of $W_P\dom W$ and strings of $n$ letters in the alphabet
$\{0,\ldots,d\}$, where $W$ acts on strings by permuting their labels,
and $W_P$ is the stabilizer of $\omega$, which is identified with the
unique weakly increasing string $0^{p_0},\ldots,d^{p_d}$.

\subsubsection{Schubert cycles}
Schubert cycles are not as central as in \cite{artic71}, but we still need the following definitions.
Schubert cells $X^\sigma_o := P_- \dom P_- \sigma B_+$ are $B_+$-orbits of fixed points $\sigma \in W_P\dom W$.
Their closures are Schubert varieties $X^\sigma$. 
The codimension of $X^\sigma$ is given by the number of inversions in the string corresponding to $\sigma$.

\subsubsection{$K$-theory}\label{sssec:K}
In what follows, we shall consider the equivariant $K$-theory ring $K_{\hat T}(T^*(P_-\backslash G))$.
It can be described explicitly as follows. 
Define the restriction map
\begin{align*}
\vert_\sigma:\ 
\ZZ[x_1^\pm,\ldots,x_n^\pm,z_1^\pm,\ldots,z_n^\pm,t^\pm]^{W_P} &\to K_{\hat T}(pt)
\cong \ZZ[z_1^\pm,\ldots,z_n^\pm,t^\pm] \\
f &\mapsto f\vert_\sigma := f(z_{\sigma^{-1}(1)},\ldots,z_{\sigma^{-1}(n)},z_1,\ldots,z_n,t)
\end{align*}
where the $W_P$ permutes the $x$ Laurent variables,
and $\sigma$ runs over $W_P\dom W$.
The $z_1,\ldots,z_n,t$ are the equivariant parameters associated to $\hat T=T\times \CC^\times$.\footnote{The $T$-action is right multiplication $x\mapsto x\,\text{diag}(z_1,\ldots,z_n)^{-1}$; this is the opposite choice that is made in \cite{artic71}.}

To present $K_{\hat T}(T^*(P_- \dom GL_n)) \iso K_{\hat T}(P_- \dom GL_n)$,
we use $GL_n \into Mat_n$ and consider the maps
$$
\ZZ[x_1^\pm,\ldots,x_n^\pm,z_1^\pm,\ldots,z_n^\pm,t^\pm]^{W_P} \iso
K_{P_- \times \hat T}(Mat_n)
\onto
K_{P_- \times \hat T}(GL_n)
\iso
K_{\hat T}(P_- \dom GL_n)
\into \bigoplus_{W_P\dom W} K_{\hat T}(pt)
$$
The second map is surjective because (as in \cite[verse]{KM-Schubert}) each
Schubert class in the target can be lifted to the class of a matrix
Schubert variety, and the final map is injective because
its kernel is torsion \cite[proposition 2.1]{Segal}, 
but $K_{\hat T}(P_- \dom GL_n)$ is a free $K_{\hat T}$-module
(on the Schubert basis). 

Consequently
\[
 K_{\hat T}(T^*(P_-\dom G)) \ \iso \ 
\ZZ[x_1^\pm,\ldots,x_n^\pm,z_1^\pm,\ldots,z_n^\pm,t^\pm]^{W_P} \bigg/
\bigcap_{\sigma\in W_P\dom W} \ker\left(\vert_\sigma\right) 
\]
where the $x_i$ are the Chern roots of the duals of the tautological bundles.
Thus to describe a class, it suffices to give a Laurent polynomial 
and check its symmetry in the $x$s.

\junk{Though the restriction map
$K_{\hat T}(T^*(P_-\backslash G)) \to K_{\hat T}(pt)^{W_P\dom W}$ is
only injective, it becomes an isomorphism after localization.
\rem{is that obvious? AK: it holds for $T$ acting on any space. PZJ: even with the limited localization above?}}


\subsection{Localization}\label{ssec:loc}
In this paper we need to extend the base ring $K_{\hat T}(pt)$.

A first (minor) modification is that we need a
square root of $t$, that is a variable $q$ such that $t=q^{-2}$. It can be formally defined by introducing a double cover
of the circle scaling the fiber; {\em in the rest of this paper, this is implicitly
done.}
Note that there is an arbitrariness in the sign of $q$ which will reappear in \S\ref{sec:coho}.

More importantly, most of the time, we shall require localization.  
When we deal with cotangent bundles of flag varieties, as in the current section,
localization corresponds, in the parametrization of \S\ref{sssec:K},
to inverting $1-q^{2}z_i/z_j$ for any $i,j$.
Note that $1-w$ is invertible, where
$w$ runs over weights at fixed points in the normal directions to the
base $P_-\backslash G$, which is all that matters for now. (When
we work with more general quiver varieties in \S\ref{sec:puzzle} and
\S\ref{sec:Nakajima}--\ref{sec:geominterp},
we shall extend slightly the localization, in a way that is compatible
with the one presented here).

Finally, to conform to the representation-theory literature, and for simplicity of notation, we tensor with $\CC(q)$
(or one could use $\QQ(q)$, as in \cite[\S9]{CP-book}).
In fact, some of the localisation above can be dispensed with, and in \S\ref{sec:coho} (and only there), we will reconsider it.

We denote the resulting base ring $K_{\hat T}^{\loc}(pt)$, and
$K_{\hat T}^{\loc}(X) := K_{\hat T}(X) \otimes_{K_{\hat T}(pt)} K_{\hat T}^{\loc}(pt)$.

\subsection{The single-number $R$-matrix}
We temporarily set aside the geometry, and introduce $R$-matrices.
Consider the vector space $V^A=\CC^{d+1}=\left<e_0,\ldots,e_d\right>$, 
as well as the evaluation representation $V^A(z)$ of $\aqg$,
where the evaluation parameter $z$ is also called the spectral parameter, see Appendix~\ref{app:qgeval}.
A word of caution is needed on the somewhat misleading notation $V^A(z)$:
we consider $q$ and $z$ as formal parameters,
so that as a vector space it is $V^A \otimes \CC(q)[z^\pm]$. 
We take tensor products of
such representations over $\CC(q)$, hence the asymmetric notation in $q$ and $z$.
Furthermore, we localize as in \S\ref{ssec:loc},
i.e., in the tensor product $V^A(z_1)\otimes\cdots\otimes V^A(z_n)$,
we allow ourselves to invert $1-q^{2} z_i/z_j$, where $z_i$ and $z_j$ are any two spectral parameters.
\junk{This localization is implicit in the rest of this section.}

We use the same diagrammatic language as in \cite[\S 3]{artic71}, 
and briefly recall it here.
Our diagrams consists of graphs embedded in the plane, whose edges are labeled. In this section
the set of labels is $\{0,1,\ldots,d\}$. The convention is that the labels of
external edges are fixed, but those of internal ones are summed over. To each edge are also attached some fixed data: a spectral parameter and
an orientation (the ``direction of time'').
This way, an edge labeled $i$ with a spectral parameter $z$ corresponds
to the standard basis element $e_i\in V^A(z)$. Juxtaposition of edges
corresponds to taking a tensor product, where the ordering is left to right
if they are oriented downwards.

The $R$-matrix is represented by the $4$-valent vertex
\[
\check R^{kl}_{ij}(z',z'')=
\tikz[baseline=0,xscale=0.5]
{
\draw[invarrow=0.75,d] (-1,-1) node[below] {$\m i$} -- node[pos=0.75,right] {$\ss z''$} (1,1) node[above] {$\m l$};
\draw[invarrow=0.75,d] (1,-1) node[below] {$\m j$} -- node[pos=0.75,left] {$\ss z'$} (-1,1) node[above] {$\m k$};
}
\]
Its entries depend on the ratio $z''/z'$ of parameters attached to the two lines crossing,
and have the following explicit expression: (see \cite{jimbo-Rmat} and references therein)
\begin{equation}\label{eq:Rsingle}
\check R_{ij}^{kl}(z',z'')=
\check R_{ij}^{kl}(z''/z')=
\frac{1}{\textstyle 1-q^2 z''/z'}
\begin{cases}
 1-q^2 z''/z'& i=j=k=l
\\
 q(1-z''/z') & i=l\ne j=k
\\
 1-q^2 & i=k<j=l
\\
 (1-q^2)z''/z' & i=k>j=l
\\
0 & \text{otherwise}
\end{cases}
\junk{
=
\begin{cases}
1& i=j=k=l
\\
\frac{\textstyle q(1-z''/z')}{\textstyle 1-q^2 z''/z'} & i=l\ne j=k
\\
\frac{\textstyle 1-q^2}{\textstyle 1-q^2 z''/z'} & i=k<j=l
\\
\frac{\textstyle (1-q^2)z''/z'}{\textstyle 1-q^2 z''/z'} & i=k>j=l
\\
0 & \text{else}
\end{cases}
}
\end{equation}
\junk{AK: I pulled the denominator out in the above, for readability;
  the original formulation is in a junk clause}
\rem[gray]{I reverted to the normal sign convention for $q$.
$q$ is now related to the scaling parameter $t$ by
$q=t^{-1/2}$ (where $t^{-1}$ is the $q$ of \cite{SZZ-Kstable} or the $\hbar$ of Okounkov;
also, in their $T^*\PP^1$ example, $e^\alpha=z_2/z_1$.}

As an operator from $V^A(z')\otimes V^A(z'')$ to $V^A(z'')\otimes V^A(z')$,
$\check R(z''/z')$ is an $\aqg$-intertwiner. This, combined with
the normalization $\check R_{ii}^{ii}=1$, implies the following well-known
properties:
\begin{prop}\label{prop:single}
\begin{itemize}
\item Yang--Baxter equation:
\begin{equation}\label{eq:ybe0}
\begin{tikzpicture}[baseline=-3pt,y=2cm]
\draw[d,arrow=0.1,arrow=0.4,arrow=0.7,rounded corners=4mm] (-0.5,0.5) -- (0.75,0) -- (1.5,-0.5) (0.5,0.5) -- (0.2,0) -- (0.5,-0.5) (1.5,0.5) -- (0.75,0) -- (-0.5,-0.5);
\end{tikzpicture}
=
\begin{tikzpicture}[baseline=-3pt,y=2cm]
\draw[d,arrow=0.1,arrow=0.4,arrow=0.7,rounded corners=4mm] (-0.5,0.5) -- (0.25,0) -- (1.5,-0.5) (0.5,0.5) -- (0.8,0) -- (0.5,-0.5) (1.5,0.5) -- (0.25,0) -- (-0.5,-0.5);
\end{tikzpicture}
\end{equation}
\item Unitarity equation:
\begin{equation}\label{eq:unit0}
\begin{tikzpicture}[baseline=-3pt]
\draw[arrow=0.07,arrow=0.57,rounded corners,d] (-0.5,1) -- (0.5,0) -- (-0.5,-1) (0.5,1) -- (-0.5,0) -- (0.5,-1);
\end{tikzpicture}
=
\begin{tikzpicture}[baseline=-3pt]
\draw[arrow=0.1,arrow=0.6,rounded corners,d] (-0.5,1) -- (-0.5,-1) (0.5,1) -- (0.5,-1);
\end{tikzpicture}
\end{equation}
\item Value at equal spectral parameters:
\begin{equation}\label{eq:equal0}
\begin{tikzpicture}[baseline=-3pt,yscale=1.5]
\draw[invarrow=0.3,d] (-0.5,-0.5) -- node[left,pos=0.3] {$\ss z$} (0.5,0.5);
\draw[invarrow=0.3,d] (0.5,-0.5) -- node[right,pos=0.3] {$\ss z$} (-0.5,0.5);
\end{tikzpicture}
=
\begin{tikzpicture}[baseline=-3pt,yscale=1.5]
\draw[invarrow=0.3,rounded corners=4mm,d] (-0.5,-0.5) -- node[left] {$\ss z$} (0,0) -- (-0.5,0.5);
\draw[invarrow=0.3,rounded corners=4mm,d] (0.5,-0.5) -- node[right] {$\ss z$} (0,0) -- (0.5,0.5);
\end{tikzpicture}
\end{equation}
\end{itemize}
\end{prop}

\subsection{Motivic Segre classes}
\junk{AK added the following paragraph because the rectangle stuff largely
  comes out of nowhere, at least for people like me who don't understand the
  on-shell/off-shell stuff}
One of the standard ways to get a hold of equivariant Schubert classes
on type $A$ flag manifolds is with {\em double Schubert polynomials,}
which were interpreted in \cite{KM-Schubert} as arising from
$$
\begin{array}{ccccl}
  H^*_T(B_-\dom GL_n)
  &\iso& H^*_{B_-\times T}(B_-\dom GL_n) &\otno& H^*_{B_-\times T}(M_n(\CC))
   \iso  \ZZ[x_1,\ldots,x_n,y_1,\ldots,y_n] \\
\rotatebox{90}{$\in$} &&\rotatebox{90}{$\in$}&& \qquad\rotatebox{90}{$\in$} \\
  {} [B_-\dom \overline{B_- w B_+}] &\mapsfrom&
  [\overline{B_- w B_+} \subseteq GL_n] &\mapsfrom&
  [\overline{B_- w B_+} \subseteq M_n(\CC)] 
\end{array}
$$
Meanwhile, under the identification
$M_n(\CC) \iso B^{2n}_-\dom
\left(B^{2n}_- (1\ n+1)(2\ n+2)\cdots (n\ 2n) B^{2n}_+\right)$
(essentially the $A_2$ case of \cite{Zelevinskii85}) we can realize
$\left[ \overline{B_- w B_+} \subseteq M_n(\CC) \right]$ as the restriction
of the equivariant Schubert class $S_{w \oplus Id_n}$ (on the flag manifold
in $\CC^{2n}$) to the point
$(1\ n+1)(2\ n+2)\cdots (n\ 2n)$. We now pursue a parallel story to obtain
the classes $S^\lambda$, where the permutation $(1\ n+1)(2\ n+2)\cdots (n\ 2n)$
will make an appearance as its square wiring diagram. 

Given a single-number string $\lambda$ of length $n$, we define
\begin{equation}\label{eq:defS}
S^\lambda
:=
\begin{tikzpicture}[baseline=1.8cm,scale=0.75]
\foreach\x/\t in {1/z_1,2/z_2,3/\cdots,4/z_n}
\draw[d,invarrow=0.9] (\x,0) node[below] {$\m d$} -- node[right=-1mm,pos=0.1] {$\ss\t$} ++(0,5);
\foreach\y/\t in {1/x_1,2/x_2,3/\vdots,4/x_n}
\draw[d,invarrow=0.9] (0,5-\y) -- node[above=-1mm,pos=0.1] {$\ss\t$} ++(5,0) node[right] {$\m d$};
\draw[decorate,decoration=brace] (-0.3,1) -- node[left] {$\m\omega$} (-0.3,4);
\draw[decorate,decoration=brace] (1,5.3) -- node[above] {$\m\lambda$} (4,5.3);
\end{tikzpicture}
\end{equation}
where we recall that the string $\omega$ is the unique weakly increasing string with the same content as $\lambda$.
Strings are always read left to right and top to bottom (which is consistent with the ordering of tensor products).
\rem[gray]{should we comment on choice of $d$ at the bottom? as far as I remember it's just to agree with Schubert polynomials,
it's irrelevant once one specializes to any fixed point}

\begin{rmk*}
This definition is related to the notion of {\em weight function}
\cite{RTV-K},
or in the language of quantum integrable systems, of {\em off-shell Bethe
vector}. 
\junk{However in an $A_d$-integrable system, one would expect a {\em nested}\/
Bethe vector, i.e., a more complicated definition than the one
above. This discrepancy is due to two factors:
 The first one is the choice of representation 
-- $(\CC^{d+1})^{\otimes n}$ -- which simplifies the Nested Bethe Ansatz.
}
Note however that we are really only interested in the
{\em on-shell}\/ Bethe vector, as will be explained right below:
that is, only in the the class $S^\lambda$ {\em modulo the relations
of $K^{\loc}_{\hat T}(T^*(P_-\backslash G))$} (the Bethe equations for infinite twist).
\junk{maybe refer to
papers like \url{https://arxiv.org/pdf/math/0610517.pdf} or
\url{https://arxiv.org/pdf/math/0610398.pdf} and refs therein.
somehow this is also related to the fact that the symmetric group (in general, Weyl group)
acts transitively on the fixed points. (should be same as minuscule?) which restricts us to $T^*(P_-\backslash G)$ (any other case among quiver
varieties?)}
See also \cite{AO-quasimap} for an extensive discussion of off-shell
Bethe vectors in a geometric context. At the moment, we do not know how to extend to
$T^*(P_-\backslash G)$ the more complicated method
of \cite{artic46,artic68} in order to expand products of off-shell Bethe vectors.
\end{rmk*}

\begin{ex}\label{ex:d1}
  We compute $S^\lambda$ for $\lambda=01,10$:
  \begin{align*}
    S^{01}&=
            \begin{tikzpicture}[baseline=(current  bounding  box.center),scale=0.75]
              \foreach\x/\t in {1/z_1,2/z_2}
              \draw[d,invarrow=0.9] (\x,0) -- node[right=-1mm,pos=0.1] {$\ss\t$} ++(0,3);
              \foreach\y/\t in {1/x_1,2/x_2}
              \draw[d,invarrow=0.9] (0,3-\y) -- node[above=-1mm,pos=0.05] {$\ss\t$} ++(3,0);
              \foreach\c in {(.5,1),(.9,.5),(1.9,.5),(2.5,1),(2.5,2),(1.9,2.5),
                (.9,1.5),(1.9,1.5),(1.5,1),(1.5,2)} { \node at \c {$1$}; }
              \foreach\c in {(.5,2),(.5,2),(.9,2.5)} { \node at \c {$0$}; }
            \end{tikzpicture}
            =\frac{1-q^2}{1-q^2x_1/z_1}
    \\
    S^{10}&=
            \begin{tikzpicture}[baseline=(current  bounding  box.center),scale=0.75]
              \foreach\x/\t in {1/z_1,2/z_2}
              \draw[d,invarrow=0.9] (\x,0) -- node[right=-1mm,pos=0.1] {$\ss\t$} ++(0,3);
              \foreach\y/\t in {1/x_1,2/x_2}
              \draw[d,invarrow=0.9] (0,3-\y) -- node[above=-1mm,pos=0.05] {$\ss\t$} ++(3,0);
              \foreach\c in {(.5,1),(.9,.5),(1.9,.5),(2.5,1),(2.5,2),(.9,2.5),
                (.9,1.5),(1.9,1.5),(1.5,1)} { \node at \c {$1$}; }
              \foreach\c in {(.5,2),(.5,2),(1.9,2.5),(1.5,2)} { \node at \c {$0$}; }
            \end{tikzpicture}
            =\frac{q(1-x_1/z_1)}{1-q^2x_1/z_1}\frac{1-q^2}{1-q^2x_1/z_2}
  \end{align*}
\end{ex}

$S^\lambda$ is a rational function in the parameters $x_1,\ldots,x_n$ and $z_1,\ldots,z_n,t$. 
Furthermore, \rem[gray]{Strictly speaking this lemma is not needed: we only care about specializations}
\begin{lem}\label{lem:Winv}
$S^\lambda$ is invariant under the action of $W_P$ on the variables $x_i$.
\end{lem}
\begin{proof}
We show invariance under the elementary transposition $i,i+1$ where
$\omega_i=\omega_{i+1}$: 
\begin{align*}
S^\lambda&=
\begin{tikzpicture}[baseline=1.8cm,scale=0.75]
\foreach\x/\t in {1/z_1,2/z_2,3/\cdots,4/z_n}
\draw[d,invarrow=0.9] (\x,0) node[below] {$\m d$} -- node[right=-1mm,pos=0.1] {$\ss\t$} ++(0,5);
\foreach\y/\t in {1/x_1,2/x_i,3/x_{i+1},4/x_n}
\draw[d,invarrow=0.9] (0,5-\y) -- node[above=-1mm,pos=0.1] {$\ss\t$} ++(5,0) node[right] {$\m d$};
\draw[decorate,decoration=brace] (-0.3,1) -- node[left] {$\m\omega$} (-0.3,4);
\draw[decorate,decoration=brace] (1,5.3) -- node[above] {$\m\lambda$} (4,5.3);
\end{tikzpicture}
=
\begin{tikzpicture}[baseline=1.8cm,scale=0.75]
\foreach\x in {1,...,4}
\draw[d,invarrow=0.9] (\x,0) node[below] {$\m d$} -- node[right=-1mm,pos=0.1] {$\ss z_\x$} ++(0,5);
\foreach\y/\s in {1/1,4/n}
\draw[d,invarrow=0.9] (0,5-\y) -- node[above=-1mm,pos=0.07] {$\ss x_\s$} ++(6,0) node[right] {$\m d$};
\draw[d,invarrow=0.9,rounded corners] (0,5-2) -- node[above=-1mm,pos=0.1] {$\ss x_i$} ++(4.5,0) -- ++(1,-1) -- ++(0.5,0) node[right] {$\m d$};
\draw[d,invarrow=0.9,rounded corners] (0,5-3) -- node[above=-1mm,pos=0.1] {$\ss x_{i+1}$} ++(4.5,0) -- ++(1,1) -- ++(0.5,0) node[right] {$\m d$};
\draw[decorate,decoration=brace] (-0.3,1) -- node[left] {$\m\omega$} (-0.3,4);
\draw[decorate,decoration=brace] (1,5.3) -- node[above] {$\m\lambda$} (4,5.3);
\end{tikzpicture}
\\
&=
\begin{tikzpicture}[baseline=1.8cm,scale=0.75]
\foreach\x in {1,...,4}
\draw[d,invarrow=0.9] (\x,0) node[below] {$\m d$} -- node[right=-1mm,pos=0.1] {$\ss z_\x$} ++(0,5);
\foreach\y/\s in {1/1,4/n}
\draw[d,invarrow=0.9] (-1,5-\y) -- node[above=-1mm,pos=0.1] {$\ss x_\s$} ++(6,0) node[right] {$\m d$};
\draw[d,invarrow=0.9,rounded corners] (-1,5-2) -- node[above=-1mm,pos=0.5] {$\ss x_i$} ++(0.5,0) -- ++(1,-1) -- ++(4.5,0) node[right] {$\m d$};
\draw[d,invarrow=0.9,rounded corners] (-1,5-3) -- node[above=-1mm,pos=0.5] {$\ss x_{i+1}$} ++(0.5,0) -- ++(1,1) -- ++(4.5,0) node[right] {$\m d$};
\draw[decorate,decoration=brace] (-1.3,1) -- node[left] {$\m\omega$} (-1.3,4);
\draw[decorate,decoration=brace] (1,5.3) -- node[above] {$\m\lambda$} (4,5.3);
\end{tikzpicture}
=
\begin{tikzpicture}[baseline=1.8cm,scale=0.75]
\foreach\x in {1,...,4}
\draw[d,invarrow=0.9] (\x,0) node[below] {$\m d$} -- node[right=-1mm,pos=0.1] {$\ss z_\x$} ++(0,5);
\foreach\y/\yy in {1/1,2/i+1,3/i,4/n}
\draw[d,invarrow=0.9] (0,5-\y) -- node[above=-1mm,pos=0.1] {$\ss x_{\yy}$} ++(5,0) node[right] {$\m d$};
\draw[decorate,decoration=brace] (-0.3,1) -- node[left] {$\m\omega$} (-0.3,4);
\draw[decorate,decoration=brace] (1,5.3) -- node[above] {$\m\lambda$} (4,5.3);
\end{tikzpicture}
\end{align*}
The argument is identical to that of \cite[lemma~3.11]{artic71}.
\end{proof}

Write $\bar\sigma$ for $\sigma^{-1}$.
\begin{lem}\label{lem:spec}
$S^\lambda$ is well-defined at every specialization $x_i=z_{\bar\sigma(i)}$,
$\sigma\in W$, and given by
\begin{equation}\label{eq:defSfp}
S^\lambda|_\sigma = 
\begin{tikzpicture}[baseline=-3pt,scale=0.9]
\draw (0,-1) rectangle (5,1); \node at (2.5,0) {$\sigma$};
\foreach\x/\t/\u in {1/z_1/z_{\bar\sigma(1)},2/z_2/z_{\bar\sigma(2)},3/\cdots/\cdots,4/z_n/z_{\bar\sigma(n)}}
{
\draw[d,invarrow=0.5] (\x,1) -- node[right,pos=0.5] {$\ss\t$} ++(0,1);
\draw[d,invarrow=0.5] (\x,-2) -- node[right=-1mm,pos=0.4] {$\ss\u$} ++(0,1);
}
\draw[decorate,decoration=brace] (4,-2.3) -- node[below] {$\m\omega$} (1,-2.3);
\draw[decorate,decoration=brace] (1,2.3) -- node[above] {$\m\lambda$} (4,2.3);
\end{tikzpicture}
\end{equation}
where the rectangle labeled $\sigma$ is any wiring diagram of $\sigma$, each crossing being an $R$-matrix. Furthermore, $S^\lambda|_{\sigma}$ only depends on the class of $\sigma$ in $W_P\dom W$.
\end{lem}
Note that time flows downwards, e.g., if $\sigma=(312)$, then the diagram of $\sigma$ is 
\begin{tikzpicture}[baseline=(current  bounding  box.center)]
\draw[d,invarrow] (1,0) -- (2,1);
\draw[d,invarrow] (2,0) -- (3,1);
\draw[d,invarrow] (3,0) -- (1,1);
\end{tikzpicture}.
\begin{proof}
\eqref{eq:defSfp} follows from proposition~\ref{prop:single}
-- the proof is identical to that of \cite[lemma~3.12]{artic71}, and we shall not repeat it here.

Because $\omega$ is invariant under $W_P$, and $\check R_{ii}^{ii}=1$ according to \eqref{eq:Rsingle},
the value of \eqref{eq:defSfp} is unaffected by left multiplication of $\sigma$ by an element of
$W_P$. This shows that $S^\lambda|_{\sigma}$ depends only on the class of $\sigma$ in $W_P\dom W$.
\end{proof}

\begin{ex}\label{ex:d1b}
  We compute $S^\lambda|_{\sigma}$ for $\lambda=01,10$ and $\sigma\in\mathcal S_2$, where we identify the identity permutation with $01$
  and the nontrivial permutation with $10$:
  \begin{align*}
    S^{01}|_{01} &=
    \begin{tikzpicture}[baseline=(current  bounding  box.center),scale=0.75]
      \draw[d,invarrow=.45] (1,0) -- node[xshift=-.15cm,pos=0.7] {$0$} node[xshift=-.15cm,pos=0.3] {$0$} node[right,pos=.9] {$\ss z_1$} (1,2);
      \draw[d,invarrow=.45] (2,0) -- node[xshift=-.15cm,pos=0.7] {$1$} node[xshift=-.15cm,pos=0.3] {$1$} node[right,pos=.9] {$\ss z_2$} (2,2);
    \end{tikzpicture}
    =1
    &
    S^{10}|_{01} &=
    \begin{tikzpicture}[baseline=(current  bounding  box.center),scale=0.75]
      \draw[d,invarrow=.45] (1,0) -- node[xshift=-.15cm,pos=0.7] {$1$} node[xshift=-.15cm,pos=0.3] {$0$} node[right,pos=.9] {$\ss z_1$} (1,2);
      \draw[d,invarrow=.45] (2,0) -- node[xshift=-.15cm,pos=0.7] {$0$} node[xshift=-.15cm,pos=0.3] {$1$} node[right,pos=.9] {$\ss z_2$} (2,2);
    \end{tikzpicture}
                   =0
                   \\
    S^{01}|_{10} &=
    \begin{tikzpicture}[baseline=(current  bounding  box.center),scale=0.75]
      \draw[d,invarrow=.6] (2,0) -- node[xshift=-.15cm,pos=0.7] {$0$} node[xshift=.15cm,pos=0.3] {$1$} node[right,pos=.9] {$\ss z_1$} (1,2);
      \draw[d,invarrow=.6] (1,0) -- node[xshift=.15cm,pos=0.7] {$1$} node[xshift=-.15cm,pos=0.3] {$0$} node[right,pos=.9] {$\ss z_2$} (2,2);
    \end{tikzpicture}
    =\frac{1-q^2}{1-q^2z_2/z_1}
    &
    S^{10}|_{10} &=
    \begin{tikzpicture}[baseline=(current  bounding  box.center),scale=0.75]
      \draw[d,invarrow=.6] (2,0) -- node[xshift=-.15cm,pos=0.7] {$1$} node[xshift=.15cm,pos=0.3] {$1$} node[right,pos=.9] {$\ss z_1$} (1,2);
      \draw[d,invarrow=.6] (1,0) -- node[xshift=.15cm,pos=0.7] {$0$} node[xshift=-.15cm,pos=0.3] {$0$} node[right,pos=.9] {$\ss z_2$} (2,2);
    \end{tikzpicture}
                   =\frac{q(1-z_2/z_1)}{1-q^2z_2/z_1}
  \end{align*}
  
  Compare with example~\ref{ex:d1}.
\end{ex}

By abuse of notation, we suppress
the dependence on the choice of word of $\sigma$ in the diagram, the justification being that the resulting
quantity is independent of the choice of word.

Lemma~\ref{lem:spec} implies that we can consider $S^\lambda|_\sigma$ (a rational function in $z_1,\ldots,z_n,q$ with only poles
at $z_i/z_j=q^2$) as the restriction
of a class in $K^{\loc}_{\hat T}(T^*(P_-\backslash G))$ to a fixed point $\sigma\in W_P \dom W$;
we identify $S^\lambda$ with this class, and call it the {\em motivic Segre class}\/ labeled by
$\lambda$. \rem[gray]{this is where we implicitly use lemma~\ref{lem:Winv} -- to say that
the identification really sends $x_i$ to the Chern root $x_i$}

We define dual classes by reversing all arrows (and conventionally rotating diagrams by 180 degrees); more precisely,
define
\begin{equation}\label{eq:defSd}
S_\lambda
:=
\begin{tikzpicture}[baseline=1.8cm,scale=0.75]
\foreach\x/\t in {1/z_1,2/z_2,3/\cdots,4/z_n}
\draw[d,invarrow=0.9] (\x,0) -- node[right=-1mm,pos=0.1] {$\ss\t$} ++(0,5) node[above] {$\m d$};
\foreach\y/\t in {1/x_n,2/\vdots,3/x_2,4/x_1}
\draw[d,invarrow=0.9] (0,5-\y) node[left] {$\m d$} -- node[above=-1mm,pos=0.1] {$\ss\t$} ++(5,0);
\draw[decorate,decoration=brace] (5.3,4) -- node[right] {$\m\alpha$} (5.3,1);
\draw[decorate,decoration=brace] (4,-0.3) -- node[below] {$\m\lambda$} (1,-0.3);
\end{tikzpicture}
\end{equation}
where we recall that the string $\alpha$ is the unique weakly decreasing string with the same content as $\lambda$.
The fixed point restriction formula reads, using the shorthand notation $\hat\sigma=\sigma^{-1}w_0$,
where $w_0$ is the longest permutation:
\begin{equation}\label{eq:defSfpd}
S_\lambda|_\sigma = 
\begin{tikzpicture}[baseline=-3pt,scale=0.9]
\draw (0,-1) rectangle (5,1); \node at (2.5,0) {$\hat\sigma$};
\foreach\x/\t/\u in {1/z_{\hat\sigma(1)}/z_1,2/z_{\hat\sigma(2)}/z_{2},3/\cdots/\cdots,4/z_{\hat\sigma(n)}/z_{n}}
{
\draw[d,invarrow=0.5] (\x,1) -- node[right=-1mm,pos=0.4] {$\ss\t$} ++(0,1);
\draw[d,invarrow=0.5] (\x,-2) -- node[right,pos=0.5] {$\ss\u$} ++(0,1);
}
\draw[decorate,decoration=brace] (4,-2.3) -- node[below] {$\m\lambda$} (1,-2.3);
\draw[decorate,decoration=brace] (1,2.3) -- node[above] {$\m\alpha$} (4,2.3);
\end{tikzpicture}
\end{equation}

\junk{in view of the polarization issues, one should check (by doing a puzzle
  calculation) whether this is really the normalization we want for dual
  S classes}

Note that reversing arrows in the $R$-matrix is equivalent
to $q,z',z''\mapsto q^{-1},z'^{-1},z''^{-1}$, according to \eqref{eq:Rsingle};
so that there is a simple relationship between $S^\lambda$ and $S_\lambda$:
\begin{lem}\label{lem:inv}
One has
\[
S_\lambda(x_1,\ldots,x_n,z_1\ldots,z_n,q)
=
S^{\lambda^*}(x_1^{-1},\ldots,x_n^{-1},z_n^{-1},\ldots,z_1^{-1},q^{-1})
\]
where $\lambda^*$ denotes the string $\lambda$ read backwards. Similarly,
\[
S_\lambda|_\sigma(z_1,\ldots,z_n,q)
=
S^{\lambda^*}|_{\sigma w_0}(z_n^{-1},\ldots,z_1^{-1},q^{-1})
\]
\end{lem}

\rem{justify in what sense they're motivic? other properties?
AK wants to add a ref here}

\subsection{Relation to stable classes}\label{sec:stabclass}
\rem[gray]{note that {\em any class}\/ related to $S^\lambda$ by multiplication by some function
of the Chern classes (including $S^\lambda$ itself!) will lead to the same exchange relation
as $\St^\lambda$ with the same $R$-matrix}

We start with the following characterization of motivic Segre classes of $P_-\backslash G$.
Denote by $\tau_i$ the elementary transposition $(i,i+1)$,
and make it act on $K_{\hat T}^{\loc}(pt)$ by permuting variables $z_i$ and $z_{i+1}$.

\begin{prop}\label{prop:charact}
The classes $S^\lambda$ are entirely determined by the following properties:
\begin{enumerate}
\item {\em Triangularity:} $S^\lambda|_\sigma=0$ unless $\sigma\le\lambda$ in the Bruhat order.
\item {\em Diagonal entries:}
\[
S^{\lambda}|_\lambda
=\prod_{i<j: \lambda_i>\lambda_j}
\frac{q(1-z_{j}/z_{i})}{1-q^2 z_{j}/z_{i}}
\]
\item {\em Exchange relation:}
\[
\tau_i S^\lambda|_{\sigma\tau_i}
=\begin{cases}
\frac{1-q^2}{1-q^2 z_{i}/z_{i+1}}
S^\lambda|_\sigma 
+ 
\frac{q(1-z_{i}/z_{i+1})}{1-q^2z_{i}/z_{i+1}}
S^{\lambda\tau_i}|_\sigma
&\lambda_i<\lambda_{i+1}
\\
S^\lambda|_\sigma
& \lambda_i=\lambda_{i+1}
\\
\frac{(1-q^2)z_{i}/z_{i+1}}{1-q^2 z_{i}/z_{i+1}}
S^\lambda|_\sigma 
+ 
\frac{q(1-z_{i}/z_{i+1})}{1-q^2 z_{i}/z_{i+1}}
S^{\lambda\tau_i}|_\sigma
& \lambda_i>\lambda_{i+1}
\end{cases}
\]
\end{enumerate}
\end{prop}
In fact, property (3) allows to define inductively the $S^\lambda$ in terms
of a single one, say $S^\alpha$, 
and then we only need (1) and (2) at $\lambda=\alpha$.
\begin{proof}
For the first two points,
pick a minimal representative $\sigma$ of $\lambda$ in $W_P \dom W$, i.e.,
a permutation of minimal length
that permutes the labels of $\omega$ to those of $\sigma$, and apply lemma~\ref{lem:spec}, i.e.,
consider a reduced decomposition $Q$ for $\sigma$.
\begin{enumerate}
\item According to \eqref{eq:Rsingle}, the summation over intermediate labels in the expression
 \eqref{eq:defSfp} of $S^\lambda|_\sigma$ is a summation over subwords $P$ of $Q$, and therefore results
in strings $\prod P\cdot\omega$ that are less of equal to $\lambda$ in the Bruhat order.
\item In the equality case $\sigma=\lambda$,
the only configuration that contributes to \eqref{eq:defSfp} is made of vertices of the form
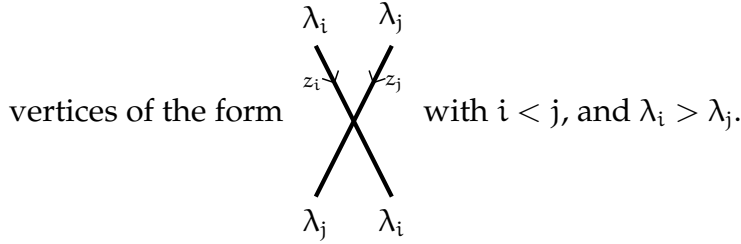
\tikz[baseline=0,xscale=0.5]
{
\draw[invarrow=0.75,d] (-1,-1) node[below] {$\m \lambda_j$} -- node[pos=0.75,right] {$\ss z_{j}$} (1,1) node[above] {$\m \lambda_j$};
\draw[invarrow=0.75,d] (1,-1) node[below] {$\m \lambda_i$} -- node[pos=0.75,left] {$\ss z_{i}$} (-1,1) node[above] {$\m \lambda_i$};
} with $i<j$, and $\lambda_i>\lambda_j$.
\item
Finally, the induction formula is derived as follows: given a fixed point $\sigma$ and an elementary transposition
$\tau_i$, diagrammatically,
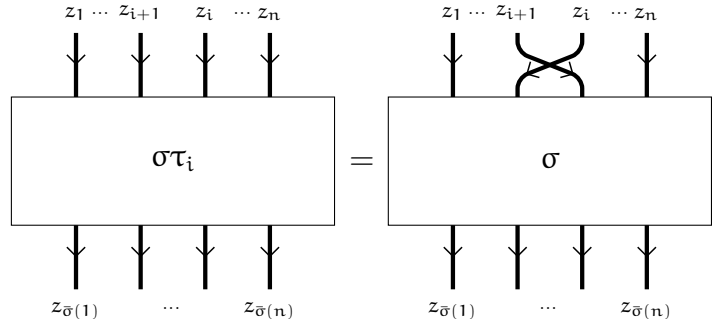
\[
\begin{tikzpicture}[baseline=-3pt,scale=0.85]
\draw (0,-1) rectangle (5,1); \node at (2.5,0) {$\sigma\tau_i$};
\foreach\x/\t/\u in {1/z_1/z_{\bar\sigma(1)},2/z_{i+1}/,3/z_i/,4/z_n/z_{\bar\sigma(n)}}
{
\draw[d,invarrow=0.5] (\x,1) -- ++(0,1) node[above] {$\ss\t$};
\draw[d,arrow=0.5] (\x,-1) -- ++(0,-1) node[below] {$\ss\u$};
}
\node at (1.4,2.3) {$\sss\cdots$};
\node at (3.6,2.3) {$\sss\cdots$};
\node at (2.5,-2.3) {$\sss\cdots$};
\end{tikzpicture}
\,
=
\,
\begin{tikzpicture}[baseline=-3pt,scale=0.85]
\draw (0,-1) rectangle (5,1); \node at (2.5,0) {$\sigma$};
\foreach\x/\t in {1/z_{\bar\sigma(1)},2/,3/,4/z_{\bar\sigma(n)}}
{
\draw[d,arrow=0.5] (\x,-1) -- ++(0,-1) node[below] {$\ss\t$};
}
\draw[d,invarrow=0.5] (1,1) -- ++(0,1) node[above] {$\ss z_1$};
\draw[d,invarrow=0.5] (4,1) -- ++(0,1) node[above] {$\ss z_n$};
\draw[d,invarrow=0.25,rounded corners] (2,1) -- ++(0,0.3) -- ++(1,0.4) -- ++(0,0.3) node[above] {$\ss z_{i}$};
\draw[d,invarrow=0.25,rounded corners] (3,1) -- ++(0,0.3) -- ++(-1,0.4) -- ++(0,0.3) node[above] {$\ss z_{i+1}$};
\node at (1.4,2.3) {$\sss\cdots$};
\node at (3.6,2.3) {$\sss\cdots$};
\node at (2.5,-2.3) {$\sss\cdots$};
\end{tikzpicture}
\]
Applying once again lemma~\ref{lem:spec} and performing the sum over the labels of the edges right under
the final $(i,i+1)$ crossing leads to the desired result.
\end{enumerate}
\end{proof}

One similarly proves the corresponding properties for the dual classes:
\begin{prop}\label{prop:charact2}
The classes $S_\lambda$ are entirely determined by the following properties:
\begin{enumerate}
\item {\em Triangularity:} $S_\lambda|_\sigma=0$ unless $\sigma\ge\lambda$ in the Bruhat order.
\item {\em Diagonal entries:}
\[
S_{\lambda}|_\lambda
=\prod_{i<j: \lambda_i<\lambda_j}
\frac{q(1-z_{i}/z_{j})}{1-q^2 z_{i}/z_{j}}
\]
\item {\em Exchange relation:}
\[
\tau_i S_\lambda|_{\sigma\tau_i}
=\begin{cases}
\frac{1-q^2}{1-q^2 z_{i+1}/z_{i}}
S_\lambda|_\sigma 
+ 
\frac{q(1-z_{i+1}/z_{i})}{1-q^2z_{i+1}/z_{i}}
S_{\lambda\tau_i}|_\sigma
&\lambda_i<\lambda_{i+1}
\\
S_\lambda|_\sigma
& \lambda_i=\lambda_{i+1}
\\
\frac{(1-q^2)z_{i+1}/z_{i}}{1-q^2 z_{i+1}/z_{i}}
S_\lambda|_\sigma 
+ 
\frac{q(1-z_{i+1}/z_{i})}{1-q^2 z_{i+1}/z_{i}}
S_{\lambda\tau_i}|_\sigma
& \lambda_i>\lambda_{i+1}
\end{cases}
\]
\end{enumerate}
\end{prop}

In both cases, one notes that these properties are very similar to the ones satisfied by the stable basis \cite{Ok-K}.
This motivates the definition
\[
\St^\lambda:=\kappa\,S^\lambda
\]
where
\[
  \kappa=[P_-\backslash G]=\prod_{i,j:\omega_i<\omega_j}(1-q^2x_i/x_j)
\]
is the class of the zero section of $T^*(P_-\backslash G)\to P_-\backslash G$.
We shall see just below that $\St^\lambda$ lives in {\em non-localized}\/ 
$K$-theory, i.e., belongs to $K_{\hat T}(T^*(P_-\backslash G))$.  Inversely,
\[
S^{\lambda}=p^*p_* \St^\lambda
\]
and the denominator of $S^\lambda$ can be blamed on the lack of properness
of the map $p$. 
\junk{should one be more specific? in particular the denominator can only
be related to the weights in the cotangent direction. that means its restrictions at fixed points will still satisfy the divisibility property}
Similarly, introduce $\St_\lambda:=\tilde\kappa\, S_\lambda$,
where
\[
  \tilde\kappa=\prod_{i,j:\omega_i<\omega_j}q(1-q^{-2}x_j/x_i)
\]
differs from $\kappa$ by a monomial (related to choosing an opposite polarization).
We can then state:

\begin{prop}\label{prop:stab}
  $\St^\lambda$ (resp.\ $\St_\lambda$) coincides with the stable basis element associated to $\lambda$,
  to the negative (resp.\ positive) chamber, to the polarization $T^*(P_-\backslash G)$ (resp.\ $T(P_-\backslash G)$),
  and to a line bundle in the positive (resp.\ negative) alcove closest to $0$.
\end{prop}
(This is $\text{Stab}_-(\lambda)$ (resp.\ $\text{Stab}_+(\lambda)$) in the notations of \cite{SZZ-Kstable}.)

As a corollary, $\St^\lambda,\St_\lambda\in K_{\hat T}(T^*(P_-\backslash G))$.
\begin{proof}
  We provide the proof for $\St^\lambda$.
  $St^\lambda$ is uniquely determined by properties strictly analogous to proposition~\ref{prop:charact}, the only difference being the normalization of the diagonal entries; using
  \begin{equation*}
    [P_-\backslash G]|_\lambda = \prod_{i,j:\lambda_i<\lambda_j}(1-q^{2} z_{i}/z_{j})
  \end{equation*}
  leads to
  \begin{equation}\label{eq:pty2}
    \St^{\lambda}|_\lambda
    =\prod_{i<j,\lambda_i<\lambda_j}(1-q^2z_{i}/z_{j})\prod_{i>j,\lambda_i<\lambda_j}q(1-z_{i}/z_{j})
  \end{equation}
  
  We now claim that the same properties are satisfied by the corresponding stable basis elements.
  Property (1) is a direct consequence of the definition of the stable envelope, and
  more precisely of the support condition \cite[9.1.3 (1)]{Ok-K}.
  
  Similarly, \eqref{eq:pty2}, which is a version of (2), is a consequence of the normalization
  condition on the diagonal \cite[9.1.3 (1)]{Ok-K}, paying attention to the choice of
  polarization, which introduces the extra factor of $-q$ in the second product, cf \cite[9.1.5]{Ok-K}.
  
  So we only need to prove that stable classes satisfy the same exchange relation (3). But this is actually
  the definition of the $R$-matrix according to \cite[9.2.11]{Ok-K}. Indeed, the substitution $S^\lambda|_{\sigma}
  \mapsto \tau_i S^\lambda|_{\tau_i\sigma}$ is equivalent to moving from the negative chamber to a
  neighboring one across the wall $z_i=z_{i+1}$. So the only check left is that our $R$-matrix \eqref{eq:Rsingle}
  coincides with the one as defined geometrically. The beauty of this construction of the $R$-matrix
  is that it needs to be done at $n=2$ only; this is the object of \cite[Exercise 9.2.25]{Ok-K}, and we shall
  not repeat it here (see also \cite[Example 1.4]{SZZ-Kstable}).
  
  The reasoning for $\St_\lambda$ is much the same; we only provide the normalization condition:
  \[
  \St_{\lambda}|_\lambda=
  \prod_{i<j,\lambda_i<\lambda_j}(1-z_j/z_i)
  \prod_{i>j,\lambda_i<\lambda_j}q(1-q^{-2}z_j/z_i)
  \]
\end{proof}

\begin{rmk*}
  The exchange formula above should be not be confused with the
  (torus-equivariant) transition formula of \cite[Proposition~3.6]{SZZ-Kstable};
  the latter only makes sense for full flags, i.e., $T^*(G/B)$.
\end{rmk*}

In particular we have the duality statement (
cf.~\cite[Remark~1.3 (2)]{SZZ-Kstable})
\begin{equation}\label{eq:Kdual}
\left< \St^\lambda \St_\mu \right>=\delta^\lambda_\mu
\end{equation}
where $\left<\cdots\right>$ denotes pushforward to a point in (localized)
equivariant $K$-theory. 
Such pushforwards are slightly subtle because the map $T^*(P_-\dom G)\to pt$
is not proper; however, the localization formula gives a natural definition
of pushforward whenever the map on {\em fixed points} is proper, with the
only price being the denominators incurred in that formula.

\begin{rmk*}
  If $i$ is the zero section of $T^*(P_-\backslash G)\to P_-\backslash G$,
  then $i^*\St^\lambda$ is the motivic Chern class of \cite{FRW-motivic}.
\end{rmk*}

\section{The puzzle rule}\label{sec:puzzle}
\newcommand\dOne[2]{
\node[gauged] at (1,0) (v1) {#1}; 
\node[gauged] at (2,0) (v2) {#2}; 
\draw (v1) -- (v2);
}
\newcommand\dOneWide[2]{
\node[gauged] at (1,0) (v1) {#1}; 
\node[gauged] at (3,0) (v2) {#2}; 
\draw (v1) -- (v2);
}
\newcommand\dtwo[4]{
\node[gauged] at (1,0) (v1) {#1}; 
\node[gauged] at (2,0) (v2) {#2}; 
\node[gauged] at (3,0) (v3) {#3}; 
\node[gauged] at (2,-1) (v4) {#4}; 
\draw (v1) -- (v2) -- (v3) (v2) -- (v4);
}
\newcommand\dtwoWide[4]{
\node[gauged] at (0,0) (v1) {#1}; 
\node[gauged] at (2,0) (v2) {#2}; 
\node[gauged] at (4,0) (v3) {#3}; 
\node[gauged] at (2,-1) (v4) {#4}; 
\draw (v1) -- (v2) -- (v3) (v2) -- (v4);
}
\newcommand\dthree[6]{
\node[gauged] at (1,0) (v1) {#1}; 
\node[gauged] at (2,0) (v2) {#2}; 
\node[gauged] at (3,0) (v3) {#3}; 
\node[gauged] at (4,0) (v4) {#4}; 
\node[gauged] at (5,0) (v5) {#5}; 
\node[gauged] at (3,-1) (v6) {#6}; 
\draw (v1) -- (v2) -- (v3) -- (v4) -- (v5) (v3) -- (v6);
}
\newcommand\dfour[8]{
\node[gauged] at (1,0) (v1) {#1}; 
\node[gauged] at (2,0) (v2) {#2}; 
\node[gauged] at (3,0) (v3) {#3}; 
\node[gauged] at (4,0) (v4) {#4}; 
\node[gauged] at (5,0) (v5) {#5}; 
\node[gauged] at (6,0) (v6) {#6}; 
\node[gauged] at (7,0) (v7) {#7}; 
\node[gauged] at (5,-1) (v8) {#8};
\draw (v1) -- (v2) -- (v3) -- (v4) -- (v5) -- (v6) -- (v7) (v5) -- (v8);
}

\newcommand\wi{\tikz[script math mode,baseline=-3pt]{
\node[framed]  {\displaystyle w^i }; }}
\newcommand\vi{\tikz[script math mode,baseline=-3pt]{
\node[gauged]  {\displaystyle v^i }; }}

In this section we several times make use of
\begin{itemize}
\item a finite-type Dynkin diagram, 
  to which we associate a quantized loop algebra $\gqg$,
\item a dominant weight, to which we associate a particular representation
  (or more precisely, a parametrized family of representations), and
\item a weight of that representation.
\end{itemize}
Although we defer discussion of Nakajima quiver varieties 
(whose $K_T$-groups will be those corresponding weight spaces) 
to \S \ref{sec:Nakajima}, we will recall here the way these varieties are indexed,
in order to similarly index our weight spaces.

The \defn{Nakajima diagram} associated to a Dynkin diagram $I$ attaches a
univalent \defn{framed vertex} hanging off each vertex of $I$, those called
the \defn{gauged vertices}. For example,
the $A_4$ Nakajima diagram looks like
$
\tikz[script math mode,baseline=0]{
\node[framed] at (0,1) (w1) { }; 
\node[framed] at (1,1) (w2) { }; 
\node[framed] at (2,1) (w3) { }; 
\node[framed] at (3,1) (w4) { }; 
\node[gauged] at (0,0) (v1) { }; 
\node[gauged] at (1,0) (v2) { }; 
\node[gauged] at (2,0) (v3) { }; 
\node[gauged] at (3,0) (v4) { }; 
\draw (w1) -- (v1) -- (v2) -- (v3) -- (v4) -- (w4);
\draw (w2) -- (v2);
\draw (w3) -- (v3);
}
$\ .
When we populate the diagram with natural numbers, the framed vertices
(\wi) indicate a dominant linear combination 
$\sum w^i \vec\omega_i$ of the fundamental weights $\vec\omega_i$ of $\lie{g}$, whereas the
gauged vertices (\vi) indicate a linear combination of the simple roots $\vec\alpha_i$,
to be subtracted from that dominant weight, 
leaving the \defn{weight} $\sum w^i \vec\omega_i - \sum_i v^i \vec\alpha_i$ of the labeled quiver.
(Note that every weight of the $\lie{g}$-irrep $V_{\sum w^i \vec\omega_i}$
is of this latter form.)
These tuples (\wi), (\vi) are called the framed and
gauged \defn{dimension vectors}. When some $w_i$ is $0$ we generally
don't bother to draw that framed vertex.

A basic example occurs when (1) the support of the gauged dimension vector
(\vi) lies on a type $A_d$ subdiagram, and (2) the framed dimension vector
is supported at one vertex, at one end of the type $A_d$ subdiagram.
For reasons to be explained in \S\ref{sec:Nakajima} we call this a
\defn{($d$-step) flag type} case.

To each (\wi) 
is associated a representation of $\gqg$, as will be discussed in more detail
in \S \ref{ssec:reptheory}.
Properly speaking, this is not a finite-dimensional complex vector space but
rather a finite rank module over a subring of $\CC(z_{i,j},\ j=1,\ldots,w^i)$.
For now we say that for general values of the parameters $(z_{i,j})$ this
representation is isomorphic to a tensor product
$\tensor_i \tensor_{j=1}^{w^i} V_{\vec\omega_i}(z_{i,j})$ 
(which, for those general values, is irreducible).

The set of weights in each of our representations is $W_G$-invariant.
We can see the action of a simple reflection $r_i \in W_G$ on a
pair (\wi), (\vi) of dimension vectors thusly: {\bf replace a gauged
label \vi\ by the sum of its neighboring labels (including the \wi\ on 
neighboring framed vertex), minus the original \vi. }

With all this we can introduce the labeled Nakajima quivers we need in
this paper, four for each of $d=1,2,3,4$. 
In figures $\ref{fig:d1}-\ref{fig:d4}$ we give a Dynkin diagram
of rank $2d$, and four labelings of the Nakajima diagram.
In each case, 
\begin{itemize}
\item the sum of the first two weights equals the third weight,
  and also the fourth;
\item the first of the four labelings is of $d$-step flag type; and
\item the second and the fourth weights can be reflected (by iterating
  the bolded recipe in the paragraph above) to be of $d$-step flag type.
  In the figure we indicate sequences of reflections to use to see this.
  The third weight usually cannot be reflected to flag type (only
  for $d=1$ is this possible).
\end{itemize}

Furthermore, the weight of the second quiver is $\tau^2$ times that of the first,
where $\tau$ is a nontrivial order $3$ automorphism of the weight space
that was defined in \cite[\S 2]{artic71}; \footnote{The explicit
form of $\tau$ is irrelevant for the present work;
note that if $d\neq 2$, $\tau$ can be chosen of the form $\chi^{h_d/3}$
where $\chi$ is a certain Coxeter element and
$h_d=3,12,30$ for $d=1,3,4$; e.g., at $d=4$,
$\chi=\rr {d'}\rr {a'}\rr a \rr b \rr c \rr d \rr {c'} \rr {b'}$
(compare with figure~\ref{fig:d4}, noting that
$-1 = w_0 = \chi^{h/2}$ implies
$-\tau = \chi^{-5}$).}
so the third = fourth weight is $-\tau$ times the first weight.

For the purposes of this section, the arrows in 
figures $\ref{fig:d1}-\ref{fig:d4}$ refer to the maps on weight spaces
induced from some corresponding $\gqg$-module homomorphisms.
The second arrow will be called the ``fusion'' arrow, following the
notion from the representation theory of quantized affine algebras.

\junk{AK: it's very tempting to refer to ``figure $d$'', but I'll wait
  until we're absolutely sure these are the first four figures.}

\begin{figure}
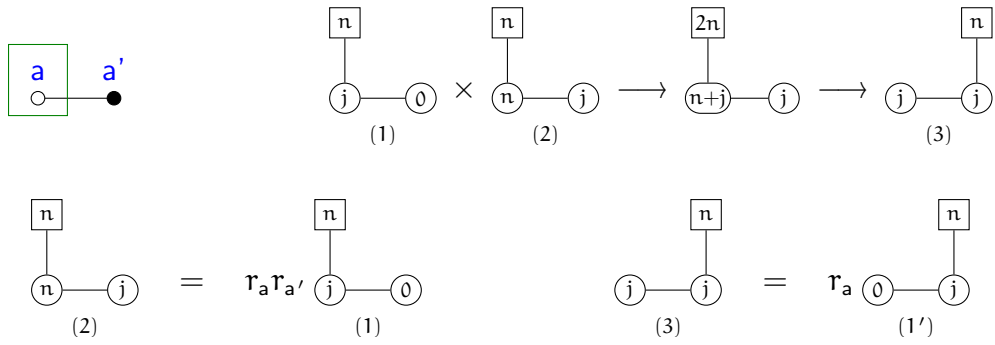

\begin{gather*}
 \tikz[baseline=0,every label/.style={rt,above},math mode,rotate=180]{\draw (0,0) node[mynode,label=a'] (a') {} -- (1,0) node[mynode,fill=white,label={[name=aa]a}] (a) {}; \node[rectangle,draw=green!50!black,fit=(a) (aa)] {}; }
\qquad \qquad \qquad 
\tikz[script math mode,baseline=0]{\dOne{j}{0} \node[framed] at (1,1) (w1) {n}; \draw (v1) -- (w1); \node at (1.5,-0.5) {(1)};}
\ \times\ 
\tikz[script math mode,baseline=0]{\dOne{n}{j} \node[framed] at (1,1) (w1) {n}; \draw (v1) -- (w1); \node at (1.5,-0.5) {(2)};}
\ \longrightarrow\ 
\tikz[script math mode,baseline=0]{\dOne{n+j}{j} \node[framed] at (1,1) (w1) {2n}; \draw (v1) -- (w1);}
\ \longrightarrow\ 
\tikz[script math mode,baseline=0]{\dOne{j}{j} \node[framed] at (2,1) (w2) {n}; \draw (v2) -- (w2); \node at (1.5,-0.5) {(3)};}
\\[4mm]
\tikz[script math mode,baseline=0]{\dOne{n}{j} \node[framed] at (1,1) (w1) {n}; \draw (v1) -- (w1); \node at (1.5,-0.5) {(2)};}
\quad=\quad
\rr a
\rr {a'}\
\tikz[script math mode,baseline=0]{\dOne{j}{0} \node[framed] at (1,1) (w1) {n}; \draw (v1) -- (w1); \node at (1.5,-0.5) {(1)};}
\qquad\qquad\qquad
\tikz[script math mode,baseline=0]{\dOne{j}{j} \node[framed] at (2,1) (w2) {n}; \draw (v2) -- (w2); \node at (1.5,-0.5) {(3)};}
\quad=\quad
\rr a\
\tikz[script math mode,baseline=0]{\dOne{0}{j} \node[framed] at (2,1) (w2) {n}; \draw (v2) -- (w2); \node at (1.5,-0.5) {(1')};}
\end{gather*}
\caption{The $d=1$ Dynkin diagram, fusion, and $\pi$s needed in proposition \ref{prop:reflect}. 
}
  \label{fig:d1}
\end{figure}

\begin{figure}[b]
\begin{gather*}
 \tikz[baseline=0,every label/.style={rt,above},math mode,rotate=180]{\draw (0,0) node[mynode,label=a'] {} -- ++(1,0) node[mynode,label={[name=aa]a}] (a) {} -- ++(1,0) node[mynode,fill=white,label={[name=bb]b}] (b) {}; \draw (1,0) -- ++(0,1) node[mynode,label={[xshift=2mm]b'}] {}; \node[rectangle,draw=green!50!black,fit=(a) (b) (aa) (bb)] {}; }
\quad
\tikz[script math mode,baseline=0]{\dtwo{k}{j}{0}{0} \node[framed] at (1,1) (w1) {n}; \draw (v1) -- (w1); \node at (1,-1) {(1)};}
\times
\tikz[script math mode,baseline=0]{\dtwo{n}{n+k}{n+j}{k} \node[framed] at (3,1) (w3) {n}; \draw (v3) -- (w3); \node at (1,-1) {(2)};}
\longrightarrow
\tikz[script math mode,baseline=0]{\dtwo{n+k}{\ss n+\atop\ss k+j}{n+j}{k} \node[framed] at (1,1) (w1) {n}; \node[framed] at (3,1) (w3) {n}; \draw (v1) -- (w1) (v3) -- (w3);}
\longrightarrow
\tikz[script math mode,baseline=0]{\dtwo{k}{k+j}{j}{k} \node[framed] at (3,-1) (w4) {n}; \draw (v4) -- (w4); \node at (1,-1) {(3)};}
\\[2mm]
\tikz[script math mode,baseline=0]{\dtwo{n}{n+k}{n+j}{k} \node[framed] at (3,1) (w3) {n}; \draw (v3) -- (w3); \node at (1,-1) {(2)};}
\, =\,
\begin{matrix}
\rr b \rr a \rr b \rr {b'} \quad \\ \quad \rr {a'} \rr a \rr b \rr {b'}  
\end{matrix}\, 
\tikz[script math mode,baseline=0]{\dtwo{0}{j}{k}{0} \node[framed] at (3,1) (w3) {n}; \draw (v3) -- (w3); \node at (1,-1) {(1')};}
\qquad
\tikz[script math mode,baseline=0]{\dtwo{k}{k+j}{j}{k} \node[framed] at (3,-1) (w4) {n}; \draw (v4) -- (w4); \node at (1,-1) {(3)};}
\, =\,
\rr b \rr a \rr b \rr {a'}\
\tikz[script math mode,baseline=0]{\dtwo{0}{j}{0}{k} \node[framed] at (3,-1) (w4) {n}; \draw (v4) -- (w4); \node at (1,-1) {(1'')};}
\end{gather*}
  \caption{The $d=2$ Dynkin diagram, fusion, and $\pi$s needed in proposition \ref{prop:reflect}.}
  \label{fig:d2}
\end{figure}

\begin{figure}
\begin{gather*}
 \tikz[baseline=0,every label/.style={rt,above},math mode,rotate=180]{\draw (0,0) node[mynode,label=c'] {} -- ++(1,0) node[mynode,label=a'] {} -- ++(1,0) node[mynode,label={[name=aa]a}] (a) {} -- ++(1,0) node[mynode,label={[name=bb]b}] (b) {} -- ++(1,0) node[mynode,fill=white,label={[name=cc]c}] (c) {}; \draw (2,0) -- (2,1) node[mynode,label={[xshift=2mm]b'}] {}; \node[rectangle,draw=green!50!black,fit=(a) (b) (c) (aa) (bb) (cc)] {};}
\\
\tikz[script math mode,baseline=0,xscale=1.2]{\dthree{l}{k}{j}{0}{0}{0} \node[framed] at (1,1) (w1) {n}; \draw (v1) -- (w1); \node at (1,-1) {(1)};}
\quad\times\quad
\tikz[script math mode,baseline=0,xscale=1.2]{\dthree{2n}{2n+l}{\ss 2n+\atop\ss l+k}{\ss n+\atop\ss l+j}{l}{n+k} \node[framed] at (1,1) (w1) {n}; \draw (v1) -- (w1); \node at (1,-1) {(2)};}
\qquad\qquad\qquad\qquad
\\
\qquad\longrightarrow\quad
\tikz[script math mode,baseline=0,xscale=1.2]{\dthree{2n+l}{\ss 2n+\atop\ss l+k}{\ss 2n+l\atop\ss +k+j}{\ss n+\atop\ss l+j}{l}{n+k} \node[framed] at (1,1) (w1) {2n}; \draw (v1) -- (w1);}
\quad\longrightarrow \quad
\tikz[script math mode,baseline=0,xscale=1.2]{\dthree{l}{l+k}{\ss l+k\atop\ss +j}{l+j}{l}{k} \node[framed] at (5,1) (w5) {n}; \draw (v5) -- (w5); \node at (1,-1) {(3)};}
\\[2mm]
\tikz[script math mode,baseline=0,xscale=1.2]{\dthree{2n}{2n+l}{\ss 2n+\atop\ss l+k}{\ss n+\atop\ss l+j}{l}{n+k} \node[framed] at (1,1) (w1) {n}; \draw (v1) -- (w1); \node at (1,-1) {(2)};}
\quad=\quad
\junk{
  \begin{matrix}
    \rr c \rr a \rr b \rr {a'} \rr c \rr a \quad \quad \quad \\
    \quad \rr b \rr {a'} \rr {b'} \rr a \rr b \rr c \quad \quad \\
    \quad\quad \rr {a'}  \rr {c'} \rr {a'} \rr {b'} \rr a \rr b \quad \\
    \quad\quad\quad \rr {a'} \rr {b'} \rr a \rr {a'} \rr {b'}
  \end{matrix}
}
(\rr c \rr b \rr a \rr {a'} \rr {b'} \rr {c'})^4 \quad
\tikz[script math mode,baseline=0,xscale=1.2]
{\dthree{l}{k}{j}{0}{0}{0} 
\node[framed] at (1,1) (w1) {n}; \draw (v1) -- (w1); \node at (1,-1) {(1)};}
\\[2mm]
\tikz[script math mode,baseline=0,xscale=1.2]{\dthree{l}{l+k}{\ss l+k\atop\ss +j}{l+j}{l}{k} \node[framed] at (5,1) (w5) {n}; \draw (v5) -- (w5); \node at (1,-1) {(3)};}
\quad=\quad
\begin{matrix}
\rr c \rr b \rr a \rr {a'} \rr b \quad \\
\quad \rr c \rr b \rr {b'} \rr a \rr b \rr {b'}
\end{matrix}\
\tikz[script math mode,baseline=0,xscale=1.2]
{\dthree{0}{0}{j}{k}{l}{0} \node[framed] at (5,1) (w5) {n}; \draw (v5) -- (w5); \node at (1,-1) {(1')};}
\end{gather*}
\caption{The $d=3$ Dynkin diagram, fusion, and $\pi$s needed in proposition \ref{prop:reflect}.}
  \label{fig:d3}
\end{figure}

\begin{figure}
\begin{gather*}
 \tikz[baseline=0,every label/.style={rt,above},math mode,rotate=180]{\draw (0,0) node[mynode,label=c'] {} -- ++(1,0) node[mynode,label=b'] {} -- ++(1,0) node[mynode,label={a'}] (a') {} -- ++(1,0) node[mynode,label={[name=aa]a}] {} -- ++(1,0) node[mynode,label=b] (b) {} -- ++(1,0) node[mynode,label=c] (c) {} -- ++(1,0) node[mynode,fill=white,label=d] (d) {}; \draw (2,0) -- (2,1) node[mynode,label={[xshift=2mm]d'}] {}; 
\node[rectangle,draw=green!50!black,fit=(aa) (b) (c) (d)] {};
}
\\
\begin{array}{ccc}
\tikz[script math mode,baseline=0,xscale=1.25]{\dfour{m}{l}{k}{j}{0}{0}{0}{0} \node[framed] at (1,1) (w1) {n}; \draw (v1) -- (w1); \node at (1,-1) {(1)};}
&
\times
&
\tikz[script math mode,baseline=0,xscale=1.25]{
\dfour{3n}{4n+m}{\ss 5n+\atop\ss m+l}{\ss 6n+m\atop\ss +l+k}
{\ss 7n+m\atop\ss +l+k+j}{\ss 5n+m\atop\ss +k+j}{\ss 2n+\atop\ss m+j}{\ss 3n +\atop\ss m+k} 
\node[framed] at (1,1) (w1) {n}; \draw (v1) -- (w1); \node at (1,-1) {(2)};}
\\[1.3cm]
\hspace{5cm}\big\downarrow
\\
\tikz[script math mode,baseline=0,xscale=1.25]{\dfour{3n+m}{\ss 4n+\atop\ss m+l}
{\ss 5n+m\atop\ss +l+k}{\ss 6n+m\atop\ss +l+k+j}{\ss 7n+m\atop\ss +l+k+j}
{\ss 5n+m\atop\ss +k+j}{\ss 2n+\atop\ss m+j}{\ss 3n +\atop\ss m+k} 
 \node[framed] at (1,1) (w1) {2n}; \draw (v1) -- (w1);}
&
\longrightarrow
&
\tikz[script math mode,baseline=0,xscale=1.25]{\dfour{n+m}{\ss n+\atop\ss m+l}
{\ss n+m\atop\ss +l+k}{\ss n+m+\atop\ss l+k+j}{\ss n+m+\atop\ss l+k+j}
{\ss n+m\atop\ss +k+j}{m+j}{m+k} 
 \node[framed] at (1,1) (w1) {n}; \draw (v1) -- (w1); \node at (1,-1) {(3)};}
\end{array}
\\[2mm]
\tikz[script math mode,baseline=0,xscale=1.25]{
\dfour{3n}{4n+m}{\ss 5n+\atop\ss m+l}{\ss 6n+m\atop\ss +l+k}
{\ss 7n+m\atop\ss +l+k+j}{\ss 5n+m\atop\ss +k+j}{\ss 2n+\atop\ss m+j}{\ss 3n+\atop\ss m+k} 
\node[framed] at (1,1) (w1) {n}; \draw (v1) -- (w1); \node at (1,-1) {(2)};}
\ = \ 
(  \rr {d'} \rr {a'} \rr a \rr b \rr c \rr d \rr {c'} \rr {b'} )^{10}
\ \tikz[script math mode,baseline=0,xscale=.7] 
{\dfour{m}{l}{k}{j}{0}{0}{0}{0} \node[framed] at (1,1) (w1) {n}; \draw
  (v1) -- (w1); \node at (1,-1) {(1)};}
\\[2mm]
\tikz[script math mode,baseline=0,xscale=1.25]{\dfour{n+m}{\ss n+\atop\ss m+l}
{\ss n+m\atop\ss +l+k}{\ss n+m+\atop\ss l+k+j}{\ss n+m+\atop\ss l+k+j}
{\ss n+m\atop\ss +k+j}{m+j}{m+k} 
 \node[framed] at (1,1) (w1) {n}; \draw (v1) -- (w1); \node at (1,-1) {(3)};}
\ = 
( \rr {b'} \rr {c'} \rr d \rr c \rr b \rr a \rr {a'}  \rr {d'} )^{5}
 \tikz[script math mode,baseline=0,xscale=.7] 
{\dfour{m}{l}{k}{j}{0}{0}{0}{0} \node[framed] at (1,1) (w1) {n}; \draw
  (v1) -- (w1); \node at (1,-1) {(1)};}
\end{gather*}
\caption{The $d=4$ Dynkin diagram, fusion, and $\pi$s needed in
  proposition \ref{prop:reflect}. }   \label{fig:d4}
\end{figure}

\subsection{The representation theory}\label{ssec:reptheory}

Let $\lie x_{2d}$ be one of the Lie algebras
$\lie a_2, \lie d_4, \lie e_6, \lie e_8$ corresponding to the Dynkin
diagrams $X_{2d}$ introduced in figures \ref{fig:d1}-\ref{fig:d4},
and $\xqg$ the corresponding {\em quantized loop algebra};
see Appendix~\ref{app:qg} for details.
\junk{PZJ: some comment about $\ZZ$ form? otherwise we need to tensor our
  $K_{\hat T}$ with $\CC$. AK: I don't understand -- the $\CC$-reps of 
  $\CC \tensor_{\ZZ} A$ are the same as those of $A$, so why need we tensor?
  you're right, my mistake. I guess the related issue is, do we tensor
  $K_T(\mathcal M)$ with $\CC$? AK: I don't see any need, since the
  reps are the same}
\junk{AK: This is an ill-defined question, probably, but do we really
  want to set the Cartan element to $1$ not $0$? Certainly I'd call
  these ``level zero'' reps}

In each of figures~\ref{fig:d1}--\ref{fig:d4}, 
the first, second, and fourth quivers each have only one framed vertex.
Denote by $V_a(z)$ the fundamental representation \cite[p399]{CP-book}
associated to that vertex, where $a=1,2,3$ for the 
first, second, and fourth quiver (i.e., the one labeled ($a$)) respectively.
\junk{
  $\calM(w_a,v_a)$, $a=1,2,3$;
  equivalently, $V_a(z):=K_{\hat T}(\bigsqcup_v\calM(w_a,v))$ where $w_a$ is as
  in figures~\ref{fig:d1}--\ref{fig:d4} at $n=1$ (i.e., the associated
  dominant weight is a single fundamental weight), but $v$ is unconstrained 
  (and $z$ is the equivariant parameter associated to the Cartan $T=GL(1)$). 
  Then $K_{\hat T}(\bigsqcup_v\calM(w_a,v)) \cong \bigotimes_{k=1}^n V_a(z_k)$ for
  general $n$.
}
All three representations $V_a(z)$ have dimension $3,8,27,248+1$ for
$d=1,2,3,4$ respectively -- in the last case, the representation is
{\em not}\/ irreducible for the finite quantized algebra
$\Uq(\mathfrak e_8)$, but is rather a direct sum of the fundamental
representation associated to the same vertex ($q$-deformation of the
adjoint representation of $e_8$) and of the trivial representation.

Each $V_a(z)$ possesses a natural weight space decomposition; furthermore, 
every weight space except the zero weight space at $d=4$ is one-dimensional.
We denote the set of weights of $V_1(z)$ by $\W$; the set of weights
of $V_2(z)$ (resp.\ $V_3(z)$) is then given by $\tau^2\W$ (resp.\ $-\tau\W$).
\junk{AK: maybe this $\tau$-relatedness can be relegated to a
  footnote that talks about the connection to paper \#1, and $\tau$ can
  otherwise be excised from this paper. Is it really needed in
  the proof of lemma \ref{lem:easy}?}

We are particularly interested in the subspaces $V_a^A(z)$ of $V_a(z)$
defined by restricting to weight spaces
given by figures~\ref{fig:d1}--\ref{fig:d4}
at $n=1$; that is, $w$ corresponds
to a single fundamental representation as above, and $v$ is constrained to 
be of the form given by quiver $(1)$ of figures~\ref{fig:d1}--\ref{fig:d4}.
We pick basis elements (weight vectors) in each weight space; the normalization
is unimportant for now and will be fixed by the geometry in \S\ref{sec:Nakajima}. We denote them
$e_{a,i}$, $i=0,\ldots,d$, where $e_{a,d}$ is the highest weight vector 
and the weight of $e_{a,i}$ is obtained from the highest weight by subtracting $d-i$ simple roots.%
\begin{ex}
At $d=3$, $a=1$, there are four possible assignments
$v_a$ of dimensions, namely
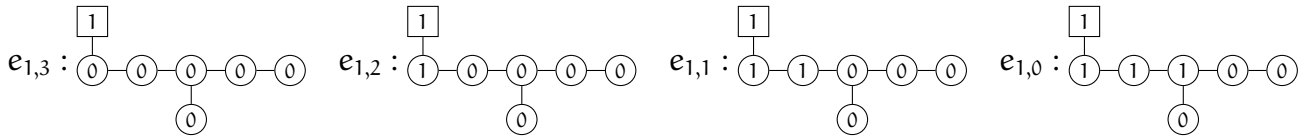
\[
e_{1,3}:
\tikz[script math mode,baseline=0,scale=0.65]{\dthree{0}{0}{0}{0}{0}{0} \node[framed] at (1,1) (w1) {1}; \draw (v1) -- (w1);}
\quad
e_{1,2}:
\tikz[script math mode,baseline=0,scale=0.65]{\dthree{1}{0}{0}{0}{0}{0} \node[framed] at (1,1) (w1) {1}; \draw (v1) -- (w1);}
\quad
e_{1,1}:
\tikz[script math mode,baseline=0,scale=0.65]{\dthree{1}{1}{0}{0}{0}{0} \node[framed] at (1,1) (w1) {1}; \draw (v1) -- (w1);}
\quad
e_{1,0}:
\tikz[script math mode,baseline=0,scale=0.65]{\dthree{1}{1}{1}{0}{0}{0} \node[framed] at (1,1) (w1) {1}; \draw (v1) -- (w1);}
\]
while at $d=2$, $a=2$, there are three possibilities:
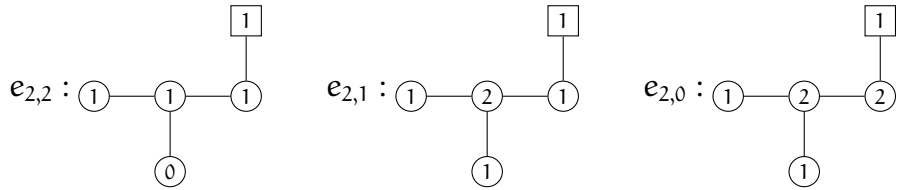
\[
e_{2,2}:
\tikz[script math mode,baseline=0]{\dtwo{1}{1}{1}{0} \node[framed] at (3,1) (w3) {1}; \draw (v3) -- (w3);}
\qquad
e_{2,1}:
\tikz[script math mode,baseline=0]{\dtwo{1}{2}{1}{1} \node[framed] at (3,1) (w3) {1}; \draw (v3) -- (w3);}
\qquad
e_{2,0}:
\tikz[script math mode,baseline=0]{\dtwo{1}{2}{2}{1} \node[framed] at (3,1) (w3) {1}; \draw (v3) -- (w3);}
\]
\end{ex}
\rem[gray]{our convention is, $v_{1,d}$ is the highest weight vector.
careful that the basis is ORDERED. important because trig $R$-matrix is NOT invariant by permutation (Weyl).
must be related to the choice of chamber}
The weights of $e_{a,i}$ for $a=1,2,3$ are denoted $\vf_i$, $\tau^2\vf_i$, $-\tau \vf_i$;
the $\vf_i$ form a subset $\W^A$ of $\W$. We can also define dual weight vectors $e_{a,i}^*$ with $\left< e_{a,i}^*, e_{a,j}\right>=\delta_{ij}$, and opposite weights.
\rem[gray]{note a subtlety, at $d=4$ there isn't a stable basis because of
the zero weight space.
there is however in the single number sector, which is all we care about}

Note that at $d=1,3$, diagrams (1) and (2) correspond to (weight spaces
of) the same representation
of $\xqg$, whereas (3) corresponds to the dual
representation -- it is related by $W$ action to a diagram (1') that is the
mirror image of (1). At $d=2$, diagrams (1), (2) and (3) correspond to
nonisomorphic representations related by triality. Finally, at $d=4$,
the same representation is associated to diagrams (1), (2) and (3).

\subsection{The $R$-matrices}\label{ssec:Rmat}
\rem[gray]{in this paper, we identify once and for all a diagram and its fugacity!!!}
In the previous section, we
\junk{
  recalled that the equivariant $K$-theory of the quiver varieties 
  $\bigsqcup_v\calM(w_a,v)$ comes equipped with an action
}
introduced fundamental representations $V_a(z)$, $a=1,2,3$, of the
quantized loop algebra $\xqg$.  We now consider tensor products of
such representations. As in \S \ref{sec:stable}, we tensor over
$\CC(q)$, and require localization. Extending that of \S \ref{sec:stable},
we invert $1-q^{2k}z''/z'$ for any $k\in\ZZ$
and any spectral parameters $z',z''$ {\em except in the following two
  cases}: if $z',z''$ appear in $V_a(z')$ and $V_b(z'')$,
\begin{enumerate}
\item[(i)] if $a=b$, we disallow $k=0$;
\item[(ii)] if $a=1$, $b=2$ we disallow $k=h_d/3$,
\end{enumerate}
where $h_d$ denotes the dual Coxeter number of $\lie x_{2d}$.

\junk{Let $V_i(z)$ be the representation of $\xqg$ with
the same underlying vector space $V(z)$, but with the action twisted by $\tau^i$.

\rem[gray]{The inclusion of finite Dynkin diagrams $\mathfrak a_d \subset \mathfrak x_{2d}$ described on
Fig.~\ref{fig:dd}
induces an inclusion of quantized loop algebras.
using Drinfeld currents. respects
inclusion of finite quantum groups. but are these *Hopf* algebra embeddings? NO. (particularly
obvious in the RTT formulation -- intermediate states are arbitrary)
and certainly the automorphisms are *not* Hopf algebra automorphisms.}}

For any $z''/z'\not\in q^{2\ZZ}$,
$V_a(z')\otimes V_b(z'')$ and
$V_b(z'')\otimes V_a(z')$ are known
to be irreducible and isomorphic (see \cite{Chari-braid} and \cite[\S3.7]{HL-cluster});
let $\check R_{a,b}(z''/z')$, $a,b=1,2,3$ be the unique $\xqg$ intertwiner from
$V_a(z')\otimes V_b(z'')$ to
$V_b(z'')\otimes V_a(z')$, with the normalization condition
\begin{equation}\label{eq:normR}
\left< e_{b,0}^*\otimes e_{a,0}^*,\check R_{a,b}(z) e_{a,0}\otimes e_{b,0}\right> 
=:\check R_{a,b}(z){}^{00}_{00}=1
\end{equation}
making $\check R_{a,b}(z)$ a rational function of $z$.
We will justify in what follows the slightly stronger statement that $\check R_{a,b}(z''/z')$ is well-defined
in the localization above, i.e., we will explain the exceptions (i) and (ii).

In \cite{MO-qg,Ok-K} Maulik and Okounkov define
$R$-matrices {\em geometrically}\/ using the stable envelope
construction, recalled here in \S\ref{sec:geominterp}.
As we shall see, our normalization \eqref{eq:normR} of the 
$R$-matrices coincides with the one defined geometrically only when $a=b$. 
For now, the choice of normalization for $a\ne b$ is ad hoc and allows
for a simpler formulation of our main theorem.  A proper explanation
(at least in cohomology) will be given in \S\ref{sec:geominterp}.

As in \S \ref{sec:stable}, we represent $R$-matrices diagrammatically
as crossings, except we now distinguish the various $V_a(z)$ by the
color of the line -- green, red, blue for $a=1,2,3$.

For $\check R_{1,2}(z)$, we also use the ``dual'' graphical depiction
which is traditional in Schubert puzzles, namely,
\[
\check R_{1,2}(z) = \rh{}{}{}{}
\]
where the parameter $z$ is implicitly determined by the location of
the rhombus, as will be discussed below; and similarly, matrix
elements in a given basis are denoted by putting labels on the edges
of the rhombus.

The following lemma is key, in that it relates the $\lie x_{2d}$
$R$-matrix we use in puzzles to the $\lie a_d$ $R$-matrix we use to
compute the motivic Segre classes on the (cotangent bundle to the) 
$d$-step flag manifold.
The lemma is an $R$-matrix analogue of the familiar statement that for 
a circle $S \into G$, the two extremal $S$-weight spaces in a $G$-irrep are
(predictable) irreps of the Levi subgroup $C_G(S)$.  As such, this
lemma should have a purely representation-theoretic proof,
\footnote{A sketch of proof is that inclusions of Dynkin diagrams $A_d\subset X_{2d}$
induce inclusions of loop algebras $\aqg \subset \xqg$ by
the Drinfield current presentation;
this immediately implies the property at $a=1$, where the $e_{1,i}$ are the standard basis
of the $(d+1)$-dimensional representation of $\Uq(\lie{a}_d)$, and $\check R_{1,1}$ is uniquely fixed
up to normalization by $\aqg$-invariance. For $a=2,3$, one uses the braid group action \cite{Beck-braid} on
$\xqg$
with any lift of $\tau$ to reduce to the case $a=1$.}
but we found it easier to exploit geometry, 
and defer the proof to \S \ref{sec:quiver} where we provide the
slightly more precise lemma \ref{lem:ressinglev2}.

\begin{lem}\label{lem:ressingle}
  The matrices of the operators $\check R_{a,a}$, $a=1,2,3$, restricted
  to $V_a^A(z_1)\otimes V_a^A(z_2)$,  in the basis $(e_{a,i})_{i=0,\ldots,d}$
  (with a normalization to be explained in \S \ref{sec:Nakajima}),
  are equal to the $R$-matrix from \eqref{eq:Rsingle}.
\end{lem}

The purpose of the following lemma is to justify exception (ii) of our localization,
and to introduce the
tensors $U$ and $D$; we don't otherwise use the result directly.
\begin{lem}\label{lem:factorR}
  At the value $z=q^{-2h_d/3}$ (where $h_d=3,6,12,30$ for $d=1,2,3,4$), 
  the matrix $\check R_{1,2}$ is well-defined (i.e., has no pole) and factors as
  \[   \check R_{1,2}(q^{-2h_d/3}) = DU   \]
  where $U: V_1(q^{h_d/3}z)\otimes V_2(q^{-h_d/3}z) \to V_{3}(z)$ and
  $D: V_{3}(z)\to V_2(q^{-h_d/3}z)\otimes V_1(q^{h_d/3}z)$ are $\xqg$ intertwiners
  (unique up to scaling one while un-scaling the other).
\end{lem}

\begin{proof}
  The $d=2$, $d=3$ $R$-matrices were provided in \cite[\S 3.6]{artic71} and
  \cite[\S 3.8]{artic71} respectively, where the factorization was already 
  pointed out.  
  The $d=1$ case is given in appendix~\ref{app:Rd1}, and its
  factorization is discussed in \S\ref{sec:exd1} below. Finally, the
  $d=4$ case requires the explicit form of the trigonometric
  $\mathfrak e_8$ $R$-matrix in its 249-dimensional representation, 
  which to the authors' knowledge has not
  appeared in the literature before.  It is provided in a 
  companion paper \cite{artic77}, where the factorization is proven.
\end{proof}

Note that the geometric meaning of setting the ratio of
spectral parameters equal to $q^{-2h_d/3}$ is not
obvious, and in fact it is only our unusual normalization of the
$R$-matrix (as mentioned above) which guarantees that
$\check R_{1,2}(z)$ is well-defined at this value of $z$.
\junk{careful that reps \sout{have been} will be defined geometrically, 
which means we don't really have freedom of
shifting the spectral parameter. overall, there's no easy proof of this lemma
except explicit calculation. in principle the construction of \S 6
would also provide it except it's in $H_T$ and probably not rigorous enough. 
At least it should fix the correct values of the spectral parameters
-- unless there's some arbitrariness?}

The graphical notation and its dual are
\begin{equation}\label{eq:factorR}
  \check R_{1,2}(q^{-2h_d/3})=\tikz[scale=0.75,baseline=(current  bounding  box.center)]{\draw[dr,invarrow=0.25,invarrow=0.75] (210:1) -- (0,0) (0,1) -- ++(30:1); \draw[dg,invarrow=0.25,invarrow=0.75] (-30:1) -- (0,0)  (0,1) -- ++(150:1); \draw[db,invarrow=0.55] (0,0) node[triv] {} -- (0,1) node[triv] {}; }
  =
  \tikz[baseline=-1mm]{\uptri{}{}{}\downtri{}{}{}}
\end{equation}
where the \Deltatri\ corresponds to $U$, and the \nablatri\ to $D$. 
Because of the normalization condition \eqref{eq:normR}, we can 
fix the scaling of $U$ and $D$ to have
\begin{equation}\label{eq:normUD}
\left< e^*_{3,0}, U e_{1,0}\otimes e_{2,0} \right>
=
\left<e^*_{2,0}\otimes e^*_{1,0}, D e_{3,0}\right>
=1
\end{equation}

These $R$-matrices collectively satisfy the following identities
(cf.~\cite[property~2]{artic71}):
\begin{prop}\label{prop:ybe}
  In the pictures below, black lines can have arbitrary (independent)
  colors, and all lines can have arbitrary spectral parameters (as
  long as they match between l.h.s.\ and r.h.s.).
\begin{itemize}
\item Weight conservation:
matrix entries (in a basis of weight vectors) 
of all $\check R_{i,j}$, $U$, $D$ are nonzero 
only if the sum of incoming weights is equal to the sum of outgoing weights.
\item Yang--Baxter equation:
\begin{equation}\label{eq:ybe}
\begin{tikzpicture}[baseline=-3pt,y=2cm]
\draw[d,arrow=0.1,arrow=0.4,arrow=0.7,rounded corners=4mm] (-0.5,0.5) -- (0.75,0) -- (1.5,-0.5) (0.5,0.5) -- (0.2,0) -- (0.5,-0.5) (1.5,0.5) -- (0.75,0) -- (-0.5,-0.5);
\end{tikzpicture}
=
\begin{tikzpicture}[baseline=-3pt,y=2cm]
\draw[d,arrow=0.1,arrow=0.4,arrow=0.7,rounded corners=4mm] (-0.5,0.5) -- (0.25,0) -- (1.5,-0.5) (0.5,0.5) -- (0.8,0) -- (0.5,-0.5) (1.5,0.5) -- (0.25,0) -- (-0.5,-0.5);
\end{tikzpicture}
\end{equation}
\item Bootstrap equations: 
\begin{align}\label{eq:qtri}
&\tikz[baseline=0]{
\draw[invarrow=0.5,dg] (0,0) -- (-1,1);
\draw[invarrow=0.5,dr] (0,0) -- (1,1);
\draw[invarrow=0.5,db] (0,-1.5) -- (0,0) node[triv] {};
\draw[invarrow=0.5,rounded corners,d] (-1,-1.5) -- (-1,-0.5) -- (1.5,0) -- (1.5,1);
}
=
\tikz[baseline=0]{
\draw[invarrow=0.6,dg] (0,-0.5) -- (-1,1);
\draw[invarrow=0.75,dr] (0,-0.5) -- (1,1);
\draw[invarrow=0.6,db] (0,-1.5) -- (0,-0.5) node[triv] {};
\draw[invarrow=0.6,rounded corners,d] (-1,-1.5) -- (-1,0) -- (1.5,0.5) -- (1.5,1);
}
&
&\tikz[baseline=0,xscale=-1]{
\draw[invarrow=0.5,dr] (0,0) -- (-1,1);
\draw[invarrow=0.5,dg] (0,0) -- (1,1);
\draw[invarrow=0.5,db] (0,-1.5) -- (0,0) node[triv] {};
\draw[invarrow=0.5,rounded corners,d] (-1,-1.5) -- (-1,-0.5) -- (1.5,0) -- (1.5,1);
}
=
\tikz[baseline=0,xscale=-1]{
\draw[invarrow=0.6,dr] (0,-0.5) -- (-1,1);
\draw[invarrow=0.75,dg] (0,-0.5) -- (1,1);
\draw[invarrow=0.6,db] (0,-1.5) -- (0,-0.5) node[triv] {};
\draw[invarrow=0.6,rounded corners,d] (-1,-1.5) -- (-1,0) -- (1.5,0.5) -- (1.5,1);
}
\\[4mm]\label{eq:qtrirev}
&\tikz[baseline=0,scale=-1]{
\draw[arrow=0.5,dg] (0,0) -- (-1,1);
\draw[arrow=0.5,dr] (0,0) -- (1,1);
\draw[arrow=0.5,db] (0,-1.5) -- (0,0) node[triv] {};
\draw[arrow=0.5,rounded corners,d] (-1,-1.5) -- (-1,-0.5) -- (1.5,0) -- (1.5,1);
}
=
\tikz[baseline=0,scale=-1]{
\draw[arrow=0.6,dg] (0,-0.5) -- (-1,1);
\draw[arrow=0.75,dr] (0,-0.5) -- (1,1);
\draw[arrow=0.6,db] (0,-1.5) -- (0,-0.5) node[triv] {};
\draw[arrow=0.6,rounded corners,d] (-1,-1.5) -- (-1,0) -- (1.5,0.5) -- (1.5,1);
}&
&\tikz[baseline=0,yscale=-1]{
\draw[arrow=0.5,dr] (0,0) -- (-1,1);
\draw[arrow=0.5,dg] (0,0) -- (1,1);
\draw[arrow=0.5,db] (0,-1.5) -- (0,0) node[triv] {};
\draw[arrow=0.5,rounded corners,d] (-1,-1.5) -- (-1,-0.5) -- (1.5,0) -- (1.5,1);
}
=
\tikz[baseline=0,yscale=-1]{
\draw[arrow=0.6,dr] (0,-0.5) -- (-1,1);
\draw[arrow=0.75,dg] (0,-0.5) -- (1,1);
\draw[arrow=0.6,db] (0,-1.5) -- (0,-0.5) node[triv] {};
\draw[arrow=0.6,rounded corners,d] (-1,-1.5) -- (-1,0) -- (1.5,0.5) -- (1.5,1);
}
\end{align}
\item Unitarity equation: 
\begin{equation}\label{eq:unit}
\begin{tikzpicture}[baseline=-3pt]
\draw[arrow=0.07,arrow=0.57,rounded corners,d] (-0.5,1) -- (0.5,0) -- (-0.5,-1) (0.5,1) -- (-0.5,0) -- (0.5,-1);
\end{tikzpicture}
=
\begin{tikzpicture}[baseline=-3pt]
\draw[arrow=0.1,arrow=0.6,rounded corners,d] (-0.5,1) -- (-0.5,-1) (0.5,1) -- (0.5,-1);
\end{tikzpicture}
\end{equation}
\item Value at equal spectral parameters: (here, the lines {\em must}\/ have the same color
for the equality to make sense)
\begin{equation}\label{eq:equal}
\begin{tikzpicture}[baseline=-3pt,yscale=1.5]
\draw[invarrow=0.3,d] (-0.5,-0.5) -- node[left,pos=0.3] {$\ss z$} (0.5,0.5);
\draw[invarrow=0.3,d] (0.5,-0.5) -- node[right,pos=0.3] {$\ss z$} (-0.5,0.5);
\end{tikzpicture}
=
\begin{tikzpicture}[baseline=-3pt,yscale=1.5]
\draw[invarrow=0.3,rounded corners=4mm,d] (-0.5,-0.5) -- node[left] {$\ss z$} (0,0) -- (-0.5,0.5);
\draw[invarrow=0.3,rounded corners=4mm,d] (0.5,-0.5) -- node[right] {$\ss z$} (0,0) -- (0.5,0.5);
\end{tikzpicture}
\end{equation}
\end{itemize}
\end{prop}
\begin{proof}
The first property is simply a reformulation of the fact that $R$-matrices
commute with the action of the Cartan torus of $\xqg$.
This property will be used repeatedly in what follows, sometimes
requiring a careful check for $d=1,2,3,4$ (possibly by computer).

All other properties can be proven as follows: first note that the pictures
represent intertwiners between two $\xqg$-modules which are irreducible
(for our generic choice of spectral parameters;
for \eqref{eq:equal}, use \cite{Chari-braid} or \cite[Prop.~6.15]{FM-qchar}). Therefore by Schur's lemma, l.h.s.\ and
r.h.s.\ must be proportional. In order to fix the constant of proportionality,
we now impose the ``boundary conditions 0'' on the pictures, i.e., compute
the matrix elements of these intertwiners between vectors of the form
of a tensor product of $e_{a,0}$. Using the first property, one can check
that $0$s propagate throughout the diagrams, so that we only get entries
of the form of \eqref{eq:normR} or \eqref{eq:normUD}, which are all equal
to $1$.

\end{proof}

\eqref{eq:equal} also justifies exception (i) in our localization: $\cR_{a,a}(z)$ is well-defined 
(and equal to $1$) at $z=1$.

\subsection{The main theorem}\label{sec:main}
Given three strings $\lambda,\mu,\nu$ of same content, we define
\begin{equation}\label{eq:defpuz}
\tikz[scale=1.8,baseline=0.5cm]{\uptri{\lambda}{\nu}{\mu}}
:=
\left< e^*_\nu, 
\overbrace{U\ldots U}^{n}\, \overbrace{\check R\ldots \check R}^{n\choose 2}
 \, e_\lambda\otimes e_\mu\right>
\end{equation}
to be the matrix entry of the product of $R$-matrices and $U$-matrices 
forming a puzzle, where
$e_\lambda := \bigotimes_{k=1}^n e_{1,\lambda_k}$ and so on; e.g., for
$\lambda = 0102$, $\mu = 0201$, $\nu = 0210$,
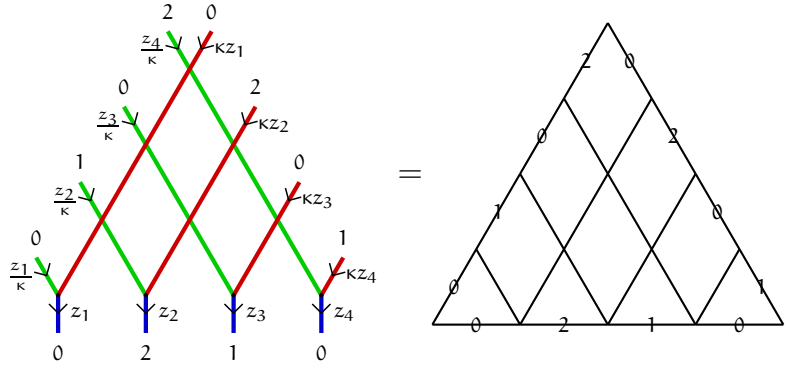
\[
\begin{tikzpicture}[math mode,nodes={\mcol},x={(-0.577cm,-1cm)},y={(0.577cm,-1cm)},baseline=(current  bounding  box.center)]
\draw[dg,arrow=0.07] (0.5,0) node[above] {\ss 2} -- node[left,pos=0.07] {\ss \frac{z_4}{\kappa}} (0.5,3.5);
\draw[dg,arrow=0.1] (1.5,0) node[above] {\ss 0} -- node[left,pos=0.1] {\ss \frac{z_3}{\kappa}} (1.5,2.5);
\draw[dg,arrow=0.17] (2.5,0) node[above] {\ss 1} -- node[left,pos=0.17] {\ss \frac{z_2}{\kappa}} (2.5,1.5);
\draw[dg,arrow] (3.5,0) node[above] {\ss 0} -- node[left] {\ss \frac{z_1}{\kappa}} (3.5,0.5);
\draw[dr,arrow=0.07] (0,0.5) node[above] {\ss 0} -- node[right,pos=0.07] {\ss \kappa z_1} (3.5,0.5);
\draw[dr,arrow=0.1] (0,1.5) node[above] {\ss 2} -- node[right,pos=0.1] {\ss \kappa z_2} (2.5,1.5);
\draw[dr,arrow=0.17] (0,2.5) node[above] {\ss 0} -- node[right,pos=0.17] {\ss \kappa z_3} (1.5,2.5);
\draw[dr,arrow] (0,3.5) node[above] {\ss 1} -- node[right] {\ss \kappa z_4} (0.5,3.5);
\draw[db,arrow] (0.5,3.5) node[triv] {} -- node[right] {\ss z_4}  ++(0.25,0.25) node[below] {\ss 0};
\draw[db,arrow] (1.5,2.5) node[triv] {} -- node[right] {\ss z_3} ++(0.25,0.25) node[below] {\ss 1};
\draw[db,arrow] (2.5,1.5) node[triv] {} -- node[right] {\ss z_2} ++(0.25,0.25) node[below] {\ss 2};
\draw[db,arrow] (3.5,0.5) node[triv] {} -- node[right] {\ss z_1} ++(0.25,0.25) node[below] {\ss 0};
\end{tikzpicture}
=
\begin{tikzpicture}[math mode,nodes={\mcol},x={(-0.577cm,-1cm)},y={(0.577cm,-1cm)},baseline=(current  bounding  box.center)]
\draw[thick] (0,0) -- node[pos=0.5] {\ss 0} ++(0,1); \draw[thick] (0,0) -- node[pos=0.5] {\ss 2} ++(1,0);
\draw[thick] (0,1) -- node[pos=0.5] {\ss 2} ++(0,1); \draw[thick] (0,1) -- node[pos=0.5] {} ++(1,0); 
\draw[thick] (0,2) -- node[pos=0.5] {\ss 0} ++(0,1); \draw[thick] (0,2) -- node[pos=0.5] {} ++(1,0); 
\draw[thick] (0,3) -- node[pos=0.5] {\ss 1} ++(0,1); \draw[thick] (0,3) -- node[pos=0.5] {} ++(1,0); \draw[thick] (0+1,3) -- node {\ss 0} ++(-1,1); 
\draw[thick] (1,0) -- node[pos=0.5] {} ++(0,1); \draw[thick] (1,0) -- node[pos=0.5] {\ss 0} ++(1,0);
\draw[thick] (1,1) -- node[pos=0.5] {} ++(0,1); \draw[thick] (1,1) -- node[pos=0.5] {} ++(1,0);
\draw[thick] (1,2) -- node[pos=0.5] {} ++(0,1); \draw[thick] (1,2) -- node[pos=0.5] {} ++(1,0); \draw[thick] (1+1,2) -- node {\ss 1} ++(-1,1); 
\draw[thick] (2,0) -- node[pos=0.5] {} ++(0,1); \draw[thick] (2,0) -- node[pos=0.5] {\ss 1} ++(1,0);
\draw[thick] (2,1) -- node[pos=0.5] {} ++(0,1); \draw[thick] (2,1) -- node[pos=0.5] {} ++(1,0); \draw[thick] (2+1,1) -- node {\ss 2} ++(-1,1); 
\draw[thick] (3,0) -- node[pos=0.5] {} ++(0,1); \draw[thick] (3,0) -- node[pos=0.5] {\ss 0} ++(1,0); \draw[thick] (3+1,0) -- node {\ss 0} ++(-1,1); 
\end{tikzpicture}
\]
where $\kappa=q^{-h_d/3}$. 
Note that $e_\lambda\otimes e_\mu$ is viewed as a basis element of $V_1(\kappa^{-1}z_1)\otimes \cdots V_1(\kappa^{-1}z_n)
\otimes V_2(\kappa z_1)\otimes\cdots\otimes V_2(\kappa z_n)$, so that, just like for lemma~\ref{lem:factorR}, we have to restrict localization
so it allows the ratio $\kappa^2=q^{-2h_d/3}$ for these specific pairs of spectral parameters.
\junk{careful with possible localization issues: in
  particular need to justify already earlier existence of triangle}

Traditionally, a (Schubert) puzzle is an assignment of certain labels to 
{\em every}\/ edge of the picture on the r.h.s.\ above. This point of view can
be recovered by fixing bases of the spaces $V_a(z)$ and then inserting
decompositions of the identity at every edge. The expression above is then
the sum over all possible puzzles of their ``fugacity'' (in the language of
\cite{artic71}), i.e., of the corresponding matrix entries. 
Our bases are always made of weight vectors,
which fixes them uniquely up to normalization -- except at $d=4$ where
the zero weight space has dimension $9$, an issue which is deferred to 
\S\ref{ssec:d4}. Whenever we use the term ``puzzle'' below, we implicitly
assume that such a choice of basis of weight vectors has been made. We
then label edges using (nonzero) weights, with the shortcut notation used 
above that the weights $\vf_i$ on the NW side 
(resp.\ $\tau^2\vf_i$ on the NE side, $-\tau \vf_i$ on the S side) 
are abbreviated as ``$i$''.\footnote{To avoid any risk of confusion between such labels ``$i$'' and the original weights,
all our weights have arrows, as e.g.~$\vec f_i$.}

\begin{lem}\label{lem:easy}
Suppose one is given 
\[ 
\Phi=\tikz[scale=1.8,baseline=0.5cm]{\uptri{\lambda}{}{\mu}}
:=
\overbrace{U\ldots U}^{n}\, \overbrace{\check R\ldots \check R}^{n\choose 2}
 \, e_\lambda\otimes e_\mu \in \bigotimes_{k=1}^n V_3(z_k)
\]
where the product of $R$-matrices
and $U$-matrices forms as in \eqref{eq:defpuz} a puzzle, and $\lambda$ and $\mu$ are strings with the same content.
Then $\Phi\in\bigotimes_{k=1}^n V^A_3(z_k)$,
and more precisely is a linear combination of $e_\nu=\bigotimes_{k=1}^n e_{3,\nu_k}$ where $\nu$ runs over strings of the same content as $\lambda$ and $\mu$.
\end{lem}

Combinatorially, this is the statement ``if the NW and NE edges are
labeled only with single-numbers (not the more general multinumbers that
appear within), then the S edge is necessarily also labeled with
single-numbers''. 

\begin{proof}
  This is a direct consequence of the first part of
  lemma~\ref{lem:ressingle} and of the weight conservation of
  proposition~\ref{prop:ybe}, noting that the weight of $e_\mu$ is
  $\tau^2$ times the weight of $e_\lambda$, so that the weight of
  $\Phi$ is $-\tau$ times that of $e_\lambda$.
\end{proof}

\begin{prop}\label{prop:trivpuzzle}
  There is a unique puzzle
  \tikz[scale=1.8,baseline=0.5cm]{\uptri{\lambda}{\omega}{\mu}} with a
  weakly increasing string $\omega$ at the bottom (where $\lambda$ and
  $\mu$ are {\em a priori}\/ arbitrary strings of the same content);
  its other two sides are also labeled by $\lambda=\mu=\omega$, and
  its labels on non-horizontal edges are constant along each diagonal
  (NW/SE or NE/SW).
\end{prop}

\begin{proof}
For $d\le 3$ this was already proven in \cite[prop.~3.4]{artic71}. 
We provide here a different proof, which works at $d=4$.
\rem[gray]{actually, \cite[prop.~3.4]{artic71} is stronger in that it
  doesn't assume anything on the content of $\lambda$ and $\mu$; 
  but we don't really need that}

The proof is based on weight conservation, and on scalar products of weights.
We normalize the Killing form $\killing{\cdot}{\cdot}$ so that the norm of weights is $2$ for $d\le 3$, and $2$ or $0$ at $d=4$; see in particular \cite[\S 2.4]{artic71}
where the scalar product of $\vf_i$ is computed explicitly (in the notations there,
$a=3,2,3/2,1$ for $d=1,2,3,4$; see also appendix~\ref{app:scal}).

Let $i$ be the first entry of $\omega$. Because $\omega$ is weakly increasing, all labels of $\lambda,\mu,\omega$ are greater or equal to $i$. We have, for $j>i$,
\[
\killing{\vf_j+\tau \vf_i}{\vf_j+\tau \vf_i}=4+2 \killing{\tau\vf_i}{\vf_j}
=8,6,5,4> 2\qquad \text{for $d=1,2,3,4$ respectively}
\]
so triangles \uptri{j}{i}{} with $j\ge i$ only exist for $j=i$, 
and we are led to the following partial filling of our puzzle $P$:
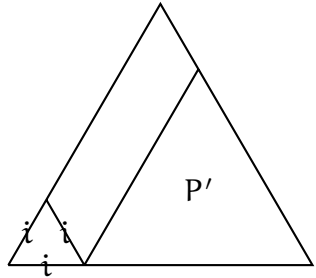
\begin{center}
\begin{tikzpicture}
\draw[thick,black] (0,0) -- node[edgelabel] {$i$} ++(1,0) -- ++(3,0)
-- ++(120:3) -- node[edgelabel] {} ++(120:1) -- ++(240:3) -- node[edgelabel] {$i$} ++(240:1)
(1,0) -- ++(60:3) (1,0) -- node[edgelabel] {$i$} ++(120:1);
\node at (2.5,1) {$P'$};
\end{tikzpicture}
\end{center}
We now want to determine the labels of the rhombi forming the left column. One has the following lemma:
\begin{lem}\label{lem:tricky}
For $0\le i\le d$, let $\mathcal A_i$ be the set of weights $\{\tau^2\vf_i,-\vf_i,\ldots,-\vf_4\}$ at $d=4$,
and $\mathcal A_i=\{\tau^2\vf_i\}$ for $d\le 3$.
Given $j\ge i$, $\vX\in \mathcal A_i$ and $\vY\in\W$, $\vZ\in\tau^2\W$,
the equality $\vX+\vY=\vZ+\vf_j$ (which is a necessary condition for the existence of a rhombus of the form  \rh{\vX}{\vY}{\vZ}{j} according to the first
point of proposition~\ref{prop:ybe}) implies $\vZ\in \mathcal A_i$. Furthermore, if $\vZ=\tau^2\vf_i$, then $\vX=\tau^2\vf_i$.
\end{lem}
\begin{proof}
In what follows, we use the following: for all $\vX\in\tau^2\W$, $\vY\in \W$, one has, for $d=4$, $-2\le \killing{\vX}{\vY}\le 2$, and for
$d\le 3$ the stronger bound $-1\le \killing{\vX}{\vY}\le 2$
(the former follows from Cauchy--Schwarz; 
noting that $\killing{\vX}{\vY}=-2$ implies that $V_1$ and $V_2$ are dual
of each other as $\mathfrak{x}_{2d}$ representations, 
which only occurs at $d=4$, results in the latter).

Given $j\ge i$, $\vX\in \mathcal A_i$ and $\vY\in\W$, $\vZ\in\tau^2\W$, noting $\killing{\vX}{\vX}=\killing{\vf_j}{\vf_j}=2$,
\begin{align}\label{eq:scalprod}
\killing{\vX}{\vZ}&=\frac{1}{2}\left(\killing{\vX}{\vX}+\killing{\vZ}{\vZ}-\killing{\vX-\vZ}{\vX-\vZ}\right)
\\\notag
&=1+\frac{1}{2}\left(\killing{\vX+\vY-\vf_j}{\vX+\vY-\vf_j}-\killing{\vY-\vf_j}{\vY-\vf_j}\right)
\\\notag
&=2+\killing{\vX}{\vY}-\killing{\vX}{\vf_j}
\end{align}
Let us first assume $\vX=\tau^2 \vf_i$. Then $\killing{\vX}{\vf_j}=-1$.
If $d\le 3$, \eqref{eq:scalprod} implies
$\killing{\vX}{\vZ}\ge 2$, and therefore $\killing{\vX}{\vZ}= 2$ and $\vZ=\vX=\tau^2\vf_i$.
If $d=4$, there are two possibilities: either $\killing{\vX}{\vZ}=2$ and we conclude likewise; 
or $\killing{\vX}{\vZ}=1$,
which means $\killing{\vX}{\vY}=-2$, a bound which is only reached if $\vY=-\vX$ and therefore $\vZ=-\vf_j$.

Now if $d=4$, $\vX=-\vf_k$, $k\ge i$, if $k\ne j$, then $\killing{\vX}{\vf_j}=-1$ and the same reasoning as above
applies, implying $\vZ=-\vf_k$ or $\vZ=-\vf_j$. If $k=j$, $\killing{\vX}{\vf_j}=-2$ and we conclude similarly that $\vY=-\vX$ and therefore $\vZ=-\vf_j$.
\end{proof}

For $d\le 3$, we conclude immediately from lemma~\ref{lem:tricky} that
all rhombi in the column are of the form \rh{i}{j}{i}{j} for some $j\ge i$.
The same conclusion is reached at $d=4$ as follows: we start from the
bottom and repeatedly apply lemma~\ref{lem:tricky}, concluding that
all NE labels of the rhombi in the column belong to $\mathcal A_i$.
When we reach the top of the column, the NE label must be a single number
and therefore must be $i$. By applying the last part of lemma~\ref{lem:tricky},
we can backtrack all the way to the bottom of the column and conclude
again that all rhombi are of the form \rh{i}{j}{i}{j}, $j\ge i$.

The statement of the proposition now follows by induction, noting that $P'$ is a puzzle that satisfies the hypotheses of the proposition.
\end{proof}
\junk{it's tempting to prove this in the $F$-basis, i.e., really, going back to fixed points...
being careful with the duality at the bottom, to pick out the coefficient of the identity, we braket
with a fixed point}

We can now state our main theorem:
\begin{thm}\label{thm:main}
Let $d=1,2,3,4$, and $P_-\backslash G$ be a $d$-step flag variety.
The product of two motivic Segre classes $S^\lambda$ and $S^\mu$ in $K_{\hat T}^{\loc}(T^*(P_-\backslash G))$ is given
by the ``puzzle'' formula
\begin{equation}\label{eq:main}
S^\lambda S^\mu=
\sum_\nu \tikz[scale=1.8,baseline=0.5cm]{\uptri{\lambda}{\nu}{\mu}}
    \ S^\nu 
  \end{equation}
\end{thm}
\junk{ note that ``puzzle'' is a bit of a strech since the formula is not 
  really combinatorial until one has chosen a basis of the representation space}
\begin{proof}
  The proof is essentially identical to that of \cite[theorem~1.4]{artic71}, 
  except for the way we fix the normalization of fugacities.  We first
  note that \eqref{eq:main} can be proven by restricting to each fixed
  point $\sigma\in W_P \dom W$.  We then have the formal series of equalities
\begin{multline*}
\sum_\nu \tikz[scale=1.8,baseline=0.5cm]{\uptri{\lambda}{\nu}{\mu}}
    \ S^\nu|_\sigma
\stackrel{\text{lemma~\ref{lem:spec},\ref{lem:ressingle}}}{=}
\sum_\nu
\begin{tikzpicture}[baseline=5mm]
\draw[thick] (0,0) -- node {$\nu$} (2,0) -- node[xshift=1mm] {$\mu$} ++(120:2) -- node[xshift=-1mm] {$\lambda$} cycle;
\draw[thick] (0,-0.5) -- node {$\omega$} (2,-0.5); \draw (0,0) -- (0,-0.5); \draw (2,0) -- (2,-0.5); \node at (1,-0.25) {$\sigma$};
\end{tikzpicture}
\stackrel{\text{lemma~\ref{lem:easy}}}{=}
\begin{tikzpicture}[baseline=5mm]
\draw[thick] (0,0) -- node {} (2,0) -- node[xshift=1mm] {$\mu$} ++(120:2) -- node[xshift=-1mm] {$\lambda$} cycle;
\draw[thick] (0,-0.5) -- node {$\omega$} (2,-0.5); \draw (0,0) -- (0,-0.5); \draw (2,0) -- (2,-0.5); \node at (1,-0.25) {$\sigma$};
\end{tikzpicture}
\stackrel{\text{prop.~\ref{prop:ybe}}}{=}
\begin{tikzpicture}[baseline=5mm]
\draw[thick] (0,0) -- node {$\omega$} (2,0) -- node[xshift=-0.5mm] {} ++(120:2) -- node[xshift=0.5mm] {} cycle;
\draw[thick] (150:0.5) -- node[xshift=-1mm] {$\lambda$} ++(60:2);
\draw (0,0) -- (150:0.5) (60:2) -- ++(150:0.5);
\path (60:1) ++(150:0.25) node[rotate=60] {$\sigma$};
\begin{scope}[xshift=2cm]
\draw[thick] (30:0.5) -- node[xshift=1mm] {$\mu$} ++(120:2);
\draw (0,0) -- (30:0.5) (120:2) -- ++(30:0.5);
\path (120:1) ++(30:0.25) node[rotate=-60] {$\sigma$};
\end{scope}
\end{tikzpicture}
\\
\stackrel{\text{lemma~\ref{lem:ressingle}}}{=}
\sum_{\tilde\lambda,\tilde\mu}
\begin{tikzpicture}[baseline=5mm]
\draw[thick] (0,0) -- node {$\omega$} (2,0) -- node[xshift=-0.5mm] {$\tilde\mu$} ++(120:2) -- node[xshift=0.5mm] {$\tilde\lambda$} cycle;
\draw[thick] (150:0.5) -- node[xshift=-1mm] {$\lambda$} ++(60:2);
\draw (0,0) -- (150:0.5) (60:2) -- ++(150:0.5);
\path (60:1) ++(150:0.25) node[rotate=60] {$\sigma$};
\begin{scope}[xshift=2cm]
\draw[thick] (30:0.5) -- node[xshift=1mm] {$\mu$} ++(120:2);
\draw (0,0) -- (30:0.5) (120:2) -- ++(30:0.5);
\path (120:1) ++(30:0.25) node[rotate=-60] {$\sigma$};
\end{scope}
\end{tikzpicture}
\stackrel{\text{prop.~\ref{prop:trivpuzzle}}}{=}
C_\omega
\begin{tikzpicture}[baseline=5mm]
\draw[thick]  (2,0) -- node[xshift=-0.5mm] {$\omega$} ++(120:2) -- node[xshift=0.5mm] {$\omega$} (0,0);
\draw[thick] (150:0.5) -- node[xshift=-1mm] {$\lambda$} ++(60:2);
\draw (0,0) -- (150:0.5) (60:2) -- ++(150:0.5);
\path (60:1) ++(150:0.25) node[rotate=60] {$\sigma$};
\begin{scope}[xshift=2cm]
\draw[thick] (30:0.5) -- node[xshift=1mm] {$\mu$} ++(120:2);
\draw (0,0) -- (30:0.5) (120:2) -- ++(30:0.5);
\path (120:1) ++(30:0.25) node[rotate=-60] {$\sigma$};
\end{scope}
\end{tikzpicture}
\stackrel{\text{lemma~\ref{lem:spec},\ref{lem:ressingle}}}{=}
C_\omega S^\lambda|_\sigma S^\mu|_\sigma
\end{multline*}
where $C_\omega$ is the value of the trivial puzzle occurring in proposition~\ref{prop:trivpuzzle}.
  
We now need to show $C_\omega=1$, which is equivalent to the following
\begin{lem}\label{lem:norm}
  One has for $0\le i\le j\le d$
  \[
  \check R_{1,2}(z)^{ji}_{ij}=\rh{i}{j}{i}{j}=1
  \]
  and for $0\le i\le d$
  \[
  U^{ii}_i=\uptri{i}{i}{i}=D^i_{ii}=\downtri{i}{i}{i}=1
  \]
\end{lem}

By using the Weyl group action as in \S\ref{sec:quiver}, 
it is not hard to reduce to the case $i=0$, with the normalization conditions
\eqref{eq:normR} and \eqref{eq:normUD} in mind;
however this is not enough to prove the first equality in the case $j\ne 0$.
Instead, one must resort to direct calculation (which we skip here) of
the relevant $R$-matrix entries based on the explicit expressions of
appendix~\ref{app:Rd1}, \cite[\S 3.6 and 3.8]{artic71},
\cite{artic77}, respectively, noting that the required entries are
``diagonal'', i.e., independent of the normalization of the weight
vectors inside each one-dimensional weight space; only the overall
normalization of the $R$- and $U$-matrices is left undetermined, and
one then concludes using \eqref{eq:normR} and \eqref{eq:normUD}.
\junk{A direct computation (using explicit expressions for the $R$-matrix) is somewhat difficult (especially at $d=4$),
  \rem[gray]{in particular the varieties involved are NOT $T^*$flags}
  so instead we us
  size examples to reduce all these equalities to the normalization conditions \eqref{eq:normR} and \eqref{eq:normUD}.
  so instead we use the Weyl group action of proposition \ref{prop:reflect} to
  show these identities without calculation.
  This is detailed in Appendix~\ref{app:norm}.
}
\end{proof}

\junk{at the end of day, what we'd have ideally is
three embeddings of the quantized
affine algebras (or loop algebras?).
using Drinfeld currents, one gets embeddings whenever there's embeddings
of {\em finite}\/ Dynkin diagrams
(for $d=1$, cf
\url{https://projecteuclid.org/download/pdf_1/euclid.cmp/1104271413}).
one can then
use the extended affine quantum Weyl (braid) group action 
(which includes Dynkin diagram automorphism).
in fact there's quite a bit of freedom because the translation part of the
action is trivial in our reps. because our reps are minuscule,
we only need to check that the action on the weight space is the right one!
(and then choose a normalization of weight vectors that agrees with the
single-number sector).
for ex for $d=1$ one can just use the affine
Dynkin diagram automorphism of order $3$! for $d=2$ one can use the affine
Dynkin diagram automorphism of order $3$ that cycles $b$, $b'$ and the affine
node, and compose it with $T_a T_{a'}$ (one can probably switch $a'$ and $b'$).
these embeddings
$\Uq(A_d^{(1)})<\Uq(X_{2d}^{(1)})$ should be such that 
the representation $\CC^{d+1}$ of 
$\Uq(A_d^{(1)})$ appears $3$ times in one ($d\ne2$) / three
($d=2$) irrep(s) of $\Uq(X_{2d}^{(1)})$, with the weights related by $\tau$.
that should be enough to formulate a general theorem up to $d=4$.
warning, at $d=4$ there's plenty of 5-dimensions subreps: so either we still need the face argument
to conclude on stability of single-number sector,
or use torus or even better other commuting $A_4$ to conclude
}

There is an obvious dual statement, that we state without proof.
Let
\[
\tikz[scale=1.8,baseline=-0.8cm]{\downtri{\lambda}{\nu}{\mu}}
:=
\left< e^*_\mu\otimes e^*_\lambda,\,
\overbrace{\check R\ldots \check R}^{n\choose 2}\,
\overbrace{D\ldots D}^{n}
\, e_\nu
\right>
\]
to be the entry of the product of $R$-matrices
and $D$-matrices forming a 180 degree rotated puzzle, where time still flows downwards,
\junk{if time still flows up, shouldn't the $D$s be right of the $R$s now? PZJ: we switched, time flows down,
but you're right it was wrong as stated}
all labels are read left to right, and the spectral parameters
at the bottom are (from left to right) $\kappa z_1,\ldots,\kappa z_n,
\kappa^{-1}z_1\,\ldots,\kappa^{-1}z_n$.
Then
\begin{thm}\label{thm:maindual}
Let $d=1,2,3,4$, and $P_-\backslash G$ be a $d$-step flag variety.
\begin{equation}\label{eq:maindual}
S_\lambda S_\mu=
\sum_\nu \tikz[scale=1.8,baseline=-0.8cm]{\downtri{\lambda}{\nu}{\mu}}
    \ S_\nu 
  \end{equation}
\end{thm}

\subsection{The hierarchy of bases}\label{ssec:limits}
The various 
limits can be summarized in the following diagram:
\begin{center}
\begin{tikzpicture}[xscale=8,yscale=4,every node/.style={align=center}]
\node (a) at (0,0) {motivic Segre classes\\ $K^{\loc}_{\hat T}$};
\node (b) at (1,0) {Schubert classes/dual\\ $K_{T}$};
\node (c) at (0,-1) {SSM classes\\ $H^{*\loc}_{\hat T}$};
\node (d) at (1,-1) {Schubert classes\\$H^*_{T}$};
\node (aa) at (0.6,0.7) {motivic Segre classes\\ $K^{\loc}_{\CC^\times}$};
\node (bb) at (1.6,0.7) {Schubert classes/dual\\ $K$};
\node (dd) at (1.6,-0.3) {Schubert classes\\$H^*$};
\draw[->] (a) -- node[above] {$\ss q^{\pm1}\to0$} (b);
\draw[->] (a) -- node[right,pos=0.4] {$\ss t=e^\hbar\to1$\\$\ss z_i=e^{y_i}\to 1$} (c);
\draw[->] (b) -- node[right] {$\ss z_i=e^{y_i}\to 1$} (d);
\draw[->] (c) -- node[above] {$\ss \hbar\to\infty$} (d);
\draw[->] (a) -- node[above left=-1mm] {$\ss z_i\to1$} (aa);
\draw[->] (b) -- node[above left=-1mm] {$\ss z_i\to1$} (bb);
\draw[->] (d) -- node[above left=-1mm] {$\ss x_i\to0$} (dd);
\draw[->] (aa) -- node[above] {$\ss q^{\pm1}\to0$} (bb);
\draw[->] (bb) -- node[right] {$\ss\mathrm{gr}$} (dd);
\begin{scope}[gray]
\node (cc) at (0.6,-0.3) {SSM classes\\ $H^{*\loc}_{\CC^\times}$};
\draw[->] (c) -- node[above left=-1mm] {$\ss x_i\to0$} (cc);
\draw[->] (aa) -- (cc);
\draw[->] (cc) -- node[above,pos=0.6] {$\ss \hbar\to\infty$} (dd);
\end{scope}
\end{tikzpicture}
\end{center}
\junk{one needs to rethink this carefully in view of the polarization issues -- there seems to
  be a hidden Serre duality in there}

The vertical direction is going from $K$-theory to cohomology
and will be discussed in more detail in \S \ref{sec:coho}.

The Northeast direction corresponds to forgetting the Cartan torus
equivariance, keeping only the $\CC^\times$ action, e.g., working in
$K_{\CC^\times}(T^*(P_-\backslash G))$. By a slight abuse of language,
we shall call the resulting specialization of theorem~\ref{thm:main} the
``nonequivariant rule'' (note that it makes no sense\footnote{Ginzburg's
  formula \cite{Ginzburg86} for Chern--Schwartz--MacPherson classes
  hides its use of dilation equivariance in a rather sneaky way, by
  specializing $\hbar\to -1$. In the cohomological limit (bottom layer
  of the hierarchy cube), the $\hbar$s can be recovered by re-homogenizing,
  but this trick is unavailable in $K$-theory.}
to remove the equivariance altogether, because the definition of motivic
Segre classes requires localization).
In this limit, the rhombi composing a puzzle each break into triangles
according to \eqref{eq:factorR}, so that a nonequivariant puzzle looks like
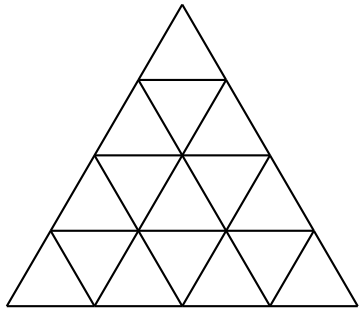
\begin{center}
\begin{tikzpicture}[math mode,nodes={edgelabel},x={(-0.577cm,-1cm)},y={(0.577cm,-1cm)}] 
\draw[thick] (0,0) -- node {} ++(0,1); \draw[thick] (0,0) -- node {} ++(1,0); \draw[thick] (0+1,0) -- node {} ++(-1,1); 
\draw[thick] (0,1) -- node {} ++(0,1); \draw[thick] (0,1) -- node {} ++(1,0); \draw[thick] (0+1,1) -- node {} ++(-1,1); 
\draw[thick] (0,2) -- node {} ++(0,1); \draw[thick] (0,2) -- node {} ++(1,0); \draw[thick] (0+1,2) -- node {} ++(-1,1); 
\draw[thick] (0,3) -- node {} ++(0,1); \draw[thick] (0,3) -- node {} ++(1,0); \draw[thick] (0+1,3) -- node {} ++(-1,1); 
\draw[thick] (1,0) -- node {} ++(0,1); \draw[thick] (1,0) -- node {} ++(1,0); \draw[thick] (1+1,0) -- node {} ++(-1,1); 
\draw[thick] (1,1) -- node {} ++(0,1); \draw[thick] (1,1) -- node {} ++(1,0); \draw[thick] (1+1,1) -- node {} ++(-1,1); 
\draw[thick] (1,2) -- node {} ++(0,1); \draw[thick] (1,2) -- node {} ++(1,0); \draw[thick] (1+1,2) -- node {} ++(-1,1); 
\draw[thick] (2,0) -- node {} ++(0,1); \draw[thick] (2,0) -- node {} ++(1,0); \draw[thick] (2+1,0) -- node {} ++(-1,1); 
\draw[thick] (2,1) -- node {} ++(0,1); \draw[thick] (2,1) -- node {} ++(1,0); \draw[thick] (2+1,1) -- node {} ++(-1,1); 
\draw[thick] (3,0) -- node {} ++(0,1); \draw[thick] (3,0) -- node {} ++(1,0); \draw[thick] (3+1,0) -- node {} ++(-1,1); 
\end{tikzpicture}
\end{center}

We now discuss the horizontal arrows, those going from left to right
side of the cube.

\rem[gray]{Note that the commutativity of the diagram is somewhat nontrivial
because naively the two limits $q\to 0,\infty$ and $t=q^{-2}\to 1$ do not
commute.}

\newcommand\Sc{\underline{S}}
\subsection{Inversion numbers and Drinfeld twist}\label{sec:twist}
The horizontal arrows were the object of \cite[\S 3.6.2]{artic71};
this limit is related to getting rid of the fiber
of the bundle $T^*(P_-\backslash G)\to P_-\backslash G$ by sending the
equivariant parameter $q$ scaling it to $0$ (or $\infty$).
We briefly review this limit here, describing how to recover the
results of \cite{artic71} as a limit of theorem~\ref{thm:main}.

We first show how to obtain Schubert classes as limits of motivic
Segre classes.  Return to the setup of \S\ref{sec:stable}. It is
convenient to extend the Cartan action of (the quantized loop algebra
of) $\mathfrak{sl}_{d+1}$ on $V^A=\CC^{d+1}=\left<e_0,\ldots,e_d\right>$
to that of $\mathfrak{gl}_{d+1}$, i.e., introduce operators $h_i$,
$i=0,\ldots,d$, such that $h_i e_j = \delta_{i,j} e_j$.
\rem[gray]{can we really define the twist using only $sl$, as Kolya suggested?}

Introduce next the {\em Drinfeld twist}
\cite{Drinfeld-QG,Drinfeld-twist} acting in $V^A\otimes V^A$ by
\[
  \Omega = q^{\frac{1}{2}\sum_{i,j=0}^d B_{i,j}h_i\otimes h_j}
\]
where $B_{i,j}=\text{sign}(i-j)=-1,0,1$ depending on whether $i<j$, $i=j$, $i>j$,
\rem[gray]{note that the twist is by powers of $q$, not $-q$ -- that's a choice, it's whether
we want to include the $(-1)^{inv}$ in the twist or not. we choose NOT TO -- the fugacity
after twisting will still include that sign}
and use it to conjugate the $R$-matrix \eqref{eq:Rsingle}:
\[
\check R(z)_{twist}=\Omega \check R(z) \Omega^{-1}
\]

Define more generally the operator $\Omega_n$ acting on $(V^A)^{\otimes n}$ by
\[
\Omega_n = q^{\frac{1}{2}\sum_{1\le k< \ell\le n}\sum_{i,j=0}^d B_{i,j}h^{(k)}_i h^{(\ell)}_j}
\]
where $h^{(k)}_i$ is $h_i$ acting on the $k^{\rm th}$ factor of the tensor product. (In particular,
$\Omega=\Omega_2$.)

Because of the cocycle property satisfied by $\Omega$, one has symbolically
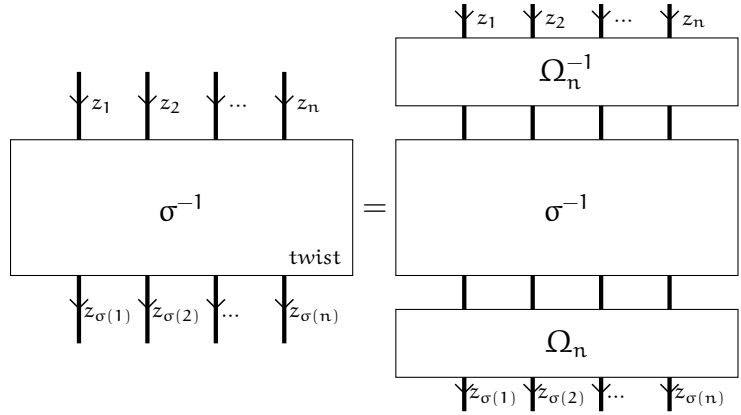
\[
\begin{tikzpicture}[baseline=-3pt,scale=0.9]
\draw (0,-1) rectangle (5,1); \node at (2.5,0) {$\sigma^{-1}$};
\node[above left] at (5,-1) {$\ss twist$};
\foreach\x/\t/\u in {1/z_1/z_{\sigma(1)},2/z_2/z_{\sigma(2)},3/\cdots/\cdots,4/z_n/z_{\sigma(n)}}
{
\draw[d,invarrow=0.5] (\x,1) -- node[right,pos=0.5] {$\ss\t$} ++(0,1);
\draw[d,invarrow=0.5] (\x,-2) -- node[right=-1mm,pos=0.4] {$\ss\u$} ++(0,1);
}
\end{tikzpicture}
=
\begin{tikzpicture}[baseline=-3pt,scale=0.9]
\draw (0,-1) rectangle (5,1); \node at (2.5,0) {$\sigma^{-1}$};
\draw (0,-1.5) rectangle (5,-2.5); \node at (2.5,-2) {$\Omega_n$};
\draw (0,1.5) rectangle (5,2.5); \node at (2.5,2) {$\Omega_n^{-1}$};
\foreach\x/\t/\u in {1/z_1/z_{\sigma(1)},2/z_2/z_{\sigma(2)},3/\cdots/\cdots,4/z_n/z_{\sigma(n)}}
{
\draw[d] (\x,1) -- ++(0,0.5);
\draw[d,invarrow=0.5] (\x,2.5) -- node[right,pos=0.5] {$\ss\t$} ++(0,0.5);
\draw[d] (\x,-1.5) -- ++(0,0.5);
\draw[d,invarrow=0.5] (\x,-3) -- node[right=-1mm,pos=0.4] {$\ss\u$} ++(0,0.5);
}
\end{tikzpicture}
\]
where the ``twisted'' rectangle uses $\check R(z)_{twist}$ at each crossing of the diagram of $\sigma^{-1}$.

With \eqref{eq:defSfp} in mind, 
we compute $\Omega_n^{-1} e_\lambda=q^{-\ell(\lambda)+D/2} e_\lambda$, 
where $D=\dim(P_-\dom G)$ (cf \cite[lemma 2.4]{artic71}), 
and similarly $e_\omega^* \Omega_n=q^{-D/2}e_\omega^*$,
so that
\begin{equation}\label{eq:q0res}
\begin{tikzpicture}[baseline=-3pt,scale=0.9]
\draw (0,-1) rectangle (5,1); \node at (2.5,0) {$\sigma^{-1}$};
\foreach\x/\t/\u in {1/z_1/z_{\sigma(1)},2/z_2/z_{\sigma(2)},3/\cdots/\cdots,4/z_n/z_{\sigma(n)}}
{
\draw[d,invarrow=0.5] (\x,1) -- node[right,pos=0.5] {$\ss\t$} ++(0,1);
\draw[d,invarrow=0.5] (\x,-2) -- node[right=-1mm,pos=0.4] {$\ss\u$} ++(0,1);
}
\draw[decorate,decoration=brace] (4,-2.3) -- node[below] {$\m\omega$} (1,-2.3);
\draw[decorate,decoration=brace] (1,2.3) -- node[above] {$\m\lambda$} (4,2.3);
\node[above left] at (5,-1) {$\ss twist$};
\end{tikzpicture}
=
q^{-\ell(\lambda)}
\begin{tikzpicture}[baseline=-3pt,scale=0.9]
\draw (0,-1) rectangle (5,1); \node at (2.5,0) {$\sigma^{-1}$};
\foreach\x/\t/\u in {1/z_1/z_{\sigma(1)},2/z_2/z_{\sigma(2)},3/\cdots/\cdots,4/z_n/z_{\sigma(n)}}
{
\draw[d,invarrow=0.5] (\x,1) -- node[right,pos=0.5] {$\ss\t$} ++(0,1);
\draw[d,invarrow=0.5] (\x,-2) -- node[right=-1mm,pos=0.4] {$\ss\u$} ++(0,1);
}
\draw[decorate,decoration=brace] (4,-2.3) -- node[below] {$\m\omega$} (1,-2.3);
\draw[decorate,decoration=brace] (1,2.3) -- node[above] {$\m\lambda$} (4,2.3);
\end{tikzpicture}
=
q^{-\ell(\lambda)}
S^\lambda|_\sigma 
\end{equation}

Let us now compute explicitly from \eqref{eq:Rsingle}
\[
\lim_{q\to0}\check R_{twist}(z)
=
\begin{cases}
1 & i=k\le j=l
\\
1-z& i=l<j=k
\\
z
& i=k>j=l
\\
0 & \text{else}
\end{cases}
\]

As explained in \cite[\S 3.6]{artic71}, this {\em nilHecke}\/ $R$-matrix
is directly related to the $K$-theory of partial flag varieties (as opposed
to their cotangent bundles); paying attention to the fact (cf \ref{sssec:K} and \cite[\S 2.1]{artic71}) that
inverse equivariant parameters and classes of line bundles are used there,
we can compute the limit as $q\to 0$ of
the l.h.s.\ of \eqref{eq:q0res} to be
the restriction $\Sc^\lambda|_\sigma$ to the fixed point $\sigma$
of the Schubert class $\Sc^\lambda$ associated to $\lambda$, composed with the map $\vee$
that takes classes of vector bundles to classes of their duals (i.e., $x_i\mapsto x_i^{-1}$, $z_i\mapsto z_i^{-1}$).
We conclude that
\[
\Sc^\lambda = \left(\lim_{q\to 0} q^{-\ell(\lambda)} S^\lambda\right)^{\vee}
\]
as elements of $K_{\hat T}(P_-\backslash G) \cong K_{\hat T}(T^*(P_-\backslash G))$.
$\vee$, being a ring map, does not affect product rules.

Now consider any $\tau$-invariant alternating form $B$ on the weight space of
$\mathfrak{x}_{2d}$ that extends $(B_{ij})$ above,
i.e., $B(\vf_i,\vf_j)=B_{ij}$.\footnote{In practice, it is convenient to augment
the weight space in a similar way as we switched from $\mathfrak{sl}_{d+1}$
to $\mathfrak{gl}_{d+1}$ above; this is implicit in \cite{artic71}.}
Use it to twist all $R$-matrices and $U$, $D$:
\[
\Omega=q^{\frac{1}{2}B(H,H)},
\qquad
\check R(z)_{ij,twist}=\Omega \check R(z)_{ij} \Omega^{-1},
\qquad
U_{twist}=U \Omega^{-1},
\qquad
D_{twist}=\Omega D
\]
where $H$ stands for the collection of Cartan generators,
and assign to puzzles twisted fugacities, which we denote with
a subscript ``twist''.
We then have
\begin{prop}
Let $d=1,2,3,4$ and $P_-\backslash G$ be a $d$-step flag variety. The product of two Schubert classes
$\Sc^\lambda$ and $\Sc^\mu$ in $K_T(P_-\backslash G)$ is given by
\[
\Sc^\lambda \Sc^\mu \ =\ 
\lim_{q\to0}
\sum_\nu \tikz[scale=1.8,baseline=0.5cm]{\uptri{\lambda}{\nu}{\mu}}^\vee_{twist}
 \   \Sc^\nu 
\]
\end{prop}

This expression is not entirely satisfactory because the summation is not positive, so that there may be compensations.
A natural question is whether one can choose bases\footnote{\label{foot:norm}Note that for $d\le 3$,
the representations $V_a(z)$ being minuscule, the only real freedom in the choice of bases is normalization of the
weight vectors.} of the $V_a(z)$ and an alternating form $B$ \junk{the freedom in $B$ is in going beyond single number -- though really isn't it the same as gauge freedom?} in such a way
the the fugacity of every (twisted) rhombus and triangle has a finite (positive) limit as $q\to0$. 

The case $d=1$, which was only sketched in \cite{artic71} because it did not lead to any new results, will be developed in
\S\ref{sec:schubd1} based on our more general framework. For $d=2$, \cite{artic71} provides a positive answer to the question above, leading to a positive $K_T$ puzzle rule.
At $d=3$, the situation is more subtle: it seems impossible to keep the fugacity
of equivariant rhombi finite as $q\to 0$. One can however find a limit
for the (nonequivariant) fugacity of every triangle, leading to a positive $K$ puzzle rule. 
In particular every triangle (with nonzero fugacity) has a nonnegative inversion number.

The situation is worse at $d=4$: even nonequivariantly, it seems impossible to get rid of triangles
with negative inversion number.
This does not preclude from formulating a puzzle rule in nonequivariant $K$-theory, but the answer is sufficiently
complicated that we prefer not to write it explicitly, providing instead in \S\ref{ssec:d4} a (mildly nonpositive) rule in nonequivariant cohomology only.

\rem[gray]{how natural is the choice of basis made in paper I? is it the stable basis? apparently not, cf.~$d=1$ below}

The limit $q\to\infty$ can be treated similarly. In view of
lemma~\ref{lem:inv}, this is actually the same as taking $q\to0$ in
the dual theorem~\ref{eq:maindual}. Explicitly, if one applies the
opposite twist $\Omega^{-1}$ to $S_\lambda$, one obtains in the $q\to 0$ limit the classes $\Sc_{\lambda}^\vee$
where the $\Sc_\lambda$ are {\em dual}\/ Schubert classes; and we derive this way
puzzle rules for the latter.

\section{Example: $d=1$}\label{sec:exd1}
\rem[gray]{mention integral formulae? stability as a function of $n$? maybe not}
\junk{provide the $A_1\to A_2$ version of things, sticking to the nondegenerate
case. of course the $R$-matrix looks very much like the $d=2$ single-number
$R$-matrix, but the devil is in the details: labeling, normalization, conjugation. signs still need careful check}
\subsection{The $d=1$ setup}\label{ssec:d1setup}
We now provide the full details of the simplest case $d=1$, i.e.,
$P_-\backslash G$ is a Grassmannian.
The explicit definition of the quantized affine algebra $\Uq(\mathfrak{a}_2^{(1)})$ is given
in appendices~\ref{app:qg} and \ref{app:Rd1}.

We consider the three $\Uq(\mathfrak{a}_2[z^\pm])$-modules $V_a(z)$, $a=1,2,3$ 
($z$ is a formal parameter); 
even though $V_1(z)$ and $V_2(z)$
are isomorphic, it is useful to differentiate them. They have bases $e_{a,X}$,
$a=1,2,3$, labeled by $X\in \{1,0,10\}$, where the first two vectors
form the usual bases of $V^A_a(z)$, while the third vector is the remaining weight vector with
some convenient normalization (the label ``10'' is traditional, 
see e.g.~\cite[\S 2.3]{artic71} for a justification);
and their weights are given by $\vf_X$ (resp.\ $\tau^2 \vf_X$, $-\tau \vf_X$). The representation
matrices are given explicitly in appendix~\ref{app:Rd1}.

\junk{
We now define $R$-matrices as follows: 
for generic $u,v$, $V_i(u)\otimes V_j(v)$ and $V_j(v)\otimes V_i(u)$
are irreducible and isomorphic, and
$\check R_{ij}(u,v)$ is the unique
intertwiner from $V_i(u)\otimes V_j(v)$ to $V_j(v)\otimes V_i(u)$  with the
normalization condition
\begin{equation}\label{eq:norm}
\check R_{ij}(u,v)^{0,0}_{0,0}=1
\end{equation}
In fact, its entries
only depend on $u/v$, and we also write $\check R_{ij}(u,v)=\check R_{ij}(u/v)$.

\begin{rmk*}
  The normalization condition \eqref{eq:norm} looks quite natural,
but a word of caution is needed.
If $i=j$, this condition is equivalent to asking that the highest-weight-to-highest-weight entry 
of the $R$-matrix be $1$, which is indeed quite natural and in particular,
is the normalization coming out of the geometry \cite{MO-qg}
However for $i\ne j$, it is not the case, due to the
fact that identical labels in different representations have no obvious
relation to each other (in particular their weights differ by powers of $\tau$);
this remains the correct normalization for our purposes, but it
will introduce some subtleties in the geometric interpretation of our construction 
(see in particular \S \ref{sec:geominterp}).
\end{rmk*}

Now all properties
are easily proven as follows: first one checks that boths l.h.s.\ and r.h.s.\ 
are quantum group intertwiners. For generic parameters, these intertwiners
are unique up to normalization. Normalization condition \eqref{eq:norm}
ensures that the particular configuration entirely made of $0$s has equal
l.h.s.\ and r.h.s. 
}

The relevant $R$-matrices are also given in appendix
\ref{app:Rd1}. In particular, one checks that
if restricted to the single-number sector $\{0,1\}$, all three $R$-matrices $\check R_{a,a}(z)$
coincide with the one in \eqref{eq:Rsingle}, in accordance with lemma~\ref{lem:ressingle}.

At the particular ratio $z''/z'=q^{-2}$, $V_1(z')\otimes V_2(z'')$
and $V_2(z'')\otimes V_1(z')$ become reducible, and the $R$-matrix factorizes
$
\check R_{1,2}(q^{-2})=DU
$
as in lemma~\eqref{lem:factorR}.
The nonzero entries of $U$ (resp.\ $D$)
are depicted as up-pointing (resp.\ down-pointing) triangles:
\begin{equation}\label{eq:tri1}
\begin{aligned}
&\uptri000=\uptri111=\uptri01{10}=\uptri{10}01=\uptri1{10}0=
1
&\quad&
\uptri{10}{10}{10}=-q^{-1}
\\
&\downtri000=\downtri111=\downtri01{10}=\downtri{10}01=\downtri1{10}0=
1
&&
\downtri{10}{10}{10}=-q
\end{aligned}
\end{equation}
Away from the ratio $q^{-2}$ of parameters, we represent the nonzero
entries of the matrix $\check R_{1,2}$ as rhombi,
and use the parametrization $\check R_{1,2}(q^{-2}z^{-1})$:
\begin{equation}\label{eq:rhombi1}
\begin{gathered}
\rh0000=
\rh0101=
\rh1111=
\rh1{10}1{10}=
\rh{10}0{10}0=
\rh{10}{10}{10}{10}=1
\\
\rh1{10}{0}{0}=
\rh{10}011=\frac{1-q^2}{1-q^2z}
\qquad
\rh001{10}=
\rh11{10}0=
\frac{(1-q^2)z}{1-q^2z}
\\
\rh1010=
\rh0{10}0{10}=
\rh{10}1{10}1=
\frac{q(1-z)}{1-q^2z}
\\
\rh01{10}{10}=-q^{-1}\frac{(1-q^2)z}{1-q^2z}
\qquad
\rh{10}{10}01=-q\frac{1-q^2}{1-q^2z}
\end{gathered}
\end{equation}
The parameter $z$ should be set to $z_j/z_i$ for a rhombus at location $i<j$.

\begin{ex}
Consider the simplest nontrivial example, of $T^*\PP^1$.  
The restrictions of motivic Segre classes at each fixed point were computed in example~\ref{ex:d1b}:
\[
\begin{matrix}
&|_{01}&\ &|_{10}\\
S^{01}=(&1&&\qquad ({1-q^2})\big/({1-q^2 z_2/z_1})&)
\\ 
S^{10}=(&0&&{q(1-z_2/z_1)}\big/({1-q^2 z_2/z_1})&)
\end{matrix}
\]

The possible products are:
\begin{itemize}
\def\thescale{1.4}
\def\posa{0.5}\def\posb{0.5}
\item $S^{01}S^{01}=S^{01}-\frac{q(1-q^2)}{1-q^2z_2/z_1}S^{10}$,
with corresponding puzzles
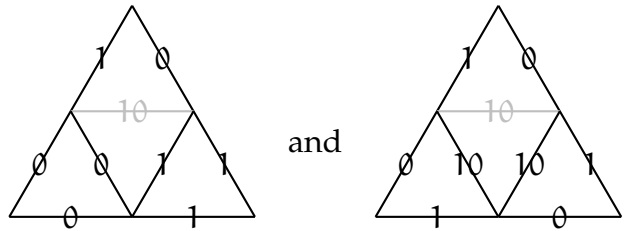
\[
\begin{tikzpicture}[math mode,nodes={edgelabel},x={(-0.577cm,-1cm)},y={(0.577cm,-1cm)},scale=\thescale]
\draw[thick] (0,0) -- node[pos=\posa] {0} ++(0,1); \draw[thick] (0,0) -- node[pos=\posb] {1} ++(1,0); \draw[thick,lightgray] (0+1,0) -- node[lightgray] {10} ++(-1,1); 
\draw[thick] (0,1) -- node[pos=\posa] {1} ++(0,1); \draw[thick] (0,1) -- node[pos=\posb] {1} ++(1,0); \draw[thick] (0+1,1) -- node {1} ++(-1,1); 
\draw[thick] (1,0) -- node[pos=\posa] {0} ++(0,1); \draw[thick] (1,0) -- node[pos=\posb] {0} ++(1,0); \draw[thick] (1+1,0) -- node {0} ++(-1,1); 
\end{tikzpicture}
\raisebox{1cm}{\quad and\quad}
\begin{tikzpicture}[math mode,nodes={edgelabel},x={(-0.577cm,-1cm)},y={(0.577cm,-1cm)},scale=\thescale]
\draw[thick] (0,0) -- node[pos=\posa] {0} ++(0,1); \draw[thick] (0,0) -- node[pos=\posb] {1} ++(1,0); \draw[thick,lightgray] (0+1,0) -- node[lightgray] {10} ++(-1,1); 
\draw[thick] (0,1) -- node[pos=\posa] {1} ++(0,1); \draw[thick] (0,1) -- node[pos=\posb] {10} ++(1,0); \draw[thick] (0+1,1) -- node {0} ++(-1,1); 
\draw[thick] (1,0) -- node[pos=\posa] {10} ++(0,1); \draw[thick] (1,0) -- node[pos=\posb] {0} ++(1,0); \draw[thick] (1+1,0) -- node {1} ++(-1,1); 
\end{tikzpicture}
\]
(we indicated in light gray the corresponding {\em nonequivariant}\/ puzzles, see \S \ref{ssec:loop}).

\item $S^{01}S^{10}=S^{10}S^{01}=\frac{1-q^2}{1-q^2z_2/z_1}S^{10}$, with one puzzle
for each of the two computations:
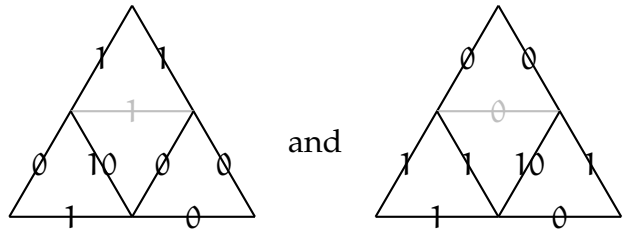
\[
\begin{tikzpicture}[math mode,nodes={edgelabel},x={(-0.577cm,-1cm)},y={(0.577cm,-1cm)},scale=\thescale]
\draw[thick] (0,0) -- node[pos=\posa] {1} ++(0,1); \draw[thick] (0,0) -- node[pos=\posb] {1} ++(1,0); \draw[thick,lightgray] (0+1,0) -- node[lightgray] {1} ++(-1,1); 
\draw[thick] (0,1) -- node[pos=\posa] {0} ++(0,1); \draw[thick] (0,1) -- node[pos=\posb] {0} ++(1,0); \draw[thick] (0+1,1) -- node {0} ++(-1,1); 
\draw[thick] (1,0) -- node[pos=\posa] {10} ++(0,1); \draw[thick] (1,0) -- node[pos=\posb] {0} ++(1,0); \draw[thick] (1+1,0) -- node {1} ++(-1,1); 
\end{tikzpicture}
\raisebox{1cm}{\quad and\quad}
\begin{tikzpicture}[math mode,nodes={edgelabel},x={(-0.577cm,-1cm)},y={(0.577cm,-1cm)},scale=\thescale]
\draw[thick] (0,0) -- node[pos=\posa] {0} ++(0,1); \draw[thick] (0,0) -- node[pos=\posb] {0} ++(1,0); \draw[thick,lightgray] (0+1,0) -- node[lightgray] {0} ++(-1,1); 
\draw[thick] (0,1) -- node[pos=\posa] {1} ++(0,1); \draw[thick] (0,1) -- node[pos=\posb] {10} ++(1,0); \draw[thick] (0+1,1) -- node {0} ++(-1,1); 
\draw[thick] (1,0) -- node[pos=\posa] {1} ++(0,1); \draw[thick] (1,0) -- node[pos=\posb] {1} ++(1,0); \draw[thick] (1+1,0) -- node {1} ++(-1,1); 
\end{tikzpicture}
\]

\item $S^{10}S^{10}=\frac{q(1-z_2/z_1)}{1-q^2z_2/z_1}S^{10}$, with puzzle
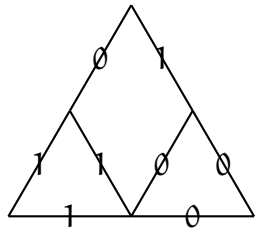
\[
\begin{tikzpicture}[math mode,nodes={edgelabel},x={(-0.577cm,-1cm)},y={(0.577cm,-1cm)},scale=\thescale]
\draw[thick] (0,0) -- node[pos=\posa] {1} ++(0,1); \draw[thick] (0,0) -- node[pos=\posb] {0} ++(1,0); 
\draw[thick] (0,1) -- node[pos=\posa] {0} ++(0,1); \draw[thick] (0,1) -- node[pos=\posb] {0} ++(1,0); \draw[thick] (0+1,1) -- node {0} ++(-1,1); 
\draw[thick] (1,0) -- node[pos=\posa] {1} ++(0,1); \draw[thick] (1,0) -- node[pos=\posb] {1} ++(1,0); \draw[thick] (1+1,0) -- node {1} ++(-1,1); 
\end{tikzpicture}
\]
(this one only contributes equivariantly, as its fugacity vanishes at
$z_1=z_2$).

\end{itemize}
The same exercise can be repeated for the dual classes:
\[
\begin{matrix}
&|_{01}&\ &|_{10}\\
S_{01}=(&\ {q(1-z_1/z_2)}/\left(1-q^2 z_1/z_2\right)&&0&)
\\
S_{10}=(&{(1-q^2)z_1/z_2}/\left(1-q^2 z_1/z_2\right)&&1&)
\end{matrix}
\]
the puzzles being 180 degree rotations of the ones above.
\junk{careful that lines are always numbered left to right
and strings read left to right, of course}
\end{ex}

\subsection{Back to Schubert calculus in $Gr(k,n)$}\label{sec:schubd1}
We now consider the limit $q^{\pm 1}\to 0$, cf.~the 
top horizontal arrow of the diagram of \S\ref{ssec:limits}. As discussed in \S\ref{sec:twist}, 
one must first twist the triangle/rhombi 
according to their inversion number. We list only those with nonzero inversion numbers:
\begin{align*}
\uptri[1]{10}{10}{10}&=-q^{-1\pm 1}
&
\downtri[1]{10}{10}{10}&=-q^{1\pm 1}
\\
\rh[1]01{10}{10}&=-q^{-1\pm 1}\frac{(1-q^2)z}{1-q^2z}
&
\rh[1]{10}{10}01&=-q^{1\pm 1}\frac{1-q^2}{1-q^2z}
&
\rh[2]{10}{10}{10}{10}&=q^{\pm 2}
\\
\rh[-1]1010&=
\frac{q^{1\mp 1}(1-z)}{1-q^2z}
&
\rh[1]0{10}0{10}&=
\rh[1]{10}1{10}1=
\frac{q^{1\pm 1}(1-z)}{1-q^2z}
\end{align*}
If we pick the $+$ sign and send $q$ to $0$, the first of the two
triangles survive, as well as the first rhombi of each row, and we recover
exactly the puzzle rule for Schubert classes in the equivariant $K$-theory
of the Grassmannian as formulated in \cite{artic68} (related
to symmetric Grothendieck polynomials). 
If we pick the $-$ sign and send $q$ to $\infty$, the second triangle and the second rhombus of the second row survive, as well as the same first rhombus in the third row;
rotating the pieces 180 degrees, we recognize the same rule as in the first limit, with the substitution $z\to z^{-1}$, cf lemma~\ref{lem:inv}. This is nothing but the usual reflection of duality in the quantum integrable setting; 
we recover this way the puzzle rule for {\em dual}\/ 
Schubert classes (i.e., classes of ideal sheaves of boundaries of Schubert varieties \rem[gray]{actually, in the Grassmannian case, they're also $O(-1)$}) 
in the equivariant $K$-theory
of the Grassmannian as formulated in \cite{artic68}.
In this sense, the finite $q$ rule interpolates between Schubert and dual
Schubert puzzle rules.
\junk{what about the subtlety of structure vs canonical sheaves of Schubert
varieties? well, there's Serre duality, plus some freedom in tensoring
by a $O(a)$ line bundle. TODO: explain this better. see also Mihalcea's slides for conventions}

\newcommand\tile[2][0]{%
\begin{tikzpicture}[math mode,x={(-0.577cm,-1cm)},y={(0.577cm,-1cm)},rotate=#1]
\draw[thick] (0,0) -- (0,1) -- (1,0) -- cycle;
\clip (0,0) -- (0,1) -- (1,0) -- cycle;
#2
\end{tikzpicture}%
}
\subsection{Nonequivariant rule and loop model}\label{ssec:loop}
Let us now discuss the nonequivariant rule (NE direction in the diagram of \S \ref{ssec:limits}).
\rem[gray]{AK: that's cool that NE is for NonEquivariant. PZJ: totally intentional}
First, note that setting $z=1$ in equation~\eqref{eq:rhombi1},
either the fugacity of the equivariant rhombi vanishes,
or it factors as a product of fugacities of triangles of
equation~\eqref{eq:tri1}, in accordance with \eqref{eq:factorR}.

The resulting $d=1$ nonequivariant puzzles can be thought of as the partition
function (with particular boundary conditions) of a quantum integrable model
which already appeared in a different context. 
In \cite{Resh-On}, Reshetikhin considered the so-called $O(n)$ loop model on the
honeycomb lattice, and observed that if loops cover the entire lattice,
then the model is exactly solvable as an $A_2$ integrable system.
We reconnect to this loop model here (in the dual graphical description).

A \defn{loop puzzle} for the Grassmannian $Gr(k,n)$
is an assignment to each elementary
triangle of a size $n$ equilateral triangle of one of the tiles
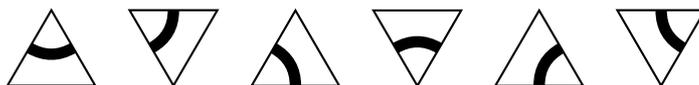
\[
\foreach\rot in {0,60,120,180,240,300} {\tile[\rot]{
\draw[line width=1.5mm,line cap=round] (0,0.5) [bend left] to (0.5,0);
}\quad}
\]
as well as $k$ incoming arrows on the North-East side, $n-k$ outgoing arrows on the North-West side, and $k$ outgoing arrows
and $n-k$ incoming arrows on the South side,
in such a way that
\begin{itemize}
\item At each internal edge, the black lines are continuous.
\item Lines end on the sides exactly at arrows, and the directions of the arrows at opposite ends must match (i.e.,
lines can only connect NW and NE side, or NW to incoming arrows on the S side, or NE to outgoing arrows on the S side, or opposite arrows on the S side).
\end{itemize}

For example, if $k=3$, $n=7$, a valid loop puzzle is
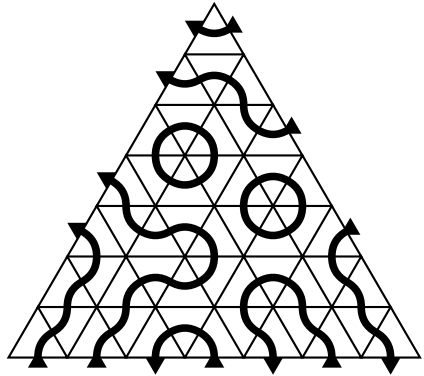
\begin{center}
\tikzset{arrow/.style={}}
\def\thescale{0.67}
\begin{tikzpicture}[math mode,nodes={edgelabel},x={(-0.577cm,-1cm)},y={(0.577cm,-1cm)},line width=1mm,scale=\thescale]
\foreach \x in {0,1,3,4}
\fill (\x+0.333,0) -- ++(0.333,0) -- ++(0,-0.333) -- cycle;
\foreach \y in {0,2,4}
\fill (0.14,\y+0.43) -- ++(-0.34,0) -- ++(0,0.34) -- cycle;
\foreach \z in {0,1,3,5}
\fill (6.43-\z,\z+0.43) -- ++(0.34,0) -- ++(-0.34,0.34) -- cycle;
\foreach \z in {2,4,6}
\fill (6.33-\z,\z+0.67) -- ++(0.34,-0.34) -- ++(0,0.34) -- cycle;
\draw[black,thick] (0,0) -- ++(0,1) -- ++(1,-1) -- cycle;
\draw[arrow,bend left] (0.0,0.5) to (0.5,0.0);
\draw[arrow,bend right] (0.5,1.0) to (1.0,0.5);
\draw[black,thick] (0,1) -- ++(0,1) -- ++(1,-1) -- cycle;
\draw[arrow,bend right] (0.5,1.5) to (0.5,1.0);
\draw[arrow,bend left] (0.5,2.0) to (0.5,1.5);
\draw[black,thick] (0,2) -- ++(0,1) -- ++(1,-1) -- cycle;
\draw[arrow,bend left] (0.0,2.5) to (0.5,2.0);
\draw[arrow,bend right] (0.5,3.0) to (1.0,2.5);
\draw[black,thick] (0,3) -- ++(0,1) -- ++(1,-1) -- cycle;
\draw[arrow,bend right] (0.5,3.5) to (0.5,3.0);
\draw[arrow,bend right] (1.0,3.5) to (0.5,3.5);
\draw[black,thick] (0,4) -- ++(0,1) -- ++(1,-1) -- cycle;
\draw[arrow,bend right] (0.0,4.5) to (0.5,4.5);
\draw[arrow,bend right] (0.5,4.5) to (0.5,5.0);
\draw[black,thick] (0,5) -- ++(0,1) -- ++(1,-1) -- cycle;
\draw[arrow,bend left] (0.5,5.0) to (0.5,5.5);
\draw[arrow,bend right] (0.5,5.5) to (0.5,6.0);
\draw[black,thick] (0,6) -- ++(0,1) -- ++(1,-1) -- cycle;
\draw[arrow,bend left] (0.5,6.0) to (0.5,6.5);
\draw[black,thick] (1,0) -- ++(0,1) -- ++(1,-1) -- cycle;
\draw[arrow,bend left] (1.0,0.5) to (1.5,0.0);
\draw[arrow,bend right] (1.5,1.0) to (2.0,0.5);
\draw[black,thick] (1,1) -- ++(0,1) -- ++(1,-1) -- cycle;
\draw[arrow,bend right] (1.5,1.5) to (1.5,1.0);
\draw[arrow,bend right] (2.0,1.5) to (1.5,1.5);
\draw[black,thick] (1,2) -- ++(0,1) -- ++(1,-1) -- cycle;
\draw[arrow,bend right] (1.0,2.5) to (1.5,2.5);
\draw[arrow,bend right] (1.5,2.5) to (1.5,3.0);
\draw[black,thick] (1,3) -- ++(0,1) -- ++(1,-1) -- cycle;
\draw[arrow,bend right] (1.5,3.0) to (1.0,3.5);
\draw[arrow,bend right] (1.5,4.0) to (2.0,3.5);
\draw[black,thick] (1,4) -- ++(0,1) -- ++(1,-1) -- cycle;
\draw[arrow,bend right] (1.5,4.5) to (1.5,4.0);
\draw[arrow,bend left] (1.5,5.0) to (1.5,4.5);
\draw[black,thick] (1,5) -- ++(0,1) -- ++(1,-1) -- cycle;
\draw[arrow,bend right] (1.5,5.5) to (1.5,5.0);
\draw[black,thick] (2,0) -- ++(0,1) -- ++(1,-1) -- cycle;
\draw[arrow,bend right] (2.0,0.5) to (2.5,0.5);
\draw[arrow,bend right] (2.5,0.5) to (2.5,1.0);
\draw[black,thick] (2,1) -- ++(0,1) -- ++(1,-1) -- cycle;
\draw[arrow,bend right] (2.5,1.0) to (2.0,1.5);
\draw[arrow,bend right] (2.5,2.0) to (3.0,1.5);
\draw[black,thick] (2,2) -- ++(0,1) -- ++(1,-1) -- cycle;
\draw[arrow,bend right] (2.5,2.5) to (2.5,2.0);
\draw[arrow,bend right] (3.0,2.5) to (2.5,2.5);
\draw[black,thick] (2,3) -- ++(0,1) -- ++(1,-1) -- cycle;
\draw[arrow,bend right] (2.0,3.5) to (2.5,3.5);
\draw[arrow,bend right] (2.5,3.5) to (2.5,4.0);
\draw[black,thick] (2,4) -- ++(0,1) -- ++(1,-1) -- cycle;
\draw[arrow,bend left] (2.5,4.0) to (2.5,4.5);
\draw[black,thick] (3,0) -- ++(0,1) -- ++(1,-1) -- cycle;
\draw[arrow,bend right] (3.5,0.5) to (3.5,0.0);
\draw[arrow,bend left] (3.5,1.0) to (3.5,0.5);
\draw[black,thick] (3,1) -- ++(0,1) -- ++(1,-1) -- cycle;
\draw[arrow,bend left] (3.0,1.5) to (3.5,1.0);
\draw[arrow,bend left] (4.0,1.5) to (3.5,2.0);
\draw[black,thick] (3,2) -- ++(0,1) -- ++(1,-1) -- cycle;
\draw[arrow,bend right] (3.5,2.0) to (3.0,2.5);
\draw[arrow,bend right] (3.5,3.0) to (4.0,2.5);
\draw[black,thick] (3,3) -- ++(0,1) -- ++(1,-1) -- cycle;
\draw[arrow,bend right] (3.5,3.5) to (3.5,3.0);
\draw[black,thick] (4,0) -- ++(0,1) -- ++(1,-1) -- cycle;
\draw[arrow,bend right] (4.5,0.5) to (4.5,0.0);
\draw[arrow,bend right] (5.0,0.5) to (4.5,0.5);
\draw[black,thick] (4,1) -- ++(0,1) -- ++(1,-1) -- cycle;
\draw[arrow,bend left] (4.5,1.5) to (4.0,1.5);
\draw[arrow,bend right] (5.0,1.5) to (4.5,1.5);
\draw[black,thick] (4,2) -- ++(0,1) -- ++(1,-1) -- cycle;
\draw[arrow,bend right] (4.0,2.5) to (4.5,2.5);
\draw[black,thick] (5,0) -- ++(0,1) -- ++(1,-1) -- cycle;
\draw[arrow,bend left] (5.5,0.5) to (5.0,0.5);
\draw[arrow,bend right] (6.0,0.5) to (5.5,0.5);
\draw[black,thick] (5,1) -- ++(0,1) -- ++(1,-1) -- cycle;
\draw[arrow,bend left] (5.5,1.5) to (5.0,1.5);
\draw[black,thick] (6,0) -- ++(0,1) -- ++(1,-1) -- cycle;
\draw[arrow,bend left] (6.5,0.5) to (6.0,0.5);
\end{tikzpicture}
\end{center}

To each such loop puzzle, one associates the fugacity
\[
(-q-q^{-1})^{\#\text{closed loops}}
(-q)^{\#\text{paths oriented rightward with endpoints on the S side}}
\]

\begin{prop}
The coefficient of $S^\nu$ in the expansion of $S^\lambda S^\mu$ in $K_{\CC^\times}^{\loc}(Gr(k,n))$ is
the sum of fugacities of loop puzzles such that
$1$s of $\lambda$ correspond to arrows on the NW side, 
$0$s of $\mu$ to arrows on the NE side,
$1$s (resp.\ $0$s) of $\nu$ to incoming (resp.\ outgoing) arrows on the S side.
\end{prop}

\tikzset{arrowmod/.style={postaction={decorate,thin,decoration={markings,mark = at position #1 with {\arrow{Latex}}}}},arrowmod/.default=0.65}
\begin{proof}
There is a bijection between $d=1$ puzzles made of the triangles of \eqref{eq:tri1} and {\em oriented}\/ loop puzzles, given by:
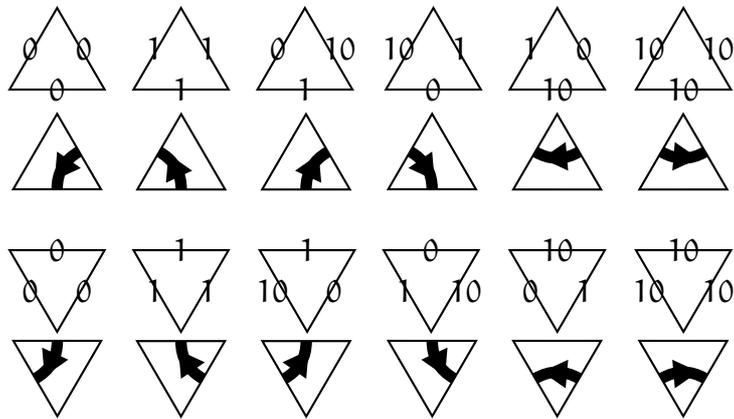
\def\oriloop{\draw[bend right,line width=1.5mm,line cap=round] (0,0.5) to (0.5,0.5); \fill[shift={(0.25,0.25)},scale=0.33,rotate=25] (0,0) -- (1,0) -- (0,1) -- cycle;}
\def\oriloopb{\begin{scope}[xscale=-1]\oriloop\end{scope}}
\[
\begin{matrix}
\uptri000&\uptri111&\uptri01{10}&\uptri{10}01&\uptri1{10}0&\uptri{10}{10}{10}
\\[5mm]
\tile{\oriloop}
&
\tile[240]{\oriloop}
&
\tile[-240]{\oriloopb}
&
\tile{\oriloopb}
&
\tile[-120]{\oriloopb}
&
\tile[120]{\oriloop}
\\[5mm]
\downtri000&\downtri111&\downtri01{10}&\downtri{10}01&\downtri1{10}0&\downtri{10}{10}{10}\\[5mm]
\tile[-60]{\oriloopb}
&
\tile[-180]{\oriloopb}
&
\tile[180]{\oriloop}
&
\tile[60]{\oriloop}
&
\tile[300]{\oriloop}
&
\tile[-300]{\oriloopb}
\end{matrix}
\]
Starting from an ordinary puzzle and erasing the arrows except at the boundaries, 
one obtains a loop puzzle. Inversely, given a loop puzzle,
one can orient each individual triangle starting from the boundaries, except
for closed loops, which have two possible orientations. The formula for
the fugacity follows.
\end{proof}

As a corollary, we find that
the coefficient of $S^\nu$ in the expansion of $S^\lambda S^\mu$ in $K_{\CC^\times}^{\loc}(Gr(k,n))$
is a polynomial with positive coefficients in $-q$ and $-(q+q^{-1})$; these coefficients
are in fact unique, and given by the numbers of puzzles with fixed number of closed loops and S side rightward oriented paths.

{\em Example.} By direct computation, one finds the coefficient $-3q-q^{-1}$ for $\lambda=\mu=0101$, $\nu=1010$.
Indeed, there are three loop puzzles:
\long\def\puz#1{%
\begin{tikzpicture}[math mode,x={(-0.577cm,-1cm)},y={(0.577cm,-1cm)}]
\draw[thick] (0,0) -- ++(0,1); \draw[thick] (0,0) --  ++(1,0); \draw[thick] (0+1,0) -- ++(-1,1); 
\draw[thick] (0,1) -- ++(0,1); \draw[thick] (0,1) --  ++(1,0); \draw[thick] (0+1,1) -- ++(-1,1); 
\draw[thick] (0,2) -- ++(0,1); \draw[thick] (0,2) --  ++(1,0); \draw[thick] (0+1,2) -- ++(-1,1); 
\draw[thick] (0,3) -- ++(0,1); \draw[thick] (0,3) --  ++(1,0); \draw[thick] (0+1,3) -- ++(-1,1); 
\draw[thick] (1,0) -- ++(0,1); \draw[thick] (1,0) --  ++(1,0); \draw[thick] (1+1,0) -- ++(-1,1); 
\draw[thick] (1,1) -- ++(0,1); \draw[thick] (1,1) --  ++(1,0); \draw[thick] (1+1,1) -- ++(-1,1); 
\draw[thick] (1,2) -- ++(0,1); \draw[thick] (1,2) --  ++(1,0); \draw[thick] (1+1,2) -- ++(-1,1); 
\draw[thick] (2,0) -- ++(0,1); \draw[thick] (2,0) --  ++(1,0); \draw[thick] (2+1,0) -- ++(-1,1); 
\draw[thick] (2,1) -- ++(0,1); \draw[thick] (2,1) --  ++(1,0); \draw[thick] (2+1,1) -- ++(-1,1); 
\draw[thick] (3,0) -- ++(0,1); \draw[thick] (3,0) --  ++(1,0); \draw[thick] (3+1,0) -- ++(-1,1); 
\fill (0.33,0) -- (0.67,0) -- (0.67,-0.33) -- cycle;
\fill (2.33,0) -- (2.67,0) -- (2.67,-0.33) -- cycle;
\fill (0.14,0.43) -- (-0.2,0.43) -- (-0.2,0.77) -- cycle;
\fill (0.14,2.44) -- (-0.2,2.44) -- (-0.2,2.77) -- cycle;
\fill (3.43,0.43) -- (3.77,0.43) -- (3.43,0.77) -- cycle;
\fill (2.33,1.67) -- (2.67,1.33) -- (2.67,1.67) -- cycle;
\fill (1.43,2.43) -- (1.77,2.43) -- (1.43,2.77) -- cycle;
\fill (0.33,3.67) -- (0.67,3.33) -- (0.67,3.67) -- cycle;
\clip (0,0) -- (4,0) -- (0,4) -- cycle;
#1
\end{tikzpicture}%
}
\[
\puz{
\begin{scope}[line width=1.5mm,line cap=round]
\draw (0.5,0) [bend right] to (0,0.5);
\draw (2.5,0) [bend right] to (2,0.5) [bend right] to (1.5,0.5) [bend left] to (1,0.5) [bend left] to (0.5,1) [bend left] to (0.5,1.5) [bend left] to (1,1.5) [bend right] to (1.5,1.5) [bend right] to (1.5,2) [bend left] to (1.5,2.5);
\draw (0,2.5) [bend right] to (0.5,2.5) [bend right] to (0.5,3) [bend left] to (0.5,3.5);
\draw[bend left,purple!80!black] (3.5,0.5) to (3,0.5) to (2.5,1) to (2.5,1.5);
\end{scope}
}
\quad
\puz{
\begin{scope}[line width=1.5mm,line cap=round]
\draw (0.5,0) [bend right] to (0,0.5);
\draw (2.5,0) [bend left] to (2.5,0.5) [bend left] to (3,0.5) [bend right] to (3.5,0.5);
\draw (0,2.5) [bend right] to (0.5,2.5) [bend right] to (0.5,3) [bend left] to (0.5,3.5);
\draw (2.5,1.5) [bend left] to (2,1.5) [bend left] to (1.5,2) [bend left] to (1.5,2.5);
\draw [bend right,brown!80!black] (1,0.5) to (1.5,0.5) to (1.5,1) to (1,1.5) to (0.5,1.5) to (0.5,1) to cycle;
\end{scope}
}
\quad
\puz{
\begin{scope}[line width=1.5mm,line cap=round,xscale=-1]
\draw (0.5,0) [bend right] to (0,0.5);
\draw (2.5,0) [bend right] to (2,0.5) [bend right] to (1.5,0.5) [bend left] to (1,0.5) [bend left] to (0.5,1) [bend left] to (0.5,1.5) [bend left] to (1,1.5) [bend right] to (1.5,1.5) [bend right] to (1.5,2) [bend left] to (1.5,2.5);
\draw (0,2.5) [bend right] to (0.5,2.5) [bend right] to (0.5,3) [bend left] to (0.5,3.5);
\draw[bend left,purple!80!black] (3.5,0.5) to (3,0.5) to (2.5,1) to (2.5,1.5);
\end{scope}
}
\]
where the loop in brown results in a fugacity of $-q-q^{-1}$, whereas
the two purple paths result in fugacities $-q$.

One defines similarly a dual fugacity to a loop puzzle, conventionally drawn upside-down, as
\[
(-q-q^{-1})^{\#\text{closed loops}}
(-q^{-1})^{\#\text{paths oriented leftward with endpoints on the N side}}
\]
Then we have the obvious dual statement:
\begin{prop}
The coefficient of $S_\nu$ in the expansion of $S_\lambda S_\mu$ in $K_{\CC^\times}^{\loc}(Gr(k,n))$ is
the sum of dual fugacities of loop puzzles such that
$1$s of $\lambda$ correspond to arrows on the SW side, 
$0$s of $\mu$ to arrows on the SE side,
$1$s (resp.\ $0$s) of $\nu$ to incoming (resp.\ outgoing) arrows on the N side.
\end{prop}

\junk{ref to Nadeau/me about FPLs?}
\junk{note that in both $q^\pm\to0$, the loop weight goes to infinity, forbidding closed loops}

\junk{There remains only the bottom right case of the commutative diagram.
The right vertical arrow is well-known, and one recovers this way
the puzzle rule for Schubert classes in the 
equivariant cohomology of the Grassmannian \cite{KT} (related to Grassmannian
Schubert polynomials a.k.a.\ factorial Schur polynomials).
One easily checks that the bottom horizontal 
arrow produces the same result, being careful
to twist by $\hbar^{-\text{inversion number}}$ before taking the limit
$\hbar\to\infty$.
--- orphaned paragraph; maybe remove entirely?}

\section{The limit to ordinary cohomology}\label{sec:coho}
We now discuss the transition from $K$-theory to ordinary cohomology,
corresponding to the vertical arrows of the diagram of \S\ref{ssec:limits}.
There is an arbitrariness in the sign of $q$, since only its square
$q^2=t^{-1}$ is geometrically meaningful; this choice of sign is
effectively equivalent to a choice of polarization in the cohomology
limit. The conventional choice is $q\to 1$; here we choose to send $q$
to $-1$, which has some technical advantages (in particular, it
trivializes the polarization, which simplifies slightly the geometric
discussion in \S\ref{sec:geominterp}).  
\rem[gray]{PZJ: for the Euler--Poincar\'e theorem, seems we want $q=1$.
  oh well. AK: does $q=-1$ makes every $K$-piece worth $-1$?
  (By a $B$-calculation we can show that) the number of $K$-pieces
  matches the $-1$ exponent we currently need in that theorem, $\bmod 2$,
  so the theorem would be cleaner at $q=-1$ if this is the case.
  (Of course we'd include a comment that the puzzles are each worth
  $1$ or each worth $-1$, predictably.) Update: PZJ says that with $q=-1$
  we get thm \ref{thm:EP} as stated, so, it should just get a comment
  about how $q=1$ would change things.
}

\subsection{Localization revisited}
In \S\ref{ssec:loc} and \S\ref{ssec:Rmat}, we have introduced the localisation that is convenient for most of the paper.
However, it prevents us from specializing at $q=\pm1$. We briefly sketch how one can get around this difficulty.

What we are formalizing here is the following procedure:
in order to obtain formul\ae\ in
$H^*_{\hat T}(P_-\dom G)$, where
$H^*_{\hat T}(pt) =\ZZ[\hbar,y_1,\ldots,y_n]$, we need to set
$q=-e^{-\hbar/2}$, $z_i=e^{y_i}$ and to expand at first nontrivial
order in $\hbar,y_1,\ldots,y_n$.

Of course, a practical point of view is that this procedure produces a well-defined result (i.e., a $\hbar\to0$ limit)
in all cases of interest to us.
A more formal way to {\em guarantee} that this construction is well-defined is to change the base ring by allowing less localisation.

The first point is that clearly one shouldn't tensor with $\CC(q)$; instead, one can for example use $\CC[q^\pm,[a]^{-1}_q,a\ge2]$,
where $[a]_q=(q^a-q^{-a})/(q-q^{-1})$. This means we can specialise at all $q$s except roots of unity distinct from $\pm1$. 
This corresponds to the standard lore that the representation theory of quantum groups is uniform except at roots of unity.

The second point is to consider the ratios $[(k\hbar-y_i+y_j)/\hbar]_q:=(1-q^{2k}z_i/z_j)/(1-q^2)$ for all $k,i,j$ and to add them to our
base ring. Let $S$ be the set of $(i,j,k)$ for which we previously made $1-q^{2k}z_i/z_j$ invertible;
we now instead make $[(k\hbar-y_i+y_j)/\hbar]_q$ invertible for all $(i,j,k)\in S$.
One can check that all expressions that we have considered so far live in this ring.

The specialisation $q\mapsto \pm1$ is now well-defined:
it is easy to see that there is a well-defined map into $\CC[\hbar^\pm,y_1,\ldots,y_n,(k\hbar-y_i+y_j)^{-1},(i,j,k)\in S]$
which sends $q$ to $\pm1$ and $[(k\hbar-y_i+y_j)/\hbar]$ to $(k\hbar-y_i+y_j)/\hbar$ for all $i,j,k$.

\subsection{Segre--Schwartz--MacPherson classes}
In this limit to cohomology, the $S^\lambda$ defined in \S\ref{sec:stable}
turn into classes $S^\lambda_H\in H^{*\loc}_{\hat T}(T^*(P_-\backslash G))$ which
are called Segre--Schwartz--MacPherson (SSM) classes
(see e.g. \cite{FeherRimanyi}); explicitly, they are given by
the same diagram \eqref{eq:defS} in which $K$-theoretic parameters
$x_1,\ldots,x_n,z_1,\ldots,z_n$ have been replaced with their $H^*$
analogues, namely Chern roots $t_1,\ldots,t_n$ and equivariant
parameters $y_1,\ldots,y_n$, and in which the $R$-matrix entries
\eqref{eq:Rsingle} become
\begin{equation}\label{eq:RsingleH}
\tikz[baseline=0,xscale=0.5]
{
\draw[invarrow=0.75,d] (-1,-1) node[below] {$\m i$} -- node[pos=0.75,right] {$\ss y''$} (1,1) node[above] {$\m l$};
\draw[invarrow=0.75,d] (1,-1) node[below] {$\m j$} -- node[pos=0.75,left] {$\ss y'$} (-1,1) node[above] {$\m k$};
}
=
\check R_{ij}^{kl}(y',y'')=
\check R_{ij}^{kl}(y''-y')=
\begin{cases}
1& i=j=k=l
\\
\frac{\textstyle y''-y'}{\textstyle \hbar+y'-y''} & i=l\ne j=k
\\
\frac{\textstyle \hbar}{\textstyle \hbar+y'-y''} & i=k\ne j=l
\\
0 & \text{else.}
\end{cases}
\end{equation}
\rem[gray]{why in \url{https://arxiv.org/pdf/1709.08697.pdf}
  they divide by $e(TX)$, not $e(T^*X)$? probably a $h\to -h$
  problem. in fact SSM/CSM are usually defined in $P_-\backslash G$
  where there's no $h$. AK: hopefully the \S\ref{ssec:SSM} discussion
  suffices for this}

We can define $S_\lambda^H$ analogously as limit to cohomology of
$S_\lambda$ (defined by \eqref{eq:defSd}); but since the entries
\eqref{eq:RsingleH} of the $R$-matrix are invariant under reversal of
both arrows (or equivalently by $180^\circ$ rotation),
lemma~\ref{lem:inv} simplifies to
\begin{equation}\label{eq:Hdual1}
  S^H_\lambda(t_1,\ldots,t_n,y_1,\ldots,y_n) = S_H^{\lambda^*}(t_1,\ldots,t_n,y_n,\ldots,y_1)
\end{equation}
where we recall that $\lambda^*$ denotes the string $\lambda$ read
backwards. Similarly, one has for stable classes
\begin{equation}\label{eq:Hdual2}
  \St^H_\lambda(t_1,\ldots,t_n,y_1,\ldots,y_n) = \St_H^{\lambda^*}(t_1,\ldots,t_n,y_n,\ldots,y_1)
\end{equation}

Also note that the $S_H^\lambda$ 
are degree $0$ classes (which is possible because of localization).

The rest of this section being devoted to cohomology only, we omit the sub/superscripts ``$\scriptstyle H$'' and simply write
$S^\lambda$, etc.

\subsection{\texorpdfstring{$H^*_{\hat T}$}{H T hat}
  puzzle rules for \texorpdfstring{$d\le 3$}{d<=3}}
\label{ssec:d123}
We now discuss puzzle rules satisfied by the $S^\lambda$.
More specifically, let us consider the leftmost vertical arrow of the
diagram of \S\ref{ssec:limits}.
In this limit, the fugacities simplify slightly,
which is the sign of the $\mathfrak{x}_{2d}$-invariance of the underlying
$R$-matrix. 

In the rest of this section, we assume $d\le 3$.
We start by fixing a basis of weight vectors of each $V_a(z)$, $a=1,2,3$,
given by the stable envelope construction (with, as mentioned above, the trivial choice of 
polarization).
We recall that we then label triangles and rhombi using the corresponding weights, i.e., write
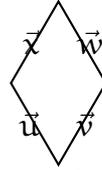
\rh{\vX}{\vY}{\vZ}{\vW} for the matrix entry between basis vectors with weights $\vX,\vZ\in\tau^2 \W$ and $\vY,\vW\in \W$,
and similarly for triangles.

All $R$-matrices for $d\le 3$ were provided in \cite{artic71} and \S\ref{sec:exd1}, and we simply
state the results.

The triangles all have fugacity $1$.\footnote{The sign convention
is somewhat unusual from the point of view of invariant
theory. E.g., at $d=1$, the
triangles are supposed to represent the fully antisymmetric tensor of $\mathfrak{sl}_3$; to recover
the usual permutation signs, one needs to set $q=1$ (as opposed to $q=-1$),
and then for example change the sign of states labeled $10$.
}\junk{proof?
it's a consequence of the rhombi weights
below plus the 180 degree invariance which should've been mentioned somewhere, up to a sign}

As to the rhombi, we can parametrize them as follows. We assume $\vX+\vY=\vW+\vZ$ (otherwise the matrix entry is zero).
We use the Killing form $\killing{\cdot}{\cdot}$
with the normalization that all weights have squared norm $2$.
For ease of comparison between different values of $d$,
we also denote $a=3$, $2$, $3/2$, $1$ for $d=1,2,3,4$ 
(see \cite[\S 2.5]{artic71} for a justification), 
and use it to parametrize our scalar products.

The matrix entry above can then only depend on the scalar products of
the various weights.  Because weight vectors have squared norm $2$,
there are only two independent scalar products, namely
$s=\killing{\vX}{\vW}=\killing{\vY}{\vZ}$ 
and 
$t=\killing{\vX}{\vY}=\killing{\vZ}{\vW}$, 
the third scalar product being given by
$r=\killing{\vX}{\vZ}=\killing{\vY}{\vW}=t-s+2$ (see \eqref{eq:scalprod}
for an identical calculation).
Because the latter is less than or equal to $2$, one has $t\le s$.
A table of scalar products can be found in appendix~\ref{app:scal}.

At $d=1$, one can easily deduce from the results of \S\ref{sec:exd1}
the following table \junk{AK: should the missing entry be $0$? Or is there
  at least a quick note to put as to why this and the $d=3$ table
  are lower triangular. The $d=4$ table is endowed with $0$s.
PZJ: right above it says $t\le s$. the zeroes in the $d=4$ table are a mistake}
\begin{equation}\label{eq:tab1}
  \begin{array}{c|ccc}
    s\backslash t & -1 & -1+a
    \\
    \hline
    -1 & 1 &  \\
    -1+a & \frac{\hbar}{\hbar-y}& \frac{y}{\hbar-y}\\
  \end{array}
  \qquad\qquad
  \rh{\vX}{\vY}{\vZ}{\vW},
  \ 
  s=\killing{\vX}{\vW},
  \ 
  t=\killing{\vX}{\vY}
\end{equation}
where the parameter $y$ should be understood as the difference $y_j-y_i$
of equivariant parameters in $H^*_{\hat T}(pt)\cong \ZZ[\hbar,y_1,\ldots,y_n]$
related to the coordinates $i<j$ of the rhombus (and recall $a=3$).

It is not hard to see that the exact same table \eqref{eq:tab1} holds at $d=2$ (except $a=2$ now);
this can in principle be derived from \cite[appendix~A]{artic71}.

Instead we jump straight to $d=3$.
The rhombi are all the ones that are allowed by weight conservation; there are 3591 of them.
There are six different types of entries, because
imposing $r,s,t\in\{-1,1/2,2\}$ leads to 6 solutions.

Diagonal rhombi, i.e.,
of the form \rh{\vX}{\vY}{\vX}{\vY}, implying $r=2$ and $s=t$, fall into 3 classes:
\begin{itemize}
\item There are 27 entries for which $s=t=2$,
i.e., $\vX=\vY=\vZ=\vW$,
of the form
$
\frac{-y(3\hbar-y)}{(\hbar-y)(4\hbar-y)}
$.
\item 270 entries with $s=t=-1$ which are equal to $1$.
\item 432 entries with $s=t=1/2$, of the form
$
\frac{y}{\hbar-y}
$.
\end{itemize}
Similarly, the nondiagonal rhombi are
\begin{itemize}
\item 2160 entries with $s=1/2$, $t=-1$, of the form
$
\frac{\hbar}{\hbar-y}
$.
\item 432 entries with $s=2$, $t=1/2$, with fugacity
$
\frac{\hbar y}{(\hbar-y)(4\hbar-y)}
$.
\item 270 entries with $s=2$, $t=-1$, with fugacity
$
\frac{4\hbar^2}{(\hbar-y)(4\hbar-y)}
$.
\end{itemize}
In the last two cases note that we have $\vX=\vW$, $\vY=\vZ$.

We can summarize these results in the table:
\begin{equation}\label{eq:tab3}
\begin{array}{c|ccc}
s\backslash t & -1 & -1+ a & -1+2a
\\
\hline
-1 & 1 &  \\
-1+a & \frac{\hbar}{\hbar-y}& \frac{y}{\hbar-y}\\
-1+2a & \frac{4\hbar^2}{(\hbar-y)(4\hbar-y)} & \frac{\hbar y}{(\hbar-y)(4\hbar-y)} & \frac{-y(3\hbar-y)}{(\hbar-y)(4\hbar-y)}\\
\end{array}
\qquad
\rh{\vX}{\vY}{\vZ}{\vW},
\ 
s=\killing{\vX}{\vW},
\ 
t=\killing{\vX}{\vY}
\end{equation}
where $a=3/2$.
The diagonal entries correspond to diagonal rhombi; the first column corresponds to ``nonequivariant'' rhombi (i.e., whose contribution does not vanish at $y=0$ and then factors into a product of two triangles).

Comparing tables \eqref{eq:tab1} and \eqref{eq:tab3}, we note that one
is a subtable of the other. Another way of understanding this is that
were we to consider a $d$-step flag variety as a $d'$-step flag
variety with $d\le d'\le 3$ (and some trivial steps), the puzzle rule
would remain the same. This is related to the functoriality property
of \cite[\S 2.2]{artic71}. (As will be explained in the next
subsection, this table containment does {\em not} continue to $d<d'=4$.)

\begin{ex}
Let us compute the product $S^{3201}S^{2013}$ in $H^*_T(T^*(\text{3-step}))$. 
One has, using the shorthand notations $y_{ij}=y_i-y_j$ and $y'_{ij}=\hbar+y_i-y_j$,
\begin{align*}
\begin{matrix}
&|_{3201}&|_{3210}\\
S^{3201}=(&
\frac{y_{21}y_{31}y_{32}y_{41}y_{42}}{y'_{12}y'_{13}y'_{23}y'_{14}y'_{24}}
&
\frac{\hbar y_{21}y_{31}y_{32}y_{41}y_{42}}{y'_{12}y'_{13}y'_{23}y'_{14}y'_{24}y'_{34}}
&)
\\ 
S^{2013}=(&
\frac{\hbar^3y_{31}y_{41}}{y'_{12}y'_{13}y'_{23}y'_{14}y'_{24}}
&
\frac{\hbar^2 y_{31}y_{41}(\hbar^2+y_2y_33^22y_4+y_3y_4)}{y'_{12}y'_{13}y'_{23}y'_{14}y'_{24}y'_{34}}
&)
\end{matrix}
\end{align*}
all other entries being irrelevant. From this we conclude
\[
S^{3201}S^{2013}=
\frac{\hbar^3 y_{13}y_{14}}{y'_{12}y'_{13}y'_{23}y'_{14}y'_{24}} 
S^{3201}
+ 
\frac{\hbar^3 y_{13}y'_{32}y_{14}}{y'_{12}y'_{13}y'_{23}y'_{14}y'_{24}y'_{34}} 
S^{3210}
\]

\junk{more $\tau$ here... but maybe people who care enough to look at
  these details are the same who can be told to look at paper \#1 for $\tau$}
In order to draw the corresponding puzzles, it is convenient to provide weights using
the traditional multinumber notation. A label $X$ translates into the weight $\vf_X$
multiplied by $1,-\tau,\tau^2$ for NE--SW, NW--SE, horizontal edges, 
where $\vf_X$ is defined inductively starting from
the single-number labels $0,\ldots,d$ by $\vf_{YX}=-\tau \vf_X-\tau^2 \vf_Y$
(see \cite[\S 2.3]{artic71} for details). We also provide the scalar products
of weights, under the form of a blue number $k \in \{0,1,2\}$ where the scalar product is $-1+k a$;
as well as the corresponding nonequivariant triangles (in light gray).
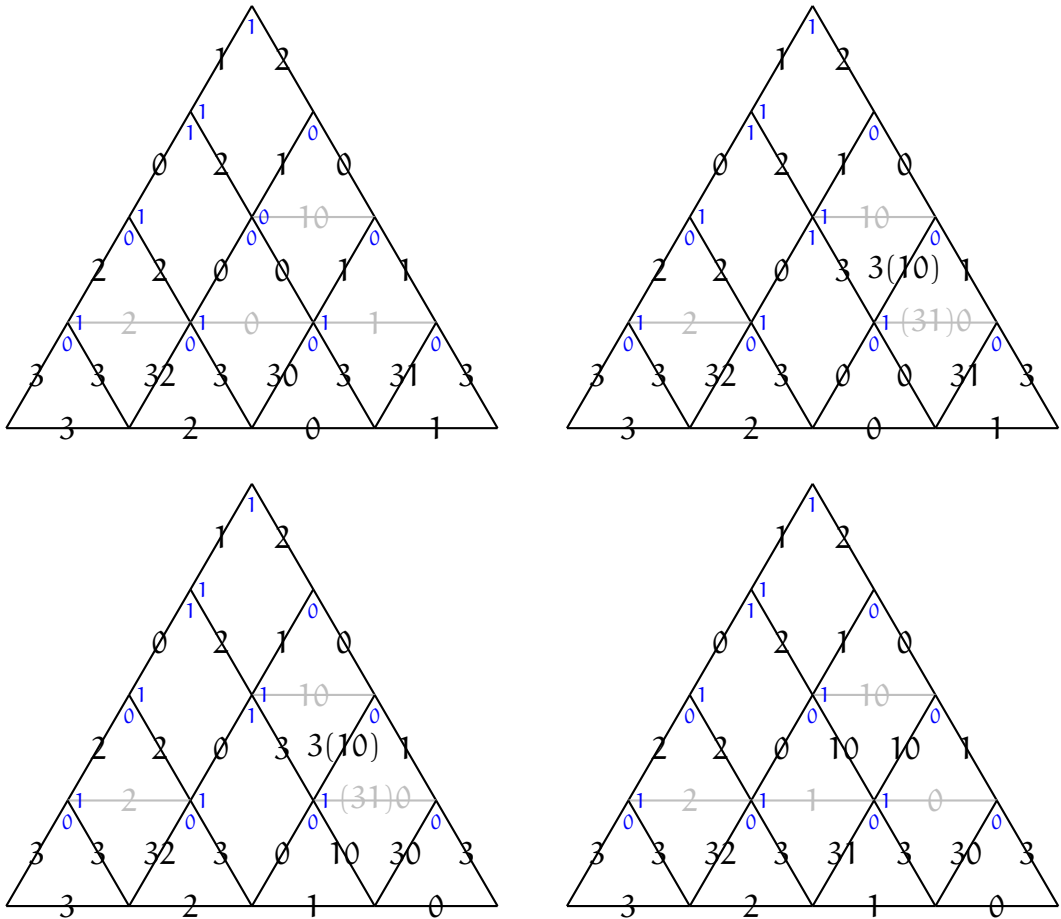
\begin{center}
\def\thescale{1.4}
\def\posa{0.5}\def\posb{0.5}
\begin{tikzpicture}[math mode,nodes={edgelabel},x={(-0.577cm,-1cm)},y={(0.577cm,-1cm)},scale=\thescale]\useasboundingbox (0,0) -- (4+0.5,0) -- (0,4+0.5);
\draw[thick] (0,0) -- node[pos=\posa] {2} ++(0,1); \draw[thick] (0,0) -- node[pos=\posb] {1} ++(1,0); \node[blue] at (0+0.1,0+0.1) {\ss 1}; \node[blue] at (0+0.9,0+0.1) {\ss 1}; 
\draw[thick] (0,1) -- node[pos=\posa] {0} ++(0,1); \draw[thick] (0,1) -- node[pos=\posb] {1} ++(1,0); \node[blue] at (0+0.1,1+0.1) {\ss 0}; \draw[thick,lightgray] (0+1,1) -- node {10} ++(-1,1); \node[blue] at (0+0.9,1+0.1) {\ss 0}; 
\draw[thick] (0,2) -- node[pos=\posa] {1} ++(0,1); \draw[thick] (0,2) -- node[pos=\posb] {1} ++(1,0); \node[blue] at (0+0.1,2+0.1) {\ss 0}; \draw[thick,lightgray] (0+1,2) -- node {1} ++(-1,1); \node[blue] at (0+0.9,2+0.1) {\ss 1}; 
\draw[thick] (0,3) -- node[pos=\posa] {3} ++(0,1); \draw[thick] (0,3) -- node[pos=\posb] {31} ++(1,0); \node[blue] at (0+0.1,3+0.1) {\ss 0}; \draw[thick] (0+1,3) -- node {1} ++(-1,1); 
\draw[thick] (1,0) -- node[pos=\posa] {2} ++(0,1); \draw[thick] (1,0) -- node[pos=\posb] {0} ++(1,0); \node[blue] at (1+0.1,0+0.1) {\ss 1}; \node[blue] at (1+0.9,0+0.1) {\ss 1}; 
\draw[thick] (1,1) -- node[pos=\posa] {0} ++(0,1); \draw[thick] (1,1) -- node[pos=\posb] {0} ++(1,0); \node[blue] at (1+0.1,1+0.1) {\ss 0}; \draw[thick,lightgray] (1+1,1) -- node {0} ++(-1,1); \node[blue] at (1+0.9,1+0.1) {\ss 1}; 
\draw[thick] (1,2) -- node[pos=\posa] {3} ++(0,1); \draw[thick] (1,2) -- node[pos=\posb] {30} ++(1,0); \node[blue] at (1+0.1,2+0.1) {\ss 0}; \draw[thick] (1+1,2) -- node {0} ++(-1,1); 
\draw[thick] (2,0) -- node[pos=\posa] {2} ++(0,1); \draw[thick] (2,0) -- node[pos=\posb] {2} ++(1,0); \node[blue] at (2+0.1,0+0.1) {\ss 0}; \draw[thick,lightgray] (2+1,0) -- node {2} ++(-1,1); \node[blue] at (2+0.9,0+0.1) {\ss 1}; 
\draw[thick] (2,1) -- node[pos=\posa] {3} ++(0,1); \draw[thick] (2,1) -- node[pos=\posb] {32} ++(1,0); \node[blue] at (2+0.1,1+0.1) {\ss 0}; \draw[thick] (2+1,1) -- node {2} ++(-1,1); 
\draw[thick] (3,0) -- node[pos=\posa] {3} ++(0,1); \draw[thick] (3,0) -- node[pos=\posb] {3} ++(1,0); \node[blue] at (3+0.1,0+0.1) {\ss 0}; \draw[thick] (3+1,0) -- node {3} ++(-1,1); 
\end{tikzpicture}
\begin{tikzpicture}[math mode,nodes={edgelabel},x={(-0.577cm,-1cm)},y={(0.577cm,-1cm)},scale=\thescale]\useasboundingbox (0,0) -- (4+0.5,0) -- (0,4+0.5);
\draw[thick] (0,0) -- node[pos=\posa] {2} ++(0,1); \draw[thick] (0,0) -- node[pos=\posb] {1} ++(1,0); \node[blue] at (0+0.1,0+0.1) {\ss 1}; \node[blue] at (0+0.9,0+0.1) {\ss 1}; 
\draw[thick] (0,1) -- node[pos=\posa] {0} ++(0,1); \draw[thick] (0,1) -- node[pos=\posb] {1} ++(1,0); \node[blue] at (0+0.1,1+0.1) {\ss 0}; \draw[thick,lightgray] (0+1,1) -- node {10} ++(-1,1); \node[blue] at (0+0.9,1+0.1) {\ss 1}; 
\draw[thick] (0,2) -- node[pos=\posa] {1} ++(0,1); \draw[thick] (0,2) -- node[pos=\posb] {3(10)} ++(1,0); \node[blue] at (0+0.1,2+0.1) {\ss 0}; \draw[thick,lightgray] (0+1,2) -- node {(31)0} ++(-1,1); \node[blue] at (0+0.9,2+0.1) {\ss 1}; 
\draw[thick] (0,3) -- node[pos=\posa] {3} ++(0,1); \draw[thick] (0,3) -- node[pos=\posb] {31} ++(1,0); \node[blue] at (0+0.1,3+0.1) {\ss 0}; \draw[thick] (0+1,3) -- node {1} ++(-1,1); 
\draw[thick] (1,0) -- node[pos=\posa] {2} ++(0,1); \draw[thick] (1,0) -- node[pos=\posb] {0} ++(1,0); \node[blue] at (1+0.1,0+0.1) {\ss 1}; \node[blue] at (1+0.9,0+0.1) {\ss 1}; 
\draw[thick] (1,1) -- node[pos=\posa] {3} ++(0,1); \draw[thick] (1,1) -- node[pos=\posb] {0} ++(1,0); \node[blue] at (1+0.1,1+0.1) {\ss 1}; \node[blue] at (1+0.9,1+0.1) {\ss 1}; 
\draw[thick] (1,2) -- node[pos=\posa] {0} ++(0,1); \draw[thick] (1,2) -- node[pos=\posb] {0} ++(1,0); \node[blue] at (1+0.1,2+0.1) {\ss 0}; \draw[thick] (1+1,2) -- node {0} ++(-1,1); 
\draw[thick] (2,0) -- node[pos=\posa] {2} ++(0,1); \draw[thick] (2,0) -- node[pos=\posb] {2} ++(1,0); \node[blue] at (2+0.1,0+0.1) {\ss 0}; \draw[thick,lightgray] (2+1,0) -- node {2} ++(-1,1); \node[blue] at (2+0.9,0+0.1) {\ss 1}; 
\draw[thick] (2,1) -- node[pos=\posa] {3} ++(0,1); \draw[thick] (2,1) -- node[pos=\posb] {32} ++(1,0); \node[blue] at (2+0.1,1+0.1) {\ss 0}; \draw[thick] (2+1,1) -- node {2} ++(-1,1); 
\draw[thick] (3,0) -- node[pos=\posa] {3} ++(0,1); \draw[thick] (3,0) -- node[pos=\posb] {3} ++(1,0); \node[blue] at (3+0.1,0+0.1) {\ss 0}; \draw[thick] (3+1,0) -- node {3} ++(-1,1); 
\end{tikzpicture}
\begin{tikzpicture}[math mode,nodes={edgelabel},x={(-0.577cm,-1cm)},y={(0.577cm,-1cm)},scale=\thescale]\useasboundingbox (0,0) -- (4+0.5,0) -- (0,4+0.5);
\draw[thick] (0,0) -- node[pos=\posa] {2} ++(0,1); \draw[thick] (0,0) -- node[pos=\posb] {1} ++(1,0); \node[blue] at (0+0.1,0+0.1) {\ss 1}; \node[blue] at (0+0.9,0+0.1) {\ss 1}; 
\draw[thick] (0,1) -- node[pos=\posa] {0} ++(0,1); \draw[thick] (0,1) -- node[pos=\posb] {1} ++(1,0); \node[blue] at (0+0.1,1+0.1) {\ss 0}; \draw[thick,lightgray] (0+1,1) -- node {10} ++(-1,1); \node[blue] at (0+0.9,1+0.1) {\ss 1}; 
\draw[thick] (0,2) -- node[pos=\posa] {1} ++(0,1); \draw[thick] (0,2) -- node[pos=\posb] {3(10)} ++(1,0); \node[blue] at (0+0.1,2+0.1) {\ss 0}; \draw[thick,lightgray] (0+1,2) -- node {(31)0} ++(-1,1); \node[blue] at (0+0.9,2+0.1) {\ss 1}; 
\draw[thick] (0,3) -- node[pos=\posa] {3} ++(0,1); \draw[thick] (0,3) -- node[pos=\posb] {30} ++(1,0); \node[blue] at (0+0.1,3+0.1) {\ss 0}; \draw[thick] (0+1,3) -- node {0} ++(-1,1); 
\draw[thick] (1,0) -- node[pos=\posa] {2} ++(0,1); \draw[thick] (1,0) -- node[pos=\posb] {0} ++(1,0); \node[blue] at (1+0.1,0+0.1) {\ss 1}; \node[blue] at (1+0.9,0+0.1) {\ss 1}; 
\draw[thick] (1,1) -- node[pos=\posa] {3} ++(0,1); \draw[thick] (1,1) -- node[pos=\posb] {0} ++(1,0); \node[blue] at (1+0.1,1+0.1) {\ss 1}; \node[blue] at (1+0.9,1+0.1) {\ss 1}; 
\draw[thick] (1,2) -- node[pos=\posa] {10} ++(0,1); \draw[thick] (1,2) -- node[pos=\posb] {0} ++(1,0); \node[blue] at (1+0.1,2+0.1) {\ss 0}; \draw[thick] (1+1,2) -- node {1} ++(-1,1); 
\draw[thick] (2,0) -- node[pos=\posa] {2} ++(0,1); \draw[thick] (2,0) -- node[pos=\posb] {2} ++(1,0); \node[blue] at (2+0.1,0+0.1) {\ss 0}; \draw[thick,lightgray] (2+1,0) -- node {2} ++(-1,1); \node[blue] at (2+0.9,0+0.1) {\ss 1}; 
\draw[thick] (2,1) -- node[pos=\posa] {3} ++(0,1); \draw[thick] (2,1) -- node[pos=\posb] {32} ++(1,0); \node[blue] at (2+0.1,1+0.1) {\ss 0}; \draw[thick] (2+1,1) -- node {2} ++(-1,1); 
\draw[thick] (3,0) -- node[pos=\posa] {3} ++(0,1); \draw[thick] (3,0) -- node[pos=\posb] {3} ++(1,0); \node[blue] at (3+0.1,0+0.1) {\ss 0}; \draw[thick] (3+1,0) -- node {3} ++(-1,1); 
\end{tikzpicture}
\begin{tikzpicture}[math mode,nodes={edgelabel},x={(-0.577cm,-1cm)},y={(0.577cm,-1cm)},scale=\thescale]\useasboundingbox (0,0) -- (4+0.5,0) -- (0,4+0.5);
\draw[thick] (0,0) -- node[pos=\posa] {2} ++(0,1); \draw[thick] (0,0) -- node[pos=\posb] {1} ++(1,0); \node[blue] at (0+0.1,0+0.1) {\ss 1}; \node[blue] at (0+0.9,0+0.1) {\ss 1}; 
\draw[thick] (0,1) -- node[pos=\posa] {0} ++(0,1); \draw[thick] (0,1) -- node[pos=\posb] {1} ++(1,0); \node[blue] at (0+0.1,1+0.1) {\ss 0}; \draw[thick,lightgray] (0+1,1) -- node {10} ++(-1,1); \node[blue] at (0+0.9,1+0.1) {\ss 1}; 
\draw[thick] (0,2) -- node[pos=\posa] {1} ++(0,1); \draw[thick] (0,2) -- node[pos=\posb] {10} ++(1,0); \node[blue] at (0+0.1,2+0.1) {\ss 0}; \draw[thick,lightgray] (0+1,2) -- node {0} ++(-1,1); \node[blue] at (0+0.9,2+0.1) {\ss 1}; 
\draw[thick] (0,3) -- node[pos=\posa] {3} ++(0,1); \draw[thick] (0,3) -- node[pos=\posb] {30} ++(1,0); \node[blue] at (0+0.1,3+0.1) {\ss 0}; \draw[thick] (0+1,3) -- node {0} ++(-1,1); 
\draw[thick] (1,0) -- node[pos=\posa] {2} ++(0,1); \draw[thick] (1,0) -- node[pos=\posb] {0} ++(1,0); \node[blue] at (1+0.1,0+0.1) {\ss 1}; \node[blue] at (1+0.9,0+0.1) {\ss 1}; 
\draw[thick] (1,1) -- node[pos=\posa] {10} ++(0,1); \draw[thick] (1,1) -- node[pos=\posb] {0} ++(1,0); \node[blue] at (1+0.1,1+0.1) {\ss 0}; \draw[thick,lightgray] (1+1,1) -- node {1} ++(-1,1); \node[blue] at (1+0.9,1+0.1) {\ss 1}; 
\draw[thick] (1,2) -- node[pos=\posa] {3} ++(0,1); \draw[thick] (1,2) -- node[pos=\posb] {31} ++(1,0); \node[blue] at (1+0.1,2+0.1) {\ss 0}; \draw[thick] (1+1,2) -- node {1} ++(-1,1); 
\draw[thick] (2,0) -- node[pos=\posa] {2} ++(0,1); \draw[thick] (2,0) -- node[pos=\posb] {2} ++(1,0); \node[blue] at (2+0.1,0+0.1) {\ss 0}; \draw[thick,lightgray] (2+1,0) -- node {2} ++(-1,1); \node[blue] at (2+0.9,0+0.1) {\ss 1}; 
\draw[thick] (2,1) -- node[pos=\posa] {3} ++(0,1); \draw[thick] (2,1) -- node[pos=\posb] {32} ++(1,0); \node[blue] at (2+0.1,1+0.1) {\ss 0}; \draw[thick] (2+1,1) -- node {2} ++(-1,1); 
\draw[thick] (3,0) -- node[pos=\posa] {3} ++(0,1); \draw[thick] (3,0) -- node[pos=\posb] {3} ++(1,0); \node[blue] at (3+0.1,0+0.1) {\ss 0}; \draw[thick] (3+1,0) -- node {3} ++(-1,1); 
\end{tikzpicture}
\end{center}
\end{ex}
For example, the first puzzle could be written more explicitly as
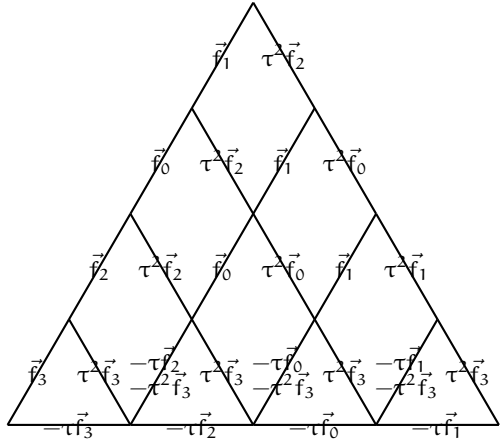
\begin{center}
\def\thescale{1.4}
\def\posa{0.5}\def\posb{0.5}
\begin{tikzpicture}[math mode,nodes={node font=\scriptsize,align=left},x={(-0.577cm,-1cm)},y={(0.577cm,-1cm)},scale=\thescale]\useasboundingbox (0,0) -- (4+0.5,0) -- (0,4+0.5);
\draw[thick] (0,0) -- node[pos=\posa] {$ \tau^2\vf_2$} ++(0,1); \draw[thick] (0,0) -- node[pos=\posb] {$ \vf_1$} ++(1,0); 
\draw[thick] (0,1) -- node[pos=\posa] {$ \tau^2\vf_0$} ++(0,1); \draw[thick] (0,1) -- node[pos=\posb] {$ \vf_1$} ++(1,0); 
\draw[thick] (0,2) -- node[pos=\posa] {$ \tau^2\vf_1$} ++(0,1); \draw[thick] (0,2) -- node[pos=\posb] {$ \vf_1$} ++(1,0); 
\draw[thick] (0,3) -- node[pos=\posa] {$ \tau^2\vf_3$} ++(0,1); \draw[thick] (0,3) -- node[pos=\posb] {$-\tau\vf_1$\\[-0.5mm]$-\tau^2\vf_3$} ++(1,0); \draw[thick] (0+1,3) -- node {$-\tau\vf_1$} ++(-1,1); 
\draw[thick] (1,0) -- node[pos=\posa] {$ \tau^2\vf_2$} ++(0,1); \draw[thick] (1,0) -- node[pos=\posb] {$ \vf_0$} ++(1,0); 
\draw[thick] (1,1) -- node[pos=\posa] {$ \tau^2\vf_0$} ++(0,1); \draw[thick] (1,1) -- node[pos=\posb] {$ \vf_0$} ++(1,0); 
\draw[thick] (1,2) -- node[pos=\posa] {$ \tau^2\vf_3$} ++(0,1); \draw[thick] (1,2) -- node[pos=\posb] {$-\tau\vf_0$\\[-0.5mm]$-\tau^2\vf_3$} ++(1,0); \draw[thick] (1+1,2) -- node {$-\tau\vf_0$} ++(-1,1); 
\draw[thick] (2,0) -- node[pos=\posa] {$ \tau^2\vf_2$} ++(0,1); \draw[thick] (2,0) -- node[pos=\posb] {$ \vf_2$} ++(1,0); 
\draw[thick] (2,1) -- node[pos=\posa] {$ \tau^2\vf_3$} ++(0,1); \draw[thick] (2,1) -- node[pos=\posb] {$-\tau\vf_2$\\[-0.5mm]$-\tau^2\vf_3$} ++(1,0); \draw[thick] (2+1,1) -- node {$-\tau\vf_2$} ++(-1,1); 
\draw[thick] (3,0) -- node[pos=\posa] {$ \tau^2\vf_3$} ++(0,1); \draw[thick] (3,0) -- node[pos=\posb] {$ \vf_3$} ++(1,0); \draw[thick] (3+1,0) -- node {$-\tau\vf_3$} ++(-1,1); 
\end{tikzpicture}
\end{center}
We leave it as an exercise to the reader to check that the fugacities of the
puzzles sum up to the correct coefficients in the expansion of $S^{3201}S^{2013}$
(a table of scalar products can be found in appendix~\ref{app:scal}).

Note that neither the rhombus \rh{0}{31}{1}{3(10)}, nor the triangle
\uptri{3(10)}{(31)0}{1}, exist in the puzzle rule for $3$-step Schubert
classes as formulated in \cite{artic71}; they are suppressed as
$\hbar\to\infty$.

The value of the scalar product $s=-1+2a$ only occurs in
cases where manual computation would be difficult (and such examples would be too long to fit in a paper).
Here is {\em one}\/ puzzle (among the 30) contributing to the coefficient of $S^{2310}$ in $S^{0213}S^{0321}$:
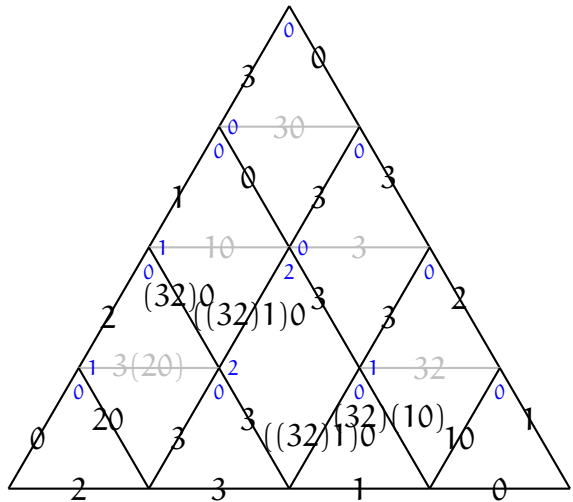
\begin{center}
\def\thescale{1.6}
\def\posa{0.42}\def\posb{0.58}
\begin{tikzpicture}[math mode,nodes={edgelabel},x={(-0.577cm,-1cm)},y={(0.577cm,-1cm)},scale=\thescale]\useasboundingbox (0,0) -- (4+0.5,0) -- (0,4+0.5);
\draw[thick] (0,0) -- node[pos=\posa] {0} ++(0,1); \draw[thick] (0,0) -- node[pos=\posb] {3} ++(1,0); \node[blue] at (0+0.1,0+0.1) {\ss 0}; \draw[thick,lightgray] (0+1,0) -- node {30} ++(-1,1); \node[blue] at (0+0.9,0+0.1) {\ss 0}; 
\draw[thick] (0,1) -- node[pos=\posa] {3} ++(0,1); \draw[thick] (0,1) -- node[pos=\posb] {3} ++(1,0); \node[blue] at (0+0.1,1+0.1) {\ss 0}; \draw[thick,lightgray] (0+1,1) -- node {3} ++(-1,1); \node[blue] at (0+0.9,1+0.1) {\ss 0}; 
\draw[thick] (0,2) -- node[pos=\posa] {2} ++(0,1); \draw[thick] (0,2) -- node[pos=\posb] {3} ++(1,0); \node[blue] at (0+0.1,2+0.1) {\ss 0}; \draw[thick,lightgray] (0+1,2) -- node {32} ++(-1,1); \node[blue] at (0+0.9,2+0.1) {\ss 1}; 
\draw[thick] (0,3) -- node[pos=\posa] {1} ++(0,1); \draw[thick] (0,3) -- node[pos=\posb] {10} ++(1,0); \node[blue] at (0+0.1,3+0.1) {\ss 0}; \draw[thick] (0+1,3) -- node {0} ++(-1,1); 
\draw[thick] (1,0) -- node[pos=\posa] {0} ++(0,1); \draw[thick] (1,0) -- node[pos=\posb] {1} ++(1,0); \node[blue] at (1+0.1,0+0.1) {\ss 0}; \draw[thick,lightgray] (1+1,0) -- node {10} ++(-1,1); \node[blue] at (1+0.9,0+0.1) {\ss 1}; 
\draw[thick] (1,1) -- node[pos=\posa] {3} ++(0,1); \draw[thick] (1,1) -- node[pos=\posb] {((32)1)0} ++(1,0); \node[blue] at (1+0.1,1+0.1) {\ss 2}; \node[blue] at (1+0.9,1+0.1) {\ss 2}; 
\draw[thick] (1,2) -- node[pos=\posa] {(32)(10)} ++(0,1); \draw[thick] (1,2) -- node[pos=\posb] {((32)1)0} ++(1,0); \node[blue] at (1+0.1,2+0.1) {\ss 0}; \draw[thick] (1+1,2) -- node {1} ++(-1,1); 
\draw[thick] (2,0) -- node[pos=\posa] {(32)0} ++(0,1); \draw[thick] (2,0) -- node[pos=\posb] {2} ++(1,0); \node[blue] at (2+0.1,0+0.1) {\ss 0}; \draw[thick,lightgray] (2+1,0) -- node {3(20)} ++(-1,1); \node[blue] at (2+0.9,0+0.1) {\ss 1}; 
\draw[thick] (2,1) -- node[pos=\posa] {3} ++(0,1); \draw[thick] (2,1) -- node[pos=\posb] {3} ++(1,0); \node[blue] at (2+0.1,1+0.1) {\ss 0}; \draw[thick] (2+1,1) -- node {3} ++(-1,1); 
\draw[thick] (3,0) -- node[pos=\posa] {20} ++(0,1); \draw[thick] (3,0) -- node[pos=\posb] {0} ++(1,0); \node[blue] at (3+0.1,0+0.1) {\ss 0}; \draw[thick] (3+1,0) -- node {2} ++(-1,1); 
\end{tikzpicture}
\end{center}

\tikzset{vertex/.style={}}
\tikzset{empty/.style={circle,fill,inner sep=1pt}}
\tikzset{edge/.style={d}}
\tikzset{diag/.style={execute at end picture={\useasboundingbox ([shift={(5mm,5mm)}]current bounding box.north east) rectangle ([shift={(-5mm,-5mm)}]current bounding box.south west);},scale=0.7}}
\newcommand\crossing[4]{\tikz[baseline=-3pt,d,diag]{
\draw (-1,1) node[vertex,above] {#4} -- (-0.1,0.1) (0.1,-0.1) -- (1,-1)  node[vertex,below] {#2};
\draw (-1,-1) node[vertex] {} -- (1,1)  node[vertex] {};
}}
\newcommand\invcrossing[1][1]{\tikz[baseline=-3pt,d,diag]{
\draw (1,1) node[vertex] {} -- (0.1,0.1) (-0.1,-0.1) -- (-1,-1)  node[vertex] {};
\draw (1,-1) node[vertex] {} -- (-1,1)  node[vertex] {};
}}
\newcommand\flatcrossing[4]{\tikz[baseline=-3pt,d,diag]{
\draw (1,1) node[vertex,above] {$#4$} -- (-1,-1)  node[vertex,below] {$#2$};
\draw (1,-1) node[vertex,below] {$#3$} -- (-1,1)  node[vertex,above] {$#1$};
}}
\newcommand\tri[1][1]{\tikz[baseline=1cm]{\draw[bend left=60] (1,1) node[vertex] {} to (-1,1) node[vertex] {}; \draw (0,1) node[vertex] {} -- (0,0.5);}}
\newcommand\triinv[1][1]{\tikz[baseline=0cm,scale=-#1]{\draw[bend left=60] (1,1) node[vertex] {} to (-1,1) node[vertex] {}; \draw (0,1) node[vertex] {} -- (0,0.5);}}
\newcommand\mycap[1][1]{\tikz[baseline=-3pt]{\draw[bend left=60] (-1,0) node[vertex] {} to (1,0) node[vertex] {};}}
\newcommand\mycup[1][1]{\tikz[baseline=-3pt]{\draw[bend left=60] (1,0) node[vertex] {} to (-1,0) node[vertex] {};}}
\newcommand\D[1][1]{\tikz[baseline=0cm,d,diag]{
\draw (-1,-1) node[vertex] {} -- (0,-0.2) -- (1,-1) node[vertex] {};
\draw (0,-0.2) -- (0,1) node[vertex] {};
}}
\newcommand\U[1][1]{\tikz[baseline=0cm,d,diag]{
\draw (0,-1) node[vertex] {} -- (0,0.2);
\draw (-1,1) node[vertex] {} -- (0,0.2) -- (1,1) node[vertex] {};
}}
\newcommand\full[4]{\tikz[baseline=-3pt,d,diag]{\draw[bend right=45] (-1,-1) node[vertex,below] {$#2$} to (-1,1) node[vertex,above] {$#1$} (1,1) node[vertex,above] {$#4$} to (1,-1) node[vertex,below] {$#3$};}}
\newcommand\identity[1][1]{\tikz[baseline=-3pt,d,diag]{\draw[bend right=45,dashed,diag] (-1,-1) node[halfvertex] {} to (-1,1) node[halfvertex] {} (1,1) node[halfvertex] {} to (1,-1) node[halfvertex] {};}}
\newcommand\cupcap[4]{\tikz[baseline=-3pt,d,diag]{\draw[bend left=45] (-1,-1) node[vertex,below] {$#2$} to (1,-1) node[vertex,below] {$#3$} (1,1) node[vertex,above] {$#4$} to (-1,1) node[vertex,above] {$#1$} ;}}
\newcommand\vertical[4]{\tikz[baseline=-3pt,d,diag]{
\draw (-1,-1) node[vertex,below] {$#2$} -- (0,-0.5) -- (1,-1) node[vertex,below] {$#3$} ;
\draw (-1,1) node[vertex,above] {$#1$} -- (0,0.5) -- (1,1) node[vertex,above] {$#4$} ;
\draw (0,-0.5) -- (0,0.5);
}}
\newcommand\horizontal[4]{\tikz[baseline=-3pt,d,diag,rotate=90]{
\draw (-1,-1) node[vertex,below] {$#3$} -- (0,-0.5) -- (1,-1) node[vertex,above] {$#4$} ;
\draw (-1,1) node[vertex,below] {$#2$} -- (0,0.5) -- (1,1) node[vertex,above] {$#1$} ;
\draw (0,-0.5) -- (0,0.5);
}}
\newcommand\mysquare[1][1]{\tikz[baseline=-3pt,d,diag]{
\draw (-1,-1) node[vertex] {} -- (-0.5,-0.5) -- (-0.5,0.5) -- (-1,1) node[vertex] {};
\draw (1,-1) node[vertex] {} -- (0.5,-0.5) -- (0.5,0.5) -- (1,1) node[vertex] {};
\draw (-0.5,-0.5) -- (0.5,-0.5) (-0.5,0.5) -- (0.5,0.5);
}}
\newcommand\doublesquare[1][1]{\tikz[baseline=-3pt,d,diag]{
\draw (-1,-1) node[vertex] {} -- (-0.5,-0.5) -- (0.5,-0.5) -- (1,-1) node[vertex] {};
\draw (-1,1) node[vertex] {} -- (-0.5,0.5) -- (0.5,0.5) -- (1,1) node[vertex] {};
\draw (-0.5,-0.5) -- (-0.25,0) -- (-0.5,0.5) (0.5,-0.5) -- (0.25,0) -- (0.5,0.5) (-0.25,0) -- (0.25,0);
}}
\newcommand\leftid[1][1]{\tikz[baseline=-3pt,d,diag]{
\draw[bend right=45] (-1,-1) node[vertex] {} to (-1,1) node[vertex] {} (1,1);
\node[empty] at (1,-1) {};
\node[empty] at (1,1) {};
}}
\newcommand\diagida[1][1]{\tikz[baseline=-3pt,d,diag]{
\draw (-1,1) node[vertex] {} to (1,-1) node[vertex] {};
\node[empty] at (-1,-1) {};
\node[empty] at (1,1) {};
}}
\newcommand\diagidb[1][1]{\tikz[baseline=-3pt,d,diag]{
\draw (1,1) node[vertex] {} to (-1,-1) node[vertex] {};
\node[empty] at (1,-1) {};
\node[empty] at (-1,1) {};
}}
\newcommand\rightid[1][1]{\tikz[baseline=-3pt,d,diag]{
\draw[bend right=45] (1,1) node[vertex] {} to (1,-1) node[vertex] {};
\node[empty] at (-1,-1) {};
\node[empty] at (-1,1) {};
}}
\newcommand\Yaa[1][1]{\tikz[baseline=-3pt,d,diag]{
\draw (-1,1) node[vertex] {} to (1,-1) node[vertex] {};
\node[empty] at (-1,-1) {};
\draw (1,1) node[vertex] {} -- (0,0);
}}
\newcommand\Yba[1][1]{\tikz[baseline=-3pt,d,diag]{
\draw (1,1) node[vertex] {} to (-1,-1) node[vertex] {};
\node[empty] at (1,-1) {};
\draw (-1,1) node[vertex] {} -- (0,0);
}}
\newcommand\Yab[1][1]{\tikz[baseline=-3pt,d,diag]{
\draw (-1,1) node[vertex] {} to (1,-1) node[vertex] {};
\node[empty] at (1,1) {};
\draw (-1,-1) node[vertex] {} -- (0,0);
}}
\newcommand\Ybb[1][1]{\tikz[baseline=-3pt,d,diag]{
\draw (1,1) node[vertex] {} to (-1,-1) node[vertex] {};
\node[empty] at (-1,1) {};
\draw (1,-1) node[vertex] {} -- (0,0);
}}
\newcommand\capa[1][1]{\tikz[baseline=-3pt,d,diag]{
\draw[bend left=45] (-1,-1) node[vertex] {} to (1,-1) node[vertex] {};
\node[empty] at (-1,1) {};
\node[empty] at (1,1) {};
}}
\newcommand\cupa[1][1]{\tikz[baseline=-3pt,d,diag]{
\draw[bend left=45] (1,1) node[vertex] {} to (-1,1) node[vertex] {} ;
\node[empty] at (-1,-1) {};
\node[empty] at (1,-1) {};
}}
\newcommand\myempty[1][1]{\tikz[baseline=-3pt,d,diag]{
\node[empty] at (-1,1) {};
\node[empty] at (1,1) {};
\node[empty] at (-1,-1) {};
\node[empty] at (1,-1) {};
}}
\subsection{Some details on $d=4$}\label{ssec:d4}
$4$-step puzzles involve several complications which have already been
mentioned; for details on the underlying representation theory of
$\Uq (\lie{e}_{8}[z^\pm])$, we refer to the companion paper
\cite{artic77}. These complications are present even in
cohomology. Note that all examples of $d=4$ puzzles are relegated to
appendix~\ref{app:exd4}.

In order to simplify the discussion, we first consider the case $q=1$,
where the $q$-deformation disappears (the required sign changes to set
$q=-1$ will be provided shortly).  The spaces $V_a(z)$ decompose under
the action of $\mathfrak{e}_8$ as the direct sum of the adjoint
representation and the trivial representation.
Each of the 240 nonzero weight spaces is one-dimensional and has a
natural basis vector which is for example provided by the stable
envelope construction; in contrast, the zero weight space is
nine-dimensional, and only has a natural decomposition as
$\CC \oplus \mathfrak{t}$, where we identify the Cartan subalgebra
$\mathfrak t$ of $\mathfrak e_8$ with its dual via the Killing form.

We start with the $R$-matrix provided in appendix~C of \cite{artic77}. We shall analyze it piece by piece.
First there is the part living in the tensor square of the adjoint representation; with our choice
\eqref{eq:normR} of normalization, and shifting the spectral parameter by $y\mapsto y-10\hbar$ to match
the difference of equivariant parameters,
it is given diagrammatically by
\begin{multline}\label{eq:Rd4}
\check R_{ad}(y)=
\frac{6\hbar^2y}{(\hbar-y)(5\hbar-y)(10\hbar-y)}
\full{}{}{}{}
+
\frac{6\hbar^2}{(\hbar-y)(5\hbar+y)}
\cupcap{}{}{}{}
\\
+
\frac{\hbar}{\hbar-y}
\vertical{}{}{}{}
+
\frac{\hbar y}{(\hbar-y)(5\hbar-y)}
\horizontal{}{}{}{}
+
\frac{y}{\hbar-y}
\flatcrossing{}{}{}{}
\end{multline}
These diagrams have the following meaning in our setting. The ``cup'' and ``cap'' are the Killing form and its inverse.
The trivalent vertex is associated to the Lie bracket on $\mathfrak e_8$. The crossing is the permutation of factors
of the tensor product.
We now provide their explicit expressions.
Given two roots $\vX$ and $\vY$, define the sign
\[
\epsilon_{\vX,\vY}=\prod_{\substack{1\le i<j\le 8\\[1pt]
\killing{\vec\alpha_i}{\vec\alpha_j}
=-1}} (-1)^{\killing{\vec\alpha_i}{\vX}\killing{\vec\alpha_j}{\vY}}
\]
where the $\vec\alpha_i$ are the simple roots of $\mathfrak e_8$.
Note that it satisfies
\[
\epsilon_{\vX,\vY}\, \epsilon_{\vY,\vX}=(-1)^{\killing{\vX}{\vY}}
\]
Given roots $\vX,\vY,\vZ$, we then have
\tikzset{edgelabel/.style={}}
\begin{align*}
\tikz[baseline=-3pt]{\draw[edge,arrow=0.3,invarrow=0.7] (0,0) node[edgelabel,left] {$\vX$} -- (1,0) node[edgelabel,right] {$\vY$};}
&=\delta_{\vX+\vY,0}\epsilon_{\vX,\vX}
&
\tikz[baseline=-3pt]{\draw[edge,arrow] (0,0) -- (0.5,0) node[right,edgelabel] {$\vZ$}; \draw[edge,arrow] (120:0.5) node[left,edgelabel] {$\vX$} -- (0,0); \draw[edge,arrow] (-120:0.5) node[left,edgelabel] {$\vY$} -- (0,0); } &= \delta_{\vX+\vY,\vZ} \epsilon_{\vX,\vY}
\\
\intertext{which also implies}
\tikz[baseline=-3pt]{\draw[edge,invarrow=0.3,arrow=0.7] (0,0) node[left,edgelabel] {$\vX$} -- (1,0) node[edgelabel,right] {$\vY$};}&=\delta_{\vX+\vY,0}\epsilon_{\vX,\vX}
&
\tikz[baseline=-3pt]{\draw[edge,invarrow] (0,0) -- (0.5,0) node[right,edgelabel] {$\vZ$}; \draw[edge,invarrow] (120:0.5) node[left,edgelabel] {$\vX$} -- (0,0); \draw[edge,invarrow] (-120:0.5) node[left,edgelabel] {$\vY$} -- (0,0); } &= \delta_{\vX+\vY,\vZ} \epsilon_{\vY,\vX}
\end{align*}
These can be plugged into the diagrams of \eqref{eq:Rd4}, where all
external lines are implicitly oriented downwards, resulting in
\begin{gather*}
\full{\vW}{\vX}{\vY}{\vZ} 
= \delta_{s,2} \qquad\qquad
\flatcrossing{\vW}{\vX}{\vY}{\vZ} 
=\delta_{r,2}\qquad\qquad
\cupcap{\vW}{\vX}{\vY}{\vZ} 
=\delta_{t,2}\epsilon_{\vX,\vX}\epsilon_{\vZ,\vZ}
\\
\horizontal{\vW}{\vX}{\vY}{\vZ} 
=(\delta_{s,1}-\delta_{s,2}\,r)\epsilon_{\vX,\vY}\epsilon_{\vW,\vZ}
\qquad
\vertical{\vW}{\vX}{\vY}{\vZ} 
=(\delta_{r+s,1}-\delta_{r+s,0}\,r)\epsilon_{\vX,\vY}\epsilon_{\vW,\vZ}
\end{gather*}
where we have defined as usual the scalar products $r=\killing{\vX}{\vZ}=\killing{\vY}{\vW}$, $s=\killing{\vX}{\vW}=\killing{\vY}{\vZ}$ and $t=\killing{\vX}{\vY}=\killing{\vZ}{\vW}$
with $r=t-s+2$.

It is convenient to switch to $q=-1$ now. It is not hard to see that it 
corresponds to twisting (in the same sense as in \S\ref{sec:twist})
with $\epsilon_{\cdot,\cdot}$; explicitly, multiplying the above by $\epsilon_{\vX,\vY}\epsilon_{\vW,\vZ}$ results in
\begin{gather*}
\full{\vW}{\vX}{\vY}{\vZ} 
= \delta_{s,2} \qquad\qquad
\flatcrossing{\vW}{\vX}{\vY}{\vZ} 
=\delta_{r,2}(-1)^s \qquad\qquad
\cupcap{\vW}{\vX}{\vY}{\vZ} 
=\delta_{t,2}
\\
\horizontal{\vW}{\vX}{\vY}{\vZ} 
=\delta_{s,1}-\delta_{s,2}\,r
\qquad
\vertical{\vW}{\vX}{\vY}{\vZ} 
=\delta_{r+s,1}-\delta_{r+s,0}\,r
\end{gather*}

If all weight vectors $\vX,\vY,\vZ,\vW$ are nonzero, there is no contribution from the trivial sub-representation,
and summing the diagrams above results in the following table:
\begin{equation}\label{eq:tab4}
\!\!\!\!\!\!\!\!\!\!\!\!\!\!\!\!\!\!
\begin{array}{c|cccccc}
s\backslash t &
-1-a & -1 & -1+ a & -1+2a & -1+3a\\ \hline
-1 -a & 
-\frac{4 \hbar+y}{5 \hbar+y} 
\\ \\
-1 & 
\frac{\hbar}{5 \hbar+y} & 1 
\\ \\
-1+a & 
\frac{6 \hbar^2}{(\hbar-y) (5 \hbar+y)} & \frac{\hbar}{\hbar-y} & \frac{y}{\hbar-y} 
\\ \\
-1+2a & 
\frac{\hbar^2 (55 \hbar-y)}{(\hbar-y) (5 \hbar-y) (5 \hbar+y)} & \frac{5 \hbar^2}{(\hbar-y) (5 \hbar-y)} & \frac{\hbar y}{(\hbar-y) (5 \hbar-y)} & -\frac{y (4 \hbar-y)}{(\hbar-y) (5 \hbar-y)} & 0 \\ \\
-1+3a & 
\frac{2 \hbar^2 \left(400 \hbar^2-5 \hbar y+y^2\right)}{(\hbar-y) (5 \hbar-y) (5 \hbar+y) (10 \hbar-y)} & 
\frac{\hbar^2(50\hbar+y)}{(\hbar-y) (5 \hbar-y) (10 \hbar-y)} 
& \frac{6 \hbar^2 y}{(\hbar-y) (5 \hbar-y) (10 \hbar-y)}  & -\frac{\hbar y (4 \hbar-y)}{(\hbar-y) (5 \hbar-y) (10 \hbar-y)} & \frac{y (4 \hbar-y)(9\hbar-y)}{(\hbar-y) (5 \hbar-y) (10 \hbar-y)} 
\end{array}
\end{equation}
where as usual
$s=\killing{\vX}{\vW}$, 
$t=\killing{\vX}{\vY}$.

Some comments are in order. Firstly, as we shall see shortly when we discuss
triangles,
the first (resp.\ second) column corresponds to rhombi which nonequivariantly can be split as two triangles, with zero (resp.\ nonzero) weight diagonal.

Secondly, comparing the table above the one in the previous section
for $d\le 3$, we notice that the latter, namely \eqref{eq:tab3}, is
{\em not}\/ a subtable of the former, \eqref{eq:tab4}.  In fact, the
functoriality of \cite[\S 2.2]{artic71} does not apply at $d=4$; this
can be traced back to the fact that the inclusion
$X_{2(d-1)}\subset X_{2d}$, which was implicitly respected in
the construction of \cite{artic71} for $d\le 3$, cannot be respected at $d=4$.
See \ref{app:exd4a} for an example of a $3$-step problem that has a
different puzzle solution at $d=4$.

The case where some weight vectors are zero can be treated similarly. 
One needs to make a choice of basis of the zero weight space (i.e., the Cartan subalgebra); for example, in the basis of simple (co)roots $h_i$ and its
dual basis $h^i$ of fundamental weights, \footnote{Elsewhere in the paper, fundamental weights are
denoted $\vec\omega_i$, but the notation cannot be used here since we already label edges using weights,
and here we want to think of the fundamental weights as forming a basis of the Cartan subalgebra.}
one has
$\tikz[baseline=-3pt]{\draw[edge,arrow=0.3,invarrow=0.7] (0,0) node[edgelabel,left] {$h_i$} -- (1,0) node[edgelabel,right] {$h_j$};}
=C_{i,j}$ and therefore
$\tikz[baseline=-3pt]{\draw[edge,invarrow=0.2,arrow=0.8] (0,0) node[left,edgelabel] {$h^i$} -- (1,0) node[edgelabel,right] {$h^j$};}=C^{i,j}$ 
(where $C_{i,j}$ is the Cartan matrix, and $C^{i,j}$ its inverse). One then has
$\tikz[baseline=-3pt]{\draw[edge,arrow=0.6] (0,0) -- (0.5,0) node[right,edgelabel] {$\vY$}; \draw[edge,arrow] (120:0.5) node[left,edgelabel] {$h_i$} -- (0,0); \draw[edge,arrow] (-120:0.5) node[left,edgelabel] {$\vX$} -- (0,0); } = \delta_{\vX,\vY} u_i =\delta_{\vX,\vY} C_{i,j}u^j$,
$\tikz[baseline=-3pt]{\draw[edge,invarrow=0.4] (0,0) -- (0.5,0) node[right,edgelabel] {$\vY$}; \draw[edge,invarrow=0.4] (120:0.5) node[left,edgelabel] {$h^i$} -- (0,0); \draw[edge,invarrow] (-120:0.5) node[left,edgelabel] {$\vX$} -- (0,0); } = -\delta_{\vX,\vY} C^{i,j}u_j = -\delta_{\vX,\vY} u^i$ (and the opposite if one switches $h_i$
and $\vX$), where $\vX = \sum_i u_i h^i = \sum_i u^i h_i$. Trivalent vertices involving two or three $h$s are zero.

These rules allow to compute the fugacity of any rhombus. 
For example, if one of the four weights is zero, say $\vY$, and the basis vector there is $h_i$,
then the last two terms of $\check R$ contribute, with coefficients 
$-w_i \delta_{\vX,\vZ+\vW}$ and $-u_i \delta_{\vX,\vZ+\vW}$ respectively.

Finally, we need to include the contribution
of the trivial representation; this is given by \cite{artic77}
\begin{multline*}
\check R_{rest}
=
\frac{60 \hbar^3}{(\hbar-y) (5 \hbar-y) (10 \hbar-y)}
\left(
\leftid
+
\rightid
\right)
\\
+
\frac{(6 \hbar-y) (\hbar+y)}{(\hbar-y) (5 \hbar-y)}
\left(
\diagida
+
\diagidb
\right)
\\
+
\frac{60 \hbar^3}{(\hbar-y) (5 \hbar-y) (5 \hbar+y)}
\left(
\capa
+
\cupa
\right)
\\
+
\frac{\sqrt{30} \hbar^2}{(\hbar-y) (5\hbar-y)}
\left(
\Yaa
+
\Yab
+
\Yba
+
\Ybb
\right)
\\
+\frac{(11\hbar-y)(4\hbar-y)(5\hbar+y)y+60\hbar^3(20\hbar+y)}{(\hbar-y) (5 \hbar-y) (10 \hbar-y) (5 \hbar+y)}
\myempty
\end{multline*}
where the isolated dots correspond to trivial representations. The graphical rules are the same as before.

\rem[gray]{nonequivariant examples are just too big}

We discuss triangles now. At $x=0$, in accordance with
lemma~\ref{lem:factorR}, the $R$-matrix factorizes as a product of two
trilinear operators. At $q=-1$ these two operators are in fact
identical up to the orientation of the lines; we have
\[
U=
\U+\sqrt{6/5}\left( \tikz[baseline=0cm,yscale=-1,diag]{
\draw[bend left=60,d] (-1,-1) node[vertex] {} to (1,-1) node[vertex] {};
\node[empty] at (0,1) {};
}
+
\tikz[baseline=0cm,yscale=-1,diag]{
\draw[bend right=40,d] (-1,-1) node[vertex] {} to (0,1) node[vertex] {};
\node[empty] at (1,-1) {};
}
+
\tikz[baseline=0cm,yscale=-1,diag]{
\draw[bend left=40,d] (1,-1) node[vertex] {} to (0,1) node[vertex] {};
\node[empty] at (-1,-1) {};
}
\right)
+2\sqrt{6/5} 
\tikz[baseline=0cm,yscale=-1,diag]{
\node[empty] at (1,-1) {};
\node[empty] at (-1,-1) {};
\node[empty] at (0,1) {};
}
\]
(all lines pointing downwards) and $D$ its $180^\circ$ rotation with
orientation reversed.

Finally, this allows us to give a relatively simple
{\em nonequivariant}\/ cohomology rule at $d=4$.
Instead of breaking a puzzle along all edges into triangles,
and multiplying the fugacities of these triangles (which depend on the
choice of basis of the weight zero space), we break only along
non-zero-weight edges (whose singleton bases are canonical).
The resulting triangles have fugacity $1$ and can be discarded;
we just need give a formula for the fugacities of the larger regions
(now independent of the choice of basis of the weight zero space), 
which we then multiply together to give that of the puzzle.
\junk{AK added most of the paragraph above, howdya like it. PZJ: OK}

\begin{prop}\label{prop:d4}
Nonequivariant $d=4$ puzzles can be described as follows.
Their boundaries form $d$-ary strings encoded as roots of $\mathfrak{e}_8$, as in the general
setup of \S\ref{sec:main}. Their insides (i.e., the internal edges in the underlying triangular lattice)
are labeled in all possible ways with weights (i.e., roots, or zero) in such a way
that weight conservation is satisfied at each elementary triangle. 
To compute their fugacity, one
considers separately each connected component of the subgraph of the dual of the puzzle made of zero weight edges.
This graph can only have 3-valent vertices or 1-valent vertices (endpoints). For each component, one sums over all possible
sets of nonintersecting paths connecting pairs of endpoints (unpaired endpoints are allowed)
the fugacity
\[
{
(6/5)^{\frac{1}{2}\#\{\text{vertices with some empty edges}\}} 
\atop
2^{\#\{\text{vertices with all empty edges}\}}
}
\prod_{\text{paired endpoints $p$ and $q$}} (-\text{weight}(p),\text{weight}(q))
\]
where an empty edge is an edge not traversed by a path, and 
the weight of an endpoint is conventionally the weight of the edge counterclockwise from its zero weight edge. 
The fugacity of the puzzle
is then the product of fugacities of these connected components.
\end{prop}
The statement is just a reformulation of the expression for $U$ and $D$ above, and we skip the details of the proof.

\begin{rmk*}
Note that the rule is independent of the parameter $\hbar$, by homogeneity.
\end{rmk*}

\begin{ex}
The simplest situation is when a single edge has zero weight, i.e.,
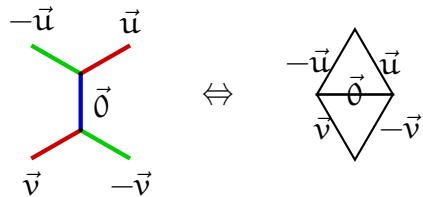
\[
\tikz[scale=0.75,baseline=(current  bounding  box.center)]{\draw[dr] (210:1) node[below] {$\vY$} -- (0,0) (0,1) -- ++(30:1) node[above] {$\vX$}; \draw[dg] (-30:1) node[below] {$-\vY$} -- (0,0)  (0,1) -- ++(150:1) node[above] {$-\vX$}; \draw[db] (0,0) node[triv] {} -- node[right] {$\vec0$} (0,1) node[triv] {}; }
\quad\Leftrightarrow\quad
\tikz[baseline=-1mm]{\uptri{-\vX}{\vec0}{\vX}\downtri{-\vY}{}{\vY}}
\]
or its $120^\circ$ rotations.
There can be either a path or none, so that the fugacity is
\[
\tikz[baseline=-1mm]{\uptri{-\vX}{\vec0}{\vX}\downtri{-\vY}{}{\vY}}
=
\frac{6}{5}-\killing{\vX}{\vY}
\qquad
\vX,\vY\ne\vec0
\]
We reach the important conclusion that fugacities may not only be rational numbers, but also negative:
$6/5-\killing{\vX}{\vY}\in \{-4/5,1/5,6/5,11/5,16/5\}$ (this is also the first column of \eqref{eq:tab4} at $y=0$).

See Appendix~\ref{app:exd4b} for a full example involving such a zero weight situation.

The next simplest situation is a vertex with only zero weights:
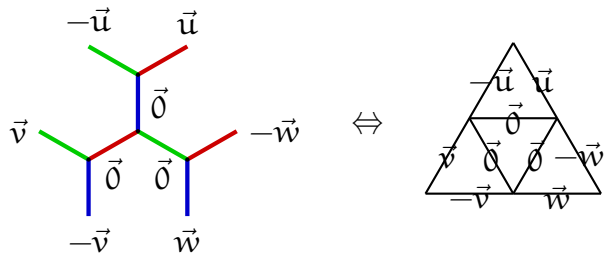
\[
\tikz[scale=0.75,baseline=(current  bounding  box.center)]{
\coordinate (O) at (0,0);
\coordinate (A) at (90:1);
\coordinate (B) at (210:1);
\coordinate (C) at (-30:1);
\draw[dr] (A) -- ++(30:1) node[above] {$\vX$} (B) -- node[below] {$\vec0$} ++(30:1) (C) -- ++(30:1) node[right] {$-\vZ$};
\draw[dg] (A) -- ++(150:1) node[above] {$-\vX$} (B) -- ++(150:1) node[left] {$\vY$} (C) -- node[below] {$\vec0$} ++(150:1);
\draw[db] (A) -- node[right] {$\vec0$} ++(-90:1) (B) -- ++(-90:1) node[below] {$-\vY$} (C) -- ++(-90:1) node[below] {$\vZ$};
\node[triv] at (O) {};
\node[triv] at (A) {};
\node[triv] at (B) {};
\node[triv] at (C) {};
}
\quad\Leftrightarrow\quad
\def\posa{0.5}\def\posb{0.5}\def\thescale{1}
\begin{tikzpicture}[math mode,nodes={edgelabel},x={(-0.577cm,-1cm)},y={(0.577cm,-1cm)},scale=\thescale,baseline=(current  bounding  box.center)]
\draw[thick] (0,0) -- node[pos=\posa,xshift=1mm] {\vX} ++(0,1); \draw[thick] (0,0) -- node[pos=\posb] {-\vX} ++(1,0); \draw[thick] (0+1,0) -- node {\vec0} ++(-1,1); 
\draw[thick] (0,1) -- node[pos=\posa] {-\vZ} ++(0,1); \draw[thick] (0,1) -- node[pos=\posb] {\vec0} ++(1,0); \draw[thick] (0+1,1) -- node[yshift=-0.5mm] {\vZ} ++(-1,1); 
\draw[thick] (1,0) -- node[pos=\posa] {\vec0} ++(0,1); \draw[thick] (1,0) -- node[pos=\posb] {\vY} ++(1,0); \draw[thick] (1+1,0) -- node[yshift=-0.5mm] {-\vY} ++(-1,1); 
\end{tikzpicture}
\]
There can be either no path, or a single path connecting any two of the outermost vertices:
\[
\def\posa{0.5}\def\posb{0.5}\def\thescale{1}
\begin{tikzpicture}[math mode,nodes={edgelabel},x={(-0.577cm,-1cm)},y={(0.577cm,-1cm)},scale=\thescale,baseline=(current  bounding  box.center)]
\draw[thick] (0,0) -- node[pos=\posa,xshift=1mm] {\vX} ++(0,1); \draw[thick] (0,0) -- node[pos=\posb] {-\vX} ++(1,0); \draw[thick] (0+1,0) -- node {\vec0} ++(-1,1); 
\draw[thick] (0,1) -- node[pos=\posa] {-\vZ} ++(0,1); \draw[thick] (0,1) -- node[pos=\posb] {\vec0} ++(1,0); \draw[thick] (0+1,1) -- node[yshift=-0.5mm] {\vZ} ++(-1,1); 
\draw[thick] (1,0) -- node[pos=\posa] {\vec0} ++(0,1); \draw[thick] (1,0) -- node[pos=\posb] {\vY} ++(1,0); \draw[thick] (1+1,0) -- node[yshift=-0.5mm] {-\vY} ++(-1,1); 
\end{tikzpicture}
=2(6/5)^2 - 6/5(\killing{\vX}{\vY}+\killing{\vY}{\vZ}+\killing{\vZ}{\vX})
\qquad
\vX,\vY,\vZ\ne\vec0
\]
\end{ex}

\junk{comments on limit to Schubert -- explain why
we can't take the $\hbar\to\infty$ as for lower $d$, but that (nonequivariantly)
one can just pick the puzzles with the right power of $\hbar$. (of course this is also a consequence of the SSM/EP theorem)
conjecture on positivity
(and how it's not a conjecture at $d\le3$). some of these comments may be redundant with next section}

\subsection{An interpretation of the
  nonequivariant-cohomology puzzle rule}\label{ssec:SSM}
In the previous sections we have provided explicit expressions for
puzzle pieces computing products of SSM classes both equivariantly
(in $H^{* \loc}_{\hat T}$) and ``nonequivariantly'' (in $H_{\CC^\times}^{* \loc}$).
In the latter case, there is a particularly appealing geometric
interpretation of the structure constants:\footnote{We thank L.~Mihalcea for pointing out and explaining
\cite{Schurmann} to us.}

\begin{thm}\label{thm:EP}
  Consider three Schubert cells $g X_o^\lambda$, $g' X_o^\mu$, $g'' X_o^\nu$
  in general position in a $d$-step flag variety $P_-\backslash G$ for
  $d\le 4$. Then the Euler 
  characteristic\footnote{It is slightly more natural to think in terms
    of ``compactly supported Euler characteristic'', but the real reason
    we use the notation $\chi_c$ is to distinguish from the Coxeter elements
    $\chi$ mentioned later.}
  $\chi_c$ of their intersection is given by
\[
\chi_c\left(g X_o^\lambda \cap g' X_o^\mu \cap g'' X_o^\nu\right)
= (-1)^\delta \tikz[scale=1.8,baseline=0.5cm]{\uptri{\lambda}{\overleftarrow\nu}{\mu}
}
= (-1)^\delta \sum_{\substack{\text{nonequivariant}\\\text{puzzles $P$ with}\\\text{sides } \lambda, \mu, \nu\\\text{read clockwise}}} \text{fug}(P)
\]
  where $\overleftarrow\nu=(\nu_n,\ldots,\nu_1)$, $\delta={\dim(P_-\backslash G)-(\ell(\lambda)+\ell(\mu)+\ell(\nu))}$, and
  the contribution $\text{fug}(P)$ of each puzzle is simply $1$ for
  $d\le 3$, and given by proposition~\ref{prop:d4} for $d=4$. (For $d>4$ one
  still enjoys the equation relating Euler 
  characteristics and SSM structure constants $c_{\lambda\mu}^{\overleftarrow\nu}$,
  but has no puzzles with which to compute the latter.)
\end{thm}
\junk{careful that notation should be consistent so codimension is $|\cdot|$}

In the simpler version with Schubert classes $\{[X^\lambda]\}$,
this $3$-fold symmetric calculation computes the number of points in the
triple intersection {\em when finite,} and gives $0$ when infinite.
That simpler statement is based on the dual-basis statement
$$ \int_{P_-\dom G} [X^\lambda] [X^\mu] = \delta_{\lambda,\overleftarrow\mu} $$
of the Schubert basis. The corresponding statement for us is (combining
the cohomology limit of \eqref{eq:Kdual}
and the nonequivariant version of \eqref{eq:Hdual2})
$$ \delta_{\lambda,\overleftarrow\mu} = 
\int_{T^*(P_-\dom G)} St^\lambda St^\mu = \int_{P_-\dom G} St^\lambda St^\mu
\frac{1}{e_{\CC^\times}(T^*(P_-\dom G))} = \int_{P_-\dom G} St^\lambda S^\mu $$
where the first integral is defined using AB/BV equivariant localization.
In particular,
$$ \alpha = \sum_\nu c_\nu S^\nu
\quad \iff\quad
c_\nu = \int_{P_-\dom G} \alpha\, St^{\overleftarrow\nu}\quad \forall \nu $$
hence
\begin{equation}
  \label{eq:cfromint}
  c_{\lambda\mu}^{\overleftarrow\nu} =
  \int_{P_-\dom G} S^\lambda\, S^\mu\, St^{\nu} =
  \int_{P_-\dom G} St^\lambda\, St^\mu\, St^{\nu} / e(T^*M)^2 
\end{equation}

The proof of theorem \ref{thm:EP} is based on two properties of the
{\em Chern--Schwartz--MacPherson} natural transformation
$csm:\, \{$constructible functions on $M\} \to H_*(M)$, although we
will work in cohomology $H^*(M)$.  We recall a little of this theory here,
primarily to provide pegs on which to hang the sign conventions that
one need grapple with in comparing CSM classes with our SSM classes.
The first key property of this natural transformation, suggesting its utility
for the statement above, is that $\int_M csm(1_A) = \chi_c(A)$ for
$A \subseteq M$ a locally closed subvariety of a proper smooth variety $M$.

\newcommand\calD{{\mathcal D}}
\newcommand\calO{{\mathcal O}}
\newcommand\Union\bigcup

Every constructible function on $M$ is (nonuniquely) a linear combination of
characteristic functions $1_A$ of locally closed sub{\em manifolds} of $M$.
In \cite{Ginzburg86} is given a formula for these $csm(1_A)$,
using the class
$[cc(\calO_A)] \in H^*_{\CC^\times}(T^*M) \iso H^*_{\CC^\times}(M) \iso H^*(M)[\hbar]$
of the characteristic cycle of the $\mathcal D_M$-module $\calO_A$, namely
$$ csm(1_A) = (-1)^{\dim A}\, [cc(\calO_A)]\,|_{\, \hbar \,\mapsto\, -1} $$

The dehomogenization $\hbar\mapsto -1$ is not so important, but the
$(-1)^{\dim_M A}$ is crucial in order to make the CSM class independent
of the expansion into characteristic functions, and more generally,
to make the natural transformation additive.
(The sign on $\hbar$ then serves to make
$csm(1_A) = +\big[\,\overline A\,\big]$,
up to higher degree terms in cohomology.) Essentially, an Euler
characteristic computation (not in topology -- rather, of a
Grothendieck--Cousin complex of $\calD$-modules)
is equipped with signs in order to become simply {\em a sum}
rather than its more usual alternating sum.
One effect of these signs is that when $A=M$, the class $csm(1_A)$
is the total Chern class of the {\em tangent} bundle of $M$,
despite its derivation from the dilation-equivariant Euler class of
the {\em cotangent} bundle.

\newcommand\phbar{{\phantom{\hbar}}}
For a small example, taking $\iota:\CP^1 \into T^*\CP^1$
to be the inclusion of the zero section, consider
$$
\begin{array}{r|ccccc} \\
  A & \CP^1 && \{\infty\} && \CC \\ \hline
  \dim A & 1 && 0 && 1 \\ 
  cc(\calO_A) & \iota(\CP^1) && T^*_\infty && \iota(\CP^1)\cup T^*_\infty \\
  \iota^*([cc(\calO_A)]) & \hbar[\CP^1] - 2[pt] && [pt] && \hbar[\CP^1]-[pt] \\
  csm(1_A) & \phbar [\CP^1] + 2[pt] &&  [pt] && \phbar [\CP^1] + [pt] 
\end{array}
$$
verifying the additivity $csm(1_{\CP^1}) = csm(1_\infty) + csm(1_{\CC^1})$,
despite the geometric equality
$[cc(\calO_{\CP^1})] = [cc(\calO_{\CC^1})] - [cc(\calO_{\infty})]$
derivable from the Grothendieck--Cousin complex.

\begin{lem}\label{lem:multCSM}
  Let $M$ be a smooth compact complex variety, and $A,B$ two
  locally closed smooth subvarieties, such that their closures
  $\overline A,\overline B$ are ``stratified-transverse'', i.e. each
  stratum in $\overline A$ is transverse to each one in $\overline B$.
  Then
  $$ [cc(\calO_A)]\, [cc(\calO_B)] = [cc(\calO_{A\cap B})]\, [M \subseteq T^*M] $$
  as elements of $H^*_{\CC^\times}(T^* M)$. If we pull these classes
  back to the base, $[M \subseteq T^*M]$ becomes the Euler class $e(T^*M)$.
\end{lem}

\begin{proof}
  Since every term is homogeneous (of, in fact, the same degree
  $\dim_{\RR} M$) it suffices to prove the specialization at $\hbar=-1$.
  After canceling signs, the equation becomes
  $$ csm(1_A)\, csm(1_B) = csm(1_{A\cap B})\, c(TM) $$
  which is \cite[theorem 10.5]{AMSS17}.
  (Being able to quote this lemma is the other reason, besides the
  connection to Euler characteristics, that we need the
  connection to CSM classes.)
\end{proof}

\begin{proof}[Proof of theorem \ref{thm:EP}.]
  \junk{note that the proof (relate SSM and $\chi_c$) is really for
    any $d$, but the puzzle formula only holds at $d\le 4$.
  }
  In the case at hand, the ambient manifold $M$ is the flag manifold
  $P_-\dom G$, the submanifold $A$ is the Bruhat cell
  $X^\lambda_\circ := P_-\dom P_- \lambda B_+$,
  and its $[cc(\calO_A)]$ is the stable basis element $St^\lambda$ from
  \S\ref{sec:stabclass}. One reference for this connection is \cite{AMSS17}.

  Using Kleiman transversality (finitely many times),
  we pick $g,g'$ generic enough that the stratified varieties
  $gX^\lambda,g'X^{\lambda'}$ are stratified-transverse, i.e. each $B_+$-orbit
  in $X^\lambda$ has been moved to be transverse to each $B_+$-orbit
  in $X^{\lambda'}$. Then we pick $g''$ generic enough so that $g'' X^{\lambda''}$
  is stratifed-transverse to $gX^\lambda \cap g'X^{\lambda'}$.
  In each case, the dilation-equivariant homology class of the characteristic
  cycle is unaffected by moving the Bruhat cell using these $g,g',g''$.

  We have all the pieces in place:
  $$     \renewcommand{\arraystretch}{1.5}
 \begin{array}{rclc}
     c_{\lambda\mu}^{\overleftarrow\nu} &=&
  \int_{P_-\dom G} St^\lambda\, St^\mu\, St^\nu / e(T^*M)^2 
  \qquad \text{from \eqref{eq:cfromint}} \\
  &=&   \int_{P_-\dom G}
  [cc(\calO_{X^\lambda_\circ})]\,[cc(\calO_{X^\mu_\circ})]\,[cc(\calO_{X^\nu_\circ})]
  / e(T^*M)^2  \\
  &=&   \int_{P_-\dom G}
  [cc(\calO_{gX^\lambda_\circ})]\,[cc(\calO_{g'X^\mu_\circ})]\,[cc(\calO_{g''X^\nu_\circ})]
  / e(T^*M)^2  \\
  &=&   \int_{P_-\dom G}
  [cc(\calO_{gX^\lambda_\circ \cap g'X^\mu_\circ \cap g''X^\nu_\circ})]
  e(T^*M)^2  / e(T^*M)^2  
      \qquad \text{from lemma \ref{lem:multCSM}, twice} \\
  &\mapsto& \int_{P_-\dom G}
             (-1)^{\dim (gX^\lambda_\circ \cap g'X^\mu_\circ \cap g''X^\nu_\circ)}
             csm(gX^\lambda_\circ \cap g'X^\mu_\circ \cap g''X^\nu_\circ)
             \qquad\text{ setting $\hbar\to -1$}\\
  &=& (-1)^\delta\ \chi_c(gX^\lambda_\circ \cap g'X^\mu_\circ \cap g''X^\nu_\circ)
 &\qedhere
  \end{array} $$
\end{proof}

\junk{
For a non-example, consider $M = (\PP^1)^2$, $A = \AA^1 \times \{\infty\}$,
$B = \{\infty\} \times \AA^1$, whose CSM classes multiply to the point class
despite $A,B$ having empty (hence transverse) intersection.
Even $\wt A = \overline A,\wt B = \overline B$ meet transversely.
The issue is that the {\em lower} strata in $\wt A,\wt B$ are not transverse.
}

\begin{cor}
At $d\le 3$, the sign of 
$\chi_c(g X_o^\lambda \cap g' X_o^\mu \cap g'' X_o^\nu)$
(if nonzero) is $(-1)^{\dim(P_-\backslash G)-(\ell(\lambda)+\ell(\mu)+\ell(\nu))}$.
\end{cor}

This is one of those rare situations where positivity in (generalized)
Schubert calculus was discovered through explicit formula before
being given a geometric proof (others being Lesieur's observation that
the Littlewood-Richardson rule computes Grassmannian Schubert calculus
decades before Kleiman transversality, and \cite{Buch} antedating
\cite{Brion-KPos}). Since releasing this paper, a geometric proof
has been found that applies to all $G/P$ \cite{SSW}.

\newcommand\ul\underline
\begin{ex}
  Consider the case $\lambda=\mu=\nu = 0101$, where the Bruhat cell
  w.r.t. a flag $F^\bullet$ is
  \begin{align*}
    X^{0101}_\circ(F)
    &= X^{0101} \setminus (X^{0110} \cup X^{1001}) \\
    &= \{V\in Gr(2,4)\, :\, \dim(V\cap F^2) \geq 1\} \ \setminus\ 
        \left( \{V\, :\, V \geq F^1\} \cup \{V\, :\, V \leq F^3\} \right)
  \end{align*}
  If we take $F^\bullet,G^\bullet,H^\bullet$ generic flags, then by dimension count 
  $X^{0101}(F) \cap X^{0101}(G) \cap X^{0101}(H)$ is a curve.
  By Kleiman--Bertini it is normal, hence smooth.
  Using $K$-theory puzzles one can determine it to be arithmetic genus $0$.
  Since we want however to intersect the open cells $X^{0101}_\circ$,
  we need to rip out the points
  $$ X^{0101}(F) \cap X^{0101}(G) \cap X^{\ul{0110}}(H) \qquad
  X^{0101}(F) \cap X^{0101}(G) \cap X^{\ul{1001}}(H) \qquad 
  X^{0101}(F) \cap X^{\ul{0110}}(G) \cap X^{0101}(H) $$ 
  $$ X^{0101}(F) \cap X^{\ul{1001}}(G) \cap X^{0101}(H) \qquad
  X^{\ul{1001}}(F) \cap X^{0101}(G) \cap X^{0101}(H) \qquad 
  X^{\ul{0110}}(F) \cap X^{0101}(G) \cap X^{0101}(H)
  $$
  where the underlining points out the superscript changed from $0101$.
  Hence $X^{0101}_\circ(F) \cap X^{0101}_\circ(G) \cap X^{0101}_\circ(H)$
  is $\PP^1$ minus $6$ points, with $\chi_c = -4$. The four
  relevant puzzles are these:
\begin{center}
\def\thescale{1.4}
\def\posa{0.5}\def\posb{0.5}
\begin{tikzpicture}[math mode,nodes={edgelabel},x={(-0.577cm,-1cm)},y={(0.577cm,-1cm)},scale=\thescale]
\draw[thick] (0,0) -- node[pos=\posa] {0} ++(0,1); \draw[thick] (0,0) -- node[pos=\posb] {1} ++(1,0); \draw[thick] (0+1,0) -- node {10} ++(-1,1); 
\draw[thick] (0,1) -- node[pos=\posa] {1} ++(0,1); \draw[thick] (0,1) -- node[pos=\posb] {1} ++(1,0); \draw[thick] (0+1,1) -- node {1} ++(-1,1); 
\draw[thick] (0,2) -- node[pos=\posa] {0} ++(0,1); \draw[thick] (0,2) -- node[pos=\posb] {1} ++(1,0); \draw[thick] (0+1,2) -- node {10} ++(-1,1); 
\draw[thick] (0,3) -- node[pos=\posa] {1} ++(0,1); \draw[thick] (0,3) -- node[pos=\posb] {10} ++(1,0); \draw[thick] (0+1,3) -- node {0} ++(-1,1); 
\draw[thick] (1,0) -- node[pos=\posa] {0} ++(0,1); \draw[thick] (1,0) -- node[pos=\posb] {0} ++(1,0); \draw[thick] (1+1,0) -- node {0} ++(-1,1); 
\draw[thick] (1,1) -- node[pos=\posa] {1} ++(0,1); \draw[thick] (1,1) -- node[pos=\posb] {10} ++(1,0); \draw[thick] (1+1,1) -- node {0} ++(-1,1); 
\draw[thick] (1,2) -- node[pos=\posa] {10} ++(0,1); \draw[thick] (1,2) -- node[pos=\posb] {0} ++(1,0); \draw[thick] (1+1,2) -- node {1} ++(-1,1); 
\draw[thick] (2,0) -- node[pos=\posa] {1} ++(0,1); \draw[thick] (2,0) -- node[pos=\posb] {1} ++(1,0); \draw[thick] (2+1,0) -- node {1} ++(-1,1); 
\draw[thick] (2,1) -- node[pos=\posa] {0} ++(0,1); \draw[thick] (2,1) -- node[pos=\posb] {0} ++(1,0); \draw[thick] (2+1,1) -- node {0} ++(-1,1); 
\draw[thick] (3,0) -- node[pos=\posa] {10} ++(0,1); \draw[thick] (3,0) -- node[pos=\posb] {0} ++(1,0); \draw[thick] (3+1,0) -- node {1} ++(-1,1); 
\end{tikzpicture}\qquad
\begin{tikzpicture}[math mode,nodes={edgelabel},x={(-0.577cm,-1cm)},y={(0.577cm,-1cm)},scale=\thescale]
\draw[thick] (0,0) -- node[pos=\posa] {0} ++(0,1); \draw[thick] (0,0) -- node[pos=\posb] {1} ++(1,0); \draw[thick] (0+1,0) -- node {10} ++(-1,1); 
\draw[thick] (0,1) -- node[pos=\posa] {1} ++(0,1); \draw[thick] (0,1) -- node[pos=\posb] {1} ++(1,0); \draw[thick] (0+1,1) -- node {1} ++(-1,1); 
\draw[thick] (0,2) -- node[pos=\posa] {0} ++(0,1); \draw[thick] (0,2) -- node[pos=\posb] {0} ++(1,0); \draw[thick] (0+1,2) -- node {0} ++(-1,1); 
\draw[thick] (0,3) -- node[pos=\posa] {1} ++(0,1); \draw[thick] (0,3) -- node[pos=\posb] {10} ++(1,0); \draw[thick] (0+1,3) -- node {0} ++(-1,1); 
\draw[thick] (1,0) -- node[pos=\posa] {0} ++(0,1); \draw[thick] (1,0) -- node[pos=\posb] {0} ++(1,0); \draw[thick] (1+1,0) -- node {0} ++(-1,1); 
\draw[thick] (1,1) -- node[pos=\posa] {10} ++(0,1); \draw[thick] (1,1) -- node[pos=\posb] {0} ++(1,0); \draw[thick] (1+1,1) -- node {1} ++(-1,1); 
\draw[thick] (1,2) -- node[pos=\posa] {1} ++(0,1); \draw[thick] (1,2) -- node[pos=\posb] {1} ++(1,0); \draw[thick] (1+1,2) -- node {1} ++(-1,1); 
\draw[thick] (2,0) -- node[pos=\posa] {0} ++(0,1); \draw[thick] (2,0) -- node[pos=\posb] {1} ++(1,0); \draw[thick] (2+1,0) -- node {10} ++(-1,1); 
\draw[thick] (2,1) -- node[pos=\posa] {1} ++(0,1); \draw[thick] (2,1) -- node[pos=\posb] {10} ++(1,0); \draw[thick] (2+1,1) -- node {0} ++(-1,1); 
\draw[thick] (3,0) -- node[pos=\posa] {10} ++(0,1); \draw[thick] (3,0) -- node[pos=\posb] {0} ++(1,0); \draw[thick] (3+1,0) -- node {1} ++(-1,1); 
\end{tikzpicture}\qquad
\begin{tikzpicture}[math mode,nodes={edgelabel},x={(-0.577cm,-1cm)},y={(0.577cm,-1cm)},scale=\thescale]
\draw[thick] (0,0) -- node[pos=\posa] {0} ++(0,1); \draw[thick] (0,0) -- node[pos=\posb] {1} ++(1,0); \draw[thick] (0+1,0) -- node {10} ++(-1,1); 
\draw[thick] (0,1) -- node[pos=\posa] {1} ++(0,1); \draw[thick] (0,1) -- node[pos=\posb] {1} ++(1,0); \draw[thick] (0+1,1) -- node {1} ++(-1,1); 
\draw[thick] (0,2) -- node[pos=\posa] {0} ++(0,1); \draw[thick] (0,2) -- node[pos=\posb] {0} ++(1,0); \draw[thick] (0+1,2) -- node {0} ++(-1,1); 
\draw[thick] (0,3) -- node[pos=\posa] {1} ++(0,1); \draw[thick] (0,3) -- node[pos=\posb] {10} ++(1,0); \draw[thick] (0+1,3) -- node {0} ++(-1,1); 
\draw[thick] (1,0) -- node[pos=\posa] {0} ++(0,1); \draw[thick] (1,0) -- node[pos=\posb] {0} ++(1,0); \draw[thick] (1+1,0) -- node {0} ++(-1,1); 
\draw[thick] (1,1) -- node[pos=\posa] {10} ++(0,1); \draw[thick] (1,1) -- node[pos=\posb] {10} ++(1,0); \draw[thick] (1+1,1) -- node {10} ++(-1,1); 
\draw[thick] (1,2) -- node[pos=\posa] {1} ++(0,1); \draw[thick] (1,2) -- node[pos=\posb] {1} ++(1,0); \draw[thick] (1+1,2) -- node {1} ++(-1,1); 
\draw[thick] (2,0) -- node[pos=\posa] {1} ++(0,1); \draw[thick] (2,0) -- node[pos=\posb] {1} ++(1,0); \draw[thick] (2+1,0) -- node {1} ++(-1,1); 
\draw[thick] (2,1) -- node[pos=\posa] {0} ++(0,1); \draw[thick] (2,1) -- node[pos=\posb] {0} ++(1,0); \draw[thick] (2+1,1) -- node {0} ++(-1,1); 
\draw[thick] (3,0) -- node[pos=\posa] {10} ++(0,1); \draw[thick] (3,0) -- node[pos=\posb] {0} ++(1,0); \draw[thick] (3+1,0) -- node {1} ++(-1,1); 
\end{tikzpicture}\qquad
\begin{tikzpicture}[math mode,nodes={edgelabel},x={(-0.577cm,-1cm)},y={(0.577cm,-1cm)},scale=\thescale]
\draw[thick] (0,0) -- node[pos=\posa] {0} ++(0,1); \draw[thick] (0,0) -- node[pos=\posb] {1} ++(1,0); \draw[thick] (0+1,0) -- node {10} ++(-1,1); 
\draw[thick] (0,1) -- node[pos=\posa] {1} ++(0,1); \draw[thick] (0,1) -- node[pos=\posb] {10} ++(1,0); \draw[thick] (0+1,1) -- node {0} ++(-1,1); 
\draw[thick] (0,2) -- node[pos=\posa] {0} ++(0,1); \draw[thick] (0,2) -- node[pos=\posb] {0} ++(1,0); \draw[thick] (0+1,2) -- node {0} ++(-1,1); 
\draw[thick] (0,3) -- node[pos=\posa] {1} ++(0,1); \draw[thick] (0,3) -- node[pos=\posb] {10} ++(1,0); \draw[thick] (0+1,3) -- node {0} ++(-1,1); 
\draw[thick] (1,0) -- node[pos=\posa] {10} ++(0,1); \draw[thick] (1,0) -- node[pos=\posb] {0} ++(1,0); \draw[thick] (1+1,0) -- node {1} ++(-1,1); 
\draw[thick] (1,1) -- node[pos=\posa] {0} ++(0,1); \draw[thick] (1,1) -- node[pos=\posb] {1} ++(1,0); \draw[thick] (1+1,1) -- node {10} ++(-1,1); 
\draw[thick] (1,2) -- node[pos=\posa] {1} ++(0,1); \draw[thick] (1,2) -- node[pos=\posb] {1} ++(1,0); \draw[thick] (1+1,2) -- node {1} ++(-1,1); 
\draw[thick] (2,0) -- node[pos=\posa] {1} ++(0,1); \draw[thick] (2,0) -- node[pos=\posb] {1} ++(1,0); \draw[thick] (2+1,0) -- node {1} ++(-1,1); 
\draw[thick] (2,1) -- node[pos=\posa] {0} ++(0,1); \draw[thick] (2,1) -- node[pos=\posb] {0} ++(1,0); \draw[thick] (2+1,1) -- node {0} ++(-1,1); 
\draw[thick] (3,0) -- node[pos=\posa] {10} ++(0,1); \draw[thick] (3,0) -- node[pos=\posb] {0} ++(1,0); \draw[thick] (3+1,0) -- node {1} ++(-1,1); 
\end{tikzpicture}
\end{center}
\end{ex}

\setcounter{MaxMatrixCols}{20}

\begin{ex}
The geometry of $X^{0101}_\circ(F) \cap X^{0101}_\circ(G) \cap X^{0011}_\circ(H)$
is more complicated, but still approachable. We list the triples
$\lambda \geq 0101$, $\mu \geq 0101$, $\nu \geq 0011$ such that
$X^\lambda(F) \cap X^\mu(G) \cap X^\nu(H)$ intersect:
$$
\begin{array}{c|ccccccccccccc}
\lambda & 0101 & 1001 & 0110 & 0101 & 0101 & 0101 \\ 
\mu     & 0101 & 0101 & 0101 & 0110 & 1001 & 0101 \\
\nu     & 0011 & 0011 & 0011 & 0011 & 0011 & 0101 \\
  \hline
\cap&\PP^1\times\PP^1& A & B & C & D & E 
\end{array}
$$
where $A \iso B \iso C \iso D \iso E \iso \PP^1$, and point intersections
$$
\begin{array}{c|ccccccccccccc}
\lambda & 1010 & 0101 & 0110 & 1001 & 1001 & 0110 & 0101 & 0101 \\
\mu     & 0101 & 1010 & 1001 & 0110 & 0101 & 0101 & 0110 & 1001 \\
\nu     & 0011 & 0011 & 0011 & 0011 & 0101 & 0101 & 0101 & 0101 \\
  \hline              
\cap    &A\cap B &C\cap D &B\cap D &A\cap C &A\cap E &B\cap E &C\cap E &D\cap E
\end{array}
$$
(note that $A\cap D = B\cap C = \emptyset$) giving
$\chi_c = \chi_c(\PP^1\times\PP^1) -5\chi_c(\PP^1) + 8\chi_c(pt) = 4 -10 + 8=2$.
The two puzzles are these:

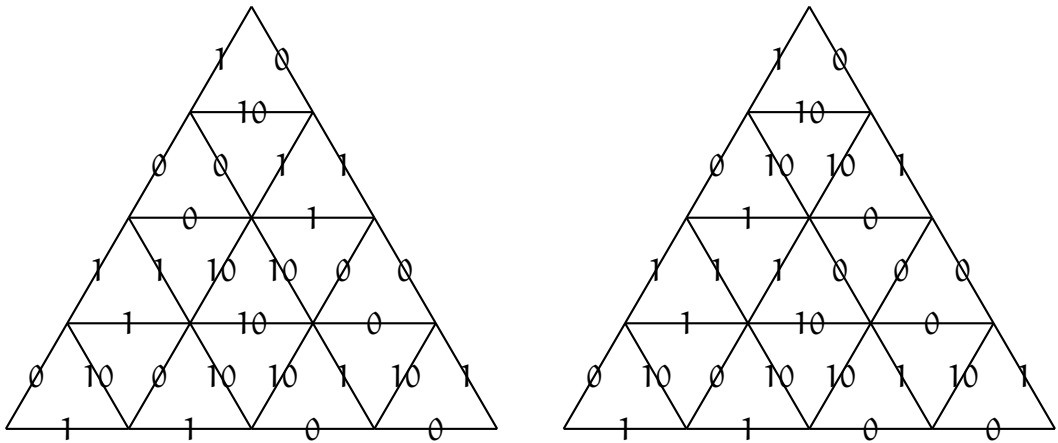
\junk{PZJ: more complicated ex: box$^2$ times empty in $Gr(2,4)$ which
  should give $\chi_c=2$.  one finds a quadric surface (4) minus 5
  $\PP^1$ plus 8 points.  indeed two of these $\PP^1$ have empty
  intersection (can't sit inside a plane and go through a point), and
  ${5\choose 2}-2=8$.}
\begin{center}
\def\thescale{1.4}
\def\posa{0.5}\def\posb{0.5}
\begin{tikzpicture}[math mode,nodes={edgelabel},x={(-0.577cm,-1cm)},y={(0.577cm,-1cm)},scale=\thescale]
\draw[thick] (0,0) -- node[pos=\posa] {0} ++(0,1); \draw[thick] (0,0) -- node[pos=\posb] {1} ++(1,0); \draw[thick] (0+1,0) -- node {10} ++(-1,1); 
\draw[thick] (0,1) -- node[pos=\posa] {1} ++(0,1); \draw[thick] (0,1) -- node[pos=\posb] {1} ++(1,0); \draw[thick] (0+1,1) -- node {1} ++(-1,1); 
\draw[thick] (0,2) -- node[pos=\posa] {0} ++(0,1); \draw[thick] (0,2) -- node[pos=\posb] {0} ++(1,0); \draw[thick] (0+1,2) -- node {0} ++(-1,1); 
\draw[thick] (0,3) -- node[pos=\posa] {1} ++(0,1); \draw[thick] (0,3) -- node[pos=\posb] {10} ++(1,0); \draw[thick] (0+1,3) -- node {0} ++(-1,1); 
\draw[thick] (1,0) -- node[pos=\posa] {0} ++(0,1); \draw[thick] (1,0) -- node[pos=\posb] {0} ++(1,0); \draw[thick] (1+1,0) -- node {0} ++(-1,1); 
\draw[thick] (1,1) -- node[pos=\posa] {10} ++(0,1); \draw[thick] (1,1) -- node[pos=\posb] {10} ++(1,0); \draw[thick] (1+1,1) -- node {10} ++(-1,1); 
\draw[thick] (1,2) -- node[pos=\posa] {1} ++(0,1); \draw[thick] (1,2) -- node[pos=\posb] {10} ++(1,0); \draw[thick] (1+1,2) -- node {0} ++(-1,1); 
\draw[thick] (2,0) -- node[pos=\posa] {1} ++(0,1); \draw[thick] (2,0) -- node[pos=\posb] {1} ++(1,0); \draw[thick] (2+1,0) -- node {1} ++(-1,1); 
\draw[thick] (2,1) -- node[pos=\posa] {10} ++(0,1); \draw[thick] (2,1) -- node[pos=\posb] {0} ++(1,0); \draw[thick] (2+1,1) -- node {1} ++(-1,1); 
\draw[thick] (3,0) -- node[pos=\posa] {10} ++(0,1); \draw[thick] (3,0) -- node[pos=\posb] {0} ++(1,0); \draw[thick] (3+1,0) -- node {1} ++(-1,1); 
\end{tikzpicture}\qquad
\begin{tikzpicture}[math mode,nodes={edgelabel},x={(-0.577cm,-1cm)},y={(0.577cm,-1cm)},scale=\thescale]
\draw[thick] (0,0) -- node[pos=\posa] {0} ++(0,1); \draw[thick] (0,0) -- node[pos=\posb] {1} ++(1,0); \draw[thick] (0+1,0) -- node {10} ++(-1,1); 
\draw[thick] (0,1) -- node[pos=\posa] {1} ++(0,1); \draw[thick] (0,1) -- node[pos=\posb] {10} ++(1,0); \draw[thick] (0+1,1) -- node {0} ++(-1,1); 
\draw[thick] (0,2) -- node[pos=\posa] {0} ++(0,1); \draw[thick] (0,2) -- node[pos=\posb] {0} ++(1,0); \draw[thick] (0+1,2) -- node {0} ++(-1,1); 
\draw[thick] (0,3) -- node[pos=\posa] {1} ++(0,1); \draw[thick] (0,3) -- node[pos=\posb] {10} ++(1,0); \draw[thick] (0+1,3) -- node {0} ++(-1,1); 
\draw[thick] (1,0) -- node[pos=\posa] {10} ++(0,1); \draw[thick] (1,0) -- node[pos=\posb] {0} ++(1,0); \draw[thick] (1+1,0) -- node {1} ++(-1,1); 
\draw[thick] (1,1) -- node[pos=\posa] {0} ++(0,1); \draw[thick] (1,1) -- node[pos=\posb] {1} ++(1,0); \draw[thick] (1+1,1) -- node {10} ++(-1,1); 
\draw[thick] (1,2) -- node[pos=\posa] {1} ++(0,1); \draw[thick] (1,2) -- node[pos=\posb] {10} ++(1,0); \draw[thick] (1+1,2) -- node {0} ++(-1,1); 
\draw[thick] (2,0) -- node[pos=\posa] {1} ++(0,1); \draw[thick] (2,0) -- node[pos=\posb] {1} ++(1,0); \draw[thick] (2+1,0) -- node {1} ++(-1,1); 
\draw[thick] (2,1) -- node[pos=\posa] {10} ++(0,1); \draw[thick] (2,1) -- node[pos=\posb] {0} ++(1,0); \draw[thick] (2+1,1) -- node {1} ++(-1,1); 
\draw[thick] (3,0) -- node[pos=\posa] {10} ++(0,1); \draw[thick] (3,0) -- node[pos=\posb] {0} ++(1,0); \draw[thick] (3+1,0) -- node {1} ++(-1,1); 
\end{tikzpicture}
\end{center}
\end{ex}

As mentioned in the proof, in the special case $\delta=0$, or 
$\ell(\lambda)+\ell(\mu)+\ell(\nu)=\dim(P_-\backslash G)$,
Theorem~\ref{thm:EP} reduces to ordinary Schubert calculus, i.e., counting points in the intersection of three
Schubert varieties in general position. In fact, for $d\le 3$, the puzzles of theorem~\ref{thm:EP} reduce to ordinary puzzles as formulated in e.g.~\cite{artic71}.
The reason is that triangles at $d\le 3$ can only have nonnegative inversion charge, which means that for $\ell(\lambda)+\ell(\mu)+\ell(\nu)=\dim(P_-\backslash G)$ to hold, the triangles with positive inversion charge cannot occur, and excluding them exactly turns our puzzles
into ordinary puzzles. 
In contradistinction, at $d=4$, triangles of both positive and negative charge exist, so no simplification occurs in the case
of ordinary Schubert calculus (furthermore, as already pointed out in the previous section, the rule is nonpositive).

If $\ell(\lambda)+\ell(\mu)+\ell(\nu)>\dim(P_-\backslash G)$, theorem~\ref{thm:EP} implies that the structure constant, i.e.,
$\sum_P \text{fug}(P)$, is zero. Because of nonpositivity at $d=4$, this does not imply that there aren't any puzzles
with such boundaries; see appendix~\ref{app:exd4c} for a counterexample.

Finally, if one considers the $q \mapsto 1$ specialization instead of
$q \mapsto -1$, some puzzle pieces at $d\leq 3$ acquire fugacity $-1$
(which is why we usually don't), making each $fug(P) = (-1)^\delta$.
In that sense, the formula becomes simpler, as the prefactor has
been absorbed.

\section{Positivity}\label{sec:positivity}
Recall from \S\ref{ssec:mainres} that a \defn{positivity notion} in Schubert calculus
is a submonoid $M$ (under $+$) of the coefficient
ring (either the cohomology of a point, or a localization thereof) s.t.
$M \cap -M = 0$.

While it is frequently required that $M$ be closed under multiplication
as well as addition,
there are sometimes stronger positivity statements to be made when this
is not done. For example Graham positivity \cite{Graham},
in which $M$ is any sum of products of simple roots, can be
tightened up to sums of products of {\em distinct positive} roots.
(This follows from a careful reading of Graham's original proof.
Note too that the formul\ae\ in \cite{KT,artic71} for Schubert
structure constants in $H^*_T(Gr(k,n))$ and $H^*_T(Fl(j,k;\ n))$ are
manifestly $M$-positive for this tighter $M$.)
[this whole paragraph should be moved

We now show that our main theorem~\ref{thm:main} provides a {\em positive}\/ product rule up to $d=2$
equivariantly, and up to $d=3$ nonequivariantly.

\subsection{Positivity in equivariant $K$-theory}
The explicit entries of the $R$-matrices involved in theorem~\ref{thm:main} are given in \S\ref{ssec:d1setup}
at $d=1$ and in \cite[appendix~A]{artic71} at $d=2$.



\newcommand\qq{q^2\phantom{z}}
We introduce the following
\begin{lem}\label{lem:M}
  Let $M$ be the set of sums of products of factors
  $$
  -q^{\pm}
  \qquad
  z^{\pm}
  \qquad
  \frac{1-\qq}{1-q^2z}
  \qquad
  \frac{z-1}{1-q^2z}    
  $$
  where $z$ varies over distinct $z_j/z_i$, $i<j$.
  Then $M$ is a positivity notion, i.e. $M\cap -M = 0$.
\end{lem}
\begin{proof}
  If we specialize to $z_i = 2^i$, $q=-2^{-n/2}$, then every factor is a
  positive real number.
\end{proof}

For $d=1,2$, by inspection, any nonzero $R$-matrix entry can be expressed as a product of the factors
of lemma~\ref{lem:M} only. We conclude that our puzzle rule for equivariant motivic Chern class is positive
for $d=1,2$.

It seems impossible to find such an elementary proof of positivity at $d=3$. This is in line with the fact
that in \cite{artic71}, we were unable to provide a $d=3$ equivariant rule in ordinary Schubert calculus.
It is not obvious to find an explicit counter-example to positivity, so the issue remains open.

We do not expect any positivity in $d=4$ equivariant $K$-theory.

\subsection{Positivity in equivariant cohomology}
The exact same statements hold in cohomology:
\begin{lem}\label{lem:MH}
  Let $M$ be the set of sums of products of factors
  $$
  \frac{\hbar}{\hbar-y}
  \qquad
  \frac{y}{\hbar-y}    
  $$
  where $y$ varies over distinct $y_j-y_i$, $i<j$.
  Then $M$ is a positivity notion, i.e. $M\cap -M = 0$.
\end{lem}
\begin{proof}
  Specialize at $y_i=i$, $\hbar=n$.
\end{proof}

The nonzero entries of the rational $R$-matrices for $d=1,2$ are precisely one of the two factors
of lemma~\ref{lem:MH}. We conclude that our puzzle rule for equivariant SSM classes is positive
for $d=1,2$.

\subsection{Positivity in nonequivariant $K$-theory and  cohomology}
In nonequivariant $K$-theory, one can easily check that $d\le 3$ rules only involve nonzero fugacities that
are powers of $-q$, so positivity follows immediately for our rule for nonequivariant motivic Chern classes
for $d=1,2,3$.

In particular, as pointed out in theorem \ref{thm:EP}, the fugacity of nonequivariant puzzles
for SSM classes is just $1$, making the positivity statement trivial.

As already discussed in \S\ref{ssec:d4}, the $d=4$ cohomology puzzle rule is not positive, cf \S\ref{app:exd4c}
for an example of puzzles whose sum of fugacities equals zero. The same puzzles also provide a counter-example
to positivity in $d=4$ nonequivariant $K$-theory.

\section{A couple of results on
  Nakajima quiver varieties}\label{sec:Nakajima}
\subsection{The varieties}\label{sec:quiver}
A Nakajima quiver variety $\calM_I(w,v,\theta)$ depends
on five data (see e.g. \cite{Ginz-naka}):
\begin{itemize}
\item a quiver (directed graph) $I$ of ``gauge vertices'' $\{\alpha_i\}$
  spanning a lattice called $\lie{t}^*$, 
  to each of which we attach a ``framed vertex'' of degree $1$, 
  giving the dual basis in $\lie{t}$,
\item a ``dimension vector'' ($\NN$-valued) $w=(w^i)$
  on the framing vertices,
\item a similar dimension vector $v=(v^i)$ on the gauged vertices,
\item the ``complex moment'' $\theta_\CC \in \CC\tensor\lie{t}$ that will be zero
  (hence ignored) until \S\ref{sec:geominterp}, and
\item the ``stability vector'' $\theta \in \RR\tensor \lie{t}$
  we only partly explain below.
\end{itemize}
The points in $\calM_I(w,v,\theta)$ index certain
equivalence classes of representations of the (doubled) Nakajima quiver. 
Specifically, the representations satisfy a closed ``moment map'' condition
at the gauged vertices, and an open ``stability'' condition 
depending on $\theta$.
We won't detail these conditions (punting, again, to \cite{Ginz-naka})
but summarize what we need about them here.

Write $\vec\alpha_i \in \lie{t}^*$ for the basis vector corresponding
to the gauge vertex $\alpha_i$.
Call the stability vector $\theta$ \defn{positive} if
$\langle \theta, \vec\alpha_i \rangle > 0$ for all $\vec\alpha$, 
and just
\defn{$v$-positive} if $\langle \theta, \vec\alpha_i \rangle > 0$ for those
$\vec\alpha_i$ with coefficient $v^i \neq 0$ (a generalization of (3.2.2) from
\cite{Ginz-naka}). The $\theta$s we use will always be $v$-positive,
in which case we omit $\theta$ from the notation.
The quiver will eventually be one of four Dynkin diagrams $X_{2d}$, 
$d=1,\ldots,4$, and will similarly be suppressed in the notation, $\calM(w,v)$.
Since each $X_{2d}$ is ADE, we can identify $\lie{t}^*$ with the
corresponding root lattice, and make reference to fundamental weights
$\vec\omega_i \in \QQ\tensor\lie{t}^*$.

Call $(w,v)$ of \defn{flag type} if $w^i \neq 0$ at only one vertex $\fbox c$, 
and $(v^i)$ is supported on a type $A_d$ subdiagram $S$ of $I$ with
$\fbox c$ attached to one end of $S$, which we call its \defn{head}.
For examples, see the first quiver in each of figures \ref{fig:d1}-\ref{fig:d4}.

\begin{prop}\label{prop:flagquiver}
  \begin{enumerate}
  \item \cite[Lemma 3.2.3(i)]{Ginz-naka} 
    If $\theta$ is $v$-positive, then a representation is $\theta$-stable
    iff it contains no subrepresentations supported on the gauged vertices.
  \item \cite[\S 7]{Nakaj-quiv1}
    If $(w,v)$ is of flag type and $\theta$ is $v$-positive, then
    $\calM_I(w,v,\theta)$ is isomorphic to 
    the cotangent bundle of a $d$-step flag variety.
    If $S$ is oriented toward its head,
    then the steps are given by $(v^i, i \in S)$.    
  \item Assume $Q$'s underlying graph is an $ADE$ Dynkin diagram.
    With $w$ fixed and $\theta$ positive, the set of vectors $v$ with
    $M_Q(w,v,\theta) \neq \emptyset$ forms the lattice points in a
    polytope $P(w)$ with integral vertices.
    \junk{
    Moreover, $P(w_1+w_2) = P(w_1)+P(w_2)$, where the r.h.s.
    is the Minkowski sum.}
  \end{enumerate}
\end{prop}

\begin{proof}
  For both (1) and (2), observe that the quiver variety is unchanged if
  we throw away all vertices $\alpha \in I$ with $v_i = 0$, then follow
  the stated references.

  For (3), since $I$ is $ADE$, it has an associated simply-laced Lie
  algebra $\lie{g}_I$.  We use \cite[Theorem 10.16]{Nakaj-quiv1} to observe that
  $M_I(w,v,\theta) \neq \emptyset$ iff a certain $\lie{g}_I$-irrep
  determined by $w$ has a nonzero weight space, weight determined by $v$.
  Then the result is standard from representation theory, e.g. see
  \cite{GuilleminLermanSternbergSFandMD}.
\end{proof}


We start with $d=1$ in figure \ref{fig:d1}, where conveniently
we need only consider $(w,v)$ of flag type, and three of the varieties are
$$
\tikz[script math mode,baseline=0]{\dOne{j}{0} 
  \node[framed] at (1,1) (w1) {n}; \draw (v1) -- (w1);}
\quad\iso\quad
\tikz[script math mode,baseline=0]{\dOne{n}{j} 
  \node[framed] at (1,1) (w1) {n}; \draw (v1) -- (w1);}
\quad\iso\quad
\tikz[script math mode,baseline=0]{\dOne{j}{j} 
  \node[framed] at (2,1) (w2) {n}; \draw (v2) -- (w2);}
\quad\iso\quad
T^*\, Gr(j,n)
$$

In each of $d=2,3,4$ (figures \ref{fig:d2}--\ref{fig:d4}) we will need
several $(w,v)$ that are {\em not} of flag type, but whose quiver varieties
are nonetheless isomorphic to cotangent bundles of $d$-step flag varieties.
To verify these isomorphisms, we will use the following proposition.

\begin{prop}\label{prop:reflect}
  \begin{enumerate}
  \item \cite{Naka-Weyl} Let $\pi\in W$, and define $v'$ by
    $$ \pi \cdot \left(
      \sum_i w^i \vec\omega_i - \sum_i v^i \vec\alpha_i \right)
    =
    \sum_i w^i \vec\omega_i - \sum_i v'^i \vec\alpha_i. $$
    Then $\calM_I(w,v',\pi\cdot\theta) \iso \calM_I(w,v,\theta)$
    as complex varieties, equivariantly w.r.t.\ the framing group action
    $\prod_{i\in I} GL(w^i)$ on both sides. \\
    To calculate $v'$ from $v$ when $\pi = r_{\vec\alpha_i}$,
    one replaces the $\vec\alpha_i$ label $v^i$ by the sum of all adjacent labels
    {\em including} labels on framed vertices, minus the original label $v^i$.
  \item Fix $\pi\in W$, and let $w,v,v',S$ be as in (1).
    Assume that $\pi\cdot \theta$ is positive.
    Assume that $(w,v)$ is of flag type, with $v$ supported on a type $A$ 
    subdiagram $S$, and that $\pi$ is chosen minimal in its $W/W_S$ coset. 
    Then
    \begin{enumerate}
    \item each entry of $v'$ is a positive combination of the entries
      of $w$ and $v$, when expressed using the reflection algorithm from (1),
      and
    \item $\theta$ is $v$-positive.
    \end{enumerate}
    Then by (1) and proposition \ref{prop:flagquiver} (2), we have
    $\calM_I(w,v',\pi\cdot\theta) \iso \calM_I(w,v,\theta) \iso$ \,
    the cotangent bundle of a $d$-step flag variety,
    where (assuming $S$ is oriented toward its head)
    the steps are given by $(v^i, i\in S)$. \\
    If $v$ is nonzero on $S$, then $\pi$'s minimality in its coset doesn't
    merely imply condition (a), but can be inferred from it.
  \end{enumerate}
\end{prop}


\begin{proof}
  The calculation at the end of (1), based
  on $r_{\vec\alpha} \lambda = \lambda - \langle \vec\alpha,\lambda \rangle \vec\alpha$,
  is straightforward. It remains to prove (2).

  The minimality condition on $\pi$ is equivalent to
  ``$\pi\cdot \vec\alpha_i$ is a positive root for each $i \in S$'',
  which we will use in both (2a) and (2b). Note that this is a
  rephrasing of the last statement of the proposition.

  First observe that $\pi\cdot\vec\omega_i$ is a weight of $V_{\vec\omega_i}$,
  hence $\vec\omega_i - \pi\cdot\vec\omega_i$ is a positive combination of
  simple roots. This and the quote in the previous paragraph prove (2a).

  For (2b) we compute
  $$ \langle \theta, \vec\alpha \rangle
  = \langle \pi\cdot \theta, \pi\cdot \vec\alpha \rangle $$
  then use the quote above and the assumed positivity of $\pi\cdot\theta$.  
\end{proof}



We give an example of proposition \ref{prop:reflect}, 
one of the cases appearing in figure \ref{fig:d2}.
At each stage we color the label(s) where we're about to 
perform (commuting) reflections,
using the recipe at the end of proposition \ref{prop:reflect} (1).
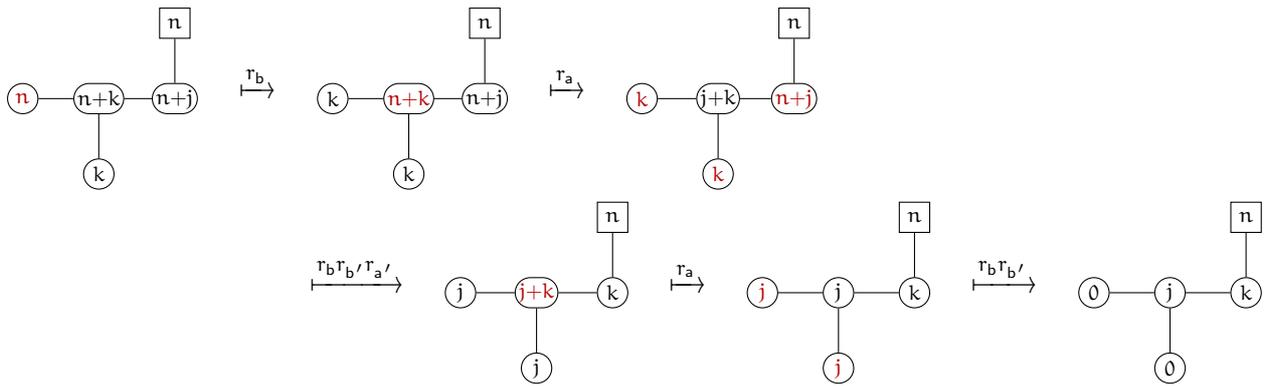
\begin{multline*}
\tikz[script math mode,baseline=0]
{\dtwo{\color{red!80!black} n}{n+k}{n+j}{k} \node[framed] at (3,1) (w3) {n}; \draw (v3) -- (w3);}
\quad\xmapsto{\rr {b}}\quad
\tikz[script math mode,baseline=0]
{\dtwo{k}{\color{red!80!black}{n+k}}{n+j}{k} \node[framed] at (3,1) (w3) {n};\draw (v3) -- (w3);}
\quad\xmapsto{\rr {a}}\quad
\tikz[script math mode,baseline=0]
{\dtwo{\color{red!80!black} k}{j+k}{\color{red!80!black}{n+j}}{\color{red!80!black} k} \node[framed] at (3,1) (w3) {n};\draw (v3) -- (w3);}
\\
\quad\xmapsto{\rr {b} \rr {b'} \rr {a'}}\quad
\tikz[script math mode,baseline=0]
{\dtwo{j}{\color{red!80!black}{j+k}}{k}{j} \node[framed] at (3,1) (w3) {n};\draw (v3) -- (w3);}
\quad\xmapsto{\rr {a}}\quad
\tikz[script math mode,baseline=0]
{\dtwo{\color{red!80!black} j}{j}{k}{\color{red!80!black} j} \node[framed] at (3,1) (w3) {n};\draw (v3) -- (w3);}
\quad\xmapsto{\rr b \rr {b'}}\quad
\tikz[script math mode,baseline=0]
{\dtwo{0}{j}{k}{0} \node[framed] at (3,1) (w3) {n};\draw (v3) -- (w3);}
\end{multline*}

The $v$-coefficients $n,n+k,n+j,k$ in the $D_4$ quiver variety
are positive combinations of $n,j,k$, so by the last statement in the
proposition, this sequence $\rr b \rr a \rr b \rr {a'} \rr {b'} \rr a \rr b \rr {b'}$
of reflections defines a $\pi$ with the minimality
required to apply part (2). Hence this sequence demonstrates that 
(with $\theta$ chosen positive) the $D_4$ quiver variety
first listed is isomorphic to $T^* Fl(j,k;\ n)$.

\subsection{The geometry of figures \ref{fig:d1}--\ref{fig:d4}}
We can now fully explain the meaning of figures \ref{fig:d1}--\ref{fig:d4},
(excepting the arrows, which will come in \S\ref{sec:geominterp}).
In figure~$d$ we list a Dynkin diagram $X_{2d}$, 
a quiver variety $\calM(w_{(1)},v_{(1)})$ of flag type, a quiver variety $\calM(w_{(2)},v_{(2)})$
not usually of flag type but susceptible to proposition \ref{prop:reflect},
an ``intermediate'' quiver variety $\calM(w_{(1)}+w_{(2)},v_{(1)}+v_{(2)})$ usually
neither of flag type nor susceptible to proposition \ref{prop:reflect}
(except at $d=1$), and finally another quiver variety $\calM(w_{(3)},v_{(3)})$
not usually of flag type but susceptible to proposition \ref{prop:reflect}.
\junk{
Note that in all cases, the weight $\sum_i w^i \vec\omega_i - \sum_i v^i \vec\alpha_i$
is the same for $\calM(w_{(1)}+w_{(2)},v_{(1)}+v_{(2)})$ as for $\calM(w_{(3)},v_{(3)})$.
In each case the second weight is $\tau^2$ times the first,
so the third = fourth weight is $-\tau$
times the first weight. }
For each of the non-intermediate quiver varieties,
we give a reduced word for the $\pi$ 
used in proposition \ref{prop:reflect} to show that those quiver varieties are
just cotangent bundles.)

Each $\calM(w,v)$ carries an action of $\prod_i GL(w^i)$ changing basis on
the framed vertices, which we shrink to a maximal torus 
$T := \prod_i T^{w^i}$. There is also a commuting action of $\CC^\times$ that 
scales the ``backward'' maps (between gauged vertices against $I$'s orientation,
or downward from framed to gauged) and $\theta_\CC$
that one should think about differently: in terms of the complex sympletic
form naturally borne by $\calM(w,v)$, $T$ {\em preserves} the symplectic
form whereas the $\CC^\times$-action {\em scales} the form.
We put the tori together into
$$ \hat T := T \times \CC^\times $$

In \cite{Nakaj-quiv3}, Nakajima defines an action of $\gqg$ on the $K$-theory
$K_{\hat T}(\bigsqcup_v\calM(w,v))$ (which we will take with complex coefficients) of a quiver variety,
and identifies the latter with
the tensor product $\bigotimes_{k=1}^{n} V_{\omega_{m_k}}(z_k)$ of
certain fundamental representations, where $n=\sum_i w^i$,
$\#\{m_k=i\}=w^i$, and the $z_k$ (resp.\ $t$) are the equivariant
parameters associated to the Cartan torus $T\cong (\CC^\times)^n$
(resp.\ to the scaling action).
These were the representations mentioned at the start of \S \ref{sec:puzzle}.

\rem[gray]{is that true? is it only a conjecture?
  generic spectral parameters? Quoth Nakajima:
  ``We also show that $M_{x,a}$ is a finite dimensional l-highest
  weight module. As a usual argument for Verma modules, $M_{x,a}$ has the
  unique (nonzero) simple quotient. The author conjectures that $M_{x,a}$
  is a tensor product of l-fundamental representations in some
  order. This is proved when the parameter is generic in \S 14.1.''
  I think the being-a-tensor-product is provable nowadays using
  stable envelopes.
  The def of $l$-fundamental is hiding in remark 1.2.17 on p9.
}

In the present case, fixing the framed dimension vector to be of the form
of figure~$d$, quiver $(1)$, and summing over all gauged dimension vectors,
we obtain an action of $\xqg$ on $V_a(z_1)\otimes \cdots \otimes V_a(z_n)$,
recovering the representation theory discussed in \S \ref{ssec:reptheory}.

In particular, at $n=1$, for all one-dimensional weight spaces of
$V_a(z)$ (which includes all of $V_a^A(z)$, and in fact, is
all weight spaces except the zero weight space at $d=4$),
one has a natural basis given
by the stable envelope construction (where we always make the same
``canonical'' choice of chamber, polarization and line bundle as for
$\St^\lambda$ in proposition~\ref{prop:stab}).


\begin{lem}\label{lem:ressinglev2}
  Let $d=1,2,3,4$ and $a=1,2,3$, with $w_{(a)}$ the framed dimension vector
  on quiver $a$ of figure $d$.
  In particular $w_{(a)}$ has \fbox{$n$} at exactly one vertex.
  Define \defn{the $a$th single-number sector}\footnote{This
    denomination comes from the labels $0,\ldots,d$ of weight vectors
    of $V_a^A(z)$ which are single numbers, as opposed to other weight
    spaces of $V_a(z)$ which are traditionally labeled in Schubert
    puzzles by multinumbers.} for that quiver in that figure as the
  set of $v_{(a)}$ pictured, a $d$-parameter rather than $2d$-parameter set.

  \begin{enumerate}
  \item The $a$th single-number sector forms a face of the polytope from
    proposition \ref{prop:flagquiver} (3).
  \item If $v_{(a)}$ runs over a single-number sector, then
    \[
 K_{\hat T}\left(\bigsqcup_{v_{(a)}}\calM(w_{(a)},v_{(a)})\right) \iso \Otimes_{k=1}^n V_a^A(z_k)
    \]
  \item The subspace $V_a^A(z_1)\otimes V_a^A(z_2)$
    is invariant under the operator $\check R_{a,a}$.
  \item The matrix of $\check R_{a,a}$ restricted
    to $V_a^A(z_1)\otimes V_a^A(z_2)$, in the basis $(e_{a,i})_{i=0,\ldots,d}$,
    matches the $R$-matrix from \eqref{eq:Rsingle}. (Note the subtlety
    mentioned before the lemma, that not all weight spaces come with
    natural bases, but those in the single-number sector do.)
  \end{enumerate}
\end{lem}

\begin{proof}
  \begin{enumerate}
  \item This statement is obvious for $a=1$ -- the conditions on $v$
    are that certain entries $\langle v, \vec\alpha \rangle$, obviously
    bounded below by $0$, are in fact $0$. To see it for $a=2,3$
    we use the sequences of reflections in figure $d$ to rotate the $a=1$ face
    to the purported $a=2,3$ faces.
  \item For $a=1$, this is essentially an $A_d$ calculation, using
    \cite{Nakaj-quiv3} as above. For $a=2,3$ we use
    proposition~\ref{prop:reflect} (and the sequences of reflections
    provided in the end of each figure) to identify the varieties with
    those from the $a=1$ case. 
  \junk{  
    This is a consequence of the results of \S \ref{sec:stabclass} and
    \S \ref{sec:quiver}, and in particular of proposition~\ref{prop:reflect}.  
    When $v_{(a)}$ is constrained to be of the form given in
    figures~\ref{fig:d1}--\ref{fig:d4}, each $\calM(w_{(a)},v_{(a)})$ is a
    cotangent bundle of a $d$-step flag variety. 
    The first part of the lemma then follows in principle from the
    general theory of stable envelopes, cf.~the proof of
    proposition~\ref{prop:stab}.  We provide here an alternative proof
    based on the fact that $\W^A$ is a face of $\W$, cf.~the proof of
    \cite[proposition~3.4]{artic71}.
    \rem[red]{AK: complete; add the face argument at $d=4$}
  }
\item This follows from part (1) of the lemma and $\check R_{a,a}$'s
  $T$-equivariance.
\item By part (2), at $n=2$,
  $V_a^A(z_1)\otimes V_a^A(z_2)\cong K_{\hat T}(\bigsqcup_{v_{(a)}}\calM(w_{(a)},v_{(a)}))$,
  and the $R$-matrix then agrees with the geometrically
  defined $R$-matrix of \S \ref{sec:stabclass}, cf.~the proof of
  proposition~\ref{prop:stab}, up to normalization. Finally, one
  verifies the normalization $\check R^{00}_{00}=1$ in \eqref{eq:Rsingle}.
\end{enumerate}
\end{proof}

\rem[gray]{the fact that $V_2$ is dual of $V_3,V_1$ for $d\ne2$ is somewhat confusing: geometrically it's clear that both defining rep and its dual produce
a $T^*$ flag variety, so the $R$-matrices should be equal; but the reps are nonisomorphic (we break the symmetry in the $RTT$ formalism
when we choose the node for the auxiliary space)}

\section{Geometric interpretation of puzzles}
\label{sec:geominterp}

The proof of theorem \ref{thm:main} comes down to equivariant localization
and variants of YBE (found in proposition \ref{prop:ybe}). 
In particular, that proof does not make clear why one might {\em expect} 
$\xqg$ $R$-matrices to be of use when studying partial flag varieties,
and indeed we didn't know ``why'' they proved so useful in \cite{artic71}
at the time of writing. 
In this section we provide a retrodiction, deriving the puzzle rules
directly from geometry. (Unfortunately we only understand the geometry
well enough to work in cohomology, rather than $K$-theory.)
We emphasize that our principal results do not depend on those
of this section, and due to that, some of the proofs will be abbreviated.

\junk{probably requires some intro on what we're trying to achieve.
  proof above not purely geometric, combinatorial part. fill the gap
  in $H_{\hat T}$. AK: how's this. Please correct history as
  needed. PZJ: history's fine, but we still need some transition
  between the rest of the paper and this section}

It was first noted in \cite[proposition 1]{artic46} that the equivariant
puzzle rule of \cite{KT} is based on an $R$-matrix. Drinfel$'$d and Jimbo
showed that many $R$-matrices arise from
the representation theory of quantized loop algebras. Nakajima
\cite{Nakaj-quiv3} showed that many representations of quantized
loop algebras arise on the $K$-theory of quiver varieties,
and Varagnolo \cite{Varagnolo} gave the corresponding result in cohomology.%
\footnote{This is a rare case -- like Atiyah--Bott's equivariant
  localization formul\ae\ -- where the $K$-theory result predates the
  cohomology result (see the arXiv references).  Publication took
  place in the opposite order. We thank Sachin Gautam for setting us
  straight on the history.}  Maulik and Okounkov \cite{MO-qg}
interpreted $R$-matrices directly in cohomology, using their ``stable
envelope'' construction of certain Lagrangian relations between quiver
varieties. We review some of this latter work.
One novel feature is that we need to mix in some Lagrangian relations
other than stable envelopes; in particular for $d=1$ we need a
Hamiltonian reduction.

\newcommand\COMfld{{\text{\bf COMfld}}}
\newcommand\Inner{{\text{\bf Inner}}}
\renewcommand\Vec{{\text{\bf Vec}}}
\newcommand\calK{{\mathcal K}}

\subsection{The ``category'' of correspondences}
Let $\calK$ denote the ``category'' (terminology due to \cite{weinstein}) 
whose objects are compact oriented manifolds, with morphisms 
$Hom_\calK(A,B) := \{$oriented cycles $K$ in $A\times B\}$. 
It is not an actual category,
because we only define a composition $K_1 \star K_2$ of
$K_1 \in Hom_\calK(A,B)$ with $K_2 \in Hom_\calK(B,C)$ when
$K_1 \times C$ and $A \times K_2$ are transverse inside $A\times B\times C$.
The composite is defined as 
$K_1 \star K_2 := \pi_{AC}((K_1 \times C) \cap (A \times K_2))$,
where $\pi_{AC}$ is the projection $A\times B\times C \to A\times C$.
Later, when we allow $A,B,C$ to be noncompact, to define a composition
we will also require that $\pi_{AC}$ be proper on 
that intersection. Only under these two conditions do we say that
$K_1,K_2$ are \defn{composable}. (We will do something very weird
in proposition \ref{prop:rightsquare} 
and compose two ``non-composable'' relations.)

The cohomology functor factors through this ``category'', as follows. 
Consider $H^*(A) = H^*(A;\RR)$ as an inner product space, 
with pairing $\langle \alpha,\beta\rangle := \int_A (\alpha\cup \beta)$. 
Then $H^*$ is a functor from the category $\COMfld$ = (compact
oriented manifolds, smooth maps) to the category $\Inner$ of real
inner product spaces.
This $\Inner$ is no different from $\Vec$ as a category, but is endowed
with a contravariant endofunctor `transpose'. We can be ambiguous
about whether $H^*$ is covariant or contravariant, thanks to this transpose.
The ``category'' $\calK$ has an obvious transpose as well, unlike $\COMfld$.

Now we factor $H^*$ as
$$ 
\begin{array}{rcccl}
\COMfld &\xrightarrow{\mathop{graph}} &
\calK  &\xrightarrow{\bullet-\text{transform}}\quad & \Inner \\[.5em]
A&\mapsto&A &\mapsto& H^*(A) \\
(f:A\to B)&\mapsto&graph(f) \\[.2em]
&& K\subseteq A\times B &\mapsto& \Upsilon_K 
:= (\alpha \mapsto P.D.\left( (\pi_B)_*( \pi_A^*(\alpha) \cap [K]) \right))
\end{array}
$$
The principal results to know, at this level of generality, are 
(1) when two correspondences are composable, 
the $(K_1\star K_2)$-transform is the composite of the two
individual transforms, and
(2) the $graph(f)$-transform is the pushforward $f_*$ in cohomology,
whose transpose is the pullback $f^*$.
In particular, the composite of the two functors above is $H^*$.

\junk{
  Our interest is actually in noncompact manifolds (more specifically,
  quiver varieties) and the ``category'' $\calK$ generalizes just
  fine to those as long as we continue to ask, when composing $K_1$ and $K_2$, 
  that the map $\pi_{AC} : (K_1 \times C) \cap (A \times K_2) \to A\times C$ 
  be proper.
}

\subsection{Weinstein's ``category'' of symplectic manifolds}
\label{ssec:weinstein}
\newcommand\calC{{\mathcal C}}
Our actual interest is in the ``category'' $\calC$ of (holomorphic)
symplectic manifolds and Lagrangian correspondences, i.e.
$$ Hom_\calC(M,N) \ :=\  \{L \subseteq M\times N\ 
:\ L\text{ is a Lagrangian cycle in } (-M) \times N\} $$
where $(-M)$ denotes $M$ with the symplectic form negated. 
\junk{
  Two morphisms
  $M \xrightarrow L N \xrightarrow{L'} P$ are composable only when
  $L\times P$ intersects $M\times L'$ transversely inside $M\times N\times P$. 
  In that case, the composition (denoted $L' \circ L$) is the projection
  to $M\times P$ of $(L\times P) \cap (M\times L')$. }
This ``category'' was introduced in \cite{weinstein} and has seen much
development since then, e.g. \cite{wehrheimwoodward,weinstein2,weinstein3}.
Note that it again enjoys a transpose.

The dimension of $L$ is the average of the dimensions of $M$ and $N$.
When $\dim M \geq \dim N$, we might expect the projection $L\to M$ to
be an immersion and $L\to N$ to be a submersion; when these hold
we call the Lagrangian relation $L \subseteq M\times N$ a \defn{reduction}.
The transpose of such an $L$ is called a \defn{co-reduction}.
\junk{
Using the trivial bijection
$$ Hom_\calC(M,N) \iso Hom_\calC(N,M) $$
based on $M\times N \iso N\times M$, we can reverse the conditions, in
which case $L$ is 
So $N$ carries two bundles:
  the subbundle $ker\ T(L\to N) \leq TL$ and the normal bundle $N_L N$.
  Is there a relation between them, in our examples?
  I guess in symplectic reduction the normal bundle is always trivial,
  and the subbundle is trivial, at least for abelian group actions.
  The normal bundle to $N$ in $M\times L$ is $T^* N$, but I'm not sure
  how that relates to these two bundles.}

We mention five examples of Lagrangian relations, four of which date from
the introduction of this ``category'' and one of which is much more recent.
Each but the fourth is a reduction or a co-reduction.
\begin{enumerate}
\item If $\phi:M\to N$ is a symplectomorphism, then 
  $\mathop{graph}(\phi) \in Hom_\calC(M,N)$.
\item Elements $L\in Hom_\calC(pt,N)$ are simply Lagrangian cycles in $N$.
\item If $G$ acts on $M$ Hamiltonianly with moment map $\Phi_G:M\to \lie{g}^*$,
  and $N = \Phi^{-1}(c)//G$ is the GIT quotient of a central level set
  (i.e. $c\in (\lie{g}^*)^G$), then $\Phi^{-1}(c) \in Hom_\calC(M,N)$.
  This is essentially the Marsden-Weinstein theorem.
\item If $R \subseteq X \times Y$ is a submanifold of a product, 
  e.g. the graph of a function,   then its conormal bundle
  $CR \subseteq T^*(X\times Y) \iso T^*X\times T^*Y$
  gives an element of $Hom_\calC(T^*X, T^*Y)$.
\item If $\CC^\times \actson M$ symplectically, with $F \subseteq M^{\CC^\times}$
  a fixed-point component, then
  $$ \attr(F) := \{m\in M\ :\ \lim_{z\to 0}\ z\cdot m \in F \} $$
  has obvious maps $\attr(F) \into M$, 
  $\attr(F) \xrightarrow{\lim_{z\to 0} z\cdot} F$ which together give an 
  inclusion $attr(F) \into M\times F$, whose image turns out to be Lagrangian. 
  In the rare occurrence that $\attr(F)$ is {\em closed} in $M$, it
  defines an element of $Hom_\calC(M,F)$.
\end{enumerate}

A basic example of the last construction\footnote{Intriguingly,
  it is also an example of the third construction, if one considers the
  symplectic reduction of $G\cdot\lambda$ by the maximal unipotent
  subgroup of $G$. The zero level set is $\coprod_W NwT/T \subseteq G/T$.
}
has $M = G \cdot \lambda \subseteq \lie{g}^*$, a generic coadjoint orbit
of the complex group, so $M \iso G/T$ as a homogeneous space.
(Don't confuse this with the projective variety $G/B$; $G/T$ is a bundle
over $G/B$ with contractible fibers, but no holomorphic sections $G/B \to G/T$,
since $G/T \iso G\cdot \lambda$ is affine.)
If the $\CC^\times$-action on this $M$ is by a regular dominant coweight, 
then the fixed points are the finite set $N(T)/T$, and for each $w\in W$
we have $\attr(wT/T) = BwT/T$. Since $B\leq G$ is a closed subgroup 
we know that $B/T \subseteq G/T$ is closed. Then since the {\em right}
action of $W$ on $G/T$ transitively permutes the submanifolds $\{BwT/T\}$,
each of those submanifolds is closed, and of the same (Lagrangian) dimension. 
This situation is in strong contrast to the $B$-orbits on $G/B$!

To compare the graph of $f:Y\to X$ and its conormal bundle, 
as correspondences, we observe the following equation in the ``category'':

\begin{lem}\label{lem:zerosection}
  Let $f:Y\to X$ be a smooth map of compact manifolds, 
  and $\iota_X,\iota_Y$ the inclusions of $X,Y$ respectively into their
  cotangent bundles, as the zero sections. The following square of
  correspondences commutes (in particular, both compositions exist):
  $$
  \begin{matrix}
    &T^*X &\xrightarrow{Cgraph(f)^T}& T^* Y \\
    graph(\iota_X)&\uparrow &=& \uparrow &graph(\iota_Y) \\
    &X &\xrightarrow{graph(f)^T}& Y 
  \end{matrix}
  $$
\end{lem}

\begin{proof}
  Three of these correspondences are easy. To get a hold of the
  conormal bundle atop, start with the map $f\times Id: Y \to X\times Y$,
  $y \mapsto (f(y),y)$, a diffeomorphism to the graph. 
  Its derivative at $y$ is $\vec a \mapsto (T_yf(\vec a), \vec a)$
  with dual $(\vec a,\vec b) \mapsto T_y^*f(\vec a) + \vec b$. 
  The conormal space at $y$ that we seek is the kernel
  $\{ (\vec a, -T^*f(\vec a))\ :\ \vec a\in T^*_{f(y)} X \}$ of that dual map. 
  In all
  $$ Cgraph(f)^T 
  = \left\{ \left( (f(y),\vec a), (y, -T_y^*f(\vec a) \right) 
    \in T^* X \times T^* Y  \right\} $$
  Now we try to compose 
  $$ graph(\iota_X) = \left\{ (x, (x,\vec 0)) \in X \times T^* X \right\} $$
  with $Cgraph(f)^T$. The intersection is
  $$ (graph(\iota_X) \times T^* X) \cap (X \times Cgraph(f)^T)
  = \left\{ (f(y), (f(y), \vec 0), (y, \vec 0) \right\} $$
  whose projection to
  $\left\{ \left(f(y), (y, \vec 0)\right)\, :\, y\in Y \right\}$
  is a diffeomorphism, hence proper.

  Also we try to compose 
  $$ graph(f)^T = \left\{ (f(y),y)\ :\ y\in Y \right\} $$
  with $graph(\iota_Y)$, obtaining the intersection
  $$ (graph(f)^T \times T^* Y) \cap (X \times graph(\iota_Y))
  = \left\{ \left(f(y), y, (y, \vec 0)\right)\ :\ y\in Y \right\} $$
  whose projection to
  $\left\{ \left(f(y), (y, \vec 0)\right)\, :\, y\in Y \right\}$
  is again a diffeomorphism, hence proper.
\end{proof}

We would like to infer a result on cohomology from that lemma, 
but 
our definition of $K$-transform for $K\subseteq M\times N$
involves integrating over the fibers of $M\times N \to N$,
so requires $M$ compact. We sidestep this in the next section.

When $X \subseteq M$ is a closed and irreducible but possibly singular
subvariety of a smooth variety, its \defn{conormal variety} $CX$
is by definition the closure of the conormal bundle to its
smooth part $X_{reg}$. For use later we bring up a characterization,
tracing to Monge (see \cite[p169]{KleimanTangency}), of
these subvarieties:

\begin{lem}\label{lem:reflex}\cite[Theorem 1.10]{TevelevDualityHomogeneous}
  A reduced closed subscheme $Y \subseteq T^*M$ is a conormal variety $CX$
  iff it is \defn{conical} (invariant under scaling the cotangent fibers),
  irreducible, and Lagrangian. In this case $X$ can be computed from $Y$
  either as the intersection of $Y$ with the zero section in $T^*M$
  (hence closed), or as the image of the projection $Y \into T^*M \onto M$
  (hence irreducible).
\end{lem}

\subsection{Correspondences and (equivariant) cohomology,
  and dividing by the zero section}
\label{sssec:divzero}
In equivariant cohomology, there is a handy trick: once one ``localizes''
to the fraction field $frac\ H^*_T$, to define our transforms
it is enough for $M^T$ to be compact,
since $$frac\ H^*_T \tensor_{H^*_T} H^*_T(M) 
\iso frac\ H^*_T \tensor_{H^*_T} H^*_T(M^T) 
\iso frac\ H^*_T \tensor_{\ZZ} H^*(M^T) 
= frac\ H^*_T \tensor_{\ZZ} H_c^*(M^T).$$
Stated more baldly, every class on $M$ is a $(frac\ H^*_T)$-linear combination
of classes on $M^T$, and by expressing classes that way we can define the
application of the $\beta$-transforms $\Upsilon_\beta$ to them.
This is more than just a trick for computation: it can happen that
$\Upsilon_\beta$ applied to a class in $H^*_T(M)$ does not lie in $H^*_T(N)$,
but lies only in $(frac\ H^*_T) \tensor_{H^*_T} H^*_T(N)$.

Consider now the situation of lemma~\ref{lem:zerosection}, with $X,Y$
compact complex. 
Their cotangent bundles carry $\CC^\times$-actions with compact fixed points, 
so as just explained lemma~\ref{lem:zerosection} gives us a commuting square
$$
\begin{array}{ccc}
  H^*_{\CC^\times}(T^* X) 
& \xrightarrow{\ \ \Upsilon_{Cgraph(f)^T}\ \ } & H^*_{\CC^\times}(T^* Y)\\
  \uparrow & & \uparrow \\
  H^*_{\CC^\times}(X) & \xrightarrow{\qquad f^*\qquad}{} & H^*_{\CC^\times}(Y) 
\end{array}
\qquad\qquad\text{all tensored with $frac\ H^*_{\CC^\times}$}
$$
where each vertical map is the degree-shifting map
``pushforward in cohomology'' $\iota_*$ along the inclusion $\iota$ 
of the zero section. Compared to / composed with the isomorphism $\iota^*$,
this $\iota_*$ amounts to multiplying by the equivariant Euler class 
$$ e(T^*X) \in H^*_{\CC^\times}(X) = H^*(X)[\hbar]
$$
(for the left vertical map; replace $X$ by $Y$ for the right vertical map).

Before we analyze this class, we emphasize the differences between
the pullback and pushforward maps induced from $X \into T^*X$.
The pullback (on ordinary, equivariant, or localized equivariant cohomology)
is a graded ring isomorphism, whereas the pushforward is an 
isomorphism only on localized equivariant cohomology, and only as an
$H^*_{\CC^\times}$-module. The pushforward of $1$ is the class
$[X \subseteq T^* X]$ of the zero section, whose pullback is $e(T^* X)$.
In particular either composite is multiplication by this class
(in one of its guises). 

Any complex vector bundle $V$ on $M$ can be regarded as a
$\CC^\times$-equivariant vector bundle by the scaling action,
giving an equivariant Euler class
$e(V) \in H^*_{\CC^\times}(M) = H^*(M)[\hbar]$. Its dehomogenization
$e(V)|_{\hbar\to 1}$ is the total Chern class $c(V)$ (we thank Shaun Martin
for this point of view, which nicely retrodicts Stiefel--Whitney and
Pontrjagin classes as well. The proof is pretty immediate from the
usual characterizations of Chern classes).
Note that this is not the dehomogenization we needed in \S\ref{ssec:SSM}). 
Since the total Chern class is $1$ + nilpotent, the vertical maps in
the square above become isomorphisms already upon inverting $\hbar$,
much less tensoring with $frac\ H^*_{\CC^\times}$. (In practice we work with 
not just $\CC^\times$ but $(T\times \CC^\times)$-equivariant cohomology, 
in which case the total Chern class is only invertible
after more fully localizing.)

In particular, from a known $Cgraph(f)^T$-transform
$\rho \mapsto \Upsilon_\beta(\rho)$, we can infer thereby that
$f^*\big(\rho/[X \subseteq T^*X]\big) 
= \Upsilon_\beta(\rho) \, /\, [Y \subseteq T^*Y]$.
This result is what motivates the denominator in our SSM classes.

\junk{This is maybe the place to do the calculation of the conormal
  bundle $\{(F,X_1,X_2,X_3)\ :\ (F,X_i) \in T^*G/P, \, X_1+X_2+X_3 = 0 \}$
  to the graph of the diagonal inclusion, and to state the
  multiplication formula as derived from that convolution
}

Hereafter our calculations happen to be in Weinstein's symplectic ``category'',
though in fact we nowhere use the symplectic or Lagrangian structure.

\subsection{Stable envelopes}\label{ssec:envelopes}
In \cite[Theorem 3.7.4]{MO-qg} Maulik and Okounkov extend the $\attr$
construction from \S \ref{ssec:weinstein} as follows, perhaps inspired
by the following phenomenon.

\begin{lem}\label{lem:attrclosed}
  Let $V$ be a linear representation of $\CC^\times$, and $X\into V$ a closed
  $\CC^\times$-invariant subscheme.
  Then for each component $F \subseteq X^{\CC^\times}$ of $X$'s fixed-point set,
  the attracting set $attr(F \subseteq X)$ is closed.
\end{lem}

\begin{proof}
  Break $V = V_- \oplus V_0 \oplus V_+$ into its negative, zero, and
  positive $\CC^\times$-weight spaces. Then $V^{\CC^\times} = V_0$ and
  $X^{\CC^\times} = X \cap V_0$, and
  \begin{eqnarray*}
    attr(F\subseteq X) 
    &=& \{\vec v \in V\ :\ \vec v\in X,\ \lim_{z\to 0} z\cdot \vec v \in F \} \\
    &=& \{\vec v \in V\ :\ \vec v\in X\}\ \cap\ \{\vec v \in V\ :\ \lim_{z\to 0} z\cdot \vec v \in F \} \quad
    = X \cap (F + V_+)
  \end{eqnarray*}
  where $F + V_+ = \{f + \vec v\ :\ f\in F, \ \vec v\in V_+ \}$ is closed,
  being the preimage of $F$ along the projection $V_0 \oplus V_+ \onto V_0$. 
  Since $attr(F\subseteq X)$ is the intersection of two closed sets, 
  it is closed.
\end{proof}

If $E \to S$ is a flat family with a fiberwise $\CC^\times$-action, 
call it a \defn{deformation to affine} if the fibers over a dense open set 
$S^\circ \subseteq S$ are subvarieties of affine space.
Then given a flat subfamily $F \subseteq E^{\CC^\times}$, 
define the \defn{stable envelope} 
$\overline{attr(F\subseteq E)} \subseteq F \times E$
as the closure of the image of the injection
$$ attr(F\subseteq E) \to F\times E, \quad
m \mapsto \left( \lim_{z\to 0} z\cdot m, \ m \right) 
$$
(Note that it is {\em not} defined as the closure of $attr(F\subseteq E)$
inside $E$, and indeed, the composite
$\overline{attr(F\subseteq E)} \into F \times E \onto E$
is typically not injective.)

By lemma~\ref{lem:attrclosed},
$\overline{attr(F\subseteq E)} \subseteq F \times E$
agrees with $attr(F\subseteq E)$ over $S^\circ$
(nothing is added in the closure), 
but inside other fibers $E|_{s\notin S^\circ}$ one may have strict containment
$\overline{attr(F\subseteq E)}\, |_s\supset\overline{attr(F|_s\subseteq E|_s)}$.
\junk{
  For proposition \ref{prop:stable} below we characterize the closure:
  \begin{equation}
    \label{eq:stableDef}
    \overline{attr(F\subseteq E)} = 
    \{ (f,e)\ :\ \exists (e_i)_{i\in \NN},\ \lim_{i\to\infty} e_i = e,
    \lim_{i\to \infty} \lim_{z\to 0} z\cdot e_i = f \} 
  \end{equation}
}
Given a non-affine fiber $E|_o$ of such a family, one can ask whether that 
fiber $\overline{attr(F\subseteq E)}\,|_o$ of the stable envelope depends on
the choice of deformation $E$ to affine; in fact it does not, as the
stable envelope in a fiber can be characterized as in \cite[\S3.3]{MO-qg}.
(Indeed, their definition of ``stable envelope'' is for one fiber
at a time, without a deformation to affine assumed, and in particular
is more general than the definition here.)
This independence matters little for quiver varieties and more generally for
symplectic resolutions $M$, as by \cite{Kaledin} any such $M$ is the central
fiber of a canonical deformation to affine, with $T$-diffeomorphic fibers.

To continue the basic example\footnote{%
  The noncommutative analogue may be more familiar, in which $T^*G/B$
  deforms to the central-character-zero algebra $(U\lie{g})_0$,
  whose Verma modules are very complicated. This deforms in an independent
  direction to $(U\lie{g})_\lambda$ with generic central character,
  whose Verma modules are irreducible and bear no $Ext$s with one another.}
from \S\ref{ssec:weinstein}, consider the Springer resolution
$M = T^* G/B$, and its Grothendieck--Springer deformation to 
$M_{def} = G\cdot \lambda$.
We can, for example, fix a regular $\lambda$ and consider the one-parameter
subfamily whose fibers are $G\cdot z\lambda$ for $z\neq 0$, $T^* G/B$ for $z=0$.
Note that this family, and its zero fiber $M$, possess a $\CC^\times$ action
(scaling of $\lie{g}^*$) that the general fiber $G\cdot \lambda$ does not.
This action is what we use to degenerate the attracting set in $M_{def}$
to the stable envelope in $M$.  Thanks to lemma~\ref{lem:attrclosed}
each attracting set $BwT/T$ in $G\cdot \lambda$ is closed, and we
obtain the stable envelope as $\lim_{z\to 0} z\cdot (BwT/T) \subseteq T^* G/B$.

Perhaps the most crucial property of stable envelopes is their
following functoriality:

\begin{prop}[{restatement of \cite[Lemma 3.6.1]{MO-qg}}]\label{prop:stable}
  Let $A_1,A_2:\CC^\times \to Aut(E\to S)$ be two commuting fiberwise-actions
  on a deformation-to-affine $E\to S$. Let $F_2$ be a component of
  $E^{\langle A_1, A_2\rangle}$, and $F_1 \subseteq E^{A_1}$ the component
  containing $F_1$. Then for $N \gg 0$, the triangle of relations
  $$
  \begin{matrix}
    & F_2 \phantom{\longrightarrow}& \xrightarrow{env(A_{1+})} & 
    \phantom{\longrightarrow} E \\
    & \scriptstyle{env(A_2)}\ \searrow && \nearrow\scriptstyle{env(A_1)} \\
     && F_1 && 
  \end{matrix}
  $$
  commutes, where $A_{1+}(z) = A_1(z^N) A_2(z)$.
\end{prop}

(In \cite{MO-qg} they only state the corresponding result in cohomology,
but their proof is actually at the level of convolution of cycles,
as stated here.)


\subsection{The flag type case}\label{ssec:grothspringer}
The deformation-to-affine of Nakajima quiver varieties is easy to describe;
this is the variation of complex moment $\theta_\CC$ from \S\ref{sec:quiver}.
\junk{
  At each of the gauge vertices $v$, we ordinarily ask that the moment map
  $ \sum_{w \to v} \phi_{wv} \circ \phi_{vw} 
  - \sum_{v \to w} \phi_{vw} \circ \phi_{wv}$ vanish; in the affine deformation,
  we instead ask the moment map to equal a scalar $\varepsilon_v$ 
  (with independent scalars for different gauge vertices). 
}
In the flag type case 
this deformation of
quiver varieties is exactly the Grothendieck--Springer deformation of
$T^*(GL_n/P)$ to $GL_n/L$ ($L$ a Levi subgroup of $P$),
generalizing the $P=B$ example from \S\ref{ssec:envelopes}.
Specifically, if we put a scalar $\varepsilon_i$ at the $n_{i}$ gauge vertex of
$
\begin{matrix}  
  \begin{tikzpicture}[script math mode,baseline=0]
    \node[gauged] at (1,0) (v1) {n_d}; 
    \coordinate (v2) at (2,0); 
    \coordinate (v3) at (3,0);
    \node[gauged] at (4,0) (v4) {n_2}; 
    \node[gauged] at (5,0) (v5) {n_1}; 
    \draw (v1) -- (v2);
    \draw[dotted] (v2) -- (v3);
    \draw (v3) -- (v4) -- (v5);
    \node[framed] at (1,1) (w1) {n}; \draw (v1) -- (w1);
  \end{tikzpicture}
\end{matrix}
$
then our deformation is to the space
$$ \{(X \in End(\CC^n), 
0 = V_0 \leq V_1 \leq V_2 \leq \ldots \leq V_d \leq V_{d+1} =\CC^n)\ :\
(X-\varepsilon_i {\bf 1})V_{i+1} \leq V_{i} \} $$
where $\dim V_i = n_i$ and $\varepsilon_{d+1} := 0$.
In particular $X$ has spectrum 
$\varepsilon_1^{(n_1)} \varepsilon_2^{(n_2-n_1)}
\ldots
\varepsilon_{d}^{(n_d-n_{d-1})}
0^{(n-n_d)} 
$
and if these $d+1$ values are distinct, then each subspace $V_i$
can be recovered as a sum of eigenspaces of $X$. Consequently, the
projection to $End(\CC^n)$ is an affine embedding, with image the space
$\calO_\varepsilon$ of semisimple operators $X$ with this specified spectrum.

The $T$-fixed points on $\calO_\varepsilon$ are just the diagonal matrices
$D \in \calO_\varepsilon$, of which there are 
$n\choose n_1,\ n_2-n_1,\ldots,n-n_d$.
With respect to the action of a regular dominant coweight, 
the attracting set $\attr(D)$ is $\{D$ plus block strictly upper
triangular matrices$\}$, with blocks of size $n_1,n_2-n_1,\ldots,n-n_d$.  
In particular each $\attr(D)$ is closed (as predicted 
by lemma~\ref{lem:attrclosed}).

\subsection{$d=1$}
In the $d=1$ case from figure \ref{fig:d1}, we give here full detail on how
to derive (the cohomology, not $K$-theory, version of) theorem \ref{thm:main}
from quiver variety geometry. The exposition is easier in this case,
in that the $X_{2d}$ quiver variety for $d=1$ is itself of flag type.

\subsubsection{Reviewing the representation theory.}
\label{sssec:reps}
The $T^n$-equivariant cohomology of the quiver scheme 
$\begin{matrix}
  \tikz[script math mode,baseline=0]{\dOne{*}{*} 
    \node[framed] at (1,1) (w1) {n}; \draw (v1) -- (w1);}
\end{matrix}
$ is a representation of $\Uq(\lie{sl}_3[z])$ \cite{Varagnolo}.
The $T^n$ fixed points form
$
\begin{pmatrix}
  \tikz[script math mode,baseline=0]{\dOne{*}{*} 
    \node[framed] at (1,1) (w1) {1}; \draw (v1) -- (w1);}
\end{pmatrix}^n
$
whose cohomology bears the representation $\tensor_{i=1}^n \CC^3(x_i)$,
a tensor product of evaluation representations.
As with $H_{T^n}^*$ of any space, when we tensor with $frac\ H^*_{T^n}$
(passing to the generic point in the space of equivariant parameters)
we can identify the equivariant cohomology of the whole with
that of the fixed points.
In particular, the equivariant parameters from $H^*_T(pt)$ enter as
the evaluation parameters, and in turn are the spectral parameters in
the (rational) $R$-matrices.
\junk{
  Then we tensor two of those together
  (though our interest is only in two of the weight spaces, coming from
  $\tikz[script math mode,baseline=0]{\dOne{j}{0} 
    \node[framed] at (1,1) (w1) {n}; \draw (v1) -- (w1);}
  $
  and
  $\tikz[script math mode,baseline=0]{\dOne{n}{j} 
    \node[framed] at (1,1) (w1) {n}; \draw (v1) -- (w1);}
  $
  ).
}

In \cite[\S 14.1]{Nakaj-quiv3} it is conjectured that this representation
is a tensor product even at nongeneric values. In the next section
we implicitly assume the truth of this for {\em motivating} the geometry
we pursue, but we don't actually make use of it in the eventual
cohomological calculations.

The subtle step to come is in constructing the map 
$\CC^3(x_i) \tensor \CC^3(y_i) \to Alt^2 \CC^3(z_i)$,
as that tensor product is irreducible for generic values of $(x_i),(y_i)$.
Indeed, this map will only exist when
\begin{equation}\label{eq:specialtorus}
  y_i=x_i+\hbar, \qquad z_i = x_i+\hbar/2
\end{equation}
(which is the rational analogue of lemma~\ref{lem:factorR} at $d=1$).
The value of $z_i$ is not so important because the $\hbar$-shift can be absorbed in the geometric construction.
The equality $y_i=x_i+\hbar$ however, is meaningful:
specialization of equivariant parameters, $H^*_T \onto H^*_S$, is equivalent
to passage to a subtorus $S\into T$. This $S$ will leave invariant more
possible equations on our quiver varieties and motivates one of our choices
in the next section.

This type of geometric fusion was first developed in \cite{artic64}, though only in type $A$ --
which in the present context encompasses the $d=1$ construction.

\subsubsection{Discovering the quiver geometry} 
\label{sssec:discovering}

Being of flag type, the quiver variety 
$\tikz[script math mode,baseline=0]{\dOne{n+j}{j} 
  \node[framed] at (1,1) (w1) {2n}; \draw (v1) -- (w1);}
$
has a well-known description in Springer co\"ordinates as
$$
\left\{ \left(X \in End(\CC^{2n}), \CC^{2n} \geq V^{n+j} \geq W^j \right)\ :\
  \CC^{2n} \xrightarrow{X} V \xrightarrow{X} W \xrightarrow{X} 0 \right\} $$
where $X$ is the composite $\fbox{$2n$} \to \ovalbox{$n+j$} \into \fbox{$2n$}$,
and $V,W$ are the images of $\ovalbox{$n+j$} \into \fbox{$2n$}$ and
$\ovalbox{$j$} \into \ovalbox{$n+j$} \into \fbox{$2n$}$. The containments
$X(V) \leq W$ etc. follow from the moment map conditions defining the
quiver variety.

We have two representation-theoretic
maps to study. The first, and simpler, one joins the two
representations $\tensor_{i=1}^n \CC^3(a_i)$ and $\tensor_{i=1}^n \CC^3(b_i)$
into their tensor product and on the quiver scheme level, is achieved
using a stable envelope we compute now.

Denote by $Z$ the circle $z\mapsto \diag
\begin{pmatrix} \overbrace{1,\ldots,1}^n, \overbrace{z,\ldots,z}^n\end{pmatrix}
 \in GL(2n)$
acting on the quiver variety above. A triple $(X,V,W)$ is $Z$-fixed if
$X$ is block diagonal (with two blocks each of size $n$) and $V,W$
are graded subspaces w.r.t. the splitting $\CC^{2n} = \CC^n \oplus \CC^n$.
The fixed point set of this $Z$-action is not connected;
rather,  the components are distinguished by
the statistics
$\dim(V\cap (\CC^n\oplus 0)), \dim(W\cap (\CC^n\oplus 0))$.

Our interest is in the component
$\tikz[script math mode,baseline=0]{\dOne{n}{j} 
  \node[framed] at (1,1) (w1) {n}; \draw (v1) -- (w1);}
\times
\tikz[script math mode,baseline=0]{\dOne{j}{0} 
  \node[framed] at (1,1) (w1) {n}; \draw (v1) -- (w1);}
$
where those statistics are maximized. 
Its $Z$-attracting set is
$$ \left\{ \left(X = \begin{pmatrix} A & B \\ 0 & C \end{pmatrix},
    V \geq (\CC^n \oplus 0) \geq W\right) \right\} $$
which is {\em already closed} and hence is a component $L_1$
of a stable envelope, $env(Z)$. In the Springer co\"ordinates of the two spaces,
$L_1 \subseteq env(Z)$ is the following Lagrangian coreduction:
\begin{gather*}
L_1 := \left\{ \left(
    (A,V',D,W'),
    \left(X,V,W\right)\right)
  \ :\
  X = \begin{pmatrix} A & * \\ 0 & D \end{pmatrix}, 
  \begin{array}[c]{ccc} V &=& \CC^n \oplus V'\\
    W &=& W'\oplus 0  \end{array}
  \right\} 
\\
\begin{matrix}
  \swarrow &&\searrow \\ \\
  \{(A \in End(\CC^n), V'^{\ j})\} \times  \{(D \in End(\CC^n), W'^{\ j})\}
  &\qquad\qquad&
  \left\{ \left(X \in End(\CC^{2n}), V^{n+j}, W^j \right) \right\}
\end{matrix}
\end{gather*}
The image of the Southeast map, an inclusion, is the attracting locus
determined above. The Southwest map fails to be an inclusion because
of the unspecified ``$*$''. 
\junk{make sure the $A,D$ are matched with the right subspaces. This'll
  likely require being clear on row vs. column vectors
}
Since $L_1$ is conical, lemma \ref{lem:reflex} applies,
saying $L_1$ is the conormal variety to the graph of
\begin{eqnarray*}
  Gr(j;\, \CC^n)^2 &\into& Fl(j,n+j;\, \CC^{2n}) \\
  (V',W') &\mapsto& (W'\oplus 0,\ \CC^n \oplus V')
\end{eqnarray*}
but we won't make use of this description.

Our second map of representations is
$\Tensor_{i=1}^n \CC^3(x_i) \tensor \Tensor_{i=1}^n \CC^3(x_i+\hbar)
\to \Tensor_{i=1}^n (Alt^2 \CC^3)\left(x_i+\frac{\hbar}{2}\right)$,
which we want to induce (in a certain weight) using a Lagrangian from
$\tikz[script math mode,baseline=0]
{\dOne{n+j}{j} \node[framed] at (1,1) (w1) {2n}; \draw (v1) -- (w1);}
$
to
$
\tikz[script math mode,baseline=0]
{\dOne{j}{j} \node[framed] at (2,1) (w2) {n}; \draw (v2) -- (w2);}
$.
The dimension of the first quiver variety is $2n^2$
more than that of the second (selected to have the same weight, i.e.
$2n\omega_1 - (n+j)\alpha_1 - j\alpha_2 = n\omega_2 - j\alpha_1 -j\alpha_2$).
Hence if we hope for our Lagrangian correspondence to be a reduction, 
we should impose $n^2$ many conditions on that
first quiver variety and submerse the resulting submanifold onto the second
quiver variety. In the next few paragraphs we motivate what will
be our choice of the $n^2$ many equations. 

Imposing equations on a scheme is the same as imposing them on the 
affinization, which in this case is
a space of nilpotent matrices $X$.
The torus acting on this space is the $2n$-torus with weights
$(x_i),(y_i)$ on $\CC^n\oplus \CC^n$, plus $\hbar$ from the circle that
acts by dilation on the cotangent fibers, or acts by scaling on these matrices.
Hence the weights of the $(2n)^2$ matrix entries are
$$
\begin{bmatrix}
  &\hbar + x_i - x_j & \hbar+x_i - y_j &\\
  &\hbar+y_i-x_j & \hbar+y_i - y_j &
\end{bmatrix} \qquad i,j \in [n]
$$
Recall from equation \eqref{eq:specialtorus} that to intertwine
our $\Uq(\lie{sl}_3[z])$-representations, one needs to specialize these
evaluation parameters as $y_i = x_i+\hbar$. That would make the weights now
$$
\begin{bmatrix}
  &\hbar + x_i - x_j & x_i - x_j  &\\
  &2\hbar+x_i-x_j & \hbar+x_i - x_j &
\end{bmatrix} \qquad i,j \in [n]
$$
with the intriguing consequence that the $n$ inhomogeneous linear equations
$m_{i,n+i} = 1$ become invariant.
\junk{In particular, since these weights are distinct,
  the only possible torus-invariant linear equations 
  would involve setting one of these variables to zero.

  Specialization of equivariant parameters $H^*_T \onto H^*_S$ corresponds
  to passage to a subtorus $S\leq T$.
  That allows for more possibilities for our $n^2$ equations --
  they need only be invariant under the
  subtorus dual to the $\hbar + x_i - y_i = 0$ specialization
  (as was determined in equation \eqref{eq:specialtorus}).
  These are exactly the weights of the $n$ entries along the diagonal 
  of the Northeast quadrant. }

We actually want to impose $n^2$ conditions, not just $n$,
but those $n$ conditions just found suggest the following: define the
submanifold $L_2$ by the matrix statement that the entire Northeast
quadrant be the identity matrix.

To make this $L_2$ into a reduction, we need to submerse it onto the 
$
\tikz[script math mode,baseline=0]
{\dOne{j}{j} \node[framed] at (2,1) (w2) {n}; \draw (v2) -- (w2);}
$
quiver variety, with fibers the null foliation of the presymplectic form
on $L_2$ restricted from the ambient quiver variety. 
In fact we are in especial luck (that will run out after $d=1$): 
$L_2$ is a {\em level set of a moment map,}
for the action of the abelian unipotent group
$U := \left\{  \begin{bmatrix}    I_n & 0 \\ * & I_n  \end{bmatrix}\right\}$.
By the Marsden--Weinstein theorem, the leaves of the null foliation on the 
moment map level set are exactly the orbits of $U$. 
To mod out by $U$, we want to use $U$-invariant functions of
$\begin{bmatrix} A & I_n \\ C & D \end{bmatrix}$, and we use $A+D$.
Redefine $L_2$ as the {\em Hamiltonian} reduction
$$ 
L_2 := \left\{ \left((X,V,W),(Y,V'') \right)  \ :\ 
  \begin{array}[c]{c}
    X = \begin{pmatrix} A & Id \\ * & D \end{pmatrix}, \   Y = A+D\\
    \begin{array}{rcl}
      V \cap (0\oplus \CC^n) &=& 0 \oplus V'' \\
      W / (0\oplus \CC^n) &=& (W' + \CC^n)/(0\oplus \CC^n)
    \end{array}
  \end{array}
\right\}   
$$
$$
\begin{matrix}
  \swarrow &&\searrow \\
  \left\{ \left(X \in End(\CC^{2n}), V^{n+j}, W^j \right) \right\}
  &\qquad\qquad&
  \{(Y \in End(\CC^n), V''^{\ j})\} 
\end{matrix}
$$
\junk{
  $$
  L_2 := \left\{ \left( 
      \begin{bmatrix} A & I_n \\ C & D \end{bmatrix},\, A+D \right)\ :\ 
    \begin{array}{ccl}
      A(A-\varepsilon_1)(A-\varepsilon_2) \ = \ 
      D(D-\varepsilon_1)(D-\varepsilon_2) &=& 0\\
      (A+D-\varepsilon_1)(A+D-(\varepsilon_1+\varepsilon_2))(A+D-\varepsilon_2)&=&0 
    \end{array}   \right\}  \qquad
  $$
}
This group $U$ is normalized by the Levi subgroup $GL_n \oplus GL_n\leq GL_{2n}$,
and $U$'s moment map ``Northeast quadrant''
is therefore $(GL_n \oplus GL_n)$-equivariant.
However, the value $I_n$ is only invariant under the diagonal subgroup
$(GL_n)_\Delta$, which descends to give the action we expect on the
target of the Hamiltonian reduction.

One very important difference between $L_1$ and $L_2$ is that $L_1$ is conical.
The reduction at $0$ would keep the full $GL_n\oplus GL_n$ action.
\junk{
and in fact one can interpret $L_1$ this way, albeit for the
transpose subgroup to $U$. 
}

\junk{We could just as well use $-A-D$ here, which would fix the sign
  issue that comes in \S \ref{sssec:dOneverify}. The hard part is identifying
  the Hamiltonian reduction with the quiver variety, {\em symplectically}.
  To see a simplified version of the issue, 
  the conormal bundle to $M_\Delta$ in $M\times M$ is
  $\{(m,X,Y) : X+Y = 0\}$, which is isotopic to 
  $\{(m,X,Y) : X-Y = 0\}$ but the latter isn't Lagrangian.

  It might be easier to do if we frame the SE corner, and gauge it later:

$$
T^* Gr(n+j; \CC^{2n+k}) \quad\iso\quad
\tikz[script math mode,baseline=0]
{
  \node[gauged] at (1,0) (v1) {n+j}; 
  \node[framed] at (2,0) (v2) {k}; \draw (v1) -- (v2);
  \node[framed] at (1,1) (w2) {2n}; \draw (v1) -- (w2) ;
}
\quad //_{Id} U \quad\iso \quad
\tikz[script math mode,baseline=0]
{
  \node[gauged] at (1,0) (v1) {j}; 
  \node[framed] at (2,0) (v2) {k}; \draw (v1) -- (v2);
  \node[framed] at (2,1) (w2) {n}; \draw (v2) -- (w2);
}
\quad\iso\quad T^* Gr(j;\CC^k) \times T^* M_{kn}
$$
}

\begin{prop}\label{prop:composite}
  The two Lagrangian correspondences $L_1,L_2$
  $$
  
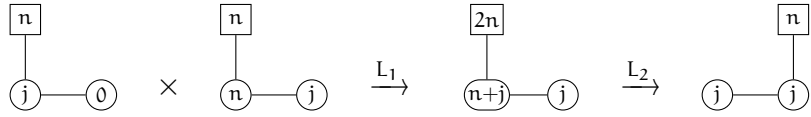
\tikz[script math mode,baseline=0]{\dOne{j}{0} 
    \node[framed] at (1,1) (w1) {n}; \draw (v1) -- (w1);}
  \quad\times\quad
  \tikz[script math mode,baseline=0]{\dOne{n}{j} 
    \node[framed] at (1,1) (w1) {n}; \draw (v1) -- (w1);}
  \quad \xrightarrow{L_1}
  \phantom{   \prod\limits_{i=1}^{n} }
  \tikz[script math mode,baseline=0]{\dOne{n+j}{j} 
    \node[framed] at (1,1) (w1) {2n}; \draw (v1) -- (w1);} 
  \quad  \xrightarrow{L_2} 
  \quad
  \tikz[script math mode,baseline=0]
  {
    \node[gauged] at (1,0) (v1) {j}; 
    \node[gauged] at (2,0) (v2) {j}; \draw (v1) -- (v2);
    \node[framed] at (2,1) (w2) {n}; \draw (v2) -- (w2);
  }
  $$
  can be composed. Under the identification of first and third spaces
  with $T^* Gr(j,n)^2$ and $T^* Gr(j,n)$, 
  the composite is the transpose $Cgraph(\Delta)^T$ of 
  the conormal bundle of the graph of the diagonal inclusion.
\end{prop}

\begin{proof}
  We check the transversality (condition \#1 of composability) by first
  projecting $L_1,L_2$ into the middle quiver variety, and checking that
  even the images of these inclusions already are transverse.
  This is essentially because the equations imposed are on disjoint
  variables (the SW and NE quadrants). Properness is easier: once
  $X,Y$ are specified, the conditions on $V,W,V'$ are closed
  conditions inside a product of Grassmannians.

  The set-theoretic composition is easy to compute:  
  $$
  L_1 \star L_2
  = \left\{ ((A,V',D',W'), (Y,V''))\ :\ Y = A+D,\ V' = W' = V'' \right\} 
  $$
  This is again conical (which is surprising, since $L_2$ isn't conical),
  and lemma \ref{lem:reflex} again lets us identify this as the
  conormal bundle to $\left\{ ((V',W'), V'')\ :\ V' = W' = V'' \right\}$
  i.e. the transpose of the graph of the diagonal inclusion.
\end{proof}

In particular, if we can follow stable classes along $L_1\star L_2$
we can use this to compute products in cohomology, as follows.

Let $\beta$ be the class of the transpose of the conormal bundle to the graph
$\{(A,A,A)\}$ of the diagonal inclusion $Gr(j,n) \into Gr(j,n) \times Gr(j,n)$,
so its $\beta$-transform goes from 
$H^*_{T\times \CC^\times}(T^*Gr(j,n))^{\tensor 2}$ to $H^*_{T\times \CC^\times}(T^*Gr(j,n))$
(both suitably localized).
Then according to \S \ref{sssec:divzero},
$$ \frac{ \Upsilon_\beta(\alpha_1 \tensor \alpha_2) }{e(T^*Gr(j,n))}
= \frac{ \alpha_1 }{e(T^*Gr(j,n))}\ \frac{ \alpha_2 }{e(T^*Gr(j,n))} 
$$

\junk{

At $d=1$ all the quiver varieties involved are of flag type
(see figure \ref{fig:d1}). 

For simplicity we work at generic deformations $\varepsilon$.
\rem[gray]{AK will add the description that holds at arbitrary $\varepsilon$}
These two considerations taken
together let us identify our deformed quiver varieties with
isospectral sets of diagonalizable matrices.

The Lagrangian combining $\tensor_{i=1}^n (\CC^3,x_i)$
and $\tensor_{i=1}^n (\CC^3,y_i)$ to get 
$(\tensor_{i=1}^n (\CC^3,x_i)) \tensor (\tensor_{i=1}^n (\CC^3,y_i))$
is the stable envelope
$$ L_1 = \left\{ \left( (A,D) \in End(\CC^n)^2,
    \begin{bmatrix}      A & B \\ 0 & D    \end{bmatrix} \right)\ :\
  A(A-\varepsilon_1)(A-\varepsilon_2) = 
  D(D-\varepsilon_1)(D-\varepsilon_2) = 0 \right\}. $$
}

\subsubsection{Following the stable classes}\label{sssec:following}
On our intermediate quiver variety
$\tikz[script math mode,baseline=0]
{
  \node[gauged] at (1,0) (v1) {n+j}; 
  \node[gauged] at (2,0) (v2) {k}; \draw (v1) -- (v2);
  \node[framed] at (1,1) (w2) {2n}; \draw (v1) -- (w2) ;}
$
we have two commuting one-parameter subgroups inside 
the ``flavor group'' $GL(2n)$ that acts at the framed \fbox{$2n$} vertex. 
One is $z\mapsto \diag(z^1,z^2,\ldots,z^n,\, z^1,z^2,\ldots,z^n)$ which we will
call $(\rhocek_n)_\Delta$. The other is $z\mapsto \diag(1,\ldots,1,z,\ldots,z)$ 
which we have been calling $Z$. This gives us the left square below 
(ignore the right square for now) of stable envelopes,
and we call these envelopes $env()$ depending on the group used:
$$
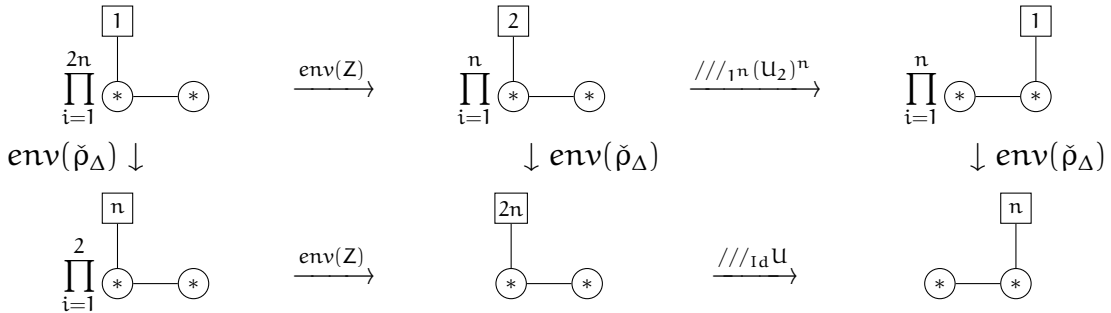
\begin{matrix}
  &\prod\limits_{i=1}^{2n}
  \tikz[script math mode,baseline=0]{\dOne{*}{*} 
    \node[framed] at (1,1) (w1) {1}; \draw (v1) -- (w1);}
  &
  \xrightarrow{env(Z)}
  &
  \prod\limits_{i=1}^{n}
  \tikz[script math mode,baseline=0]{\dOne{*}{*} 
    \node[framed] at (1,1) (w1) {2}; \draw (v1) -- (w1);} 
  & 
  \xrightarrow{///_{1^n} (U_2)^n}
  &
  \prod\limits_{i=1}^{n}
  \tikz[script math mode,baseline=0]
  {
    \node[gauged] at (1,0) (v1) {*}; 
    \node[gauged] at (2,0) (v2) {*}; \draw (v1) -- (v2);
    \node[framed] at (2,1) (w2) {1}; \draw (v2) -- (w2);
  }
  \\[1em]
  &  env(\rhocek_\Delta) \downarrow \phantom{  env(\rhocek_\Delta)}
  && 
  \phantom{  env(\rhocek_\Delta)} \downarrow env(\rhocek_\Delta) 
  && 
  \phantom{  env(\rhocek_\Delta)} \downarrow env(\rhocek_\Delta) 
  \\[.5em]
  &
  \prod\limits_{i=1}^{2}
  \tikz[script math mode,baseline=0]{\dOne{*}{*} 
    \node[framed] at (1,1) (w1) {n}; \draw (v1) -- (w1);}
  &
  \xrightarrow{env(Z)}
  &
  \phantom{   \prod\limits_{i=1}^{n} }
  \tikz[script math mode,baseline=0]{\dOne{*}{*} 
    \node[framed] at (1,1) (w1) {2n}; \draw (v1) -- (w1);} 
  & 
  \xrightarrow{///_{Id} U}
  &
  \tikz[script math mode,baseline=0]
  {
    \node[gauged] at (1,0) (v1) {*}; 
    \node[gauged] at (2,0) (v2) {*}; \draw (v1) -- (v2);
    \node[framed] at (2,1) (w2) {n}; \draw (v2) -- (w2);
  }
\end{matrix}
$$

Note that this left square does {\em not} commute; rather, proposition
\ref{prop:stable} gives us ways to complete either the down-then-right
or right-then-down paths in that square to {\em different} stable
envelopes from Northwest to Southeast.
Apply proposition \ref{prop:stable} to down-then-right giving the circle
$$ \diag(z,z^2,\ldots,z^n,z,z^2,\ldots,z^n)\cdot
  \diag(1,\ldots,1,z,\ldots,z)^N
$$ 
whose exponents are increasing 
(for $N \gg 0$, in particular, whenever $N\geq n$).
Apply proposition \ref{prop:stable} to right-then-down giving the circle
$$\diag(z,z^2,\ldots,z^n,z,z^2,\ldots,z^n)^N\cdot
  \diag(1,\ldots,1,z,\ldots,z)
$$ 
whose exponents are ordered in the riffle shuffle permutation of $2n$ cards.
To express the first stable basis in the second requires multiplying
$n\choose 2$ $R$-matrices, as in \cite[\S 4.1]{MO-qg}.

Consider the top of the right hand square, made of
$ 
 \tikz[script math mode,baseline=0]{\dOne{*}{*} 
    \node[framed] at (1,1) (w1) {2}; \draw (v1) -- (w1);} 
\xrightarrow{///_1 U_2}
  \tikz[script math mode,baseline=1]
  {
    \node[gauged] at (1,0) (v1) {*}; 
    \node[gauged] at (2,0) (v2) {*}; \draw (v1) -- (v2);
    \node[framed] at (2,1) (w2) {1}; \draw (v2) -- (w2);
  }
$
reductions. As explained in \S\ref{sssec:reps}, this is modeling the
quotient map $(\CC^3)^{\tensor 2} \onto Alt^2 \CC^3$, losing $Sym^2 \CC^3$.
\junk{
  In detail, the moment map values of the points
  $
  \tikz[script math mode,baseline=0]{\dOne{0}{0} 
    \node[framed] at (1,1) (w1) {2}; \draw (v1) -- (w1);} ,\
  \tikz[script math mode,baseline=0]{\dOne{2}{0} 
    \node[framed] at (1,1) (w1) {2}; \draw (v1) -- (w1);} ,\
  \tikz[script math mode,baseline=0]{\dOne{2}{2} 
    \node[framed] at (1,1) (w1) {2}; \draw (v1) -- (w1);} 
  $
  are $\left[{0\atop 0}{0\atop 0}\right],
  \left[{\epsilon_2\atop 0}{0\atop \epsilon_2}\right],
  \left[{\epsilon_1\atop 0}{0\atop \epsilon_1}\right]$ respectively,
  so imposing that the NE corner be $1$ gives an empty intersection.
}
The Hamiltonian reductions of the points
$
\tikz[script math mode,baseline=0]{\dOne{0}{0} 
  \node[framed] at (1,1) (w1) {2}; \draw (v1) -- (w1);} ,\
\tikz[script math mode,baseline=0]{\dOne{2}{0} 
  \node[framed] at (1,1) (w1) {2}; \draw (v1) -- (w1);} ,\
\tikz[script math mode,baseline=0]{\dOne{2}{2} 
  \node[framed] at (1,1) (w1) {2}; \draw (v1) -- (w1);} 
$
are empty, whereas the reductions of the 
$
 \tikz[script math mode,baseline=0]{\dOne{1}{0} 
    \node[framed] at (1,1) (w1) {2}; \draw (v1) -- (w1);} ,\ 
 \tikz[script math mode,baseline=0]{\dOne{1}{1} 
    \node[framed] at (1,1) (w1) {2}; \draw (v1) -- (w1);} ,\ 
 \tikz[script math mode,baseline=0]{\dOne{2}{1} 
    \node[framed] at (1,1) (w1) {2}; \draw (v1) -- (w1);} 
$
spaces, each a $T^* \PP^1$, are points.

A curious thing happens in the right-hand square: the down-then-right composite
is not ordinarily defined insofar as the intersection is not transverse,
but if one ignores this, the square commutes set-theoretically.
This agreement is close enough that one can exploit it:

\begin{prop}\label{prop:rightsquare}
  The induced map on localized equivariant cohomology, 
  going down then right in the right-hand square above, is
  $ \prod_{1\leq i<j \leq n} \frac{x_j - x_i }{ \hbar+x_i - x_{j} } $
  times the induced map from going right then down. In particular,
  it takes stable classes to either zero or to stable classes times that factor.
\end{prop}

\junk{Adjust those factors to match the ones from the end of the proof.
  The numerator weights should be
  $x_{n+i} - x_j + \hbar \mapsto x_i - x_j + 2\hbar$.  Does this fix
  \S \ref{sssec:dOneverify}?}
 
\begin{proof}
  \newcommand\ve\varepsilon
  We start with right-then-down, where the relations are composable.
  In all cases we work with the affine deformations, which as explained
  in \S\ref{ssec:grothspringer} are varieties $\calO_{\ve}$ of matrices 
  with fixed spectrum.
  The right-moving, then down-moving, relations are
  \begin{eqnarray*}
    \left\{ \left( \left(
    \begin{bmatrix} a_1 & 1 \\ c_1 & d_1 \end{bmatrix}, \ldots
        \begin{bmatrix} a_n & 1 \\ c_n & d_n \end{bmatrix} \right),
      (a_1+d_1, \ldots, a_n+d_n) 
    \right) \right\}
  &\ : \ &
  \prod\limits_{i=1}^{n}
  \tikz[script math mode,baseline=0]{\dOne{*}{*} 
    \node[framed] at (1,1) (w1) {2}; \draw (v1) -- (w1);} 
  \ \to \
   \prod\limits_{i=1}^{n}
  \tikz[script math mode,baseline=0]
  {
    \node[gauged] at (1,0) (v1) {*}; 
    \node[gauged] at (2,0) (v2) {*}; \draw (v1) -- (v2);
    \node[framed] at (2,1) (w2) {1}; \draw (v2) -- (w2);
  } 
                   \\
  \left\{ \left( (e_1,\ldots,e_n), 
      \begin{bmatrix} e_1 & * & * \\ 0 & \ddots & * \\ 0 & 0 & e_n \end{bmatrix}
      \right) \right\}
  &\ : \ &
  \prod\limits_{i=1}^{n}
  \tikz[script math mode,baseline=0]
  {
    \node[gauged] at (1,0) (v1) {*}; 
    \node[gauged] at (2,0) (v2) {*}; \draw (v1) -- (v2);
    \node[framed] at (2,1) (w2) {1}; \draw (v2) -- (w2);
  }
  \ \to \
  \tikz[script math mode,baseline=0]
  {
    \node[gauged] at (1,0) (v1) {*}; 
    \node[gauged] at (2,0) (v2) {*}; \draw (v1) -- (v2);
    \node[framed] at (2,1) (w2) {n}; \draw (v2) -- (w2);
  }
  \end{eqnarray*}
  
  This intersection is transverse, matching up each $a_i+d_i$ with its $e_i$.
  Also, the projection to the image is an isomorphism, resulting in the
  Lagrangian
  $$
  \left\{ \left( \left(
        \begin{bmatrix} a_1 & 1 \\ c_1 & d_1 \end{bmatrix}, \ldots,
        \begin{bmatrix} a_n & 1 \\ c_n & d_n \end{bmatrix} \right),
      \begin{bmatrix} a_1+d_1&*&* \\ 0 & \ddots & * \\ 0&0&a_n+d_n \end{bmatrix}
    \right) \right\}
  \ : \
  \prod\limits_{i=1}^{n}
  \tikz[script math mode,baseline=0]{\dOne{*}{*} 
    \node[framed] at (1,1) (w1) {2}; \draw (v1) -- (w1);} 
  \ \to \
  \tikz[script math mode,baseline=0]
  {
    \node[gauged] at (1,0) (v1) {*}; 
    \node[gauged] at (2,0) (v2) {*}; \draw (v1) -- (v2);
    \node[framed] at (2,1) (w2) {n}; \draw (v2) -- (w2);
  }
  $$  
  Now we write out the down-then-right Lagrangian relations 
  we will try to compose, with partial success:
  \begin{eqnarray*}
    L_3 = \left\{ \left( \left(
    \begin{bmatrix} a_i & b_i \\ c_i & d_i \end{bmatrix}
   \right)_{i=1\ldots n},
      \left[
        \begin{array}{ccc|ccc}
          a_1 & * & * & b_1 & * & * \\
          0 & \ddots & * & 0 & \ddots & * \\
          0 & 0 & a_n & 0 & 0 & b_n \\
          \hline
          c_1 & * & * & d_1 & * & * \\
          0 & \ddots & * & 0 & \ddots & * \\
          0 & 0 & c_n & 0 & 0 & d_n \\
        \end{array} \right] \right) \right\}
    &: &
  \prod\limits_{i=1}^{n}
  \tikz[script math mode,baseline=0]{\dOne{*}{*} 
    \node[framed] at (1,1) (w1) {2}; \draw (v1) -- (w1);} 
  \ \to \
  \tikz[script math mode,baseline=0]{\dOne{*}{*} 
    \node[framed] at (1,1) (w1) {2n}; \draw (v1) -- (w1);} 
              \\
  L_2 = \left\{ \left(
    \begin{bmatrix} A & I \\ C & D \end{bmatrix}, A+D \right) \right\}
    &: &
  \tikz[script math mode,baseline=0]{\dOne{*}{*} 
    \node[framed] at (1,1) (w1) {2n}; \draw (v1) -- (w1);} 
  \ \to \
  \tikz[script math mode,baseline=0]
  {
    \node[gauged] at (1,0) (v1) {*}; 
    \node[gauged] at (2,0) (v2) {*}; \draw (v1) -- (v2);
    \node[framed] at (2,1) (w2) {n}; \draw (v2) -- (w2);
  }
  \end{eqnarray*}
  The intersection is nontransverse, as the 
  $B := \begin{bmatrix} b_1 &*&* \\ 0 & \ddots & * \\ 0&0& b_n \end{bmatrix}$
  matrix already has $0$s in the lower triangle, so asking $B = I$ 
  imposes $n\choose 2$ equations that are already satisfied.
  We deal with this as follows: remove those conditions on $B$.
  The result is a larger (non-Lagrangian) correspondence $L_3' \supset L_3$,
  with class $[L_3'] = [L_3] / \prod \{ x_j - x_{n+i} + \hbar\ :\ i<j \}$
  as it has been divided by the weights of those equations.

  Now the intersection {\em is} transverse, but its projection to 
  $  \Bigg( \prod\limits_{i=1}^{n}
  \tikz[script math mode,baseline=0]{\dOne{*}{*} 
    \node[framed] at (1,1) (w1) {2}; \draw (v1) -- (w1);}  \Bigg)
  \times
  \tikz[script math mode,baseline=0]
  {    \node[gauged] at (1,0) (v1) {*}; 
    \node[gauged] at (2,0) (v2) {*}; \draw (v1) -- (v2);
    \node[framed] at (2,1) (w2) {n}; \draw (v2) -- (w2);  }
  $
  is improper; specifically, the strict upper triangle of $C$ cannot
  be reconstructed from the first and third factors, nor can the
  strict upper triangles of $A,D$ be determined individually.
  We will fix these two problems (soon) by asking that 
  $A := \begin{bmatrix} a_1 &*&* \\ 0 & \ddots & * \\ 0&0& a_n \end{bmatrix}$
  be not just triangular but diagonal. The result is a smaller correspondence
  $L_3'' \subset L_3'$, with class
  $[L_3''] = [L_3'] \prod \{ x_i - x_{j} + \hbar\ :\ i<j \}$,
  having been multiplied by the weights of those newly imposed equations.

  The intersection of the correspondences $L_3'',L_2$ is now
  $$
  \left\{ \left( \left(
        \begin{bmatrix} a_i & b_i \\ c_i & d_i \end{bmatrix}
      \right)_{i=1\ldots n}, \quad
      \left[
        \begin{array}{ccc|ccc}
          a_1 & 0 & 0 & b_1 & * & * \\
          0 & \ddots & 0 & * & \ddots & * \\
          0 & 0 & a_n & * & * & b_n \\
          \hline
          c_1 & * & * & d_1 & * & * \\
          0 & \ddots & * & 0 & \ddots & * \\
          0 & 0 & c_n & 0 & 0 & d_n \\
        \end{array} \right] = 
      \begin{bmatrix} A & I \\ C & D \end{bmatrix}, \quad A+D \right) \right\}
  $$
  To show the properness of its projection to the first times third factors,
  it suffices to be able to uniquely reconstruct $A,C,D$ from
  $\left(\begin{bmatrix} a_i & b_i \\ c_i & d_i \end{bmatrix}\right)_{i=1\ldots n},
  A+D$. Since $A$ is diagonal we can get it from $(a_i)_{i=1\ldots n}$.
  Knowing $A$ and $A+D$ we get $D$. Obtaining $C$ is trickier: since
  our $2n\times 2n$ matrix satisfies a cubic equation
  $(X-\ve_1)(X-\ve_2)(X-\ve_3) = 0$, we get a linear relation between
  the NE quadrants of
  $$ I_{2n} =  \begin{bmatrix}    * & 0 \\ * & *  \end{bmatrix}\qquad
  X =  \begin{bmatrix}    * & I_n \\ * & *  \end{bmatrix}\qquad
  X^2 =  \begin{bmatrix}    * & A+D \\ * & *  \end{bmatrix}\qquad
  X^3 =  \begin{bmatrix}    * & A^2+AD+D^2+C \\ * & *  \end{bmatrix} $$
  (coefficient $1$ on $X^3$) with which to determine $C$ from $A$ and $D$.

  In order to bring in $L_2$, we have to specialize as explained in
  \S\ref{sssec:discovering} to $x_{n+i} = x_i + \hbar$.

  Now the $L_3'' \star L_2$ convolution is defined and gives the
  same result as the right-then-down convolution, so, they induce the
  same map on localized equivariant cohomology. 
  Consequently, the map on localized equivariant cohomology given by
  the original $L_3$, then $L_2$, is off by the factors above
  (specialized to $x_{n+i} = x_i + \hbar$).
\end{proof}

We now give a schematic way of indexing the stable basis elements,
on each of the six spaces in our double square diagram. 
There is a complication in that the intermediate space (bottom middle)
has two relevant stable bases, related by a composition of $n\choose 2$
$R$-matrices. In each of the below, a stable basis element corresponds
to labeling \NE, \NW\ using $0,1,2$, and \horizbar\ 
using $\{0,1\},\{0,2\},\{1,2\}$. In these pictures, each stable envelope amounts
to concatenating strings, whereas each of the reductions amounts to
flattening an \tikz[baseline=0]{\NE[i]\NW[j]} to an
\underline{$\{i,j\}$} if $i\neq j$.
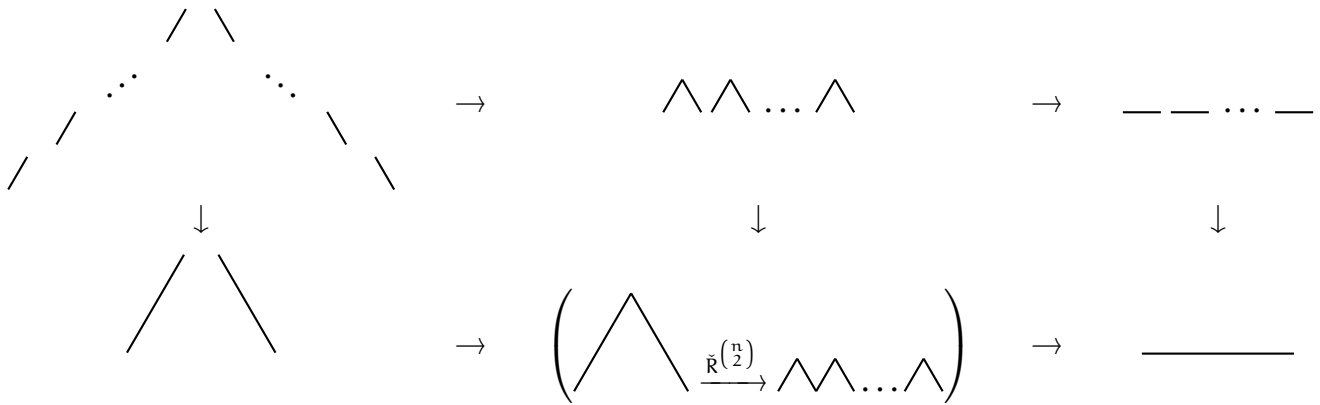
$$
\begin{matrix}
  \begin{matrix}
    &&&\NE&\NW \\ &&\iddots &&& \ddots \\ &\NE&&&&&\NW \\
    \NE&&&&&&&\NW
  \end{matrix}
  &\quad\to\quad&
  \tikz[baseline=0]{\NE\NW} \  \tikz[baseline=0]{\NE\NW}\ \ldots\ \tikz[baseline=0]{\NE\NW} 
  &\quad\to\quad&
  \horizbar\  
  \horizbar\ \ldots\ \horizbar 
  \\
  \downarrow &&   \downarrow &&   \downarrow 
  \\[.5em]
  \tikz[scale=3]{\NE} \quad \tikz[scale=3]{\NW}
  &\quad\to\quad&
  \left( \raisebox{-0.5cm}{$\tikz[scale=3,baseline=0]{\NE\NW} \xrightarrow{\check R^{n\choose 2}}
\tikz[baseline=0]{\NE\NW\begin{scope}[xshift=0.5cm]{\NE\NW}\end{scope}
}
\ldots  \tikz[baseline=0]{\NE\NW}
$}  \right)
  &\quad\to\quad&
\tikz[baseline=0] { \draw[thick] (0,0) -- (2,0); }
\end{matrix}
$$

We have nearly arrived at puzzles. 
\begin{enumerate}
\item We start with strings $\lambda$ of content $1^k 2^{n-k}$, $\mu$
  of content $0^k 1^{n-k}$, giving a point in the NW corner. 
  Going down to the SW corner,
  we get the stable basis element $MO_\lambda \tensor MO_\mu$ 
  on $(T^* Gr(k,n))^2$, which is of course where we want to start
  our product calculation (modulo the $Gr$ vs. $T^* Gr$ issue dealt
  with in \S\ref{sssec:divzero}).
\item Going East from the SW corner to the bottom middle,
  we get the stable basis element
  $MO_{\lambda\mu}$ on $T^* Fl(k,n+k;\ \CC^{2n})$.
  So far nothing has happened on the puzzle side; we are just interpreting
  the NW-then-NE labels on the puzzle as a single stable class.
\item To rewrite our stable basis element $MO_{\lambda\mu}$ in terms of
  the stable basis coming from right-then-down in the left square,
  we use ``the'' reduced word for the riffle shuffle permutation.
  (It is unique up to commuting moves, no braiding necessary.)
  These correspond to filling in the $n\choose 2$ rhombi
  in the puzzle, from top to bottom.
  We now have a linear combination of stable classes in
  the riffle-shuffle-permuted stable basis. We deal with them one by one.
  Let $\eta$ be the labels on   \tikz[baseline=0]{\NE\NW} \  \tikz[baseline=0]{\NE\NW}\ \ldots\ \tikz[baseline=0]{\NE\NW}.
\item Before attempting the reduction (bottom edge of the right square)
  of the $\eta$ class, we pull this string apart into $\eta_1,\ldots,\eta_n$
  labeling the individual sawteeth, and so seeing our $\eta$ class as 
  coming from a stable class from the top middle of the double square.
\item As we traverse the top of the right square, we either get a point
  (if $i\neq j$ in each \tikz[baseline=0]{\NE[i]\NW[j]}) or the empty set.
  These correspond to being able or unable to complete each 
  \tikz[baseline=0]{\NE[i]\NW[j]} to a triangular puzzle piece
  (with \underline{$\{i,j\}$} on the bottom).
\item Going down, we glue these \underline{$\{i,j\}$} into a string, 
  taking the point to a class $MO_\nu$ on $T^*Gr(k,n)$.
\item The actual coefficient of $MO_\nu$ in the convolution we care about,
  down-right-right from the Northwest corner of the double square,
  comes from the $n\choose 2$ $R$-matrix entries times the overall factor from
  proposition \ref{prop:rightsquare}, summed over all the puzzles.
\end{enumerate}

\subsubsection{$d=1$, computing fugacities}\label{sssec:dOneverify}
It remains to define the fugacities of the puzzle rhombus
in position $(i,j)$, $i<j$, which in the planar dual cross the lines bearing 
rapidity $x_{n+i}=x_i+\hbar$ and $x_j$. These fugacities are
the entries of the rational $R$-matrix $\check R \in End(\CC^3 \tensor \CC^3)$
(which was already given in \eqref{eq:RsingleH}):
\rem[gray]{not quite ``standard'' $R$-matrix because we use trivial polarization,
so signs are unusual}
$$ \vec v_a \tensor \vec v_b \mapsto 
\begin{cases}
  \frac{x_{n+i}-x_j}{\hbar+x_j - x_{n+i}}   \vec v_b \tensor \vec v_a +
  \frac{\hbar}{\hbar+x_j-x_{n+i}}   \vec v_a \tensor \vec v_b = 
  \frac{\hbar+x_i-x_j}{x_j - x_i}   \vec v_b \tensor \vec v_a +
  \frac{\hbar}{x_j - x_i}   \vec v_a \tensor \vec v_b
  &\text{if }a\neq b \\[.5em]
  \vec v_a \tensor \vec v_a &\text{if }a = b 
\end{cases}
$$
but we will scale the rhombus at position $(i,j)$ by 
$\frac{x_j-x_i}{\hbar+x_i-x_j}$ in order to account for the 
overall factor from proposition~\ref{prop:rightsquare}.
We then obtain fugacities
\[
\vec v_a \tensor \vec v_b \mapsto 
\begin{cases}
  \vec v_b \tensor \vec v_a +
  \frac{\hbar}{\hbar+x_i - x_j}   \vec v_a \tensor \vec v_b
  &\text{if }a\neq b \\[.8em]
  \frac{x_j-x_i}{\hbar+x_i-x_j}   \vec v_a \tensor \vec v_a &\text{if }a = b 
\end{cases}
\]
This is the desired result:
compare with table~\ref{eq:tab1}.
In particular, the rhombi with all labels $0$
have fugacity $1$,
thus explaining the unusual choice of normalization of the $R$-matrix in \eqref{eq:normR}.
Equivalently, the factor of proposition\ref{prop:rightsquare} matches the lower right entry
of table~\ref{eq:tab1}, which is due to the fact that highest weight $R$-matrix entries are always $1$
in the geometric normalization.

With these, one can compute the action of the convolution 
$T^* Gr(k,n)^2 \xrightarrow{Cgraph(\Delta)^T} T^* Gr(k,n)$
on stable basis elements, as a sum over puzzles.
Using \S\ref{sssec:divzero}, we use that to compute the structure
constants of multiplication of SSM classes, finally reproducing
theorem \ref{thm:main} (in cohomology, for $d=1$). 

\subsection{$d=2$.}\label{ssec:d2}
The intermediate quiver varieties for $d=2,3,4$ are not cotangent bundles,
so to set the stage for them we revisit the $d=1$ case
from the quiver variety point of view.  Consider the reduction $L_2$
as going between these not-quite-Nakajima quiver varieties:
$$
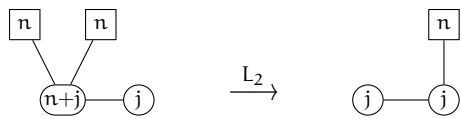
\tikz[script math mode,baseline=0]{\dOne{n+j}{j} 
\node[framed] at (0.5,1) (w1) {n}; \draw (v1) -- (w1);
\node[framed] at (1.5,1) (w2) {n}; \draw (v1) -- (w2);
}
\qquad \stackrel{L_2}{\longrightarrow}\qquad
\tikz[script math mode,baseline=0]{\dOne{j}{j} \node[framed] at (2,1) (w2) {n}; \draw (v2) -- (w2);}
$$
The image of $L_2$ embedded in the left quiver variety consists of
those representations such that the composite map
$
\tikz[script math mode,baseline=0]{
\node[gauged] at (  1,0) (v1) {n+j}; 
\node[framed] at (0.5,1) (w1) {n}; \draw [
-Latex] (w1) -- (v1);
\node[framed] at (1.5,1) (w2) {n}; \draw [
-Latex] (v1) -- (w2);
}
$
is the identity. This assumption produces a splitting of the $n+j$ vertex:
the $n$-dimensional image of the map in, $\oplus$, the $j$-dimensional
kernel of the map out.
The splitting isn't really equitable -- 
the $n$-dimensional image, despite sitting at a gauged vertex, 
has coordinates inherited from (either of) the two gauged \fbox{$n$}.

In this picture, the map from $L_2$ to the right-hand quiver variety
does the following to a representation in the left-hand quiver variety:
delete the two \fbox{$n$} vertices, and carry the newly found $n$-plane 
from the $n+j$ gauged vertex Northeast to become a framed \fbox{$n$}
attached to the old $j$ vertex. Written with spaces instead of
their dimensions, this is
$$
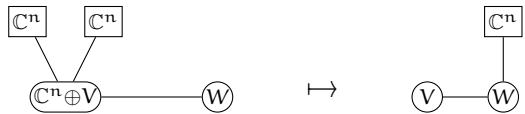
\tikz[script math mode,baseline=0]{\dOneWide{\CC^n \oplus V}{W} 
\node[framed] at (0.5,1) (w1) {\CC^n}; \draw (v1) -- (w1);
\node[framed] at (1.5,1) (w2) {\CC^n}; \draw (v1) -- (w2);
}
\qquad \mapsto \qquad
\tikz[script math mode,baseline=0]{\dOne{V}{W}
  \node[framed] at (2,1) (w2) {\CC^n}; \draw (v2) -- (w2);}
$$
where each of the four morphisms in the second diagram (recall that
each edge carries a morphism in each direction) is inferred from one
of the horizontal morphisms in the first diagram.

Now we are ready to describe the Lagrangian correspondence
required for $d=2$, in
$$
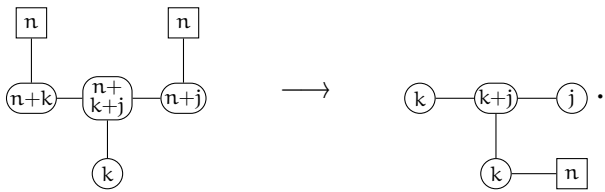
\tikz[script math mode,baseline=0]{\dtwo{n+k}{\ss n+\atop\ss k+j}{n+j}{k} \node[framed] at (1,1) (w1) {n}; \node[framed] at (3,1) (w3) {n}; \draw (v1) -- (w1) (v3) -- (w3);}
\qquad\longrightarrow\qquad
\tikz[script math mode,baseline=0]{\dtwo{k}{k+j}{j}{k} \node[framed] at (3,-1) (w4) {n}; \draw (v4) -- (w4);}
\ .
$$
Once again, we seek a reduction losing $2n^2$ dimensions, 
so want to impose $n^2$ equations. The obvious gauge-invariant equations 
involve paths from one framed vertex to another. In addition,
we want to split an \fbox{$n$} off of each of the three gauge vertices between
the two frame vertices.

Accordingly, we ask that the composite map from the left
\fbox{$n$} to the right \fbox{$n$} be the identity.
This induces splittings of each of the $n+k,n+k+j,n+j$ gauge vertices into 
\fbox{$n$} plus a complement; throw all these $n$-planes away except the
$n$-plane in the middle space (connected to the $k$ below it), which we
rip out and make into the \fbox{$n$} in the Southeast.
$$
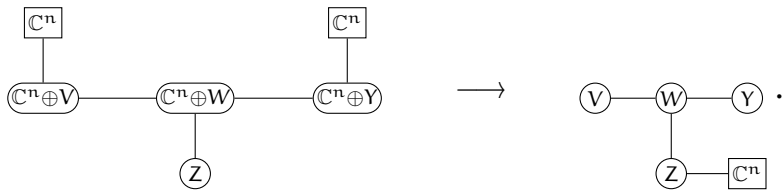
\tikz[script math mode,baseline=0]
{\dtwoWide{\CC^n\oplus V}{\CC^n\oplus W}{\CC^n\oplus Y}{Z} 
  \node[framed] at (0,1) (w1) {\CC^n}; 
  \node[framed] at (4,1) (w3) {\CC^n}; 
  \draw (v1) -- (w1) (v3) -- (w3);}
\qquad\longrightarrow\qquad
\tikz[script math mode,baseline=0]
{\dtwo{V}{W}{Y}{Z} \node[framed] at (3,-1) (w4) {\CC^n}; \draw (v4) -- (w4);}
\ .
$$

It is easy to derive an analogue of proposition~\ref{prop:rightsquare}
by following stable classes: we find that the length four path from the left
\fbox{$n$} to the right \fbox{$n$}
already has $0$s in the lower
triangle, resulting in a factor of $\prod_{i<j} (x_j-x_i)$;
and the projection is improper,
which can be fixed by zeroing out the upper triangle
of e.g. the length 2 path from the left \fbox{$n$} to itself,
resulting in a factor of  $\prod_{i<j} (\hbar+x_i-x_j)^{-1}$.
In the end we obtain the same table \ref{eq:tab1} of fugacities.

In contrast with the $d=1$ case, the equations imposed do {\em not}
form the moment map for a subgroup of the $(GL_n)^2$ that acts at the
frame vertices. (The equations are not linear in the two $GL_n$ moment maps, 
which are the endomorphisms of the \fbox{$n$} vertices.)

\subsection{$d=3,4$.}
The principal new feature at $d=3,4$ is that the drop in dimension
from the intermediate quiver variety to the final one is not $2n^2$
but is $4n^2,6n^2$ respectively. Consequently, our reductions should result
from imposing $2$ or $3$ matrix equations respectively instead of the
$1$ we imposed at each of $d=1,2$. We present here without proof
how one can ``guess'' the form of the equations at $d=3,4$.

In equation \eqref{eq:specialtorus} for $d=1$, we learned that some of
the equations should have weight $\hbar+x_i-y_i$. Since half the
steps along a path contribute $\hbar$, there the path had two steps.
By lemma \ref{lem:factorR}, at $d=1,2,3,4$ the difference in
spectral parameters is $h_d/3 = 1,2,4,10$ times $\hbar$.
This retrodicts the $d=2$ path having four steps.
For all $d$ we therefore expect an {\em inhomogeneous}\/ equation of the form
``some path of length $h_d/3$, from one framed vertex to another,
equals the identity matrix''.

The $\hbar$-weights of {\em all} the equations can be predicted from the
factors in the numerator of the SE entry of the table of fugacities (see
\S\ref{ssec:d123} and \S\ref{ssec:d4}): we add $h_d/3$ to
their $\hbar$ coefficients
to get the $\hbar$-weights, or half-lengths, of the paths:
\[
\begin{array}{c|cccc}
d&  \quad1\quad & \quad2\quad &  \quad3\quad & \quad4\quad \\ 
\hline\vphantom{\big|}
h_d/3 & 1 & 2 & 4 & 10 \\
\text{numerator} &
x & x & x(x-3\hbar) & x(x-4\hbar)(x-9\hbar) \\
\text{$\hbar$-weights} &  1 & 2 & 4,1 & 10,6,1 \\
\text{path lengths} &  2 & 4 & 8,2 & 20,12,2
\end{array}
\]
The reasoning is that each path creates a redundant set of equations
which is that its lower triangle is zero; necessarily, such paths have the same
endpoints as the one of the inhomogeneous equation
(trivially at $d=1,2$ because there's only
one path; and trivially at $d=3,4$ because there's only one possible endpoint);
the resulting weights of a path
of length $\ell$ are of the form $\frac{\ell}{2}\hbar
+x_j-y_i$, $i<j$, which upon substitution $y_i=x_i+\frac{h_d}{3}\hbar$,
$x=x_j-x_i$, gives $x-(\frac{h_d}{3}-\frac{\ell}{2})\hbar$.

Studying the algebras of paths (modulo the moment map equations) leads to
a unique solution to the constraints above for $d=3,4$ which we present now.
(A related result is \cite[Theorem 3.4.1]{Etingof}, applied to
the $X_{2d}$ quiver plus the framed vertex.)

The $d=3$ reduction
$$
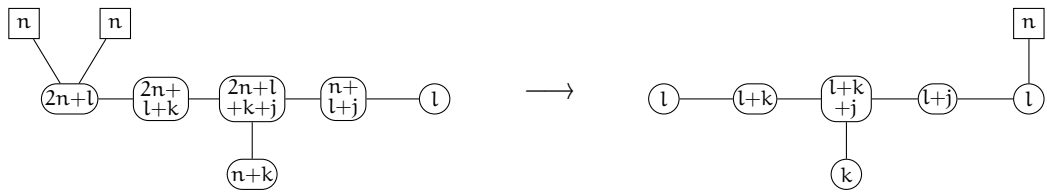
\tikz[script math mode,baseline=0,xscale=1.2]
{\dthree{2n+l}{\ss 2n+\atop\ss l+k}
{\ss 2n+l\atop\ss +k+j}{\ss n+\atop\ss l+j}{l}{n+k} 
\node[framed] at ( .5,1) (w1) {n}; \draw (v1) -- (w1);
\node[framed] at (1.5,1) (w2) {n}; \draw (v1) -- (w2);}
\qquad\longrightarrow \qquad
\tikz[script math mode,baseline=0,xscale=1.2]
{\dthree{l}{l+k}{\ss l+k\atop\ss +j}{l+j}{l}{k} 
\node[framed] at (5,1) (w5) {n}; \draw (v5) -- (w5);}
$$
is obtained from the following two matrix equations, both described using
paths from the left \fbox{$n$} to the right \fbox{$n$}:
the length $8$ path going to and from the $n+k$ vertex gives the identity,
and the direct path of length $2$ gives $0$:
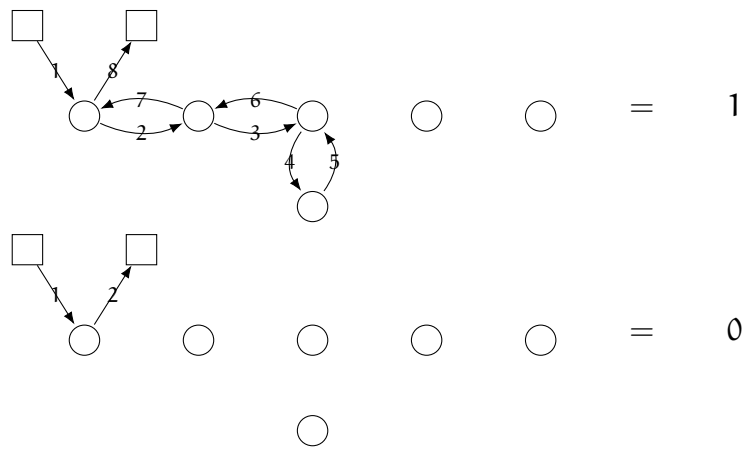
\newcommand\dthreebare[1]{
  \begin{tikzpicture}[xscale=1.5,yscale=1.2,baseline=0,script math mode]
    \node[gauged] (v1) at (1,0) {};
    \node[gauged] (v2) at (2,0) {};
    \node[gauged] (v3) at (3,0) {};
    \node[gauged] (v4) at (3,-1) {};
    \node[gauged] (v5) at (4,0) {};
    \node[gauged] (v6) at (5,0) {};
    \node[framed] (w1) at (0.5,1) {};
    \node[framed] (w1') at (1.5,1) {};
    #1
  \end{tikzpicture}
}
\begin{align*}
  \dthreebare{
  \draw[-Latex] (w1) to node{1} (v1);
  \begin{scope}[-Latex,bend right]
    \draw (v1) to node{2} (v2);
    \draw (v2) to node{3} (v3);
    \draw (v3) to node{4} (v4);
    \draw (v4) to node{5} (v3);
    \draw (v3) to node{6} (v2);
    \draw (v2) to node{7} (v1);
  \end{scope}
  \draw[Latex-] (w1') to node{8} (v1);
  }
  \qquad&=\qquad1
  \\
  \dthreebare{
  \draw[-Latex] (w1) to node{1} (v1);
  \draw[Latex-] (w1') to node{2} (v1);
  }
  \qquad&=\qquad0
\end{align*}

We already know how to use the length $8$ path to split off an \fbox{$n$} from
each of the vertices it passes through, but that doesn't help with
the $n+l+j$ vertex. So we observe now that the other length $8$ path,
to $n+l+j$ and back, must also give the identity matrix (up to a sign choice).
Proof: by the moment map condition at the trivalent vertex,
the difference of the paths gives (up to another sign choice)
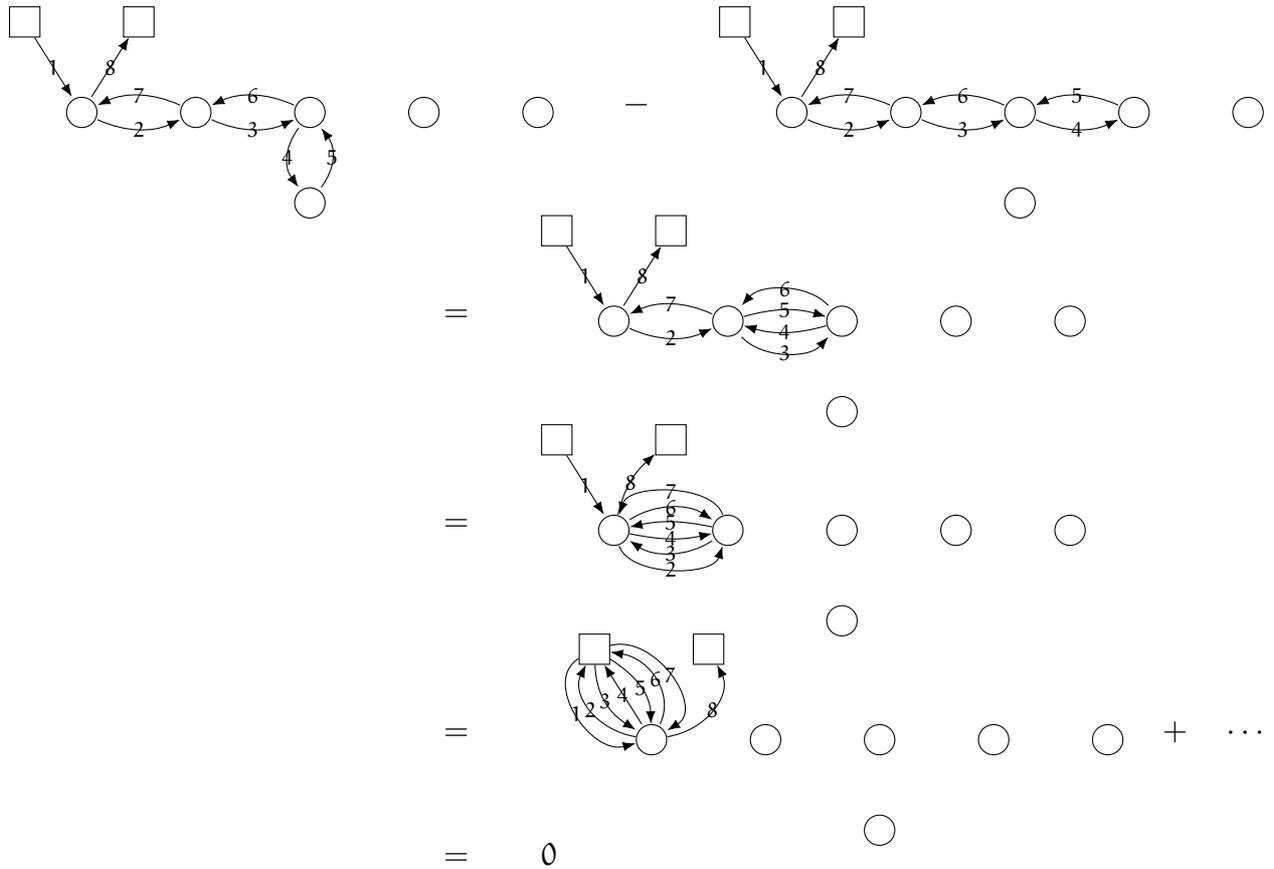
\begin{multline*}
  \dthreebare{
  \draw[-Latex] (w1) to node{1} (v1);
  \begin{scope}[-Latex,bend right]
    \draw (v1) to node{2} (v2);
    \draw (v2) to node{3} (v3);
    \draw (v3) to node{4} (v4);
    \draw (v4) to node{5} (v3);
    \draw (v3) to node{6} (v2);
    \draw (v2) to node{7} (v1);
  \end{scope}
  \draw[Latex-] (w1') to node{8} (v1);
  }
  \qquad-\qquad
    \dthreebare{
    \draw[-Latex] (w1) to node{1} (v1);
    \begin{scope}[-Latex,bend right]
      \draw (v1) to node{2} (v2);
      \draw (v2) to node{3} (v3);
      \draw (v3) to node{4} (v5);
      \draw (v5) to node{5} (v3);
      \draw (v3) to node{6} (v2);
      \draw (v2) to node{7} (v1);
    \end{scope}
    \draw[Latex-] (w1') to node{8} (v1);
    }
  \\[-2mm]
\begin{aligned}
\qquad&=\qquad
    \dthreebare{
    \draw[-Latex] (w1) to node{1} (v1);
    \begin{scope}[-Latex,bend right]
      \draw (v1) to node{2} (v2);
      \draw[bend right=55] (v2) to node{3} (v3);
      \draw[bend left=20] (v3) to node{4} (v2);
      \draw[bend left=20] (v2) to node{5} (v3);
      \draw[bend right=55] (v3) to node{6} (v2);
      \draw (v2) to node{7} (v1);
    \end{scope}
    \draw[Latex-] (w1') to node{8} (v1);
    }
  \\[-2mm]
  \qquad&=\qquad
    \dthreebare{
    \begin{scope}[-Latex]
      \draw (w1) to node{1} (v1);
      \draw[bend right=75] (v1) to node{2} (v2);
      \draw[bend left=40] (v2) to node{3} (v1);
      \draw[bend right=15] (v1) to node{4} (v2);
      \draw[bend right=15] (v2) to node{5} (v1);
      \draw[bend left=40] (v1) to node{6} (v2);
      \draw[bend right=75] (v2) to node{7} (v1);
      \draw[bend left=15] (v1) to node{8} (w1');
    \end{scope}
    }
  \\[-2mm]
  \qquad&=\qquad
    \dthreebare{
    \begin{scope}[-Latex]
      \draw[bend right=80] (w1) to node{1} (v1);
      \draw[bend left=50] (v1) to node{2} (w1);
      \draw[bend right=25] (w1) to node{3} (v1);
      \draw (v1) to node{4} (w1);
      \draw[bend left=25] (w1) to node{5} (v1);
      \draw[bend right=50] (v1) to node{6} (w1);
      \draw[bend left=80] (w1) to node{7} (v1);
      \draw[bend right=50] (v1) to node{8} (w1');
    \end{scope}
    }
    \quad+\quad \cdots
  \\[-2mm]
  \qquad&=\qquad 0
\end{aligned}
\end{multline*}
I.e., since the length $2$ path (as an endomorphism of \fbox{$2n$})
is assumed to have NE quadrant zero, the
same is true of its fourth power, verifying that the two length $8$ paths
(as endomorphisms of \fbox{$2n$}) agree in their NE quadrants,
so {\em both} have the identity matrix there.

It is then a fun linear algebra exercise to show that these give
direct sum decompositions of the $2n+\cdots$ vertices into
$\fbox{$n$} \oplus \fbox{$n$} \oplus \cdots$, and also split an \fbox{$n$}
off each of the $n+l+j,n+k$ vertices. We take the \fbox{$n$} split off the 
$n+l+j$ vertex and attach it as a new \fbox{$n$} to the $l$ vertex.

\junk{In {\tt exfusion.tex} one either imposes one path being $I$ or 
  another path being $0$. These are only equivalent at $\varepsilon=0$,
  so we should spell this out better. AK: Huh, how do we split off
  a $\CC^n$ from a composite that is $0$ rather than $I$?}

The $d=4$ reduction
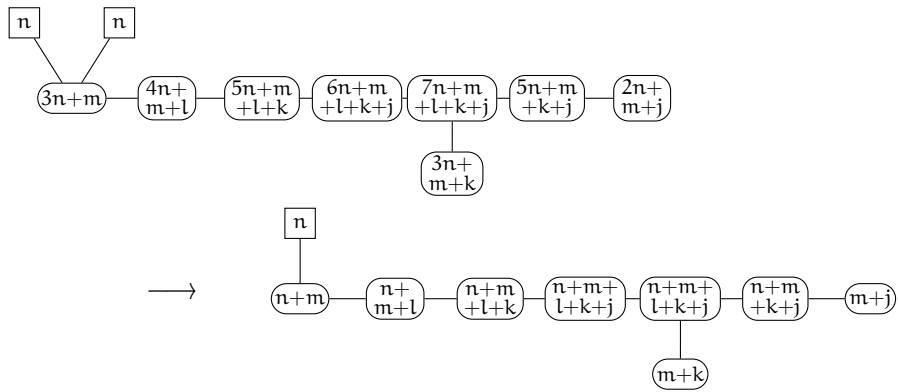
\begin{align*}
&\tikz[script math mode,baseline=0,xscale=1.25]{\dfour{3n+m}{\ss 4n+\atop\ss m+l}
{\ss 5n+m\atop\ss +l+k}{\ss 6n+m\atop\ss +l+k+j}{\ss 7n+m\atop\ss +l+k+j}
{\ss 5n+m\atop\ss +k+j}{\ss 2n+\atop\ss m+j}{\ss 3n +\atop\ss m+k} 
\node[framed] at ( .5,1) (w1) {n}; \draw (v1) -- (w1);
\node[framed] at (1.5,1) (w2) {n}; \draw (v1) -- (w2);
}
\\
&\qquad\qquad\longrightarrow \qquad
\tikz[script math mode,baseline=0,xscale=1.25]{\dfour{n+m}{\ss n+\atop\ss m+l}
{\ss n+m\atop\ss +l+k}{\ss n+m+\atop\ss l+k+j}{\ss n+m+\atop\ss l+k+j}
{\ss n+m\atop\ss +k+j}{m+j}{m+k} 
 \node[framed] at (1,1) (w1) {n}; \draw (v1) -- (w1);}
\end{align*}
similarly results from the three paths:
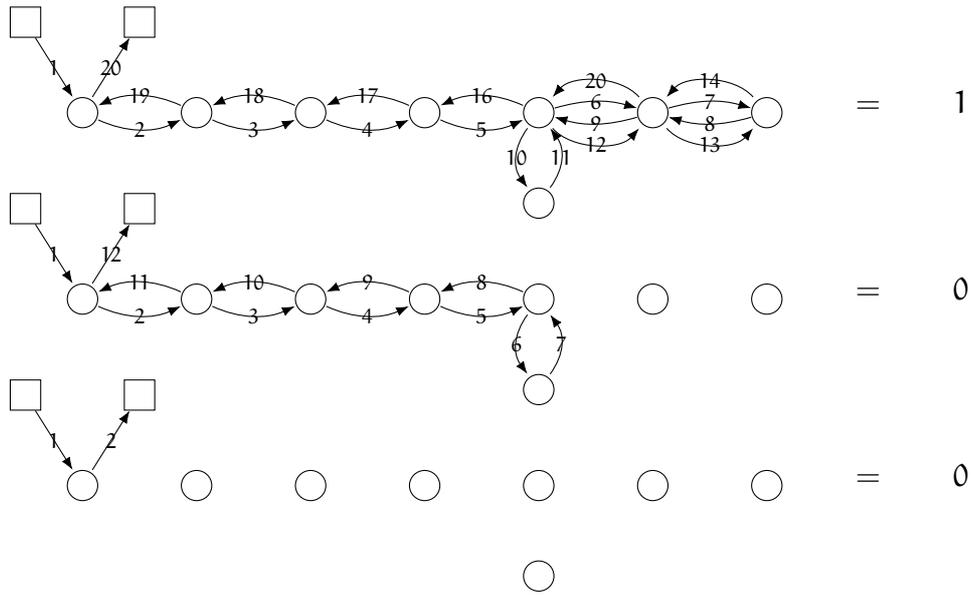
\newcommand\dfourbare[1]{
  \begin{tikzpicture}[xscale=1.5,yscale=1.2,baseline=0,script math mode]
    \node[gauged] (v1) at (1,0) {};
    \node[gauged] (v2) at (2,0) {};
    \node[gauged] (v3) at (3,0) {};
    \node[gauged] (v4) at (4,0) {};
    \node[gauged] (v5) at (5,0) {};
    \node[gauged] (v6) at (5,-1) {};
    \node[gauged] (v7) at (6,0) {};
    \node[gauged] (v8) at (7,0) {};
    \node[framed] (w1) at (0.5,1) {};
    \node[framed] (w1') at (1.5,1) {};
    #1
  \end{tikzpicture}
}
\begin{align*}
  \dfourbare{
  \draw[Latex-] (w1') to node{20} (v1);
  \draw[-Latex] (w1) to node{1} (v1);
  \draw[bend right,-Latex] (v1) to node{2} (v2);
  \draw[bend left,Latex-] (v1) to node{19} (v2);
  \draw[bend right,-Latex] (v2) to node{3} (v3);
  \draw[bend left,Latex-] (v2) to node{18} (v3);
  \draw[bend right,-Latex] (v3) to node{4} (v4);
  \draw[bend left,Latex-] (v3) to node{17} (v4);
  \draw[bend right,-Latex] (v4) to node{5} (v5);
  \draw[bend left,Latex-] (v4) to node{16} (v5);
  \draw[bend right,-Latex] (v5) to node{10} (v6);
  \draw[bend left,Latex-] (v5) to node{11} (v6);
  \draw[bend right=55,-Latex] (v5) to node{12} (v7);
  \draw[bend left=55,Latex-] (v5) to node{20} (v7);
  \draw[bend right=55,-Latex] (v7) to node{13} (v8);
  \draw[bend left=55,Latex-] (v7) to node{14} (v8);
  \draw[bend right=20,Latex-] (v5) to node{9} (v7);
  \draw[bend left=20,-Latex] (v5) to node{6} (v7);
  \draw[bend right=20,Latex-] (v7) to node{8} (v8);
  \draw[bend left=20,-Latex] (v7) to node{7} (v8);
  }
  \qquad&=\qquad 1
  \\[-5mm]
  \dfourbare{
  \draw[Latex-] (w1') to node{12} (v1);
  \draw[-Latex] (w1) to node{1} (v1);
  \draw[bend right,-Latex] (v1) to node{2} (v2);
  \draw[bend left,Latex-] (v1) to node{11} (v2);
  \draw[bend right,-Latex] (v2) to node{3} (v3);
  \draw[bend left,Latex-] (v2) to node{10} (v3);
  \draw[bend right,-Latex] (v3) to node{4} (v4);
  \draw[bend left,Latex-] (v3) to node{9} (v4);
  \draw[bend right,-Latex] (v4) to node{5} (v5);
  \draw[bend left,Latex-] (v4) to node{8} (v5);
  \draw[bend right,-Latex] (v5) to node{6} (v6);
  \draw[bend left,Latex-] (v5) to node{7} (v6);
  }
  \qquad&=\qquad 0
  \\[-5mm]
  \dfourbare{
  \draw[Latex-] (w1') to node{2} (v1);
  \draw[-Latex] (w1) to node{1} (v1);
  }
  \qquad&=\qquad 0
\end{align*}


\vfill\eject

\appendix

\section{Quantized affine algebras}\label{app:qg}
\rem[gray]{Note that we use the opposite coproduct than usual, and invert spectral parameters (which is
equivalent to using the opposite gradation). This is exactly what's needed to accommodate for time flowing
downwards}
\subsection{Generators and relations}
Let $\mathfrak g$ be a simple or affine algebra, with a Cartan matrix $C_{ij}$.
Since we are only interested in simply-laced simple
Lie algebras and the corresponding untwisted affine algebras,
all the formul\ae\ below are written for $C_{ij}$ {\em symmetric}\/
(and $C=2-A$ where $A$ is the adjacency matrix of the Dynkin diagram 
of $\mathfrak g$).

The quantized affine algebra $\Uq(\mathfrak g)$ is the $\CC(q)$-algebra given by
generators $\{E_i,F_i,K_i^{\pm 1}\}$ and relations
\begin{gather*}
K_i K_j=K_jK_i
\\
K_i E_j K_i^{-1} = q^{C_{ij}} E_j
\quad 
K_i F_j K_i^{-1} = q^{-C_{ij}} F_j
\quad
[E_i,F_j] = \delta_{ij} \frac{K_i-K_i^{-1}}{q-q^{-1}}
\\
\sum_{k=0}^{1-C_{ij}} 
(-)^k \left[ \begin{array}{c} 1-C_{ij} \\ k \end{array} \right]_{q}
E_i^{1-C_{ij}-k} E_j E_i^k = 0
\\
\sum_{k=0}^{1-C_{ij}} 
(-)^k \left[ \begin{array}{c} 1-C_{ij} \\ k \end{array} \right]_{q}
F_i^{1-C_{ij}-k} F_j F_i^k = 0
\end{gather*}
in terms of the $q$-binomials
$
\left[\begin{array}{c} m \\ n \end{array}\right]_q
:=
\frac{(q^m-q^{-m})\dots (q^{m-n+1}-q^{-m+n-1})}{(q^n-q^{-n})\dots (q-q^{-1})}$.
Their coproduct is
\[
\Delta(E_i) = E_i \otimes K_i + 1 \otimes E_i
\qquad
\Delta(F_i) = F_i \otimes 1 + K_i^{-1} \otimes F_i
\qquad
\Delta(K_i) = K_i \otimes K_i
\]
(we omit counit and antipode since they will not be needed here).

\subsection{Affine case}
In the affine case, the labels are traditionally chosen to be $0,\ldots,r$
where $0$ is the affine root.

Because the Cartan matrix has rank $r$ (one less than its size),
the algebra possesses a nontrivial gradation (not induced by a Cartan element).
A possible choice (homogeneous gradation) is
\[
\delta(E_i)=\delta_{i,0}
\qquad
\delta(F_i)=-\delta_{i,0}
\qquad
\delta(K_i)=0
\]
We can then extend the algebra by adding a degree generator $q^D$
such that $q^D x q^{-D} = q^{\delta(x)}x$ for any homogeneous element $x$.

Relatedly, the Cartan matrix has a nullspace: $\sum_i m_i C_{ij}=0$,
which implies that $K:=\prod_{i=0}^r K_i^{m_i}$ is central. The quotient of $\Uq(\mathfrak g)$ (including the degree generator $q^D$)
by the relation $K=1$ is the corresponding quantized loop algebra $\Uq(\mathfrak h[z^\pm])$ where
$\mathfrak g=\mathfrak h^{(1)}$.

\subsection{Evaluation representation}\label{app:qgeval}
Consider the case $\mathfrak g = \mathfrak a_d^{(1)}$, as in \S \ref{sec:stable}. We make $\aqg$ act on $V^A(z)$, with basis $e_0,\ldots,e_d$
over $\CC(q)[z^\pm]$, by
\begin{align*}
E_i \, e_{j-1} &= \delta_{i,j}\,z^{-\delta_{i,0}}\,e_j
\\
F_i \, e_{j} &= \delta_{i,j} \,z^{\delta_{i,0}}\,e_{j-1} 
\\
K_i\, e_j &= q^{\delta_{i,j}-\delta_{i,j+1}} e_j
\end{align*}
where indices $i,j$ are considered in $\ZZ/(d+1)\ZZ$. The degree operator is given by $D=-z\frac{d}{dz}$.

One can then check that $\check R(z''/z')$ given by \eqref{eq:Rsingle} is the unique (up to normalization)
intertwiner from $V^A(z')\otimes V^A(z'')$ to $V^A(z'')\otimes V^A(z')$.

\section{$R$-matrices at $d=1$}\label{app:Rd1}
As a special case of the previous appendix, consider the $d=1$ case of $\mathfrak{x}_{2d}^{(1)}=\mathfrak a_2^{(1)}$, i.e.,
the Cartan matrix $C=\left(\begin{smallmatrix}
2&-1&-1
\\
-1&2&-1
\\
-1&-1&2
\end{smallmatrix}\right)
$ corresponding to the 
affine Dynkin
diagram 
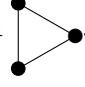
\tikz[baseline=0,scale=0.5] { \draw (0:1) node[mynode] {} -- (120:1) node[mynode] {} -- (240:1) node[mynode] {} -- cycle; }.


The representation matrices $\rho_a$ on $V_a(1)$ in the basis $\{e_{a,1},e_{a,0},e_{a,10}\}$ are:
\begin{align*}
\rho_1(E_0)&=
\begin{pmatrix}
 0 & 0 & 0 \\
 0 & 0 & 0 \\
 -q & 0 & 0 
\end{pmatrix}
&
\rho_1(E_1)&=
\begin{pmatrix}
0 & 1 & 0 \\
 0 & 0 & 0 \\
 0 & 0 & 0
\end{pmatrix}
&
\rho_1(E_2)&=
\begin{pmatrix}
0 & 0 & 0 \\
 0 & 0 & -q^{-1} \\
 0 & 0 & 0 
\end{pmatrix}
\\
\rho_1(F_0)&=
\begin{pmatrix}
 0 & 0 & -q^{-1} \\
 0 & 0 & 0 \\
 0 & 0 & 0 
\end{pmatrix}
&
\rho_1(F_1)&=
\begin{pmatrix}
 0 & 0 & 0 \\
 1 & 0 & 0 \\
 0 & 0 & 0
\end{pmatrix}           
&
\rho_1(F_2)&=
\begin{pmatrix}
0 & 0 & 0 \\
 0 & 0 & 0 \\
 0 & -q & 0
\end{pmatrix}
\\
\rho_1(K_0)&=
\begin{pmatrix}
q^{-1} & 0 & 0 \\
 0 & 1 & 0 \\
 0 & 0 & q
\end{pmatrix}
&
\rho_1(K_1)&=
\begin{pmatrix}
q & 0 & 0 \\
 0 & q^{-1} & 0 \\
 0 & 0 & 1
\end{pmatrix}
&
\rho_1(K_2)&=
\begin{pmatrix}
1 & 0 & 0 \\
 0 & q & 0 \\
 0 & 0 & q^{-1} 
\end{pmatrix}
\\[0.5cm]
\rho_2(E_0)&=
\begin{pmatrix}
 0 & 0 & 0 \\
 0 & 0 & -q^{-1} \\
 0 & 0 & 0
\end{pmatrix}
&
\rho_2(E_1)&=
\begin{pmatrix}
 0 & 0 & 0 \\
 0 & 0 & 0 \\
 -q & 0 & 0
\end{pmatrix}
&
\rho_2(E_2)&=
\begin{pmatrix}
0 & 1 & 0 \\
 0 & 0 & 0 \\
 0 & 0 & 0 
\end{pmatrix}
\\
\rho_2(F_0)&=
\begin{pmatrix}
0 & 0 & 0 \\
 0 & 0 & 0 \\
 0 & -q & 0
\end{pmatrix}
&
\rho_2(F_1)&=
\begin{pmatrix}
 0 & 0 & -q^{-1} \\
 0 & 0 & 0 \\
 0 & 0 & 0
\end{pmatrix}           
&
\rho_2(F_2)&=
\begin{pmatrix}
0 & 0 & 0 \\
 1 & 0 & 0 \\
 0 & 0 & 0
\end{pmatrix}
\\
\rho_2(K_0)&=
\begin{pmatrix}
1 & 0 & 0 \\
 0 & q & 0 \\
 0 & 0 & q^{-1}
\end{pmatrix}
&
\rho_2(K_1)&=
\begin{pmatrix}
q^{-1} & 0 & 0 \\
 0 & 1 & 0 \\
 0 & 0 & q
\end{pmatrix}
&
\rho_2(K_2)&=
\begin{pmatrix}
q & 0 & 0 \\
 0 & q^{-1} & 0 \\
 0 & 0 & 1
\end{pmatrix}
\\[0.5cm]
\rho_3(E_0)&=
\begin{pmatrix}
0 & 0 & 0 \\
-1 & 0 & 0 \\
 0 & 0 & 0
\end{pmatrix}
&
\rho_3(E_1)&=
\begin{pmatrix}
0 & 0 & 0 \\
 0 & 0 & 0 \\
 0 & 1 & 0
\end{pmatrix}
&
\rho_3(E_2)&=
\begin{pmatrix}
0 & 0 & 1 \\
 0 & 0 & 0 \\
 0 & 0 & 0
\end{pmatrix}
\\
\rho_3(F_0)&=
\begin{pmatrix}
 0 & -1 & 0 \\
 0 & 0 & 0 \\
 0 & 0 & 0
\end{pmatrix}
&
\rho_3(F_1)&=
\begin{pmatrix}
0 & 0 & 0 \\
 0 & 0 & 1 \\
 0 & 0 & 0
\end{pmatrix}           
&
\rho_3(F_2)&=
\begin{pmatrix}
0 & 0 & 0 \\
 0 & 0 & 0 \\
 1 & 0 & 0
\end{pmatrix}
\\
\rho_3(K_0)&=
\begin{pmatrix}
q^{-1} & 0 & 0 \\
 0 & q & 0 \\
 0 & 0 & 1
\end{pmatrix}
&
\rho_3(K_1)&=
\begin{pmatrix}
1 & 0 & 0 \\
 0 & q^{-1} & 0 \\
 0 & 0 & q
\end{pmatrix}
&
\rho_3(K_2)&=
\begin{pmatrix}
q & 0 & 0 \\
 0 & 1 & 0 \\
 0 & 0 & q^{-1}
\end{pmatrix}
\end{align*}
We then include the effect of the gradation by writing, for $V_a(z)$,
\[
\rho_{a,z}(g)=z^{-\delta(g)} \rho_a(g)
\]
for all homogeneous elements $g$.

\begin{rmk*}
  One could have chosen the entries of $E_i$ and $F_i$
to be $0$ or $1$ (as in appendix \ref{app:qgeval}) since these representations are minuscule;
it is however convenient to renormalize the basis elements labeled $10$ by powers of $-q$,
which is related to the different role of this label compared to single numbers
(see \S \ref{sec:twist}, in particular footnote \ref{foot:norm})
\rem[gray]{the fact that in the single number sector
it's $0/1$ means the $R$-matrix in this sector will be the usual Hecke one.
the $10$ normalization is arbitrary -- just a good way to share the
fugacities, in particular for $q\to0$. but maybe all this belongs in the
general section. also, is $10$ also a stable basis element? well, no because of these
factors of $q$. note that such changes only affect $R_{i,j}$ if $i\ne j$.
also, the weird signs are to make sure that nonequivariant pieces all have
fugacity 1}
\end{rmk*}

We list here the single-color $R$-matrices at $d=1$:
\begin{align*}
\check R_{1,1}(z',z'')&=
\begin{pmatrix}
 1 & 0 & 0 & 0 & 0 & 0 & 0 & 0 & 0 \\
 0 & \frac{\left(q^2-1\right) {z''}}{q^2 {z''}-{z'}} & 0 & \frac{q ({z''}-{z'})}{q^2
   {z''}-{z'}} & 0 & 0 & 0 & 0 & 0 \\
 0 & 0 & \frac{\left(q^2-1\right) {z''}}{q^2 {z''}-{z'}} & 0 & 0 & 0 & \frac{q ({z''}-{z'})}{q^2
   {z''}-{z'}} & 0 & 0 \\
 0 & \frac{q ({z''}-{z'})}{q^2 {z''}-{z'}} & 0 & \frac{\left(q^2-1\right) {z'}}{q^2
   {z''}-{z'}} & 0 & 0 & 0 & 0 & 0 \\
 0 & 0 & 0 & 0 & 1 & 0 & 0 & 0 & 0 \\
 0 & 0 & 0 & 0 & 0 & \frac{\left(q^2-1\right) {z''}}{q^2 {z''}-{z'}} & 0 & \frac{q
   ({z''}-{z'})}{q^2 {z''}-{z'}} & 0 \\
 0 & 0 & \frac{q ({z''}-{z'})}{q^2 {z''}-{z'}} & 0 & 0 & 0 & \frac{\left(q^2-1\right) {z'}}{q^2
   {z''}-{z'}} & 0 & 0 \\
 0 & 0 & 0 & 0 & 0 & \frac{q ({z''}-{z'})}{q^2 {z''}-{z'}} & 0 & \frac{\left(q^2-1\right)
   {z'}}{q^2 {z''}-{z'}} & 0 \\
 0 & 0 & 0 & 0 & 0 & 0 & 0 & 0 & 1
\end{pmatrix}
\\
\check R_{2,2}(z',z'')&=
\begin{pmatrix}
1 & 0 & 0 & 0 & 0 & 0 & 0 & 0 & 0 \\
 0 & \frac{\left(q^2-1\right) {z''}}{q^2 {z''}-{z'}} & 0 & \frac{q ({z''}-{z'})}{q^2{z''}-{z'}} & 0 & 0 & 0 & 0 & 0 \\
 0 & 0 & \frac{\left(q^2-1\right) {z'}}{q^2 {z''}-{z'}} & 0 & 0 & 0 & \frac{q ({z''}-{z'})}{q^2{z''}-{z'}} & 0 & 0 \\
 0 & \frac{q ({z''}-{z'})}{q^2 {z''}-{z'}} & 0 & \frac{\left(q^2-1\right) {z'}}{q^2{z''}-{z'}} & 0 & 0 & 0 & 0 & 0 \\
 0 & 0 & 0 & 0 & 1 & 0 & 0 & 0 & 0 \\
 0 & 0 & 0 & 0 & 0 & \frac{\left(q^2-1\right) {z'}}{q^2 {z''}-{z'}} & 0 & \frac{q({z''}-{z'})}{q^2 {z''}-{z'}} & 0 \\
 0 & 0 & \frac{q ({z''}-{z'})}{q^2 {z''}-{z'}} & 0 & 0 & 0 & \frac{\left(q^2-1\right) {z''}}{q^2{z''}-{z'}} & 0 & 0 \\
 0 & 0 & 0 & 0 & 0 & \frac{q ({z''}-{z'})}{q^2 {z''}-{z'}} & 0 & \frac{\left(q^2-1\right){z''}}{q^2 {z''}-{z'}} & 0 \\
 0 & 0 & 0 & 0 & 0 & 0 & 0 & 0 & 1
\end{pmatrix}
\\
\check R_{3,3}(z',z'')&=
\begin{pmatrix}
1 & 0 & 0 & 0 & 0 & 0 & 0 & 0 & 0 \\
 0 & \frac{\left(q^2-1\right) {z''}}{q^2 {z''}-{z'}} & 0 & \frac{q ({z''}-{z'})}{q^2{z''}-{z'}} & 0 & 0 & 0 & 0 & 0 \\
 0 & 0 & \frac{\left(q^2-1\right) {z''}}{q^2 {z''}-{z'}} & 0 & 0 & 0 & \frac{q ({z''}-{z'})}{q^2{z''}-{z'}} & 0 & 0 \\
 0 & \frac{q ({z''}-{z'})}{q^2 {z''}-{z'}} & 0 & \frac{\left(q^2-1\right) {z'}}{q^2{z''}-{z'}} & 0 & 0 & 0 & 0 & 0 \\
 0 & 0 & 0 & 0 & 1 & 0 & 0 & 0 & 0 \\
 0 & 0 & 0 & 0 & 0 & \frac{\left(q^2-1\right) {z'}}{q^2 {z''}-{z'}} & 0 & \frac{q({z''}-{z'})}{q^2 {z''}-{z'}} & 0 \\
 0 & 0 & \frac{q ({z''}-{z'})}{q^2 {z''}-{z'}} & 0 & 0 & 0 & \frac{\left(q^2-1\right) {z'}}{q^2{z''}-{z'}} & 0 & 0 \\
 0 & 0 & 0 & 0 & 0 & \frac{q ({z''}-{z'})}{q^2 {z''}-{z'}} & 0 & \frac{\left(q^2-1\right){z''}}{q^2 {z''}-{z'}} & 0 \\
 0 & 0 & 0 & 0 & 0 & 0 & 0 & 0 & 1
\end{pmatrix}
\end{align*}
Note that these matrices are subtly different -- even though $V_1(z)$ and $V_2(z)$ are isomorphic,
the bases we choose for them are different.

\junk{An alternate picture:
$$
\check R_{33}(z',z'') =
\left( \begin{array}{c|ccc|c|ccc|c}
1 &&&&&&&&\\ \hline 
  & \frac{\left(q^2-1\right) {z''}}{q^2 {z''}-{z'}} & 0 & \frac{q ({z''}-{z'})}{q^2
   {z''}-{z'}} &  & 0 & 0 & 0 &  \\
  & 0 & \frac{\left(q^2-1\right) {z'}}{q^2 {z''}-{z'}} & 0 &  & 0 & \frac{q ({z''}-{z'})}{q^2
   {z''}-{z'}} & 0 &  \\
  & \frac{q ({z''}-{z'})}{q^2 {z''}-{z'}} & 0 & \frac{\left(q^2-1\right) {z'}}{q^2
   {z''}-{z'}} &  & 0 & 0 & 0 &  \\
&&&&&&&& \\ \hline   &  &  &  & 1 &  &  &  &  \\ \hline 
  & 0 & 0 & 0 &  & \frac{\left(q^2-1\right) {z'}}{q^2 {z''}-{z'}} & 0 & \frac{q
   ({z''}-{z'})}{q^2 {z''}-{z'}} &  \\
  & 0 & \frac{q ({z''}-{z'})}{q^2 {z''}-{z'}} & 0 &  & 0 & \frac{\left(q^2-1\right) {z''}}{q^2
   {z''}-{z'}} & 0 &  \\
  & 0 & 0 & 0 &  & \frac{q ({z''}-{z'})}{q^2 {z''}-{z'}} & 0 & \frac{\left(q^2-1\right)
   {z''}}{q^2 {z''}-{z'}} &  \\
&&&&&&&& \\ \hline   &  &  &  &  &  &  &  & 1
\end{array} \right)
$$
}

\newcommand\V{\vee}
\renewcommand\L{\wedge}

We also provide $\check R_{1,2}(z',z'')$, which is the building block of puzzles
(here $z := z''/z'$):
\[
  \frac{
\kbordermatrix{  
  & 1\L1  & 1\L0  & 1\L10  && 0\L1  & 0\L0  & 0\L10  && 10\L1  &10\L0 & 10\L10 \\
 1\V1  & {1-z} & 0 & 0 &\vline& 0 & 0 & \mathclap{1-q^{-2}} &\vline& 0 & 0 & 0 \\
 1\V0  & 0 & 0 & 0 &\vline& {q^{-1}-q z} & 0 & 0 &\vline& 0 & 0 & 0 \\
 1\V10  & 0 & 0 & 0 &\vline& 0 & \mathclap{(q^2-1) z} & 0 &\vline& {1-z} & 0 & 0 \\
\cline{2-12}
 0\V1  & 0 & {1-z} & 0 &\vline& 0 & 0 & 0 &\vline& 0 & 0 & {q^{-1}(q^{-2}-1)} \\
 0\V0  & 0 & 0 & 0 &\vline& 0 & {1-z} & 0 &\vline& \mathclap{1-q^{-2}} & 0 & 0 \\
 0\V10  & 0 & 0 & 0 &\vline& 0 & 0 & 0 &\vline& 0 & {q^{-1}-q z} & 0 \\
\cline{2-12}
 10\V1  & 0 & 0 & {q^{-1}-q z} &\vline& 0 & 0 & 0 &\vline& 0 & 0 & 0 \\
 10\V0  & {(q^2-1) z} & 0 & 0 &\vline& 0 & 0 & {1-z} &\vline& 0 & 0 & 0 \\
 10\V10  & 0 & \mathclap{q(1-q^2) z} & 0 &\vline& 0 & 0 & 0 &\vline& 0 & 0 & {1-z} \\
 } } {1-z}
\]

At the specialization $z'' = q^{-2}z'$ (or $z=q^{-2}$) 
from lemma \ref{lem:factorR}, this matrix drops to rank $3$ and
factors (nonuniquely) as

$$
\kbordermatrix{
  & \sout 1 & \sout{0} & \sout{10} \\
1\V 1  & 1 & 0 & 0 \\
1\V 0 & 0 & 0 & 0 \\
1\V 10  & 0 & 1 & 0 \\
\cline{2-4}
0\V 1 & 0 & 0 & 1 \\
0\V0 & 0 & 1 & 0 \\
0\V 10 & 0 & 0 & 0 \\
\cline{2-4}
10\V 1  & 0 & 0 & 0 \\
10\V0  & 1 & 0 & 0 \\
10\V 10  & 0 & 0 & -q
}
\quad
\kbordermatrix{
& 1\L 1 & 1\L 0& 1\L 10&& 0\L 1& 0\L 0& 0\L 10&& 10\L 1&10\L 0& 10\L 10 \\
 \sout 1 & 1 & 0 & 0 &\vline& 0 & 0 & 1&\vline& 0 & 0 & 0 \\
\sout{0} & 0 & 0 & 0 &\vline& 0 & 1 & 0&\vline& 1 & 0 & 0 \\
\sout{10}  & 0 & 1 & 0 &\vline& 0 & 0 & 0&\vline& 0 & 0 & -q^{-1}
    }
$$
It is interesting to note that this rule again enjoys the $\ZZ_3$
rotational symmetry that was already notable in $K$-theory:
$$ c_{\lambda\mu}^{\overleftarrow\nu} 
= \pi_*\left( [X_\lambda] [X_\mu] [X_\nu] [\calO(1)] \right),
\quad \pi:Gr(k,n)\to pt $$
(see \cite[\S 8]{Buch}).

\section{Examples of $d=4$ puzzles}\label{app:exd4}
We depict $d=4$ puzzles using the standard multinumber notation. A
word of warning is needed: this notation is {\em not}\/ unique at
$d=4$, i.e., the same weight can be represented by several multinumbers.
We make here an arbitrary choice (ordering equivalent multinumbers by
length / lexicographically, then picking the smallest).

\subsection{$d=3$ vs $d=4$}\label{app:exd4a}
The coefficient of $S^{2301}$ in $S^{2103}S^{0321}$ in $H_{\hat T}^{*\loc}$
is given by 26 $d=3$ puzzles; at $d=4$, there are 11 additional puzzles, such as
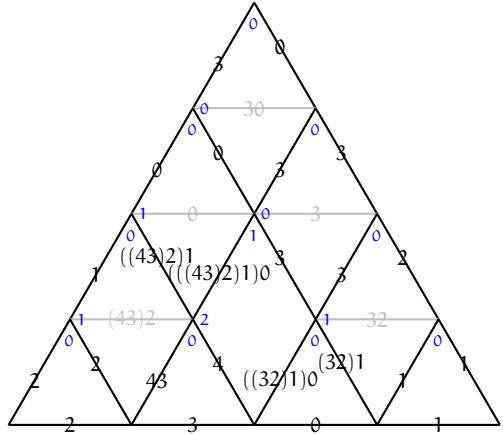
\begin{center}
\def\thescale{1.4}
\def\posa{0.42}\def\posb{0.58}
\begin{tikzpicture}[script math mode,nodes={edgelabel},x={(-0.577cm,-1cm)},y={(0.577cm,-1cm)},scale=\thescale]\useasboundingbox (0,0) -- (4+0.5,0) -- (0,4+0.5);
\draw[thick] (0,0) -- node[pos=\posa] {0} ++(0,1); \draw[thick] (0,0) -- node[pos=\posb] {3} ++(1,0); \node[blue] at (0+0.1,0+0.1) {{}_0}; \draw[thick,lightgray] (0+1,0) -- node {30} ++(-1,1); \node[blue] at (0+0.9,0+0.1) {{}_0}; 
\draw[thick] (0,1) -- node[pos=\posa] {3} ++(0,1); \draw[thick] (0,1) -- node[pos=\posb] {3} ++(1,0); \node[blue] at (0+0.1,1+0.1) {{}_0}; \draw[thick,lightgray] (0+1,1) -- node {3} ++(-1,1); \node[blue] at (0+0.9,1+0.1) {{}_0}; 
\draw[thick] (0,2) -- node[pos=\posa] {2} ++(0,1); \draw[thick] (0,2) -- node[pos=\posb] {3} ++(1,0); \node[blue] at (0+0.1,2+0.1) {{}_0}; \draw[thick,lightgray] (0+1,2) -- node {32} ++(-1,1); \node[blue] at (0+0.9,2+0.1) {{}_1}; 
\draw[thick] (0,3) -- node[pos=\posa] {1} ++(0,1); \draw[thick] (0,3) -- node[pos=\posb] {1} ++(1,0); \node[blue] at (0+0.1,3+0.1) {{}_0}; \draw[thick] (0+1,3) -- node {1} ++(-1,1); 
\draw[thick] (1,0) -- node[pos=\posa] {0} ++(0,1); \draw[thick] (1,0) -- node[pos=\posb] {0} ++(1,0); \node[blue] at (1+0.1,0+0.1) {{}_0}; \draw[thick,lightgray] (1+1,0) -- node {0} ++(-1,1); \node[blue] at (1+0.9,0+0.1) {{}_1}; 
\draw[thick] (1,1) -- node[pos=\posa] {3} ++(0,1); \draw[thick] (1,1) -- node[pos=\posb] {(((43)2)1)0} ++(1,0); \node[blue] at (1+0.1,1+0.1) {{}_1}; \node[blue] at (1+0.9,1+0.1) {{}_2}; 
\draw[thick] (1,2) -- node[pos=\posa] {(32)1} ++(0,1); \draw[thick] (1,2) -- node[pos=\posb] {((32)1)0} ++(1,0); \node[blue] at (1+0.1,2+0.1) {{}_0}; \draw[thick] (1+1,2) -- node {0} ++(-1,1); 
\draw[thick] (2,0) -- node[pos=\posa] {((43)2)1} ++(0,1); \draw[thick] (2,0) -- node[pos=\posb] {1} ++(1,0); \node[blue] at (2+0.1,0+0.1) {{}_0}; \draw[thick,lightgray] (2+1,0) -- node {(43)2} ++(-1,1); \node[blue] at (2+0.9,0+0.1) {{}_1}; 
\draw[thick] (2,1) -- node[pos=\posa] {4} ++(0,1); \draw[thick] (2,1) -- node[pos=\posb] {43} ++(1,0); \node[blue] at (2+0.1,1+0.1) {{}_0}; \draw[thick] (2+1,1) -- node {3} ++(-1,1); 
\draw[thick] (3,0) -- node[pos=\posa] {2} ++(0,1); \draw[thick] (3,0) -- node[pos=\posb] {2} ++(1,0); \node[blue] at (3+0.1,0+0.1) {{}_0}; \draw[thick] (3+1,0) -- node {2} ++(-1,1); 
\end{tikzpicture}
\end{center}
One can check that the sum of fugacities is the same in both cases:
\[
\frac{\hbar^2 \left(x_2-x_3\right) \left(x_1-x_4\right)}{\left(\hbar+x_1-x_3\right) \left(\hbar+x_2-x_3\right) \left(\hbar+x_1-x_4\right) \left(\hbar+x_2-x_4\right)}
\]
Note that the sum of fugacities of the $d=4$ extra puzzles does {\em not}\/ vanish! This is possible because fugacities
are different at $d=4$ (even for puzzles which would already appear at $d=3$), as explained in \S\ref{ssec:d4}.

\subsection{Zero weight states}\label{app:exd4b}
As explained in \S\ref{ssec:d4}, a new feature of $d=4$ puzzles is that the corresponding representation
has a weight space of dimension greater than one, namely, the zero weight space. These zero weight states
must be treated separately, cf proposition~\ref{prop:d4}.
Here is a full example involving them. Let us compute the expansion
of $\Sc^{01432}\Sc^{21043}$ in ordinary Schubert calculus in the full flag variety of $\CC^5$;
equivalently, we work in $H^{*\loc}_{\CC^\times}$ and keep only the terms
in the expansion of $S^{01432}S^{21043}$ whose labels have lowest inversion number $\ell(01432)+\ell(21043)=7$:
\[
S^{01432}S^{21043}
=
S^{41302}
+S^{34102}
+S^{43012}
+S^{24301}
+S^{42031}
+S^{24130}
+S^{41230}
+\cdots
\]

Each of the terms in the r.h.s.\ except the last one is obtained from a single $d=4$ puzzle:
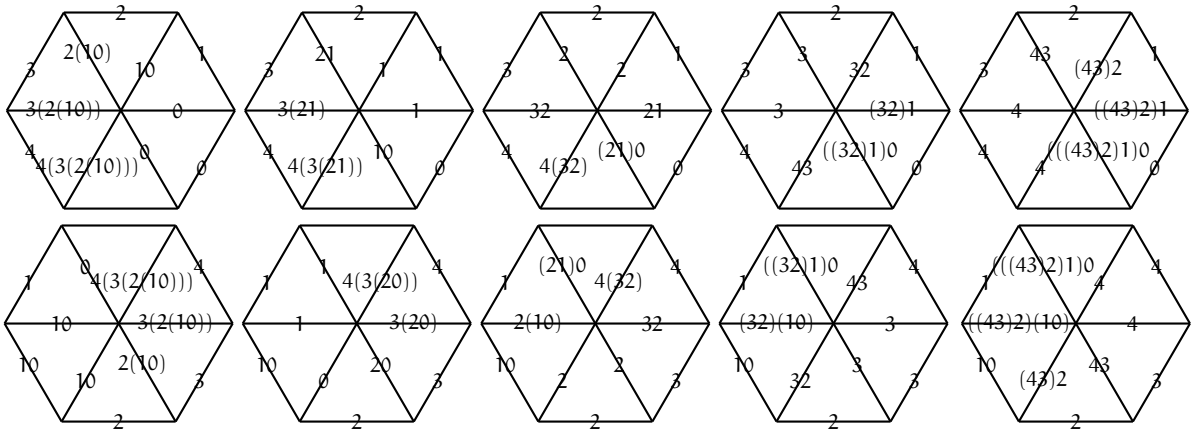
\begin{center}\def\thescale{1}\def\posa{0.42}\def\posb{0.58}%
\hspace*{-1cm}%
\begin{tikzpicture}[script math mode,nodes={edgelabel},x={(-0.577cm,-1cm)},y={(0.577cm,-1cm)},scale=\thescale]\useasboundingbox (0,0) -- (5+0.5,0) -- (0,5+0.5);
\draw[thick] (0,0) -- node[pos=\posa] {2} ++(0,1); \draw[thick] (0,0) -- node[pos=\posb] {2} ++(1,0); \draw[thick] (0+1,0) -- node {2} ++(-1,1); 
\draw[thick] (0,1) -- node[pos=\posa] {1} ++(0,1); \draw[thick] (0,1) -- node[pos=\posb] {32} ++(1,0); \draw[thick] (0+1,1) -- node {(32)1} ++(-1,1); 
\draw[thick] (0,2) -- node[pos=\posa] {0} ++(0,1); \draw[thick] (0,2) -- node[pos=\posb] {32} ++(1,0); \draw[thick] (0+1,2) -- node {(32)0} ++(-1,1); 
\draw[thick] (0,3) -- node[pos=\posa] {4} ++(0,1); \draw[thick] (0,3) -- node[pos=\posb] {4((32)0)} ++(1,0); \draw[thick] (0+1,3) -- node {(32)0} ++(-1,1); 
\draw[thick] (0,4) -- node[pos=\posa] {3} ++(0,1); \draw[thick] (0,4) -- node[pos=\posb] {32} ++(1,0); \draw[thick] (0+1,4) -- node {2} ++(-1,1); 
\draw[thick] (1,0) -- node[pos=\posa] {3} ++(0,1); \draw[thick] (1,0) -- node[pos=\posb] {3} ++(1,0); \draw[thick] (1+1,0) -- node {3} ++(-1,1); 
\draw[thick] (1,1) -- node[pos=\posa] {1} ++(0,1); \draw[thick] (1,1) -- node[pos=\posb] {43} ++(1,0); \draw[thick] (1+1,1) -- node {(43)1} ++(-1,1); 
\draw[thick] (1,2) -- node[pos=\posa] {4} ++(0,1); \draw[thick] (1,2) -- node[pos=\posb] {43} ++(1,0); \draw[thick] (1+1,2) -- node {3} ++(-1,1); 
\draw[thick] (1,3) -- node[pos=\posa] {0} ++(0,1); \draw[thick] (1,3) -- node[pos=\posb] {0} ++(1,0); \draw[thick] (1+1,3) -- node {0} ++(-1,1); 
\draw[thick] (2,0) -- node[pos=\posa] {4} ++(0,1); \draw[thick] (2,0) -- node[pos=\posb] {4} ++(1,0); \draw[thick] (2+1,0) -- node {4} ++(-1,1); 
\draw[thick] (2,1) -- node[pos=\posa] {1} ++(0,1); \draw[thick] (2,1) -- node[pos=\posb] {1} ++(1,0); \draw[thick] (2+1,1) -- node {1} ++(-1,1); 
\draw[thick] (2,2) -- node[pos=\posa] {30} ++(0,1); \draw[thick] (2,2) -- node[pos=\posb] {0} ++(1,0); \draw[thick] (2+1,2) -- node {3} ++(-1,1); 
\draw[thick] (3,0) -- node[pos=\posa] {41} ++(0,1); \draw[thick] (3,0) -- node[pos=\posb] {1} ++(1,0); \draw[thick] (3+1,0) -- node {4} ++(-1,1); 
\draw[thick] (3,1) -- node[pos=\posa] {10} ++(0,1); \draw[thick] (3,1) -- node[pos=\posb] {0} ++(1,0); \draw[thick] (3+1,1) -- node {1} ++(-1,1); 
\draw[thick] (4,0) -- node[pos=\posa] {40} ++(0,1); \draw[thick] (4,0) -- node[pos=\posb] {0} ++(1,0); \draw[thick] (4+1,0) -- node {4} ++(-1,1); 
\end{tikzpicture}%
\begin{tikzpicture}[script math mode,nodes={edgelabel},x={(-0.577cm,-1cm)},y={(0.577cm,-1cm)},scale=\thescale]\useasboundingbox (0,0) -- (5+0.5,0) -- (0,5+0.5);
\draw[thick] (0,0) -- node[pos=\posa] {2} ++(0,1); \draw[thick] (0,0) -- node[pos=\posb] {2} ++(1,0); \draw[thick] (0+1,0) -- node {2} ++(-1,1); 
\draw[thick] (0,1) -- node[pos=\posa] {1} ++(0,1); \draw[thick] (0,1) -- node[pos=\posb] {32} ++(1,0); \draw[thick] (0+1,1) -- node {(32)1} ++(-1,1); 
\draw[thick] (0,2) -- node[pos=\posa] {0} ++(0,1); \draw[thick] (0,2) -- node[pos=\posb] {32} ++(1,0); \draw[thick] (0+1,2) -- node {(32)0} ++(-1,1); 
\draw[thick] (0,3) -- node[pos=\posa] {4} ++(0,1); \draw[thick] (0,3) -- node[pos=\posb] {4((32)0)} ++(1,0); \draw[thick] (0+1,3) -- node {(32)0} ++(-1,1); 
\draw[thick] (0,4) -- node[pos=\posa] {3} ++(0,1); \draw[thick] (0,4) -- node[pos=\posb] {32} ++(1,0); \draw[thick] (0+1,4) -- node {2} ++(-1,1); 
\draw[thick] (1,0) -- node[pos=\posa] {3} ++(0,1); \draw[thick] (1,0) -- node[pos=\posb] {3} ++(1,0); \draw[thick] (1+1,0) -- node {3} ++(-1,1); 
\draw[thick] (1,1) -- node[pos=\posa] {1} ++(0,1); \draw[thick] (1,1) -- node[pos=\posb] {1} ++(1,0); \draw[thick] (1+1,1) -- node {1} ++(-1,1); 
\draw[thick] (1,2) -- node[pos=\posa] {4} ++(0,1); \draw[thick] (1,2) -- node[pos=\posb] {41} ++(1,0); \draw[thick] (1+1,2) -- node {1} ++(-1,1); 
\draw[thick] (1,3) -- node[pos=\posa] {0} ++(0,1); \draw[thick] (1,3) -- node[pos=\posb] {0} ++(1,0); \draw[thick] (1+1,3) -- node {0} ++(-1,1); 
\draw[thick] (2,0) -- node[pos=\posa] {31} ++(0,1); \draw[thick] (2,0) -- node[pos=\posb] {4} ++(1,0); \draw[thick] (2+1,0) -- node {4(31)} ++(-1,1); 
\draw[thick] (2,1) -- node[pos=\posa] {4} ++(0,1); \draw[thick] (2,1) -- node[pos=\posb] {4} ++(1,0); \draw[thick] (2+1,1) -- node {4} ++(-1,1); 
\draw[thick] (2,2) -- node[pos=\posa] {10} ++(0,1); \draw[thick] (2,2) -- node[pos=\posb] {0} ++(1,0); \draw[thick] (2+1,2) -- node {1} ++(-1,1); 
\draw[thick] (3,0) -- node[pos=\posa] {31} ++(0,1); \draw[thick] (3,0) -- node[pos=\posb] {1} ++(1,0); \draw[thick] (3+1,0) -- node {3} ++(-1,1); 
\draw[thick] (3,1) -- node[pos=\posa] {40} ++(0,1); \draw[thick] (3,1) -- node[pos=\posb] {0} ++(1,0); \draw[thick] (3+1,1) -- node {4} ++(-1,1); 
\draw[thick] (4,0) -- node[pos=\posa] {30} ++(0,1); \draw[thick] (4,0) -- node[pos=\posb] {0} ++(1,0); \draw[thick] (4+1,0) -- node {3} ++(-1,1); 
\end{tikzpicture}%
\begin{tikzpicture}[script math mode,nodes={edgelabel},x={(-0.577cm,-1cm)},y={(0.577cm,-1cm)},scale=\thescale]\useasboundingbox (0,0) -- (5+0.5,0) -- (0,5+0.5);
\draw[thick] (0,0) -- node[pos=\posa] {2} ++(0,1); \draw[thick] (0,0) -- node[pos=\posb] {2} ++(1,0); \draw[thick] (0+1,0) -- node {2} ++(-1,1); 
\draw[thick] (0,1) -- node[pos=\posa] {1} ++(0,1); \draw[thick] (0,1) -- node[pos=\posb] {32} ++(1,0); \draw[thick] (0+1,1) -- node {(32)1} ++(-1,1); 
\draw[thick] (0,2) -- node[pos=\posa] {0} ++(0,1); \draw[thick] (0,2) -- node[pos=\posb] {4((32)1)} ++(1,0); \draw[thick] (0+1,2) -- node {(4((32)1))0} ++(-1,1); 
\draw[thick] (0,3) -- node[pos=\posa] {4} ++(0,1); \draw[thick] (0,3) -- node[pos=\posb] {4((32)1)} ++(1,0); \draw[thick] (0+1,3) -- node {(32)1} ++(-1,1); 
\draw[thick] (0,4) -- node[pos=\posa] {3} ++(0,1); \draw[thick] (0,4) -- node[pos=\posb] {32} ++(1,0); \draw[thick] (0+1,4) -- node {2} ++(-1,1); 
\draw[thick] (1,0) -- node[pos=\posa] {3} ++(0,1); \draw[thick] (1,0) -- node[pos=\posb] {3} ++(1,0); \draw[thick] (1+1,0) -- node {3} ++(-1,1); 
\draw[thick] (1,1) -- node[pos=\posa] {4} ++(0,1); \draw[thick] (1,1) -- node[pos=\posb] {43} ++(1,0); \draw[thick] (1+1,1) -- node {3} ++(-1,1); 
\draw[thick] (1,2) -- node[pos=\posa] {0} ++(0,1); \draw[thick] (1,2) -- node[pos=\posb] {1} ++(1,0); \draw[thick] (1+1,2) -- node {10} ++(-1,1); 
\draw[thick] (1,3) -- node[pos=\posa] {1} ++(0,1); \draw[thick] (1,3) -- node[pos=\posb] {1} ++(1,0); \draw[thick] (1+1,3) -- node {1} ++(-1,1); 
\draw[thick] (2,0) -- node[pos=\posa] {4} ++(0,1); \draw[thick] (2,0) -- node[pos=\posb] {4} ++(1,0); \draw[thick] (2+1,0) -- node {4} ++(-1,1); 
\draw[thick] (2,1) -- node[pos=\posa] {31} ++(0,1); \draw[thick] (2,1) -- node[pos=\posb] {1} ++(1,0); \draw[thick] (2+1,1) -- node {3} ++(-1,1); 
\draw[thick] (2,2) -- node[pos=\posa] {0} ++(0,1); \draw[thick] (2,2) -- node[pos=\posb] {0} ++(1,0); \draw[thick] (2+1,2) -- node {0} ++(-1,1); 
\draw[thick] (3,0) -- node[pos=\posa] {41} ++(0,1); \draw[thick] (3,0) -- node[pos=\posb] {1} ++(1,0); \draw[thick] (3+1,0) -- node {4} ++(-1,1); 
\draw[thick] (3,1) -- node[pos=\posa] {30} ++(0,1); \draw[thick] (3,1) -- node[pos=\posb] {0} ++(1,0); \draw[thick] (3+1,1) -- node {3} ++(-1,1); 
\draw[thick] (4,0) -- node[pos=\posa] {40} ++(0,1); \draw[thick] (4,0) -- node[pos=\posb] {0} ++(1,0); \draw[thick] (4+1,0) -- node {4} ++(-1,1); 
\end{tikzpicture}
\hspace*{-1cm}%
\begin{tikzpicture}[script math mode,nodes={edgelabel},x={(-0.577cm,-1cm)},y={(0.577cm,-1cm)},scale=\thescale]\useasboundingbox (0,0) -- (5+0.5,0) -- (0,5+0.5);
\draw[thick] (0,0) -- node[pos=\posa] {2} ++(0,1); \draw[thick] (0,0) -- node[pos=\posb] {2} ++(1,0); \draw[thick] (0+1,0) -- node {2} ++(-1,1); 
\draw[thick] (0,1) -- node[pos=\posa] {1} ++(0,1); \draw[thick] (0,1) -- node[pos=\posb] {1} ++(1,0); \draw[thick] (0+1,1) -- node {1} ++(-1,1); 
\draw[thick] (0,2) -- node[pos=\posa] {0} ++(0,1); \draw[thick] (0,2) -- node[pos=\posb] {31} ++(1,0); \draw[thick] (0+1,2) -- node {(31)0} ++(-1,1); 
\draw[thick] (0,3) -- node[pos=\posa] {4} ++(0,1); \draw[thick] (0,3) -- node[pos=\posb] {4((31)0)} ++(1,0); \draw[thick] (0+1,3) -- node {(31)0} ++(-1,1); 
\draw[thick] (0,4) -- node[pos=\posa] {3} ++(0,1); \draw[thick] (0,4) -- node[pos=\posb] {31} ++(1,0); \draw[thick] (0+1,4) -- node {1} ++(-1,1); 
\draw[thick] (1,0) -- node[pos=\posa] {21} ++(0,1); \draw[thick] (1,0) -- node[pos=\posb] {3} ++(1,0); \draw[thick] (1+1,0) -- node {3(21)} ++(-1,1); 
\draw[thick] (1,1) -- node[pos=\posa] {3} ++(0,1); \draw[thick] (1,1) -- node[pos=\posb] {3} ++(1,0); \draw[thick] (1+1,1) -- node {3} ++(-1,1); 
\draw[thick] (1,2) -- node[pos=\posa] {4} ++(0,1); \draw[thick] (1,2) -- node[pos=\posb] {43} ++(1,0); \draw[thick] (1+1,2) -- node {3} ++(-1,1); 
\draw[thick] (1,3) -- node[pos=\posa] {0} ++(0,1); \draw[thick] (1,3) -- node[pos=\posb] {0} ++(1,0); \draw[thick] (1+1,3) -- node {0} ++(-1,1); 
\draw[thick] (2,0) -- node[pos=\posa] {21} ++(0,1); \draw[thick] (2,0) -- node[pos=\posb] {4} ++(1,0); \draw[thick] (2+1,0) -- node {4(21)} ++(-1,1); 
\draw[thick] (2,1) -- node[pos=\posa] {4} ++(0,1); \draw[thick] (2,1) -- node[pos=\posb] {4} ++(1,0); \draw[thick] (2+1,1) -- node {4} ++(-1,1); 
\draw[thick] (2,2) -- node[pos=\posa] {30} ++(0,1); \draw[thick] (2,2) -- node[pos=\posb] {0} ++(1,0); \draw[thick] (2+1,2) -- node {3} ++(-1,1); 
\draw[thick] (3,0) -- node[pos=\posa] {21} ++(0,1); \draw[thick] (3,0) -- node[pos=\posb] {1} ++(1,0); \draw[thick] (3+1,0) -- node {2} ++(-1,1); 
\draw[thick] (3,1) -- node[pos=\posa] {40} ++(0,1); \draw[thick] (3,1) -- node[pos=\posb] {0} ++(1,0); \draw[thick] (3+1,1) -- node {4} ++(-1,1); 
\draw[thick] (4,0) -- node[pos=\posa] {20} ++(0,1); \draw[thick] (4,0) -- node[pos=\posb] {0} ++(1,0); \draw[thick] (4+1,0) -- node {2} ++(-1,1); 
\end{tikzpicture}%
\begin{tikzpicture}[script math mode,nodes={edgelabel},x={(-0.577cm,-1cm)},y={(0.577cm,-1cm)},scale=\thescale]\useasboundingbox (0,0) -- (5+0.5,0) -- (0,5+0.5);
\draw[thick] (0,0) -- node[pos=\posa] {2} ++(0,1); \draw[thick] (0,0) -- node[pos=\posb] {2} ++(1,0); \draw[thick] (0+1,0) -- node {2} ++(-1,1); 
\draw[thick] (0,1) -- node[pos=\posa] {1} ++(0,1); \draw[thick] (0,1) -- node[pos=\posb] {1} ++(1,0); \draw[thick] (0+1,1) -- node {1} ++(-1,1); 
\draw[thick] (0,2) -- node[pos=\posa] {0} ++(0,1); \draw[thick] (0,2) -- node[pos=\posb] {41} ++(1,0); \draw[thick] (0+1,2) -- node {(41)0} ++(-1,1); 
\draw[thick] (0,3) -- node[pos=\posa] {4} ++(0,1); \draw[thick] (0,3) -- node[pos=\posb] {41} ++(1,0); \draw[thick] (0+1,3) -- node {1} ++(-1,1); 
\draw[thick] (0,4) -- node[pos=\posa] {3} ++(0,1); \draw[thick] (0,4) -- node[pos=\posb] {31} ++(1,0); \draw[thick] (0+1,4) -- node {1} ++(-1,1); 
\draw[thick] (1,0) -- node[pos=\posa] {21} ++(0,1); \draw[thick] (1,0) -- node[pos=\posb] {3} ++(1,0); \draw[thick] (1+1,0) -- node {3(21)} ++(-1,1); 
\draw[thick] (1,1) -- node[pos=\posa] {4} ++(0,1); \draw[thick] (1,1) -- node[pos=\posb] {4(3(21))} ++(1,0); \draw[thick] (1+1,1) -- node {3(21)} ++(-1,1); 
\draw[thick] (1,2) -- node[pos=\posa] {0} ++(0,1); \draw[thick] (1,2) -- node[pos=\posb] {3} ++(1,0); \draw[thick] (1+1,2) -- node {30} ++(-1,1); 
\draw[thick] (1,3) -- node[pos=\posa] {3} ++(0,1); \draw[thick] (1,3) -- node[pos=\posb] {3} ++(1,0); \draw[thick] (1+1,3) -- node {3} ++(-1,1); 
\draw[thick] (2,0) -- node[pos=\posa] {4} ++(0,1); \draw[thick] (2,0) -- node[pos=\posb] {4} ++(1,0); \draw[thick] (2+1,0) -- node {4} ++(-1,1); 
\draw[thick] (2,1) -- node[pos=\posa] {21} ++(0,1); \draw[thick] (2,1) -- node[pos=\posb] {1} ++(1,0); \draw[thick] (2+1,1) -- node {2} ++(-1,1); 
\draw[thick] (2,2) -- node[pos=\posa] {0} ++(0,1); \draw[thick] (2,2) -- node[pos=\posb] {0} ++(1,0); \draw[thick] (2+1,2) -- node {0} ++(-1,1); 
\draw[thick] (3,0) -- node[pos=\posa] {41} ++(0,1); \draw[thick] (3,0) -- node[pos=\posb] {1} ++(1,0); \draw[thick] (3+1,0) -- node {4} ++(-1,1); 
\draw[thick] (3,1) -- node[pos=\posa] {20} ++(0,1); \draw[thick] (3,1) -- node[pos=\posb] {0} ++(1,0); \draw[thick] (3+1,1) -- node {2} ++(-1,1); 
\draw[thick] (4,0) -- node[pos=\posa] {40} ++(0,1); \draw[thick] (4,0) -- node[pos=\posb] {0} ++(1,0); \draw[thick] (4+1,0) -- node {4} ++(-1,1); 
\end{tikzpicture}%
\begin{tikzpicture}[script math mode,nodes={edgelabel},x={(-0.577cm,-1cm)},y={(0.577cm,-1cm)},scale=\thescale]\useasboundingbox (0,0) -- (5+0.5,0) -- (0,5+0.5);
\draw[thick] (0,0) -- node[pos=\posa] {2} ++(0,1); \draw[thick] (0,0) -- node[pos=\posb] {2} ++(1,0); \draw[thick] (0+1,0) -- node {2} ++(-1,1); 
\draw[thick] (0,1) -- node[pos=\posa] {1} ++(0,1); \draw[thick] (0,1) -- node[pos=\posb] {1} ++(1,0); \draw[thick] (0+1,1) -- node {1} ++(-1,1); 
\draw[thick] (0,2) -- node[pos=\posa] {0} ++(0,1); \draw[thick] (0,2) -- node[pos=\posb] {0} ++(1,0); \draw[thick] (0+1,2) -- node {0} ++(-1,1); 
\draw[thick] (0,3) -- node[pos=\posa] {4} ++(0,1); \draw[thick] (0,3) -- node[pos=\posb] {40} ++(1,0); \draw[thick] (0+1,3) -- node {0} ++(-1,1); 
\draw[thick] (0,4) -- node[pos=\posa] {3} ++(0,1); \draw[thick] (0,4) -- node[pos=\posb] {30} ++(1,0); \draw[thick] (0+1,4) -- node {0} ++(-1,1); 
\draw[thick] (1,0) -- node[pos=\posa] {21} ++(0,1); \draw[thick] (1,0) -- node[pos=\posb] {3} ++(1,0); \draw[thick] (1+1,0) -- node {3(21)} ++(-1,1); 
\draw[thick] (1,1) -- node[pos=\posa] {10} ++(0,1); \draw[thick] (1,1) -- node[pos=\posb] {3} ++(1,0); \draw[thick] (1+1,1) -- node {3(10)} ++(-1,1); 
\draw[thick] (1,2) -- node[pos=\posa] {4} ++(0,1); \draw[thick] (1,2) -- node[pos=\posb] {4(3(10))} ++(1,0); \draw[thick] (1+1,2) -- node {3(10)} ++(-1,1); 
\draw[thick] (1,3) -- node[pos=\posa] {3} ++(0,1); \draw[thick] (1,3) -- node[pos=\posb] {3} ++(1,0); \draw[thick] (1+1,3) -- node {3} ++(-1,1); 
\draw[thick] (2,0) -- node[pos=\posa] {21} ++(0,1); \draw[thick] (2,0) -- node[pos=\posb] {4} ++(1,0); \draw[thick] (2+1,0) -- node {4(21)} ++(-1,1); 
\draw[thick] (2,1) -- node[pos=\posa] {4} ++(0,1); \draw[thick] (2,1) -- node[pos=\posb] {4} ++(1,0); \draw[thick] (2+1,1) -- node {4} ++(-1,1); 
\draw[thick] (2,2) -- node[pos=\posa] {10} ++(0,1); \draw[thick] (2,2) -- node[pos=\posb] {0} ++(1,0); \draw[thick] (2+1,2) -- node {1} ++(-1,1); 
\draw[thick] (3,0) -- node[pos=\posa] {21} ++(0,1); \draw[thick] (3,0) -- node[pos=\posb] {1} ++(1,0); \draw[thick] (3+1,0) -- node {2} ++(-1,1); 
\draw[thick] (3,1) -- node[pos=\posa] {40} ++(0,1); \draw[thick] (3,1) -- node[pos=\posb] {0} ++(1,0); \draw[thick] (3+1,1) -- node {4} ++(-1,1); 
\draw[thick] (4,0) -- node[pos=\posa] {20} ++(0,1); \draw[thick] (4,0) -- node[pos=\posb] {0} ++(1,0); \draw[thick] (4+1,0) -- node {2} ++(-1,1); 
\end{tikzpicture}%
\end{center}

On the other hand, there are $25$ puzzles with bottom boundary $41230$.
They have the common part
\begin{center}
\def\thescale{1.4}
\def\posa{0.42}\def\posb{0.58}
\begin{tikzpicture}[script math mode,nodes={edgelabel},x={(-0.577cm,-1cm)},y={(0.577cm,-1cm)},scale=\thescale]\useasboundingbox (0,0) -- (5+0.5,0) -- (0,5+0.5);
\draw[thick] (0,0) -- node[pos=\posa] {2} ++(0,1); \draw[thick] (0,0) -- node[pos=\posb] {2} ++(1,0); \draw[thick] (0+1,0) -- node {2} ++(-1,1); 
\draw[thick] (0,1) -- node[pos=\posa] {1} ++(0,1);
\draw[thick] (0,2) -- node[pos=\posa] {0} ++(0,1); \draw[thick] (0,2) -- node[pos=\posb] {0} ++(1,0); \draw[thick] (0+1,2) -- node {0} ++(-1,1); 
\draw[thick] (0,3) -- node[pos=\posa] {4} ++(0,1); \draw[thick] (0,3) -- node[pos=\posb] {40} ++(1,0); \draw[thick] (0+1,3) -- node {0} ++(-1,1); 
\draw[thick] (0,4) -- node[pos=\posa] {3} ++(0,1); \draw[thick] (0,4) -- node[pos=\posb] {30} ++(1,0); \draw[thick] (0+1,4) -- node {0} ++(-1,1); 
\draw[thick] (1,0) -- node[pos=\posb] {3} ++(1,0);
\draw[thick] (1,2) -- node[pos=\posa] {4} ++(0,1);
\draw[thick] (1,3) -- node[pos=\posa] {3} ++(0,1); \draw[thick] (1,3) -- node[pos=\posb] {3} ++(1,0); \draw[thick] (1+1,3) -- node {3} ++(-1,1); 
\draw[thick] (2,0) -- node[pos=\posa] {4} ++(0,1); \draw[thick] (2,0) -- node[pos=\posb] {4} ++(1,0); \draw[thick] (2+1,0) -- node {4} ++(-1,1); 
\draw[thick] (2,1) -- node[pos=\posb] {1} ++(1,0);
\draw[thick] (2+1,2) -- node {2} ++(-1,1); 
\draw[thick] (3,0) -- node[pos=\posa] {41} ++(0,1); \draw[thick] (3,0) -- node[pos=\posb] {1} ++(1,0); \draw[thick] (3+1,0) -- node {4} ++(-1,1); 
\draw[thick] (3,1) -- node[pos=\posa] {10} ++(0,1); \draw[thick] (3,1) -- node[pos=\posb] {0} ++(1,0); \draw[thick] (3+1,1) -- node {1} ++(-1,1); 
\draw[thick] (4,0) -- node[pos=\posa] {40} ++(0,1); \draw[thick] (4,0) -- node[pos=\posb] {0} ++(1,0); \draw[thick] (4+1,0) -- node {4} ++(-1,1); 
\draw[thick] (1+1,1) -- ++(-1,1); 
\end{tikzpicture}
\end{center}
where the central unlabeled edge has zero weight, i.e., one must sum over the 9-dimensional zero weight space.

Each of the hexagons have five possible fillings:
\begin{center}
\def\thescale{1.3}
\def\posa{0.42}\def\posb{0.58}
\begin{tikzpicture}[script math mode,nodes={edgelabel},x={(-0.577cm,-1cm)},y={(0.577cm,-1cm)},scale=\thescale]
\draw[thick] (0+1,0) -- node {2} ++(-1,1); 
\draw[thick] (0,1) -- node[pos=\posa] {1} ++(0,1); \draw[thick] (0,1) -- node[pos=\posb] {10} ++(1,0); \draw[thick] (0+1,1) -- node {0} ++(-1,1); 
\draw[thick] (0,2) -- node[pos=\posb] {0} ++(1,0);
\draw[thick] (1,0) -- node[pos=\posa] {2(10)} ++(0,1); \draw[thick] (1,0) -- node[pos=\posb] {3} ++(1,0); \draw[thick] (1+1,0) -- node {3(2(10))} ++(-1,1); 
\draw[thick] (1,1) -- node[pos=\posa] {0} ++(0,1); \draw[thick] (1,1) -- node[pos=\posb] {4(3(2(10)))} ++(1,0); \draw[thick] (1+1,1) -- node {} ++(-1,1); 
\draw[thick] (2,0) -- node[pos=\posa] {4} ++(0,1);
\end{tikzpicture}
\begin{tikzpicture}[script math mode,nodes={edgelabel},x={(-0.577cm,-1cm)},y={(0.577cm,-1cm)},scale=\thescale]
\draw[thick] (0+1,0) -- node {2} ++(-1,1); 
\draw[thick] (0,1) -- node[pos=\posa] {1} ++(0,1); \draw[thick] (0,1) -- node[pos=\posb] {1} ++(1,0); \draw[thick] (0+1,1) -- node {1} ++(-1,1); 
\draw[thick] (0,2) -- node[pos=\posb] {0} ++(1,0);
\draw[thick] (1,0) -- node[pos=\posa] {21} ++(0,1); \draw[thick] (1,0) -- node[pos=\posb] {3} ++(1,0); \draw[thick] (1+1,0) -- node {3(21)} ++(-1,1); 
\draw[thick] (1,1) -- node[pos=\posa] {10} ++(0,1); \draw[thick] (1,1) -- node[pos=\posb] {4(3(21))} ++(1,0); \draw[thick] (1+1,1) -- node {} ++(-1,1); 
\draw[thick] (2,0) -- node[pos=\posa] {4} ++(0,1);
\end{tikzpicture}
\begin{tikzpicture}[script math mode,nodes={edgelabel},x={(-0.577cm,-1cm)},y={(0.577cm,-1cm)},scale=\thescale]
\draw[thick] (0+1,0) -- node {2} ++(-1,1); 
\draw[thick] (0,1) -- node[pos=\posa] {1} ++(0,1); \draw[thick] (0,1) -- node[pos=\posb] {2} ++(1,0); \draw[thick] (0+1,1) -- node {21} ++(-1,1); 
\draw[thick] (0,2) -- node[pos=\posb] {0} ++(1,0);
\draw[thick] (1,0) -- node[pos=\posa] {2} ++(0,1); \draw[thick] (1,0) -- node[pos=\posb] {3} ++(1,0); \draw[thick] (1+1,0) -- node {32} ++(-1,1); 
\draw[thick] (1,1) -- node[pos=\posa] {(21)0} ++(0,1); \draw[thick] (1,1) -- node[pos=\posb] {4(32)} ++(1,0); \draw[thick] (1+1,1) -- node {} ++(-1,1); 
\draw[thick] (2,0) -- node[pos=\posa] {4} ++(0,1);
\end{tikzpicture}
\begin{tikzpicture}[script math mode,nodes={edgelabel},x={(-0.577cm,-1cm)},y={(0.577cm,-1cm)},scale=\thescale]
\draw[thick] (0+1,0) -- node {2} ++(-1,1); 
\draw[thick] (0,1) -- node[pos=\posa] {1} ++(0,1); \draw[thick] (0,1) -- node[pos=\posb] {32} ++(1,0); \draw[thick] (0+1,1) -- node {(32)1} ++(-1,1); 
\draw[thick] (0,2) -- node[pos=\posb] {0} ++(1,0);
\draw[thick] (1,0) -- node[pos=\posa] {3} ++(0,1); \draw[thick] (1,0) -- node[pos=\posb] {3} ++(1,0); \draw[thick] (1+1,0) -- node {3} ++(-1,1); 
\draw[thick] (1,1) -- node[pos=\posa] {((32)1)0} ++(0,1); \draw[thick] (1,1) -- node[pos=\posb] {43} ++(1,0); \draw[thick] (1+1,1) -- node {} ++(-1,1); 
\draw[thick] (2,0) -- node[pos=\posa] {4} ++(0,1);
\end{tikzpicture}
\begin{tikzpicture}[script math mode,nodes={edgelabel},x={(-0.577cm,-1cm)},y={(0.577cm,-1cm)},scale=\thescale]
\draw[thick] (0+1,0) -- node {2} ++(-1,1); 
\draw[thick] (0,1) -- node[pos=\posa] {1} ++(0,1); \draw[thick] (0,1) -- node[pos=\posb] {(43)2} ++(1,0); \draw[thick] (0+1,1) -- node {((43)2)1} ++(-1,1); 
\draw[thick] (0,2) -- node[pos=\posb] {0} ++(1,0);
\draw[thick] (1,0) -- node[pos=\posa] {43} ++(0,1); \draw[thick] (1,0) -- node[pos=\posb] {3} ++(1,0); \draw[thick] (1+1,0) -- node {4} ++(-1,1); 
\draw[thick] (1,1) -- node[pos=\posa] {(((43)2)1)0} ++(0,1); \draw[thick] (1,1) -- node[pos=\posb] {4} ++(1,0); \draw[thick] (1+1,1) -- node {} ++(-1,1); 
\draw[thick] (2,0) -- node[pos=\posa] {4} ++(0,1);
\end{tikzpicture}

\begin{tikzpicture}[script math mode,nodes={edgelabel},x={(-0.577cm,-1cm)},y={(0.577cm,-1cm)},scale=\thescale]
\draw[thick] (1,2) -- node[pos=\posa] {4} ++(0,1); \draw[thick] (1,2) -- node[pos=\posb] {4(3(2(10)))} ++(1,0); \draw[thick] (1+1,2) -- node {3(2(10))} ++(-1,1);
\draw[thick] (1,3) -- node[pos=\posb] {3} ++(1,0);
\draw[thick] (2,1) -- node[pos=\posa] {0} ++(0,1); \draw[thick] (2,1) -- node[pos=\posb] {1} ++(1,0); \draw[thick] (2+1,1) -- node {10} ++(-1,1);
\draw[thick] (2,2) -- node[pos=\posa] {2(10)} ++(0,1); \draw[thick] (2,2) -- node[pos=\posb] {10} ++(1,0); \draw[thick] (2+1,2) -- node {2} ++(-1,1);
\draw[thick] (3,1) -- node[pos=\posa] {10} ++(0,1);
\draw[thick] (1+1,1) -- ++(-1,1); 
\end{tikzpicture}
\begin{tikzpicture}[script math mode,nodes={edgelabel},x={(-0.577cm,-1cm)},y={(0.577cm,-1cm)},scale=\thescale]
\draw[thick] (1,2) -- node[pos=\posa] {4} ++(0,1); \draw[thick] (1,2) -- node[pos=\posb] {4(3(20))} ++(1,0); \draw[thick] (1+1,2) -- node {3(20)} ++(-1,1);
\draw[thick] (1,3) -- node[pos=\posb] {3} ++(1,0);
\draw[thick] (2,1) -- node[pos=\posa] {1} ++(0,1); \draw[thick] (2,1) -- node[pos=\posb] {1} ++(1,0); \draw[thick] (2+1,1) -- node {1} ++(-1,1); 
\draw[thick] (2,2) -- node[pos=\posa] {20} ++(0,1); \draw[thick] (2,2) -- node[pos=\posb] {0} ++(1,0); \draw[thick] (2+1,2) -- node {2} ++(-1,1); 
\draw[thick] (3,1) -- node[pos=\posa] {10} ++(0,1);
\draw[thick] (1+1,1) -- ++(-1,1); 
\end{tikzpicture}
\begin{tikzpicture}[script math mode,nodes={edgelabel},x={(-0.577cm,-1cm)},y={(0.577cm,-1cm)},scale=\thescale]
\draw[thick] (1,2) -- node[pos=\posa] {4} ++(0,1); \draw[thick] (1,2) -- node[pos=\posb] {4(32)} ++(1,0); \draw[thick] (1+1,2) -- node {32} ++(-1,1);
\draw[thick] (1,3) -- node[pos=\posb] {3} ++(1,0);
\draw[thick] (2,1) -- node[pos=\posa] {(21)0} ++(0,1); \draw[thick] (2,1) -- node[pos=\posb] {1} ++(1,0); \draw[thick] (2+1,1) -- node {2(10)} ++(-1,1); 
\draw[thick] (2,2) -- node[pos=\posa] {2} ++(0,1); \draw[thick] (2,2) -- node[pos=\posb] {2} ++(1,0); \draw[thick] (2+1,2) -- node {2} ++(-1,1); 
\draw[thick] (3,1) -- node[pos=\posa] {10} ++(0,1);
\draw[thick] (1+1,1) -- ++(-1,1); 
\end{tikzpicture}
\begin{tikzpicture}[script math mode,nodes={edgelabel},x={(-0.577cm,-1cm)},y={(0.577cm,-1cm)},scale=\thescale]
\draw[thick] (1,2) -- node[pos=\posa] {4} ++(0,1); \draw[thick] (1,2) -- node[pos=\posb] {43} ++(1,0); \draw[thick] (1+1,2) -- node {3} ++(-1,1);
\draw[thick] (1,3) -- node[pos=\posb] {3} ++(1,0);
\draw[thick] (2,1) -- node[pos=\posa] {((32)1)0} ++(0,1); \draw[thick] (2,1) -- node[pos=\posb] {1} ++(1,0); \draw[thick] (2+1,1) -- node {(32)(10)} ++(-1,1); 
\draw[thick] (2,2) -- node[pos=\posa] {3} ++(0,1); \draw[thick] (2,2) -- node[pos=\posb] {32} ++(1,0); \draw[thick] (2+1,2) -- node {2} ++(-1,1); 
\draw[thick] (3,1) -- node[pos=\posa] {10} ++(0,1);
\draw[thick] (1+1,1) -- ++(-1,1); 
\end{tikzpicture}
\begin{tikzpicture}[script math mode,nodes={edgelabel},x={(-0.577cm,-1cm)},y={(0.577cm,-1cm)},scale=\thescale]
\draw[thick] (1,2) -- node[pos=\posa] {4} ++(0,1); \draw[thick] (1,2) -- node[pos=\posb] {4} ++(1,0); \draw[thick] (1+1,2) -- node {4} ++(-1,1); 
\draw[thick] (1,3) -- node[pos=\posb] {3} ++(1,0);
\draw[thick] (2,1) -- node[pos=\posa] {(((43)2)1)0} ++(0,1); \draw[thick] (2,1) -- node[pos=\posb] {1} ++(1,0); \draw[thick] (2+1,1) -- node {((43)2)(10)} ++(-1,1); 
\draw[thick] (2,2) -- node[pos=\posa] {43} ++(0,1); \draw[thick] (2,2) -- node[pos=\posb] {(43)2} ++(1,0); \draw[thick] (2+1,2) -- node {2} ++(-1,1); 
\draw[thick] (3,1) -- node[pos=\posa] {10} ++(0,1);
\draw[thick] (1+1,1) -- ++(-1,1); 
\end{tikzpicture}
\end{center}

The fugacities depend on the scalar product of the edges surrounding the zero weight edge. With the same ordering
as above, we find:
\[
\frac{1}{5}
\begin{bmatrix}
 -4 & 1 & 1 & 1 & 1 \\
 1 & 1 & 1 & 1 & 1 \\
 1 & 1 & -4 & 1 & 1 \\
 1 & 1 & 1 & -4 & 1 \\
 1 & 1 & 1 & 1 & -4
\end{bmatrix}
\]
which sums up to $1$, as expected.

\subsection{Violation of inversion number inequalities}\label{app:exd4c}
Similarly, the 25 puzzles that contribute to the coefficient of $S^{41230}$ in the product $S^{10432}S^{21043}$ in $H^{*\loc}_{\CC^\times}$ have the common part:
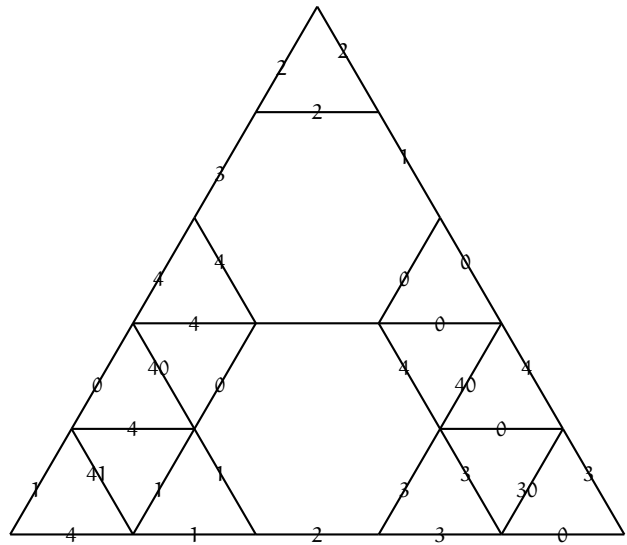
\begin{center}
\def\thescale{1.4}
\def\posa{0.42}\def\posb{0.58}
\begin{tikzpicture}[script math mode,nodes={edgelabel},x={(-0.577cm,-1cm)},y={(0.577cm,-1cm)},scale=\thescale]\useasboundingbox (0,0) -- (5+0.5,0) -- (0,5+0.5);
\draw[thick] (0,0) -- node[pos=\posa] {2} ++(0,1); \draw[thick] (0,0) -- node[pos=\posb] {2} ++(1,0); \draw[thick] (0+1,0) -- node {2} ++(-1,1); 
\draw[thick] (0,1) -- node[pos=\posa] {1} ++(0,1);
\draw[thick] (0,2) -- node[pos=\posa] {0} ++(0,1); \draw[thick] (0,2) -- node[pos=\posb] {0} ++(1,0); \draw[thick] (0+1,2) -- node {0} ++(-1,1); 
\draw[thick] (0,3) -- node[pos=\posa] {4} ++(0,1); \draw[thick] (0,3) -- node[pos=\posb] {40} ++(1,0); \draw[thick] (0+1,3) -- node {0} ++(-1,1); 
\draw[thick] (0,4) -- node[pos=\posa] {3} ++(0,1); \draw[thick] (0,4) -- node[pos=\posb] {30} ++(1,0); \draw[thick] (0+1,4) -- node {0} ++(-1,1); 
\draw[thick] (1,0) -- node[pos=\posb] {3} ++(1,0);
\draw[thick] (1,2) -- node[pos=\posa] {4} ++(0,1);
\draw[thick] (1,3) -- node[pos=\posa] {3} ++(0,1); \draw[thick] (1,3) -- node[pos=\posb] {3} ++(1,0); \draw[thick] (1+1,3) -- node {3} ++(-1,1); 
\draw[thick] (2,0) -- node[pos=\posa] {4} ++(0,1); \draw[thick] (2,0) -- node[pos=\posb] {4} ++(1,0); \draw[thick] (2+1,0) -- node {4} ++(-1,1); 
\draw[thick] (2,1) -- node[pos=\posb] {0} ++(1,0);
\draw[thick] (2+1,2) -- node {2} ++(-1,1); 
\draw[thick] (3,0) -- node[pos=\posa] {40} ++(0,1); \draw[thick] (3,0) -- node[pos=\posb] {0} ++(1,0); \draw[thick] (3+1,0) -- node {4} ++(-1,1); 
\draw[thick] (3,1) -- node[pos=\posa] {1} ++(0,1); \draw[thick] (3,1) -- node[pos=\posb] {1} ++(1,0); \draw[thick] (3+1,1) -- node {1} ++(-1,1); 
\draw[thick] (4,0) -- node[pos=\posa] {41} ++(0,1); \draw[thick] (4,0) -- node[pos=\posb] {1} ++(1,0); \draw[thick] (4+1,0) -- node {4} ++(-1,1); 
\draw[thick] (1+1,1) -- ++(-1,1); 
\end{tikzpicture}
\end{center}

The hexagons have the same possible fillings as in previous section, up to 180 degree rotation for the bottom part.

This time the table of fugacities
\[
\frac{1}{5}
\begin{bmatrix}
 -4 & 1 & 1 & 1 & 1 \\
 1 & -4 & 1 & 1 & 1 \\
 1 & 1 & -4 & 1 & 1 \\
 1 & 1 & 1 & -4 & 1 \\
 1 & 1 & 1 & 1 & -4
\end{bmatrix}
\]
sums up to $0$,
consistent with the fact that
$\ell(10432)+\ell(21043)=4+4=8>7=\ell(41230)$.

\section{Scalar products}\label{app:scal}
We provide here the scalar products (for $d\le 3$) of weights given in terms of multinumbers. More precisely,
given two multinumbers $X$ and $Y$, the table below provides the scalar products
$\killing{\vf_X}{\tau \vf_Y}$ corresponding to any of the following configurations:
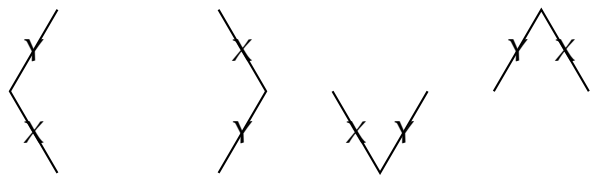
\[
\tikz[baseline=0,scale=1.25]{
      \draw[thick,black] (-0.5,0) ++(60:-1) -- node[edgelabel] {$X$} ++(-60:-1) -- node[edgelabel] {$Y$} ++(60:1);
}
\qquad
\tikz[baseline=0,scale=-1.25]{
      \draw[thick,black] (-0.5,0) ++(60:-1) -- node[edgelabel] {$X$} ++(-60:-1) -- node[edgelabel] {$Y$} ++(60:1);
}
\qquad
\tikz[baseline=0,scale=1.25]{
      \draw[thick,black] (-0.5,0) -- node[edgelabel] {$X$} ++(-60:1) -- node[edgelabel] {$Y$} ++(60:1);
}
\qquad
\tikz[baseline=0,scale=-1.25]{
      \draw[thick,black] (-0.5,0) -- node[edgelabel] {$X$} ++(-60:1) -- node[edgelabel] {$Y$} ++(60:1);
}
\]
In the notations of \S\ref{sec:coho}, the first (resp.\ last) two compute the scalar product $s$ (resp.\ $t$).

As in \S\ref{sec:coho}, we parametrize the scalar products as $-1+ka$ where
$a=3,2,3/2$ for $d=1,2,3$; and only indicate $k$ in the table. This way, the tables for $d=1,2$ are subtables
of the $d=3$ table shown below.

\junk{At $d=1$,
\[
\begin{tikzpicture}
\matrix[matrix of math nodes,column sep={28pt,between origins},row sep={28pt,between origins}] {
X\backslash Y &0&1&10\\
0&-1&-1&2\\
1&2&-1&-1\\
10&-1&2&-1\\
};
\end{tikzpicture}
\]

At $d=2$,
\[
\begin{tikzpicture}
\matrix[matrix of math nodes,column sep={28pt,between origins},row sep={28pt,between origins}] {
X\backslash Y  & 0 & 1 & 2 & 10 & 20 & 21 & 2(10) & (21)0 \\
 0 & -1 & -1 & -1 & 1 & 1 & -1 & 1 & 1 \\
 1 & 1 & -1 & -1 & -1 & 1 & 1 & -1 & 1 \\
 2 & 1 & 1 & -1 & 1 & -1 & -1 & -1 & 1 \\
 10 & -1 & 1 & -1 & -1 & -1 & 1 & 1 & 1 \\
 20 & -1 & 1 & 1 & 1 & -1 & -1 & 1 & -1 \\
 21 & 1 & -1 & 1 & 1 & 1 & -1 & -1 & -1 \\
 2(10) & 1 & 1 & 1 & -1 & -1 & 1 & -1 & -1 \\
 (21)0 & -1 & -1 & 1 & -1 & 1 & 1 & 1 & -1 \\
};
\end{tikzpicture}
\]

\vfill\eject
}
\newcommand\minusone{0}
\newcommand\oh{1}
\newcommand\two{2}
\hspace{-1.8cm}
\begin{tikzpicture}
\matrix[matrix of math nodes,column sep={18pt,between origins},row sep={18pt,between origins}] {
\scriptscriptstyle X\backslash Y  &\scriptscriptstyle 0 &\scriptscriptstyle 1 &\scriptscriptstyle 2 &\scriptscriptstyle 3 &\scriptscriptstyle 10 &\scriptscriptstyle 20 &\scriptscriptstyle 21 &\scriptscriptstyle 30 &\scriptscriptstyle 31 &\scriptscriptstyle 32 &\scriptscriptstyle 2(10) &\scriptscriptstyle 3(10) &\scriptscriptstyle 3(20) &\scriptscriptstyle (21)0 &\scriptscriptstyle 3(21) &\scriptscriptstyle (31)0 &\scriptscriptstyle (32)0
   &\scriptscriptstyle (32)1 &|[yshift=12pt]|\scriptscriptstyle (32)(10) &\scriptscriptstyle 3(2(10)) &|[yshift=12pt]|\scriptscriptstyle 3((21)0) &\scriptscriptstyle (32)((21)0) &|[yshift=12pt]|\scriptscriptstyle (3(21))0 &\scriptscriptstyle (3(21))(10) &|[yshift=12pt]|\scriptscriptstyle ((32)1)0 &\scriptscriptstyle
   (3(2(10)))0 &|[yshift=12pt]|\scriptscriptstyle 3(((32)1)0) \\|[xshift=-5pt]|\scriptscriptstyle
 0 &\minusone &\minusone &\minusone &\minusone &\oh &\oh &\minusone &\oh &\minusone &\minusone &\oh &\oh &\oh &\oh &\minusone &\oh &
   \oh &\minusone &\oh &\minusone &\oh &\oh &\oh &\oh &\oh &\two &\oh \\|[xshift=-5pt]|\scriptscriptstyle
 1 &\oh &\minusone &\minusone &\minusone &\minusone &\oh &\oh &\oh &\oh &\minusone &\minusone &\minusone &\minusone &\oh &\oh &\oh &
   \oh &\oh &\minusone &\minusone &\oh &\oh &\two &\oh &\oh &\oh &\oh \\|[xshift=-5pt]|\scriptscriptstyle
 2 &\oh &\oh &\minusone &\minusone &\oh &\minusone &\minusone &\oh &\oh &\oh &\minusone &\oh &\minusone &\oh &\minusone &\two &
   \oh &\oh &\oh &\minusone &\minusone &\minusone &\oh &\oh &\oh &\oh &\oh \\|[xshift=-5pt]|\scriptscriptstyle
 3 &\oh &\oh &\oh &\minusone &\oh &\oh &\oh &\minusone &\minusone &\minusone &\oh &\minusone &\minusone &\oh &\minusone &
   \oh &\oh &\oh &\oh &\minusone &\minusone &\oh &\oh &\oh &\two &\oh &\minusone \\|[xshift=-5pt]|\scriptscriptstyle
 10 &\minusone &\oh &\minusone &\minusone &\minusone &\minusone &\oh &\minusone &\oh &\minusone &\oh &\oh &\minusone &\oh &\minusone &\oh &\minusone &
   \oh &\oh &\oh &\oh &\oh &\oh &\two &\oh &\oh &\oh \\|[xshift=-5pt]|\scriptscriptstyle
 20 &\minusone &\oh &\oh &\minusone &\oh &\minusone &\minusone &\minusone &\minusone &\oh &\oh &\oh &\oh &\minusone &\minusone &\oh &
   \oh &\oh &\two &\oh &\minusone &\oh &\minusone &\oh &\oh &\oh &\oh \\|[xshift=-5pt]|\scriptscriptstyle
 21 &\oh &\minusone &\oh &\minusone &\oh &\oh &\minusone &\oh &\minusone &\oh &\minusone &\minusone &\oh &\minusone &\oh &
   \oh &\two &\oh &\oh &\minusone &\minusone &\oh &\oh &\minusone &\oh &\oh &\oh \\|[xshift=-5pt]|\scriptscriptstyle
 30 &\minusone &\oh &\oh &\oh &\oh &\oh &\oh &\minusone &\minusone &\minusone &\two &\oh &\oh &\oh &\minusone &
   \minusone &\minusone &\minusone &\oh &\oh &\oh &\oh &\minusone &\oh &\oh &\oh &\minusone \\|[xshift=-5pt]|\scriptscriptstyle
 31 &\oh &\minusone &\oh &\oh &\oh &\two &\oh &\oh &\minusone &\minusone &\oh &\minusone &\oh &\oh &
   \oh &\minusone &\oh &\minusone &\minusone &\minusone &\oh &\oh &\oh &\minusone &\oh &\oh &\minusone \\|[xshift=-5pt]|\scriptscriptstyle
 32 &\oh &\oh &\minusone &\oh &\oh &\oh &\oh &\oh &\oh &\minusone &\oh &\oh &\minusone &
   \two &\minusone &\oh &\minusone &\minusone &\minusone &\minusone &\oh &\minusone &\oh &\oh &\oh &\oh &\minusone \\|[xshift=-5pt]|\scriptscriptstyle
 2(10) &\oh &\oh &\oh &\minusone &\minusone &\minusone &\oh &\minusone &\oh &\oh &\minusone &\minusone &\minusone &\minusone &\oh &
   \oh &\oh &\two &\oh &\oh &\minusone &\oh &\oh &\oh &\oh &\minusone &\oh \\|[xshift=-5pt]|\scriptscriptstyle
 3(10) &\oh &\oh &\oh &\oh &\minusone &\oh &\two &\minusone &\oh &\minusone &\oh &\minusone &\minusone &\oh &
   \oh &\minusone &\minusone &\oh &\minusone &\oh &\oh &\oh &\oh &\oh &\oh &\minusone &\minusone \\|[xshift=-5pt]|\scriptscriptstyle
 3(20) &\oh &\two &\oh &\oh &\oh &\minusone &\oh &\minusone &\oh &\oh &\oh &\oh &\minusone &
   \oh &\minusone &\oh &\minusone &\oh &\oh &\oh &\minusone &\minusone &\minusone &\oh &\oh &\minusone &\minusone \\|[xshift=-5pt]|\scriptscriptstyle
 (21)0 &\minusone &\minusone &\oh &\minusone &\minusone &\oh &\oh &\minusone &\minusone &\minusone &\oh &\minusone &\oh &\minusone &\oh &\minusone &
   \oh &\oh &\oh &\oh &\oh &\two &\oh &\oh &\oh &\oh &\oh \\|[xshift=-5pt]|\scriptscriptstyle
 3(21) &\oh &\oh &\oh &\oh &\two &\oh &\minusone &\oh &\minusone &\oh &\oh &\oh &
   \oh &\oh &\minusone &\oh &\oh &\minusone &\oh &\minusone &\minusone &\minusone &\minusone &\minusone &\oh &\oh &\minusone \\|[xshift=-5pt]|\scriptscriptstyle
 (31)0 &\minusone &\minusone &\minusone &\oh &\minusone &\oh &\oh &\oh &\oh &\minusone &\oh &\oh &\oh &
   \oh &\oh &\minusone &\minusone &\minusone &\minusone &\oh &\two &\oh &\oh &\oh &\minusone &\oh &\oh \\|[xshift=-5pt]|\scriptscriptstyle
 (32)0 &\minusone &\oh &\minusone &\oh &\oh &\minusone &\minusone &\oh &\oh &\oh &\oh &\two &\oh &
   \oh &\minusone &\oh &\minusone &\minusone &\oh &\oh &\oh &\minusone &\minusone &\oh &\minusone &\oh &\oh \\|[xshift=-5pt]|\scriptscriptstyle
 (32)1 &\oh &\minusone &\minusone &\oh &\oh &\oh &\minusone &\two &\oh &\oh &\minusone &\oh &\oh &
   \oh &\oh &\oh &\oh &\minusone &\minusone &\minusone &\oh &\minusone &\oh &\minusone &\minusone &\oh &\oh \\|[xshift=-5pt]|\scriptscriptstyle
 (32)(10) &\oh &\oh &\minusone &\oh &\minusone &\minusone &\oh &\oh &\two &\oh &\minusone &\oh &\minusone &\oh &
   \oh &\oh &\minusone &\oh &\minusone &\oh &\oh &\minusone &\oh &\oh &\minusone &\minusone &\oh \\|[xshift=-5pt]|\scriptscriptstyle
 3(2(10)) &\two &\oh &\oh &\oh &\oh &\oh &\oh &\oh &\oh &\oh &\minusone &\minusone &
   \minusone &\oh &\oh &\oh &\oh &\oh &\minusone &\minusone &\minusone &\minusone &\oh &\minusone &\oh &\minusone &\minusone \\|[xshift=-5pt]|\scriptscriptstyle
 3((21)0) &\oh &\oh &\two &\oh &\oh &\oh &\oh &\minusone &\minusone &\oh &\oh &\minusone &
   \oh &\minusone &\oh &\minusone &\oh &\oh &\oh &\oh &\minusone &\oh &\minusone &\minusone &\oh &\minusone &\minusone \\|[xshift=-5pt]|\scriptscriptstyle
 (32)((21)0) &\oh &\oh &\oh &\oh &\oh &\minusone &\minusone &\oh &\oh &\two &\minusone &\oh &
   \oh &\minusone &\oh &\oh &\oh &\oh &\oh &\oh &\minusone &\minusone &\minusone &\minusone &\minusone &\minusone &\oh \\|[xshift=-5pt]|\scriptscriptstyle
 (3(21))0 &\minusone &\minusone &\oh &\oh &\oh &\oh &\minusone &\oh &\minusone &\oh &\oh &\oh &\two &\minusone &
   \oh &\minusone &\oh &\minusone &\oh &\oh &\oh &\oh &\minusone &\minusone &\minusone &\oh &\oh \\|[xshift=-5pt]|\scriptscriptstyle
 (3(21))(10) &\oh &\minusone &\oh &\oh &\minusone &\oh &\oh &\oh &\oh &\oh &\minusone &\minusone &
   \oh &\minusone &\two &\minusone &\oh &\oh &\minusone &\oh &\oh &\oh &\oh &\minusone &\minusone &\minusone &\oh \\|[xshift=-5pt]|\scriptscriptstyle
 ((32)1)0 &\minusone &\minusone &\minusone &\minusone &\minusone &\minusone &\minusone &\oh &\oh &\oh &\minusone &\oh &\oh &\minusone &\oh &\oh
   &\oh &\oh &\oh &\oh &\oh &\oh &\oh &\oh &\minusone &\oh &\two \\|[xshift=-5pt]|\scriptscriptstyle
 (3(2(10)))0 &\minusone &\oh &\oh &\oh &\minusone &\minusone &\oh &\minusone &\oh &\oh &\oh &\oh &
   \oh &\minusone &\oh &\minusone &\minusone &\oh &\oh &\two &\oh &\oh &\minusone &\oh &\minusone &\minusone &\oh \\|[xshift=-5pt]|\scriptscriptstyle
 3(((32)1)0) &\oh &\oh &\oh &\two &\oh &\oh &\oh &\oh &\oh &\oh &
   \oh &\oh &\oh &\oh &\oh &\minusone &\minusone &\minusone &\minusone &\oh &\oh &\minusone &\minusone &\minusone &\minusone &\minusone &\minusone \\
};
\end{tikzpicture}

\vfill\eject

\gdef\MRshorten#1 #2MRend{#1}%
\gdef\MRfirsttwo#1#2{\if#1M%
MR\else MR#1#2\fi}
\def\MRfix#1{\MRshorten\MRfirsttwo#1 MRend}
\renewcommand\MR[1]{\relax\ifhmode\unskip\spacefactor3000 \space\fi
\MRhref{\MRfix{#1}}{{\scriptsize \MRfix{#1}}}}
\renewcommand{\MRhref}[2]{%
\href{http://www.ams.org/mathscinet-getitem?mr=#1}{#2}}
\bibliographystyle{amsalphahyper}
\bibliography{biblio}

\end{document}